\documentclass[11pt]{article}
\usepackage{fullpage}

\usepackage[T1]{fontenc}
\usepackage[english]{babel}

\usepackage{amsmath,amssymb,amsthm,mathtools,esint}
\usepackage{tikz-cd}
\usetikzlibrary{decorations.markings}
\usetikzlibrary{calc,intersections,positioning}
\usetikzlibrary{angles,quotes}
\usetikzlibrary{fit}
\usetikzlibrary{shapes.misc}

\usepackage{graphicx,subcaption}
\usepackage{enumerate}
\usepackage{dsfont}
\usepackage{verbatim}
\usepackage{hyperref}
\usepackage{bm}
\usepackage{blkarray}
\usepackage{mathrsfs}
\usepackage{subcaption}
\usetikzlibrary{arrows.meta}
\usepackage[normalem]{ulem}

\usepackage[dvipsnames]{xcolor}

\theoremstyle{plain}
\newtheorem{theorem}{Theorem}[section] 
\newtheorem{lemma}[theorem]{Lemma}
\newtheorem{proposition}[theorem]{Proposition}
\newtheorem{corollary}[theorem]{Corollary}

\newtheorem{assumption}[theorem]{Assumption}

\theoremstyle{definition}
\newtheorem{definition}[theorem]{Definition}

\newtheorem{example}[theorem]{Example}
\newtheorem{remark}[theorem]{Remark}
\numberwithin{equation}{section}

\mathtoolsset{showonlyrefs}

\newcommand{\ZZ}{\mathbb{Z}}
\newcommand{\RR}{\mathbb{R}}
\newcommand{\CC}{\mathbb{C}}

\newcommand{\Ordo}{O}
\newcommand{\EE}{\mathbb{E}}
\newcommand{\PP}{\mathbb{P}}

\renewcommand{\i}{\mathrm{i}}
\newcommand{\e}{\mathrm{e}}

\def\bl{\mathsf b}
\def\wh{\mathsf w}
\def\f{\mathsf f}
\def\e{\mathsf e}
\def\zz{\bm \alpha}
\def\angle{\nu}
\def\dabel{\bm d}
\def\sign{\operatorname{sign}}
\def\col{\varepsilon}
\def\x{\mathsf x}
\def\y{\mathsf y}
\def\wt{\mathsf{wt}}
\def\gauge{\lambda}

\def\line{\bm l}
\def\point{\bm p}
\def\SC{\operatorname{SC}}

\def\torus{\mathbb T}
\def\graphtor{\mathsf{G}_\torus}
\def\graphpl{{\mathsf{G}}}
\def\graphbeta{\mathsf G_{c, \f_0}}
\def\graph{\mathsf G}
\def\graphbeta{\mathsf G_{\bm \beta}}
\def\graphbetan{\mathsf G_{\bm \beta^n}}
\def\Rin{R^{\mathrm{in}}}
\def\graphregion{\mathsf{R}}
\def\sector{\mathfrak{N}}

\def\cR{R}
\def\dR{\mathsf R}
\def\compact{\mathsf C}
\def\domain{\mathcal D_c}
\def\arctic{\mathscr{A}}
\def\ls{\mathfrak h}
\def\sector{\mathfrak{N}}
\def\zero{\zeta_*}

\def\kast{\mathsf K}

\def\curve{\Sigma}
\def\jac{\operatorname{Jac}}
\def\liq{\mathcal L}

\DeclareMathOperator{\ord}{ord}
\DeclareMathOperator{\res}{res}
\DeclareMathOperator{\divisor}{div}

\DeclareMathOperator{\one}{\mathbf{1}}

\newcommand{\ii}{\mathrm{i}}

\tikzset{nvert/.style={draw,circle,fill=black,minimum size=3pt,inner sep=0pt}  }  
\tikzset{redvert/.style={draw = Orange, circle, fill=Orange, minimum size=3pt, inner sep=0pt}  } 
\tikzset{greenvert/.style={draw = black ,circle,fill=black,minimum size=3pt,inner sep=0pt}  } 
\tikzset{Greenvert/.style={draw = Green ,circle,fill=Green,minimum size=3pt,inner sep=0pt}  } 
\tikzset{bluevert/.style={draw = black,circle,fill=black,minimum size=3pt,inner sep=0pt}  }

\tikzset{bvert/.style={draw,circle,fill=black,minimum size=5pt,inner sep=0pt}  }  
\tikzset{wvert/.style={draw,circle,fill=white,minimum size=5pt,inner sep=0pt}  } 

\usetikzlibrary{arrows.meta}
\tikzset{>={Latex}}
\tikzset{
  hollowedge/.style={
    draw=black,
    line cap=round,
    line join=round,
    line width=0.5pt,
    double=white,
    double distance=1.5pt,
    -{Latex[length=6pt,width=7pt, fill=white]}
  },
  solidedge/.style={
    draw=black,
    line cap=round,
    line join=round,
    line width=2pt, 
    -{Latex[length=6pt,width=7pt]}
  }
}
\newcommand{\midp}[3]{\coordinate (#1) at ($(#2)!0.5!(#3)$);}

\tikzset{
  pics/hexpiece/.style={
    code={
      \coordinate[wvert] (w0) at (0,0);
      \coordinate[wvert] (w1) at (30:1);
      \coordinate[wvert] (w2) at (90:1);
      \coordinate[wvert] (w3) at (150:1);
      \coordinate[wvert] (w4) at (210:1);
      \coordinate[wvert] (w5) at (270:1);
      \coordinate[wvert] (w6) at (330:1);

      \coordinate[bvert] (b12) at ($(w1)!0.5!(w2)$);
      \coordinate[bvert] (b23) at ($(w2)!0.5!(w3)$);
      \coordinate[bvert] (b34) at ($(w3)!0.5!(w4)$);
      \coordinate[bvert] (b45) at ($(w4)!0.5!(w5)$);
      \coordinate[bvert] (b56) at ($(w5)!0.5!(w6)$);
      \coordinate[bvert] (b61) at ($(w6)!0.5!(w1)$);

      \draw (w0) -- (b23);
      \draw (w0) -- (b34);
      \draw (w0) -- (b45);
      \draw (w0) -- (b56);
      \draw (w0) -- (b61);

      \draw (b12) -- (w1) (b12) -- (w2);
      \draw (b23) -- (w2) (b23) -- (w3);
      \draw (b34) -- (w3) (b34) -- (w4);
      \draw (b45) -- (w4) (b45) -- (w5);
      \draw (b56) -- (w5) (b56) -- (w6);
      \draw (b61) -- (w6) (b61) -- (w1);
    }
  }
}

\tikzcdset{arrow style=tikz, diagrams={>={Stealth[round,length=4pt,width=4.95pt,inset=2.75pt]}}}
\tikzset{nvert/.style={draw,circle,fill=black,minimum size=3pt,inner sep=0pt}  } 
\usetikzlibrary{decorations.markings}

\newcommand\restr[2]{{
  \left.\kern-\nulldelimiterspace 
  #1 
  \vphantom{\big|} 
  \right|_{#2} 
  }}

\graphicspath{ {Images/} }

\title{Dimer models on astroidal zig-zag graphs}

\author{Tomas Berggren\footnote{Department of Mathematics \& Statistics, University of South Florida, USA. E-mail: tberggren@usf.edu} 
\and Alexei Borodin\footnote{Department of Mathematics, Massachusetts Institute of Technology, USA. E-mail: borodin@math.mit.edu}
\and Terrence George\footnote{Department of Mathematics, Massachusetts Institute of Technology, USA. E-mail: tegeorge@mit.edu}
}

\begin{document}

\maketitle

\begin{abstract} 
On a finite weighted graph, the dimer model is a probability
measure on its dimer covers, also known as perfect matchings, that
assigns to any cover a probability proportional to the product of
the weights of its edges. For planar bipartite graphs,
dimer correlations are encoded by the inverse of the so-called
Kasteleyn matrix, which is a version of the adjacency matrix for the
graph. 

The inverse Kasteleyn matrix for a large graph, typically
taken as a finite domain in a periodic graph, is known explicitly only 
for a handful of examples. In all previously known examples, the Newton 
polygon -- a convex lattice polygon that
classifies periodic graphs up to local moves -- is either a triangle or a quadrilateral.

Our main results are the following.

For any (minimal) periodic planar bipartite graph, we construct an
$(n-3)$-dimensional family of finite subgraphs for which we obtain an
explicit inverse Kasteleyn matrix; here $n$ is the number of sides of
the Newton polygon. Their boundaries are formed by zig-zag
paths and their overall shape is reminiscent of an astroid; we call
them astroidal zig-zag graphs (AZ graphs).
If the Newton polygon is the unit square then the corresponding AZ graph 
is the celebrated Aztec diamond with its size as the parameter. 

Our inverse Kasteleyn matrices are given by a double
contour integral on the corresponding spectral curve for any Fock weighting of the graph. 
This includes, in particular, all periodic weightings.

For periodic weightings, we asymptotically analyze the resulting
inverse Kasteleyn matrices. We establish a phase separation 
in large AZ graphs into asymptotically frozen, rough (liquid), and smooth (gaseous)
regions, and obtain an explicit parametrization of the `arctic curve'.

We also compute the deterministic limit of the height function,
known as the limit shape, and prove the convergence of the
local dimer correlations to the translation-invariant Gibbs
measure of the slope predicted by the limit shape.
\end{abstract}

\tableofcontents

\section{Introduction}

The dimer model is a statistical mechanical model that studies random perfect matchings or dimer covers of a large graph. Via Thurston's height function~\cite{Thu90}, random perfect matchings on planar bipartite graphs may be viewed as random discrete surfaces with prescribed boundary conditions. Classical examples include lozenge tilings and domino tilings, which have been studied extensively from the points of view of combinatorics, probability and integrable systems; see Kenyon's lectures~\cite{Ken09} and Gorin's book~\cite{Gor21} for an introduction.

A fundamental result in the subject is the variational principle, proved for domino tilings by Cohn--Kenyon--Propp~\cite{CKP00} and generalized by Kenyon--Okounkov--Sheffield~\cite{KOS06} and Kuchumov~\cite{Kuc17}. For a periodic planar bipartite graph with periodic edge weights, and for suitable exhaustions by finite subgraphs, the associated random height functions concentrate, after scaling, near a deterministic \emph{limit shape} characterized as the minimizer of an explicit surface tension functional. {This can also be viewed as a variant of the Wulff construction~\cite{Wulff}; see also Nienhuis~\cite{Nie84}, Dobrushin--Koteck{\'y}--Shlosman~\cite{DKS92}, and Cerf--Kenyon~\cite{CerfKenyon}.}

The limit shape is highly sensitive to the choice of finite subgraphs, or, equivalently, on the boundary conditions. For example, the limit shape of a growing sequence of rectangles on the square lattice is constant, whereas for sequences of Aztec diamonds it develops distinct phases separated by a curve known as the \emph{arctic curve}, as shown by Jockusch--Propp--Shor~\cite{JPS98}. This raises a natural problem: how does one systematically choose sequences of finite graphs that exhibit nontrivial phase transitions?

The general variational theory is not accompanied by an equally general supply of concrete models with this behavior. Outside a small number of classical examples---lozenge tilings, domino tilings, and the Schur process/rail-yard graph models of Okounkov--Reshetikhin~\cite{OR03} and Boutillier--Bouttier--Chapuy--Corteel--Ramassamy~\cite{BBCCR17}---there are no theoretical results or any simulations. The first main contribution of this paper is to identify natural finite subgraphs called \emph{astroidal zig-zag graphs} (or \emph{AZ graphs} for short) of an arbitrary (minimal) periodic planar bipartite graph exhibiting nontrivial phase transitions.

Once the limit shape has been identified, the next step is to study local and global fluctuations around it. 
{In the smooth part of the limit shape,}
it is conjectured---and proved for the hexagonal lattice with uniform weights by Aggarwal~\cite{Agg23} ({see also~Gorin~\cite{Gorin_bulk}})---that the local statistics are governed by the translation-invariant ergodic Gibbs measure of Kenyon--Okounkov--Sheffield~\cite{KOS06} whose slope matches that of the limit shape. It is also of considerable interest to study the local fluctuations near non-smooth points, such as regular points of the arctic curve as in Johansson~\cite{Joh05a}, Chhita--Dauvergne--Finn~\cite{CDF26}, Huang~\cite{Huang} and Aggarwal--Huang~\cite{AggarwalHuang}; turning points as in Johansson--Nordenstam~\cite{JN06}, Okounkov--Reshetikhin~\cite{OR06}, Aggarwal--Gorin~\cite{AG21} and Berggren--Bradinoff~\cite{BB26}; cusps as in Tracy--Widom~\cite{TracyWidomPearcey}, Okounkov--Reshetikhin~\cite{OR07}, Huang--Yang--Zhang~\cite{HYZ}; and at other special points as discussed in Astala--Duse--Prause--Zhong~\cite{ADPZ20}. Likewise, understanding the global fluctuations around the limit shape remains an important open problem; see Kenyon~\cite{Ken99} and {Kenyon--Okounkov~\cite{KO07}}. Kasteleyn theory~\cite{Kas61,TF61,Kas63} provides a powerful tool for tackling these problems through the inverse Kasteleyn matrix, which, by a result of Kenyon~\cite{Ken97}, encodes all probabilistic information about the dimer model.

Explicit formulas for the inverse Kasteleyn matrix (equivalently, for the correlation kernel of the associated nonintersecting lattice path model) are, however, known only in certain special cases, including:
\begin{itemize}
    \item domino tilings of the Aztec diamond: the uniform case by Helfgott~\cite{Helfgott} and Johansson~\cite{Joh05a}; \(q^{\mathrm{vol}}\) and other one-periodic weightings due to Chhita--Young~\cite{CY14} and Chhita--Johansson--Young~\cite{CJY15}; the two-periodic and, more generally, doubly periodic models studied by Chhita--Young~\cite{CY14}, Chhita--Johansson~\cite{CJ16}, Duits--Kuijlaars~\cite{DK21}, Berggren--Duits~\cite{BD19}, Berggren~\cite{Ber21}, Borodin--Duits~\cite{BD22}, Berggren--Borodin~\cite{BB23}, and Boutillier--de Tili\`ere~\cite{BdT24}; the double Aztec diamond of Adler--Chhita--Johansson--van Moerbeke~\cite{ACJvM}; and the split Aztec diamond of Shea~\cite{Shea25}.
    \item lozenge tilings of hexagonal and polygonal regions: the uniform hexagon by Johansson~\cite{Joh02}; \(q^{\mathrm{vol}}\)-weighted hexagons studied by Borodin--Gorin--Rains~\cite{BGR10}; polygonal domains studied by Petrov~\cite{Pet14}; holey hexagons studied by Gilmore~\cite{Gilmore}; non-convex polygons studied by Adler--Johansson--van Moerbeke~\cite{AJvM}; and periodic hexagons studied by Charlier--Duits--Kuijlaars--Lenells~\cite{CDKL19} and Kuijlaars~\cite{Kui25}.
    \item more general graph families related to the Schur processes of Okounkov--Reshetikhin~\cite{OR03}, rail-yard graphs of Boutillier--Bouttier--Chapuy--Corteel--Ramassamy~\cite{BBCCR17}, and tower graphs by Borodin--Ferrari~\cite{BF15} and Nicoletti~\cite{Nicoletti}.
\end{itemize}
The complexity of a periodic bipartite graph is captured by an invariant called the Newton polygon. All the cases listed above correspond to Newton polygons that are triangles or quadrilaterals. \emph{By contrast, AZ graphs exist for arbitrary Newton polygons and, to the best of our knowledge, provide the first analysis of finite subgraphs beyond triangles and quadrilaterals.}

The second main contribution of the present paper is a formula for the inverse Kasteleyn matrix for AZ graphs, generalizing the recent Aztec diamond formula of Boutillier--de Tili\`ere~\cite{BdT24} (which is a streamlined version of the formula obtained in Berggren--Borodin~\cite{BB23}). This formula allows for a detailed asymptotic analysis similar to that performed in Berggren--Borodin~\cite{BB23} for the Aztec diamond. In particular, we obtain an explicit parameterization of the limit shape, prove that local statistics are governed by the ergodic Gibbs measures classified in Kenyon--Okounkov--Sheffield~\cite{KOS06}, and establish geometric properties of the arctic curve analogous to Kenyon--Okounkov~\cite{KO07}, Berggren--Borodin~\cite{BB23}, and Bobenko--Bobenko~\cite{BB24}. In this sense, AZ graphs provide a general new class of dimer models whose large-scale behavior can be described with a level of detail previously available only in a few special examples.

Although we do not currently have an exact sampling algorithm for random dimer covers on general AZ graphs, we show in the appendix that certain AZ graphs can be sampled via a tropical limit of the Aztec diamond. This yields concrete simulations that can be compared with our theoretical predictions.

As we were writing this paper, we learned that Boutillier--de Tili\`ere--Deb~\cite{BdT2} have independent work that derives a formula for the inverse Kasteleyn matrix for Speyer's ``crosses and wrenches'' graphs~\cite{Speyer_oct}. We expect that our class of graphs overlaps with theirs but neither class includes the other.


\begin{figure}[t]
\centering

\begin{subfigure}{0.31\textwidth}
\centering
\begin{tikzpicture}[baseline=(current bounding box.center)]
\begin{scope}[rotate=0]
    \node[bvert] (1) at (0,1){}; 
    \node[bvert] (2) at (1,0){}; 
    \node[bvert] (3) at (2,1){}; 
    \node[bvert] (4) at (1,1){};
    \node[bvert] (5) at (0,0){}; 
    \node[bvert] (6) at (1,2){};

    \draw[->]  (5) -- (2);
    \draw[->]  (3) -- (6);
    \draw[->]  (1) -- (5);
    \draw[->]  (2) -- (3);
    \draw[->]  (6) -- (1);
\end{scope}
\end{tikzpicture}
\end{subfigure}
\hfill
\begin{subfigure}{0.31\textwidth}
\centering
\begin{tikzpicture}[baseline=(current bounding box.center),scale=1.15]

\pgfmathsetmacro{\dx}{sqrt(3)}
\pgfmathsetmacro{\hx}{sqrt(3)/2}
\pgfmathsetmacro{\hy}{3/2}

\pgfmathsetmacro{\cmA}{1/sqrt(3)}
\pgfmathsetmacro{\cmB}{1/3}
\pgfmathsetmacro{\cmC}{2/3}

\begin{scope}[cm={\cmA,0,\cmB,\cmC,(0,0)}]

  \coordinate (CellBase) at (0,0);
  \coordinate (Psw) at (CellBase);
  \coordinate (Pse) at ($(Psw)+(\dx,0)$);
  \coordinate (Pnw) at ($(Psw)+(-\hx,\hy)$);
  \coordinate (Pne) at ($(Pse)+(-\hx,\hy)$);

  \fill[gray!30] (Psw) -- (Pse) -- (Pne) -- (Pnw) -- cycle;

  \begin{scope}
    \coordinate[wvert] (w0) at (0,0);
    \coordinate[wvert] (w1) at (30:1);
    \coordinate[wvert] (w2) at (90:1);
    \coordinate[wvert] (w3) at (150:1);
    \coordinate[wvert] (w4) at (210:1);
    \coordinate[wvert] (w5) at (270:1);
    \coordinate[wvert] (w6) at (330:1);

    \coordinate[bvert] (b12) at ($(w1)!0.5!(w2)$);
    \coordinate[bvert] (b23) at ($(w2)!0.5!(w3)$);
    \coordinate[bvert] (b34) at ($(w3)!0.5!(w4)$);
    \coordinate[bvert] (b45) at ($(w4)!0.5!(w5)$);
    \coordinate[bvert] (b56) at ($(w5)!0.5!(w6)$);
    \coordinate[bvert] (b61) at ($(w6)!0.5!(w1)$);

    \draw (w0) -- (b23) (w0) -- (b34) (w0) -- (b45) (w0) -- (b56) (w0) -- (b61);
    \draw (b12) -- (w1) (b12) -- (w2);
    \draw (b23) -- (w2) (b23) -- (w3);
    \draw (b34) -- (w3) (b34) -- (w4);
    \draw (b45) -- (w4) (b45) -- (w5);
    \draw (b56) -- (w5) (b56) -- (w6);
    \draw (b61) -- (w6) (b61) -- (w1);
  \end{scope}

  \begin{scope}[shift={(1*\dx +0*\hx,0*\hy)}]
    \coordinate[wvert] (w0) at (0,0);
    \coordinate[wvert] (w1) at (30:1);
    \coordinate[wvert] (w2) at (90:1);
    \coordinate[wvert] (w3) at (150:1);
    \coordinate[wvert] (w4) at (210:1);
    \coordinate[wvert] (w5) at (270:1);
    \coordinate[wvert] (w6) at (330:1);

    \coordinate[bvert] (b12) at ($(w1)!0.5!(w2)$);
    \coordinate[bvert] (b23) at ($(w2)!0.5!(w3)$);
    \coordinate[bvert] (b34) at ($(w3)!0.5!(w4)$);
    \coordinate[bvert] (b45) at ($(w4)!0.5!(w5)$);
    \coordinate[bvert] (b56) at ($(w5)!0.5!(w6)$);
    \coordinate[bvert] (b61) at ($(w6)!0.5!(w1)$);

    \draw (w0) -- (b23) (w0) -- (b34) (w0) -- (b45) (w0) -- (b56) (w0) -- (b61);
    \draw (b12) -- (w1) (b12) -- (w2);
    \draw (b23) -- (w2) (b23) -- (w3);
    \draw (b34) -- (w3) (b34) -- (w4);
    \draw (b45) -- (w4) (b45) -- (w5);
    \draw (b56) -- (w5) (b56) -- (w6);
    \draw (b61) -- (w6) (b61) -- (w1);
  \end{scope}

  \begin{scope}[shift={(-1*\dx +0*\hx,0*\hy)}]
    \coordinate[wvert] (w0) at (0,0);
    \coordinate[wvert] (w1) at (30:1);
    \coordinate[wvert] (w2) at (90:1);
    \coordinate[wvert] (w3) at (150:1);
    \coordinate[wvert] (w4) at (210:1);
    \coordinate[wvert] (w5) at (270:1);
    \coordinate[wvert] (w6) at (330:1);

    \coordinate[bvert] (b12) at ($(w1)!0.5!(w2)$);
    \coordinate[bvert] (b23) at ($(w2)!0.5!(w3)$);
    \coordinate[bvert] (b34) at ($(w3)!0.5!(w4)$);
    \coordinate[bvert] (b45) at ($(w4)!0.5!(w5)$);
    \coordinate[bvert] (b56) at ($(w5)!0.5!(w6)$);
    \coordinate[bvert] (b61) at ($(w6)!0.5!(w1)$);

    \draw (w0) -- (b23) (w0) -- (b34) (w0) -- (b45) (w0) -- (b56) (w0) -- (b61);
    \draw (b12) -- (w1) (b12) -- (w2);
    \draw (b23) -- (w2) (b23) -- (w3);
    \draw (b34) -- (w3) (b34) -- (w4);
    \draw (b45) -- (w4) (b45) -- (w5);
    \draw (b56) -- (w5) (b56) -- (w6);
    \draw (b61) -- (w6) (b61) -- (w1);
  \end{scope}

  \begin{scope}[shift={(0*\dx -1*\hx,1*\hy)}]
    \coordinate[wvert] (w0) at (0,0);
    \coordinate[wvert] (w1) at (30:1);
    \coordinate[wvert] (w2) at (90:1);
    \coordinate[wvert] (w3) at (150:1);
    \coordinate[wvert] (w4) at (210:1);
    \coordinate[wvert] (w5) at (270:1);
    \coordinate[wvert] (w6) at (330:1);

    \coordinate[bvert] (b12) at ($(w1)!0.5!(w2)$);
    \coordinate[bvert] (b23) at ($(w2)!0.5!(w3)$);
    \coordinate[bvert] (b34) at ($(w3)!0.5!(w4)$);
    \coordinate[bvert] (b45) at ($(w4)!0.5!(w5)$);
    \coordinate[bvert] (b56) at ($(w5)!0.5!(w6)$);
    \coordinate[bvert] (b61) at ($(w6)!0.5!(w1)$);

    \draw (w0) -- (b23) (w0) -- (b34) (w0) -- (b45) (w0) -- (b56) (w0) -- (b61);
    \draw (b12) -- (w1) (b12) -- (w2);
    \draw (b23) -- (w2) (b23) -- (w3);
    \draw (b34) -- (w3) (b34) -- (w4);
    \draw (b45) -- (w4) (b45) -- (w5);
    \draw (b56) -- (w5) (b56) -- (w6);
    \draw (b61) -- (w6) (b61) -- (w1);
  \end{scope}

  \begin{scope}[shift={(1*\dx -1*\hx,1*\hy)}]
    \coordinate[wvert] (w0) at (0,0);
    \coordinate[wvert] (w1) at (30:1);
    \coordinate[wvert] (w2) at (90:1);
    \coordinate[wvert] (w3) at (150:1);
    \coordinate[wvert] (w4) at (210:1);
    \coordinate[wvert] (w5) at (270:1);
    \coordinate[wvert] (w6) at (330:1);

    \coordinate[bvert] (b12) at ($(w1)!0.5!(w2)$);
    \coordinate[bvert] (b23) at ($(w2)!0.5!(w3)$);
    \coordinate[bvert] (b34) at ($(w3)!0.5!(w4)$);
    \coordinate[bvert] (b45) at ($(w4)!0.5!(w5)$);
    \coordinate[bvert] (b56) at ($(w5)!0.5!(w6)$);
    \coordinate[bvert] (b61) at ($(w6)!0.5!(w1)$);

    \draw (w0) -- (b23) (w0) -- (b34) (w0) -- (b45) (w0) -- (b56) (w0) -- (b61);
    \draw (b12) -- (w1) (b12) -- (w2);
    \draw (b23) -- (w2) (b23) -- (w3);
    \draw (b34) -- (w3) (b34) -- (w4);
    \draw (b45) -- (w4) (b45) -- (w5);
    \draw (b56) -- (w5) (b56) -- (w6);
    \draw (b61) -- (w6) (b61) -- (w1);
  \end{scope}

  \begin{scope}[shift={(2*\dx -1*\hx,1*\hy)}]
    \coordinate[wvert] (w0) at (0,0);
    \coordinate[wvert] (w1) at (30:1);
    \coordinate[wvert] (w2) at (90:1);
    \coordinate[wvert] (w3) at (150:1);
    \coordinate[wvert] (w4) at (210:1);
    \coordinate[wvert] (w5) at (270:1);
    \coordinate[wvert] (w6) at (330:1);

    \coordinate[bvert] (b12) at ($(w1)!0.5!(w2)$);
    \coordinate[bvert] (b23) at ($(w2)!0.5!(w3)$);
    \coordinate[bvert] (b34) at ($(w3)!0.5!(w4)$);
    \coordinate[bvert] (b45) at ($(w4)!0.5!(w5)$);
    \coordinate[bvert] (b56) at ($(w5)!0.5!(w6)$);
    \coordinate[bvert] (b61) at ($(w6)!0.5!(w1)$);

    \draw (w0) -- (b23) (w0) -- (b34) (w0) -- (b45) (w0) -- (b56) (w0) -- (b61);
    \draw (b12) -- (w1) (b12) -- (w2);
    \draw (b23) -- (w2) (b23) -- (w3);
    \draw (b34) -- (w3) (b34) -- (w4);
    \draw (b45) -- (w4) (b45) -- (w5);
    \draw (b56) -- (w5) (b56) -- (w6);
    \draw (b61) -- (w6) (b61) -- (w1);
  \end{scope}

\end{scope}
\end{tikzpicture}

\end{subfigure}
\hfill
\begin{subfigure}{0.31\textwidth}
\centering
\begin{tikzpicture}[baseline=(current bounding box.center),scale=2.5,
  elab/.style={midway, sloped, font=\scriptsize, inner sep=1.2pt,
               fill=white, fill opacity=.85, text opacity=1}
]

\coordinate (S1) at (3,0);
\coordinate (S2) at (2,0);
\coordinate (S3) at (3,-1);

\coordinate[wvert]      (w0a)  at (3,0);
\coordinate[wvert]      (w0)   at (2,-1);
\coordinate             (w1a)  at ({11/3},{ 1/3});
\coordinate             (w2a)  at ({10/3},{ 2/3});
\coordinate[wvert]      (w4a)  at ({ 7/3},{-1/3});
\coordinate[wvert]      (w5a)  at ({ 8/3},{-2/3});

\coordinate[bvert]      (b34a) at ({5/2},0);
\coordinate[bvert]      (b45a) at ({5/2},{-1/2});
\coordinate[bvert]      (b56a) at (3,{-1/2});

\draw (w0a) -- node[elab, above, xshift=1pt] {$1$}    (b34a);
\draw (w0a) -- node[elab, above, yshift=1pt] {$1$}    (b45a);
\draw (w0a) -- node[elab, below, yshift=-1pt] {$1$}   (b56a);

\draw (b34a) -- node[elab, below, xshift=-2pt] {$1$}  (w4a);

\draw (b45a) -- node[elab, above] {$\phi$}            (w4a);
\draw (b45a) -- node[elab, below] {$\phi$}            (w5a);

\draw (b56a) -- node[elab, above] {$1$}               (w5a);

\coordinate[wvert]      (w0b)  at (2,0);
\coordinate             (w1b)  at ({ 8/3},{ 1/3});
\coordinate             (w5b)  at ({ 5/3},{-2/3});
\coordinate[wvert]      (w6b)  at ({ 7/3},{-1/3});

\coordinate[bvert]      (b56b) at (2,{-1/2});
\coordinate[bvert]      (b61b) at ({5/2},0);

\draw (w0b) -- node[elab, below] {$1$}               (b56b);
\draw (w0b) -- node[elab, above] {$1$}               (b61b);
\draw (b56b) -- node[elab, above] {$\phi$}           (w6b);

\draw (b56b) -- node[elab, below] {$1$}              (w0);

\coordinate[wvert]      (w0c)  at (3,-1);
\coordinate             (w2c)  at ({10/3},{-1/3});
\coordinate[wvert]      (w3c)  at ({ 8/3},{-2/3});
\coordinate             (w4c)  at ({ 7/3},{-4/3});

\coordinate[bvert]      (b23c) at (3,{-1/2});
\coordinate[bvert]      (b34c) at ({5/2},{-1});

\draw (w0c) -- node[elab, above] {$1$}               (b23c);
\draw (w0c) -- node[elab, below] {$1$}               (b34c);
\draw (b34c) -- node[elab, above, xshift=-2pt] {$\phi$} (w3c);
\draw (b34c) -- node[elab, above] {$1$}              (w0);

\end{tikzpicture}

\end{subfigure}
\caption{A pentagonal lattice polygon $N$ (left), a minimal graph $\graphpl$ with Newton polygon $N$ together with a fundamental domain (middle), and a choice of isoradial weights (modulo gauge) where the angles are chosen evenly spaced around the unit circle (right). Here $\phi=(1+\sqrt5)/2$ denotes the golden ratio. See Figure~\ref{fig:isoradial_pentagon} below for an illustration of the isoradial property. 
}
\label{fig:pentagon_0}
\end{figure}
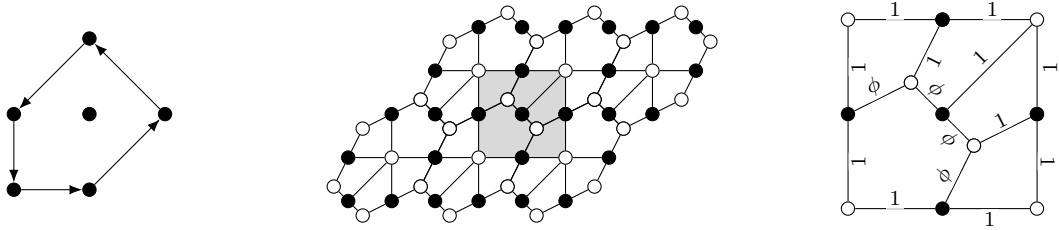

Let us now describe our results in more detail. Let \(\graphpl\) be a minimal \(\ZZ^2\)-periodic planar bipartite graph, and let
\(
\graphtor := \graphpl/\ZZ^2
\)
be the associated torus graph. We denote the black vertices, white vertices and faces of $\graphpl$ by $B(\graphpl), W(\graphpl)$ and $F(\graphpl)$ respectively. A \emph{zig-zag path} of \(\graphpl\) or \(\graphtor\) is a path that turns maximally right at white vertices and maximally left at black vertices. Every zig-zag path of \(\graphpl\) is a lift of a zig-zag path in \(\graphtor\). We write \(\zz\) for the set of zig-zag paths of \(\graphpl\).

The homology classes of the zig-zag paths of \(\graphtor\) determine a convex lattice polygon \(N\), called the \emph{Newton polygon} of \(\graphpl\); see Section~\ref{sec:minimal_graphs}. By a theorem of Goncharov--Kenyon, minimal periodic planar bipartite graphs are classified, up to certain local moves, by their Newton polygon. These local moves do not affect the statistical-mechanical properties of the model. Figure~\ref{fig:pentagon_0} shows an example of a minimal graph for a convex lattice pentagon.

Let $V(N)$ and $E(N)$ denote the vertices and edges of $N$, respectively, where by an edge we mean a side of $N$ rather than a primitive edge vector. For each \(\alpha \in \zz\), let \(e(\alpha) \in E(N)\) be the edge of \(N\) corresponding to the homology class of \(\alpha\). For each edge \(e \in E(N)\), define
\[
\zz_e := \{\alpha \in \zz : e(\alpha)=e\}.
\]
We say that the zig-zag paths in \(\zz_e\) are \emph{parallel to} \(e\). 

\subsection{Astroidal zig-zag graphs}

\begin{figure}
\begin{center}
\begin{tikzpicture}[line cap=round, line join=round]

  \node at (4,0) {\includegraphics[width=6cm]{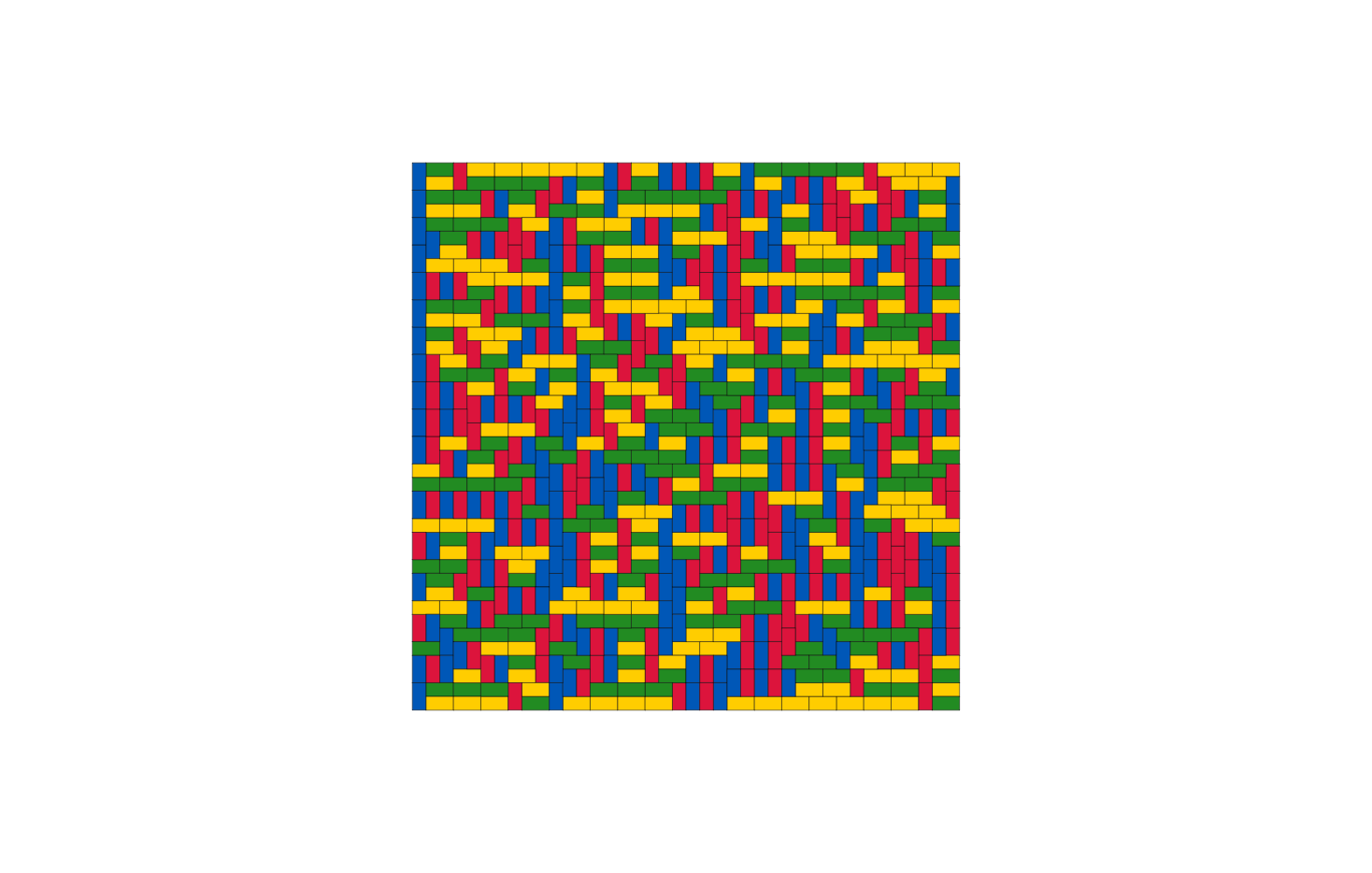}};

  \node at (12,0) {\includegraphics[width=4cm]{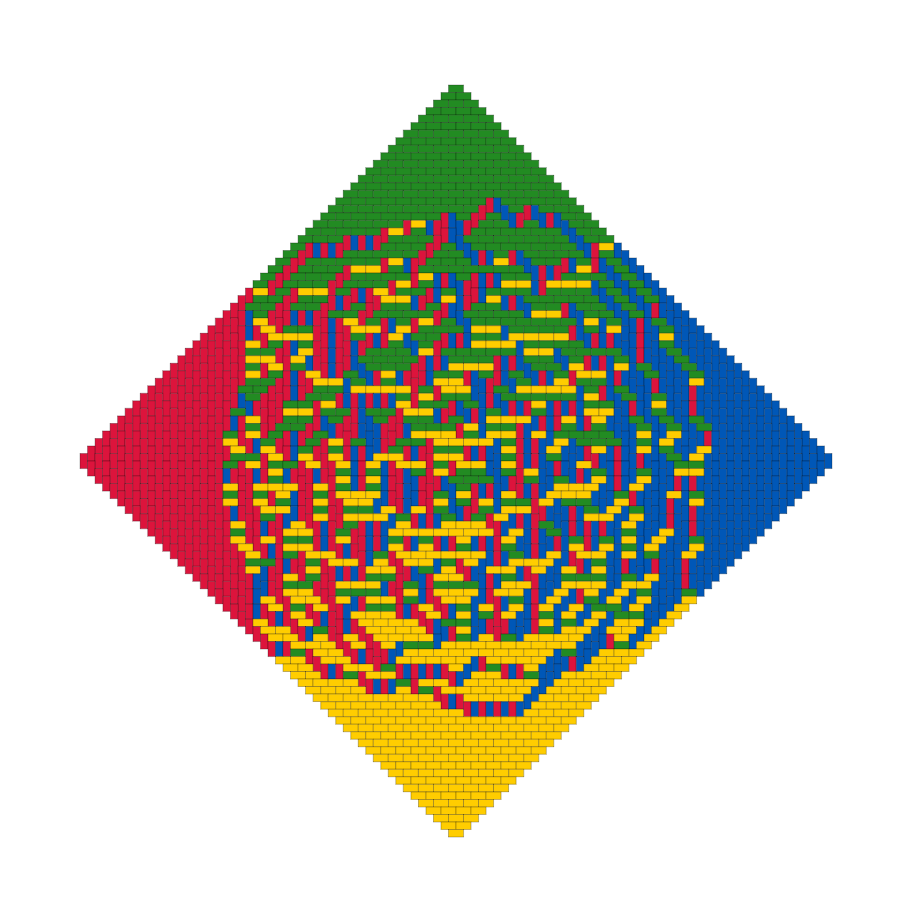}};

  \node at (0,0) {
    \begin{tikzpicture}[scale=0.57]
      \def\N{4}

      \foreach \i in {0,...,\N}{
        \foreach \j in {0,...,\N}{
          \pgfmathtruncatemacro{\parity}{mod(\i+\j,2)}
          \ifnum\parity=0
            \coordinate[bvert] (v-\i-\j) at (\i,\j);
          \else
            \coordinate[wvert] (v-\i-\j) at (\i,\j);
          \fi
        }
      }

      \foreach \i in {0,...,\N}{
        \foreach \j in {0,...,\N}{
          \ifnum\i<\N
            \pgfmathtruncatemacro{\ip}{\i+1}
            \draw (v-\i-\j) -- (v-\ip-\j);
          \fi
          \ifnum\j<\N
            \pgfmathtruncatemacro{\jp}{\j+1}
            \draw (v-\i-\j) -- (v-\i-\jp);
          \fi
        }
      }
    \end{tikzpicture}
  };

  \node at (8,0) {
    \begin{tikzpicture}[scale=0.4,rotate=45]

      \foreach \y in {0,2,4,6} {
        \foreach \x in {1,3,5} {
          \node[bvert] (b-\x-\y) at (\x,\y) {};
        }
      }

      \foreach \y in {1,3,5} {
        \foreach \x in {0,2,4,6} {
          \node[wvert] (w-\x-\y) at (\x,\y) {};
        }
      }

      \foreach \y in {0,2,4,6} {
        \foreach \x in {1,3,5} {

          \ifcsname pgf@sh@ns@w-\the\numexpr\x-1\relax-\the\numexpr\y+1\relax\endcsname
            \draw (b-\x-\y) -- (w-\the\numexpr\x-1\relax-\the\numexpr\y+1\relax);
          \fi

          \ifcsname pgf@sh@ns@w-\the\numexpr\x+1\relax-\the\numexpr\y+1\relax\endcsname
            \draw (b-\x-\y) -- (w-\the\numexpr\x+1\relax-\the\numexpr\y+1\relax);
          \fi

          \ifcsname pgf@sh@ns@w-\the\numexpr\x-1\relax-\the\numexpr\y-1\relax\endcsname
            \draw (b-\x-\y) -- (w-\the\numexpr\x-1\relax-\the\numexpr\y-1\relax);
          \fi

          \ifcsname pgf@sh@ns@w-\the\numexpr\x+1\relax-\the\numexpr\y-1\relax\endcsname
            \draw (b-\x-\y) -- (w-\the\numexpr\x+1\relax-\the\numexpr\y-1\relax);
          \fi
        }
      }

    \end{tikzpicture}
  };

\end{tikzpicture}
\end{center}
\caption{Two natural finite domains for the square lattice are the square and the Aztec diamond, together with their associated bipartite graphs. Only the Aztec diamond gives rise to a limit shape with a nontrivial phase transition. The images were generated using Leonid Petrov's domino tiling simulator; see~\url{https://lpetrov.cc/simulations/}.}
\label{fig:intro:square_and_diamond}
\end{figure}
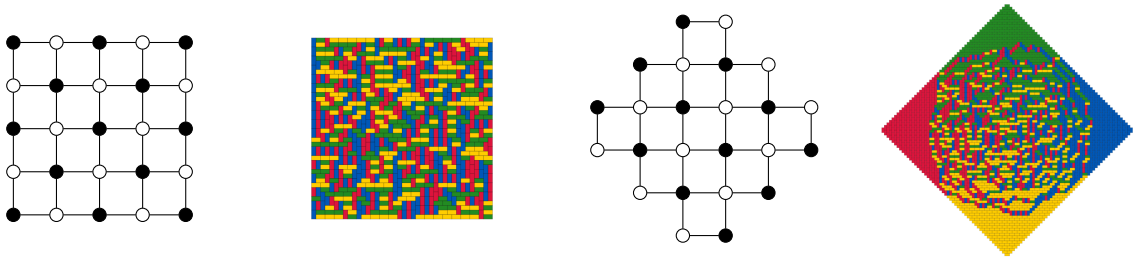

Fix a convex lattice polygon \(N\), and let \(\graphpl\) be a graph with Newton polygon \(N\). Our first main contribution is the definition of a natural class of finite subgraphs of \(\graphpl\) that give rise to limit shapes with nontrivial phase transitions. To motivate this definition, consider the square lattice. Two natural families of finite domains are squares and Aztec diamonds. As illustrated in Figure~\ref{fig:intro:square_and_diamond}, only the Aztec diamond produces a limit shape exhibiting a phase transition. Our goal is therefore to isolate the features of the Aztec diamond responsible for this phenomenon.

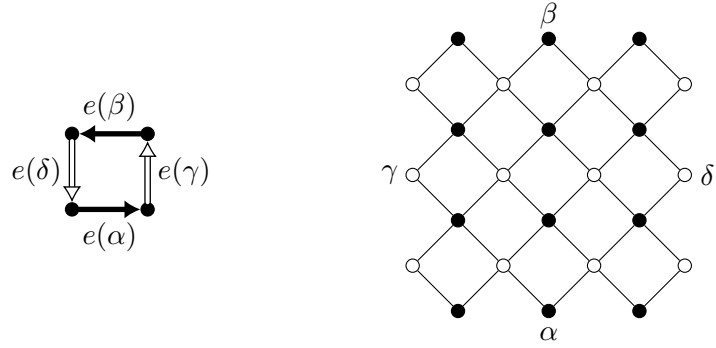
\begin{figure}
  \centering
  \begin{tikzpicture}

    \node[anchor=center] (A) {
      \begin{tikzpicture}[scale=1]
       \begin{scope}
    \node[bvert] (1) at (0,0){}; 
     \node[bvert] (2) at (1,0){}; 
      \node[bvert] (3) at (1,1){}; 
       \node[bvert] (4) at (0,1){};

        \draw[solidedge] (1) -- node[below]{${e(\alpha)}$} 
       (2);
       \draw[solidedge] (3) -- node[above]{${ e(\beta)}$} 
       (4) ;
       \draw[hollowedge] (2) -- node[right]{${ e(\gamma)}$} 
       (3);
       \draw[hollowedge] (4) -- node[left]{${ e(\delta)}$} 
       (1);

       \end{scope}
      \end{tikzpicture}
    };

    \node[anchor=center] (C) at ([xshift=4.2cm]A.east) {%
      \begin{tikzpicture}[scale=0.6]

        \foreach \y in {0,2,4,6} {
          \foreach \x in {1,3,5} {
            \node[bvert] (b-\x-\y) at (\x,\y) {};
          }
        }

        \foreach \y in {1,3,5} {
          \foreach \x in {0,2,4,6} {
            \node[wvert] (w-\x-\y) at (\x,\y) {};
          }
        }

        \foreach \y in {0,2,4,6} {
          \foreach \x in {1,3,5} {

            \ifcsname pgf@sh@ns@w-\the\numexpr\x-1\relax-\the\numexpr\y+1\relax\endcsname
              \draw (b-\x-\y) -- (w-\the\numexpr\x-1\relax-\the\numexpr\y+1\relax);
            \fi

            \ifcsname pgf@sh@ns@w-\the\numexpr\x+1\relax-\the\numexpr\y+1\relax\endcsname
              \draw (b-\x-\y) -- (w-\the\numexpr\x+1\relax-\the\numexpr\y+1\relax);
            \fi

            \ifcsname pgf@sh@ns@w-\the\numexpr\x-1\relax-\the\numexpr\y-1\relax\endcsname
              \draw (b-\x-\y) -- (w-\the\numexpr\x-1\relax-\the\numexpr\y-1\relax);
            \fi

            \ifcsname pgf@sh@ns@w-\the\numexpr\x+1\relax-\the\numexpr\y-1\relax\endcsname
              \draw (b-\x-\y) -- (w-\the\numexpr\x+1\relax-\the\numexpr\y-1\relax);
            \fi
          }
        }



         


        \node at (3,6.5) {$\beta$};
        \node at (6.5,3) {$\delta$};
         \node at (3,-.5) {$\alpha$};
        \node at (-.5,3) {$\gamma$};
      \end{tikzpicture}
    };

  \end{tikzpicture}

  \caption{
  The square Newton polygon \(N\) (left) and the Aztec diamond (right). The cyclic order of zig-zag paths along the boundary of the Aztec diamond is opposite to the cyclic order of the corresponding edges of \(N\). The edges of \(N\) are colored according to whether the black or white vertices are outside.}
  \label{fig:intro:square_aztec}
\end{figure}

First of all, we observe that the boundary of the Aztec diamond is formed by zig-zag paths whose cyclic order along the boundary is opposite to the cyclic order of the corresponding edges of \(N\) (see Figure~\ref{fig:intro:square_aztec}). This condition alone, however, is not sufficient: one can also construct rectangular regions satisfying it, but such regions do not admit dimer covers. This leads to a second requirement, which we call \emph{admissibility}. These observations motivate the definition of the class of AZ graphs.

\begin{figure}
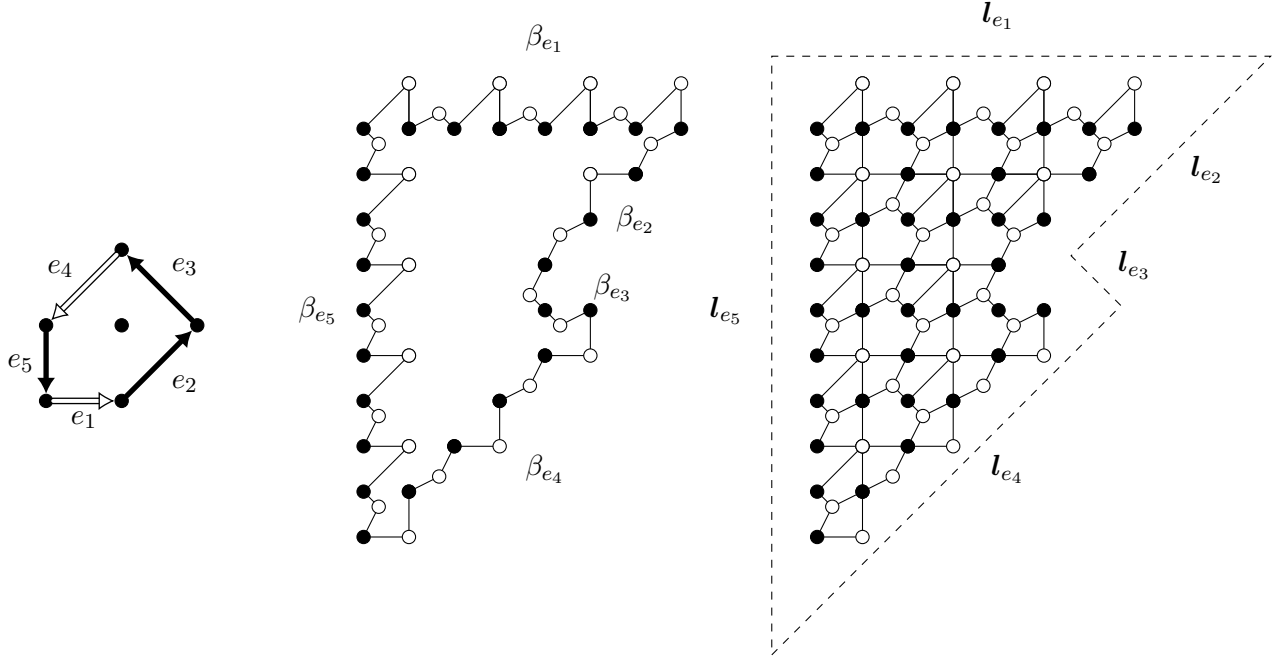

\centering


\caption{The construction of an AZ graph for the graph $\graphpl$ in Figure~\ref{fig:pentagon_0}. We choose zig-zag paths forming the boundary $\partial \graphbeta$ (middle), and obtain $\graphbeta$ as the enclosed subgraph (right). The dashed polygon is an astroidal domain that $\graphbeta$ approximates.} \label{fig:pentagon_intro_1}
\end{figure}

An AZ graph is constructed by choosing, for each edge \(e\in E(N)\), a zig-zag path \(\beta_e\) parallel to \(e\), such that these boundary zig-zag paths appear in the cyclic order opposite to that of the corresponding edges of \(N\). We write
\[
\bm\beta:=(\beta_e)_{e\in E(N)}
\]
for the resulting collection of boundary zig-zag paths. They cut out a finite region of \(\graphpl\), and the resulting subgraph \(\graphbeta\subset\graphpl\) is the portion of \(\graphpl\) enclosed by the corresponding zig-zag boundary \(\partial\graphbeta\). Each boundary zig-zag path \(\beta_e\) also determines a color of the edge \(e\), according to whether the black or white vertices lie outside \(\graphbeta\). We now give an informal definition of AZ graphs; see Definition~\ref{def:azgraph} for the precise definition. 

\begin{definition}
The graph \(\graphbeta\) is called an \emph{astroidal zig-zag graph}, or \emph{AZ graph}, if its zig-zag boundary is simple and it is \emph{admissible}, meaning that for any face, the sum of the distances to \(\beta_e\) over black edges \(e\) is equal to the sum of the distances to \(\beta_e\) over white edges \(e\).
\end{definition}

Here the distances are measured using the \emph{discrete Abel map} \[\dabel: W(\graphpl) \sqcup F(\graphpl) \sqcup B(\graphpl) \rightarrow \ZZ^{\zz},\] 
which assigns to each vertex or face an integer linear combination of zig-zag paths recording which zig-zag paths we cross as we move from a fixed reference face \(\f_0\) to that vertex or face. In Lemma~\ref{lem:four_convex_vertices}, we show that every AZ graph has exactly four convex vertices, connected by boundary arcs that may contain concave vertices. Thus, these graphs resemble the astroid curve, which explains the term ``astroidal,'' while ``zig-zag'' refers to the fact that the boundary is formed by zig-zag paths.

In Proposition~\ref{prop:AZ_graphs_height_function} and Corollary~\ref{cor:dimer_cover}, we show that among finite subgraphs whose boundary consists of zig-zag paths, one for each edge of the Newton polygon, in opposite cyclic order, the ones that admit dimer covers are precisely the AZ graphs. Thus admissibility is not an auxiliary technical condition, but exactly the condition for the graph to carry a nontrivial dimer model.

AZ graphs form a \((\#E(N)-3)\)-parameter family. Thus, when \(N\) is a triangle, no AZ graphs exist. {This can be seen directly: any finite subgraph bounded by three zig-zag paths, one in each direction, has unequal numbers of black and white vertices.} When \(N\) is the unit square, the Aztec diamond is, up to size and local moves, the unique AZ graph.

\subsection{Formula for the inverse Kasteleyn matrix}

We next turn to the dimer model itself. The periodic graph \(\graphpl\) is equipped with Fock's dimer model, built from the spectral data of an M-curve \(\curve\) called the \emph{spectral curve}, a point \(t\) in the Jacobian \(\jac(\curve)\), and an \emph{angle function} \(\nu\) assigning to each zig-zag path a point in a distinguished component \(A_0\) of the real part \(\curve(\RR)\). 
As is customary, we identify a zig-zag path with the corresponding angle in \(A_0\). We write
\[
\curve(\RR)=A_0\sqcup \left(\bigsqcup_{j=1}^g A_j\right),
\]
and denote by \(\curve_+\) the closure of one of the two connected components of \(\curve\setminus \curve(\RR)\). We also write \(\theta\) for the Riemann theta function on \(\curve\) and \(E\) for the prime form. We recall the necessary background on M-curves, the Abel map, the theta function, the prime form and Fock's dimer model in Sections~\ref{sec:M-curves}--\ref{sec:focks_dimer_model}. For the purposes of the introduction, it is sufficient to know that out of this geometric data, one can construct a Kasteleyn matrix \(\kast\) called \emph{Fock's Kasteleyn matrix}. Moreover, Fock~\cite{Foc15} showed that this construction realizes all periodic weights on $\graphpl$ (see also~Boutillier--Cimasoni--de Tili\`ere~\cite{BCT22} and~George--Goncharov--Kenyon~\cite{GGK22}).

We now describe the whole-plane inverse Kasteleyn matrices introduced in Boutillier--Cimasoni--de Tili\`ere~\cite[Section~3.5]{BCT22}, which will be used both in the definition of \(\kast^{-1}\) for finite AZ graphs and in the description of local statistics. For a black vertex \(\bl\) and a white vertex \(\wh\) of \(\graphpl\), there is an associated meromorphic form $g_{\bl,\wh}$ on \(\curve\), lying in the kernel/cokernel of \(\kast\). For \(\zeta \in \curve_+\setminus \zz\), set
\[
\mathsf A^{\zeta}_{\bl,\wh}
:=
\frac{1}{2\pi\ii}\int_{C^\zeta_{\bl,\wh}} g_{\bl,\wh},
\]
where \(C^\zeta_{\bl,\wh}\) is a contour from \(\sigma(\zeta)\) to \(\zeta\), with \(\sigma:\curve\to\curve\) the antiholomorphic involution of the M-curve; see Section~\ref{sec:whole_plane_inverse_K}.
Then, for each $\zeta$, \(\mathsf A^\zeta\) is an inverse of Fock's Kasteleyn matrix on the full graph \(\graphpl\).

For \(\zeta \in A_0\setminus \zz\), the matrix \(\mathsf A^\zeta\) depends only on the connected component of \(A_0\setminus \zz\) containing \(\zeta\). These components are naturally indexed by the vertices \(u\in V(N)\), and we denote the corresponding whole-plane inverse by \(\mathsf A^u\).

We now state the main theorem of the paper which gives an explicit formula for \(\kast^{-1}\) for AZ graphs. Since the dimer model is determinantal, this formula gives complete access to all local correlations on \(\graphbeta\). 
The boundary conditions determine a divisor \(\bm D_{\bm\beta}\), defined by recording, via the discrete Abel map, the signed distances from the reference face \(\f_0\) to the zig-zag paths forming the boundary of \(\graphbeta\). For a black vertex \(\bl\) and a white vertex \(\wh\) of \(\graphbeta\), we define
\[
\omega_{\bl,\wh}(\zeta,\eta)
:=
\frac{1}{(2\pi \ii)^2}
\frac{1}{E(\zeta,\eta)}
\frac{\theta(-t+\bm D_{\bm\beta}+\eta-\zeta)}
     {\theta(-t+\bm D_{\bm\beta})}
\frac{g_{{\f_0},\wh}(\eta)\Psi_{\bm\beta}(\eta)}
     {g_{{\f_0},\bl}(\zeta)\Psi_{\bm\beta}(\zeta)},
\]
where
\[
\Psi_{\bm\beta}(\zeta)
:=
\prod_{\alpha\in\zz} E(\alpha,\zeta)^{(\bm D_{\bm\beta})_\alpha}
\]
is the product of prime forms associated with the boundary divisor \(\bm D_{\bm\beta}\). The key point is that admissibility is precisely the condition under which this kernel is globally well-defined on \(\curve\times\curve\).

\begin{remark}
While we use the formalism of Fock's dimer model to state the results in full generality, our results are already new and interesting in the case of genus-zero spectral curves. From the probabilistic point of view, this case is the familiar critical/isoradial dimer model; see~Kenyon~\cite{Kenyonisoradial} and Kenyon--Okounkov~\cite[Section~5]{KO06}. The genus-zero case is especially concrete: the theta functions are identically equal to $1$ and the prime form $E(\alpha,\zeta)$ can be written as $\zeta-\alpha$, so the kernel reduces to an explicit rational function. For example, choosing the angles to be evenly spaced around the unit circle for the pentagonal Newton polygon in Figure~\ref{fig:pentagon_0}(left) yields the isoradial weights shown in Figure~\ref{fig:pentagon_0}(right).
\end{remark}

For each variable, the zeros and poles of \(\omega_{\bl,\wh}\) partition the boundary of \(N\) into four intervals, alternating between zeros and poles. To each of \(\bl\) and \(\wh\), we assign one of the two pole intervals, denoted
\[
I_\bl=[\ell_\bl,r_\bl],
\qquad
I_\wh=[\ell_\wh,r_\wh].
\]
We then let \(C(I_\bl)\) and \(C(I_\wh)\) be small counterclockwise-oriented contours around the corresponding angles on \(\curve\).

The inverse Kasteleyn matrix is then defined by
\begin{equation}\label{eq:intro_inverse_formula}
\kast^{-1}_{\bl,\wh}
=
\iint_{C(I_\wh)\prec C(I_\bl)} \omega_{\bl,\wh}
+\mathbf 1_{\{\ell_\bl\in I_\wh\}}\mathsf A^{\ell_\bl}_{\bl,\wh}
-\mathbf 1_{\{r_\bl\in I_\wh\}}\mathsf A^{r_\bl}_{\bl,\wh},
\end{equation}
where \(
C(I_\wh)\prec C(I_\bl)
\) denotes the product contour \(C(I_\bl)\times C(I_\wh)\) with the convention that whenever both contours contain small circles around a common angle \(\alpha\), the circle in \(C(I_\wh)\) has smaller radius than the one in \(C(I_\bl)\). This formula has the same general structure as many correlation kernels appearing in the dimer and integrable-probability literature: a double contour integral with a single-integral correction term whose integrand is the residue along the diagonal $\zeta=\eta$. 

\begin{theorem}[cf.~Theorem~\ref{thm:main}]\label{thm:intro_main}
The matrix \(\kast^{-1}\) defined by~\eqref{eq:intro_inverse_formula} is the two-sided inverse of Fock's Kasteleyn matrix \(\kast\) on \(\graphbeta\).
\end{theorem}

In particular, this immediately yields the following consequence which is not obvious from the definition of AZ graphs.

\begin{corollary}[cf. Corollary~\ref{cor:dimer_cover}]
Every AZ graph has equally many black and white vertices and admits a dimer cover.
\end{corollary}

{
We obtained the formula~\eqref{eq:intro_inverse_formula} experimentally, guided by the corresponding formula for the Aztec diamond in~\cite{BdT24}. The integrand has the same structure, while the contours and the single-integral correction term are more complicated. It would, however, be very interesting to find a direct derivation explaining this formula.
}

We now turn to the asymptotic consequences of Theorem~\ref{thm:intro_main} which are obtained by a steepest descent analysis of the exact formula~\eqref{eq:intro_inverse_formula}.

\subsection{Phase regions and the arctic curve}

\begin{figure}
    \centering
    \includegraphics[width=0.7\linewidth]{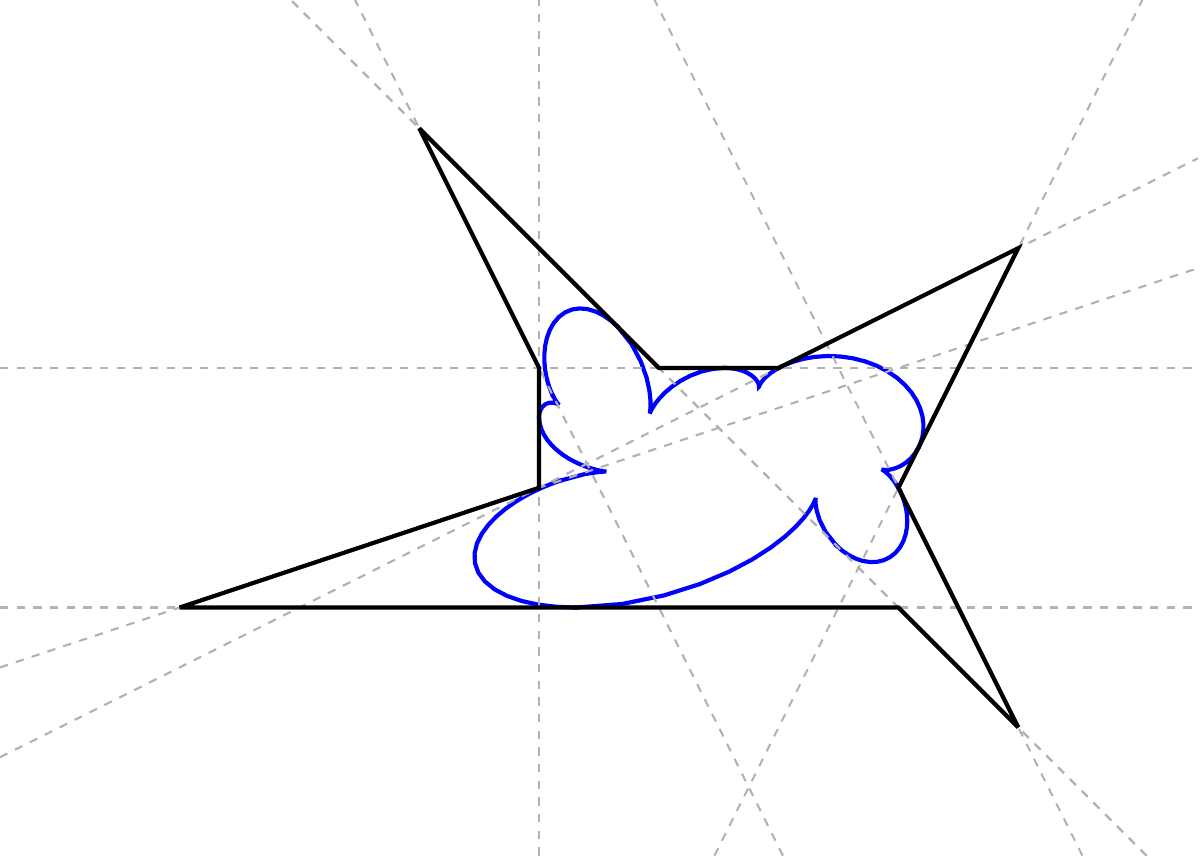}
    \caption{An astroidal domain for a decagonal Newton polygon, and an arctic curve for a choice of isoradial weights.}
    \label{fig:decagon_arctic_curve}
\end{figure}

We now consider limits of AZ graphs as the size tends to infinity. Let \(\graphbeta\) be an AZ graph. For each side \(e\in E(N)\), we label the zig-zag paths in the family \(\zz_e\) from left to right as
\[
\zz_e=\{\alpha_e^i : i\in \ZZ+\tfrac12\},
\]
choosing the indexing so that the reference face \(\f_0\) lies between \(\alpha_e^{-1/2}\) and \(\alpha_e^{1/2}\). We then define
\[
c_{\bm\beta}:E(N)\to \ZZ+\tfrac12
\]
by the requirement that
\[
\alpha_e^{ c_{\bm\beta}(e)}=\beta_e.
\]
In other words, \(c_{\bm\beta}(e)\) measures the location of the boundary zig-zag path \(\beta_e\) among the zig-zag paths parallel to \(e\).

Now consider a sequence of AZ graphs \(
\graphbetan, n=1,2,\dots, \) such that for every \(e\in E(N)\),
\[
c_{\bm\beta^n}(e)=n |e|_{\ZZ} c(e)+O(1)
\qquad\text{as }n\to\infty,
\]
for some function \(c:E(N)\to\RR\), where \(|e|_{\ZZ}\) is the lattice length of \(e\). In the limit, the admissibility condition for \(\graphbetan\) becomes
\begin{equation}\label{eq:adm_intro}
    \sum_{e\in E(N)} |e|_{\ZZ} c(e)=0.
\end{equation}

For each \(e\in E(N)\), the boundary zig-zag path \(\beta_e^n\) has as its scaling limit the line
\[
\line_e:=\{(x,y)\in \RR^2 : -b_e x+a_e y+c(e)=0\},
\qquad e\in E(N),
\]
where
\[
(a_e,b_e):=\frac{\vec e}{|e|_{\ZZ}}\in \ZZ^2
\]
is the primitive integer vector in the direction of \(e\) viewed as a vector in the plane. The resulting collection of lines \((\line_e)_{e\in E(N)}\) determines a domain
\[
\domain\subset \RR^2,
\]
called an \emph{astroidal domain}. \footnote{Rote--Santos--Streinu \cite{RSS08} called similar-looking polygons
\emph{pseudo-quadrilaterals}, in parallel to the
well-established term of \emph{pseudo-triangles}. They did not involve
the continuous version of the admissibility
condition \eqref{eq:adm_intro}, however. We are grateful to Richard Kenyon and Pavel
Galashin for pointing out the connection with
pseudo-quadrilaterals, and to Pavel Galashin for also drawing our attention
to~\cite{RSS08}.} This domain is the scaling limit of the AZ graphs \(\graphbetan\); see Figure~\ref{fig:pentagon_intro_1}(right).

In Section~\ref{sec:char_astroidal_domains}, we show that if a simple polygonal domain carries a continuous piecewise linear boundary height function whose slopes are prescribed by the Newton polygon, then it is an astroidal domain. In this sense, astroidal domains are exactly the domains compatible with the desired boundary behavior of the limit shape.

{For the asymptotic results below, we assume that the angle function is periodic. This assumption is slightly more general than
periodicity of the edge weights; see~\cite[Section~4.2]{BCT22}}. The asymptotics of the dimer model are governed by a meromorphic \(1\)-form
\[
\lambda_c(x,y)
\]
on \(\curve\), called the \emph{action \(1\)-form}, which is characterized as the unique imaginary-normalized meromorphic \(1\)-form satisfying
\[
\res_\alpha\bigl(\lambda_c(x,y)\bigr)
=
-b_{e(\alpha)}x+a_{e(\alpha)}y+c(e(\alpha)),
\qquad \text{for all}~\alpha\in \zz.
\]
The signs of the residues determine the locations of all but two zeroes of $\lambda_c(x,y)$. The two remaining zeroes are either both real or form a \(\sigma\)-conjugate pair. This leads to a phase decomposition of \(\mathcal D_c^\circ\): the \emph{rough} (or \emph{liquid}) region \(\mathcal L\), where the remaining zeroes form a \(\sigma\)-conjugate pair; the \emph{frozen} regions \(\mathcal F_{(\alpha,\beta)}\), where the remaining zeroes lie in an open interval \((\alpha,\beta)\subset A_0\) between consecutive angles with \(e(\alpha)\neq e(\beta)\); the \emph{quasi-frozen} regions \(\mathcal Q_{(\alpha,\beta)}\), where the remaining zeroes lie in an interval \((\alpha,\beta)\subset A_0\) with \(e(\alpha)=e(\beta)\); and the \emph{smooth} (or \emph{gaseous}) regions \(\mathcal G_j\), \(1\le j\le g\), where the remaining zeroes lie on the oval \(A_j\). The boundary between the rough and non-rough phases is the \emph{arctic curve}, denoted by \(\mathscr A\).

In the rough region there is a unique zero \(\zeta\in\curve_+^\circ\), and this defines a map
\[
\Omega:\mathcal L\to \curve_+^\circ,
\qquad
\Omega(x,y)=\zeta.
\]

\begin{proposition}[cf. Proposition~\ref{prop:omega_inverse}]\label{prop:omega_inverse_intro}
The map \(\Omega:\mathcal L\to \curve_+^\circ\) is a diffeomorphism. Its inverse extends smoothly to a map
\[
\Omega^{-1}:\curve_+\to\overline{\mathcal L}
\]
which restricts to a bijection
\[
\Omega^{-1}:\curve(\RR)\to\mathscr A.
\]
\end{proposition}

Thus, the real locus \(\curve(\RR)\) gives a global parametrization of the arctic curve. Figure~\ref{fig:decagon_arctic_curve} shows an example of the arctic curve plotted using this global parameterization for a decagonal Newton polygon and a choice of isoradial weights.

{
Several methods have been developed to compute arctic curves. One prominent approach is the tangent method of
Colomo--Sportiello~\cite{ColomoSportiello16} and its variants; see Di Francesco--Lapa~\cite{DiFrancescoLapa18}, Di Francesco--Guitter~\cite{DiFrancescoGuitter18,
DiFrancescoGuitter19Aztec}, Aggarwal~\cite{AggarwalTangent}, and Ruelle
~\cite{Ruelle22}. Another approach uses analytic combinatorics in several variables, particularly for dimer models arising from the
octahedron recurrence and \(T\)-systems; see Petersen--Speyer~\cite{PSgroves}, Baryshnikov--Pemantle~\cite{BaryshnikovPemantle11},
Pemantle--Wilson--Melczer~\cite{PemantleWilson13}, Di Francesco--Soto-Garrido~\cite{DiFrancescoSotoGarrido14},
Di Francesco--Vu~\cite{DiFrancescoVu24}, and references
therein. In contrast, the inverse Kasteleyn approach gives access not
only to the arctic curve, but also to the local statistics throughout
the limit shape.
}

\subsection{Limit shape}

Let \(h_M\) be the height function on \(\graphbetan\), normalized with respect to the extremal matching corresponding to a fixed vertex \(u_0\in V(N)\). If \(\f_n\) is a face whose location is \(n(x,y)+O(1)\), define
\[
\mathfrak h(x,y):=\lim_{n\to\infty}\EE \left[\frac1n h_M(\f_n)\right].
\]
For periodic dimer models the existence of this limit is guaranteed by the variational principle. In our setting of AZ graphs, we derive an explicit parameterization for $\mathfrak h(x,y)$.

Define \(A_{(x,y)}\) by
\[
A_{(x,y)}:=
\begin{cases}
(\alpha,\beta), & \text{if \((x,y)\) lies in a frozen or quasi-frozen region corresponding to }(\alpha,\beta)\subset A_0,\\
A_j, & \text{if \((x,y)\in \mathcal G_j\).}
\end{cases}
\]

\begin{theorem}[cf. Theorem~\ref{thm:limit_shape}]\label{thm:intro_limit_shape}
For \((x,y)\in \mathcal D_c^\circ\), let \(\zero=\Omega(x,y)\) if \((x,y)\in\mathcal L\), and otherwise let \(\zero\) be any point of \(A_{(x,y)}\). Then
\[
\mathfrak h(x,y)
=
-\frac{1}{2\pi\ii}\int_{C_{u_0}^{\zero}}\lambda_c(x,y),
\]
where \(C_{u_0}^{\zero}\) is a contour from $\sigma(\zero)$ to $\zero$ intersecting the component of $A_0 \setminus \zz$ corresponding to \(u_0\).
\end{theorem}

\subsection{Local statistics}

Fix a finite graph \(\mathsf C\subset\graphpl\) and translate it to a macroscopic point \((x,y)\in \mathcal D_c^\circ\) by setting
\[
\mathsf C_n=(j_n,k_n)+\mathsf C,
\qquad (j_n,k_n)=n(x,y)+O(1) \in \ZZ^2.
\]

\begin{theorem}[cf. Theorem~\ref{thm:inverse_convergence}]\label{thm:intro_local_stats}
Let \((x,y)\in \mathcal D_c^\circ\setminus \mathscr A\). In the rough region, let \(\zero=\Omega(x,y)\), and in a frozen, quasi-frozen or smooth region, let \(\zero\) be any point in the corresponding component \(A_{(x,y)}\subset \curve(\RR)\setminus \zz\). If \((x,y)\in \mathcal F_{(\alpha,\beta)}\), assume in addition that
\[
(x,y)\notin \line_{e(\alpha)}\cup \line_{e(\beta)}.
\] 
For a black vertex $\bl$ and a white vertex $\wh$ of $\compact$, let
\[
\bl_n:=\bl+(j_n,k_n),
\qquad
\wh_n:=\wh+(j_n,k_n)
\]
denote their translates in $\compact_n$. Then
\[
\kast^{-1}_{\bl_n,\wh_n}\rightarrow \mathsf A^{\zero}_{\bl,\wh}
\qquad\text{as }n\to\infty,
\]
uniformly over $\bl$ and $\wh$.
\end{theorem}

\begin{corollary}[cf. Corollary~\ref{cor:local_fluctuations}]\label{cor:local_fluctuations_intro}
The local correlations in a neighborhood of~$(x,y)$ of the dimer model on~$\graphbetan$ defined by the Kasteleyn matrix \eqref{eq:kast_fock} converge to those of the 
Gibbs measure defined by the corresponding whole-plane Kasteleyn matrix with slope~$\nabla \ls(x,y)$.
\end{corollary}
\begin{remark}
{If the edge weights of the dimer model are periodic, the Gibbs measures obtained in Corollary~\ref{cor:local_fluctuations_intro} are the translation-invariant Gibbs measures classified by Kenyon--Okounkov--Sheffield~\cite{KOS06}.}
\end{remark}

\subsection{Tropical limit and simulations}

\begin{figure}
\centering

\begin{subfigure}{0.38\textwidth}
\centering
\scalebox{1}[1.5]{\includegraphics[
  width=0.9\linewidth,
  trim=2.5mm 2.5mm 2.5mm 2.5mm,
  clip
]{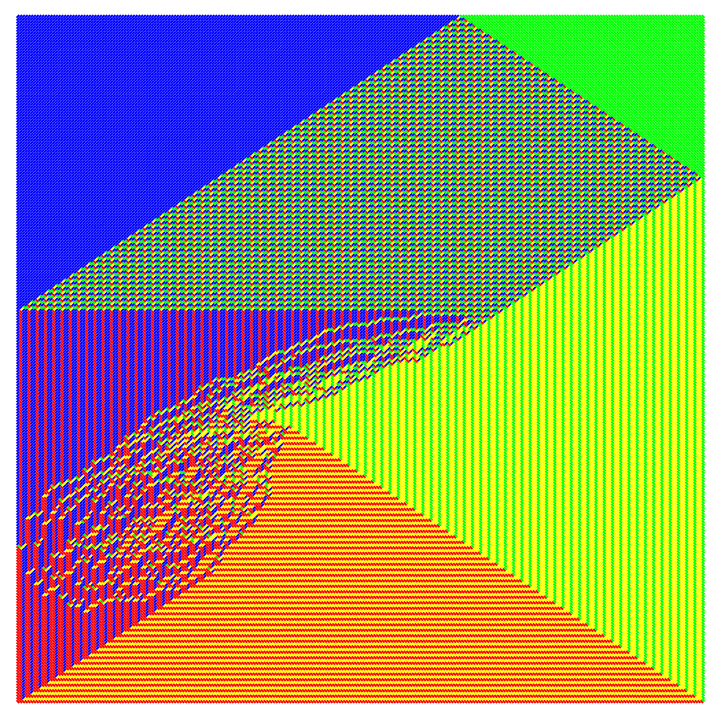}}
\end{subfigure}
\hspace{0.03\textwidth}
\begin{subfigure}{0.38\textwidth}
\centering
\includegraphics[width=0.9\linewidth]{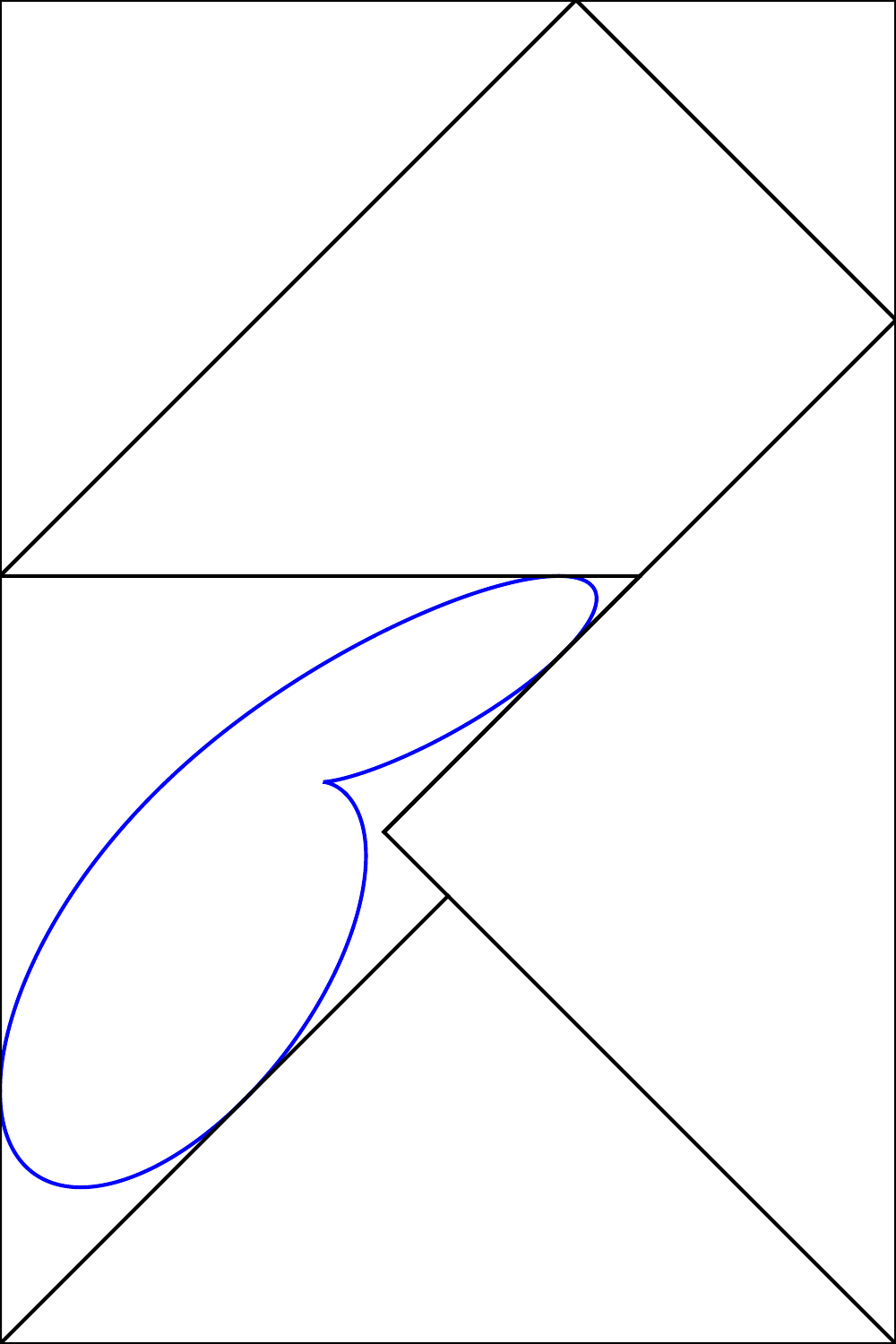}
\end{subfigure}

\caption{A simulation of the arctic curve in the tropical limit of the Aztec diamond (left) and the analytically computed arctic curve (right). The Aztec diamond has been vertically rescaled by a factor of $3/2$ to match $\mathcal D_c$.
}
\label{fig:tropical_astroidal_arctic_curve}
\end{figure}

In Appendix~\ref{appendix:tropical}, we outline, via an example, how AZ graphs can be indirectly simulated inside the Aztec diamond by choosing a suitable tropicalization in which the tropical spectral curve is not smooth. This generalizes Berggren--Borodin~\cite{BBtropical}, which treats smooth tropical curves that do not leave any dimer randomness. Figure~\ref{fig:tropical_astroidal_arctic_curve} shows one such simulation, containing an astroidal domain for the pentagonal Newton polygon in Figure~\ref{fig:pentagon_0} together with the arctic curve predicted by our theory. Understanding this phenomenon was one of the starting points for the present work, as it naturally led us to the notion of astroidal domains. Trying to understand their discrete counterparts in turn led us to AZ graphs.

\subsection{Further directions}\label{sec:further_directions}

We conclude the introduction with a few directions that we believe are promising for further study.

\begin{itemize}
\item
It might be possible to give a direct proof that the limit shape of Theorem~\ref{thm:intro_limit_shape} solves the variational problem, in the spirit of Bobenko--Bobenko~\cite{BB24} for the Aztec diamond.
\item
By construction, each edge of the Newton polygon contributes exactly one zig-zag path to the boundary of an AZ graph. However, there are natural domains that do not fit into this framework, for example lozenge tilings of a hexagon, or, more generally, the polygonal domains of Kenyon--Okounkov~\cite{KO07}, where a given parallel class of zig-zag paths may appear several times around the boundary. It would be very interesting to extend our constructions to this setting.
\item
If we drop the assumption that the zig-zag boundary of an AZ graph is simple, we obtain graphs with branching as in Lai--Musiker~\cite[Figures~46--49]{dungeons}. As the number of edges of \(N\) grows, the subset of such graphs among admissible graphs grows, and we would hope that our results extend to this setting.
\item
AZ graphs form a \((\#E(N)-3)\)-parameter family. On the other hand, the cluster modular group of generalized shufflings also gives rise to a \((\#E(N)-3)\)-parameter family, as shown by George--Inchiostro~\cite{GeorgeInchiostro} (see also Fock--Marshakov~\cite{FockMarshakov} and George--Ramassamy~\cite{GeorgeRamassamy}). It is therefore natural to expect that AZ graphs can be studied via shuffling, generalizing the relationship between domino shuffling and the Aztec diamond.
\item
Is it possible to extend the results of Berggren--Borodin~\cite{BBtropical} from the case of smooth tropical curves to the non-smooth tropical curves in Appendix~\ref{appendix:tropical}?
\item
It would be very nice to characterize which AZ graphs can be simulated inside the Aztec diamond as in Figure~\ref{fig:tropical_astroidal_arctic_curve} and Appendix~\ref{appendix:tropical}. For instance, this procedure never produces a graph with self-intersecting boundary.
\end{itemize}

\subsection*{Organization of the paper}

The paper is organized as follows. Section~\ref{sec:background} collects preliminaries on the dimer model. Section~\ref{sec:AZ_graphs} introduces AZ graphs and develops their basic combinatorics. Section~\ref{sec:inv_K} proves the exact inverse Kasteleyn formula for AZ graphs. Section~\ref{sec:scaling_limit_astroidal_domain} turns to scaling limits and introduces astroidal domains. Section~\ref{sec:phase_regions} analyzes the geometry of the phase regions and arctic curve. Section~\ref{sec:loc_stats_limit_shape} uses steepest descent analysis to derive the local statistics and the explicit formula for the limit shape. Appendix~\ref{appendix:tropical} explains how certain astroidal domains arise from a tropical limit of the Aztec diamond and provides simulations illustrating the theory.

\subsection*{Acknowledgements}
We are grateful to C\'{e}dric Boutillier, B\'{e}atrice de Tili\`ere, and Bishal Deb for communicating to us their ongoing work~\cite{BdT2}. 
We also thank C\'{e}dric Boutillier, B\'{e}atrice de Tili\`ere, and Marianna Russkikh for sharing with us the proof of Proposition~\ref{prop:telescopic}, and for allowing us to reproduce it here. {We thank Gregg Musiker for drawing our attention to further examples of AZ graphs in the literature.} A. B. was partially supported by the NSF grant DMS-2450323. A. B. and T. G. were also partially supported by the Simons Investigator program.


\section{Preliminaries}\label{sec:background}

\subsection{Conventions}\label{sec:conventions}

Since we will simultaneously refer to vertices and edges of both the graph and
the Newton polygon, we use the usual font for objects associated with the
polygon and the sans-serif font for those associated with
the graph. Whenever it makes sense, we write
\[
V(\cdot),\quad E(\cdot),\quad F(\cdot)
\]
for the sets of vertices, edges, and faces of the object under consideration
(for example, for graphs, polygons, strands, etc).

Given a lattice polygon $N \subset \RR^2$, we orient its boundary $\partial N$ counterclockwise. {This induces a cyclic order on $\partial N$: for $u,v,w \in V(N)$, we write $u<v<w$ if they appear in counterclockwise order around $\partial N$.} By an \emph{edge}
of $N$ we mean a side of the polygon (possibly containing interior lattice
points), while a \emph{vertex} always means a corner. {For \(a,b\in \partial N\), we write
\[
I=[a,b]
\]
for the closed cyclic interval in \(\partial N\) from \(a\) to \(b\), oriented counterclockwise.}

We use the subscript $-$ (resp. $+$) to indicate the object immediately clockwise (resp. counterclockwise) from a given one.  For example, when we write
\[
v = e_- \cap e_+,
\]
the edge $e_-$ (resp. $e_+$) is the edge immediately clockwise (resp. counterclockwise) from $v$.  Similarly, for an oriented edge we write
\[
e=[v_-,v_+],
\]
where $v_-$ and $v_+$ are its clockwise and counterclockwise endpoints. We freely identify edges of $N$ with the corresponding closed intervals in $\partial
N$ and with their displacement vectors in $\mathbb{R}^2$.  Thus we may write an
edge either as $e=[v_-,v_+]$ or as the vector $\vec{e}$.

For a side \(e \in E(N)\), let \(|e|_{\ZZ}\) denote the \emph{integral length} of the vector \(e\), \emph{i.e.} the number of primitive lattice segments contained in \(e\), or equivalently \(|e \cap \ZZ^2|-1\). Let
\[
(a_e,b_e):= \frac{\vec e}{|e|_\ZZ} \in \ZZ^2
\]
denote the primitive vector along \(\vec e\).

\subsection{The dimer model}

Let $\graph$ be a planar bipartite graph. We denote the black vertices, white vertices, edges and faces of $\graph$ by $B(\graph), W(\graph), E(\graph)$ and $F(\graph)$ respectively. If an edge $\e \in E(\graph)$ is incident to a black vertex $\bl \in B(\graph)$ and a white vertex 
$\wh \in W(\graph)$, we write $\e=\bl\wh$. A \emph{dimer cover} on $\graph$ is a subset of $E(\graph)$ such that every vertex is used exactly once. 
 
An \emph{edge weight} on $\graph$ is a function $\wt : E(\graphpl) \rightarrow \RR_{>0}$. Two edge weights $\wt_1$ and $\wt_2$ are said to be \emph{gauge equivalent} if there is a function $\gauge:B(\graphpl) \sqcup W(\graphpl) \rightarrow \RR_{>0}$ such that for every edge $\e=\bl\wh$, we have $\wt_2(\e) = \gauge(\bl) \wt_1(\e) \gauge(\wh)$. 

When $\graph$ is finite, the \emph{Boltzmann measure} associated to an edge weight $\wt$ assigns to each dimer cover $M$ the probability
\[
\mathbb P(M) := \frac{\wt(M)}{Z},
\]
where $\wt(M):=\prod_{\e \in M} \wt(\e)$ and $Z:=\sum_{\text{dimer covers $M$}} \wt(M)$ is the \emph{partition function}. Gauge equivalent edge weights define the same Boltzmann measure. 

Motivated by this, we define a \emph{dimer model} to be a pair $(\graph,[\wt])$, where $\graph$ is a (possibly infinite) planar bipartite graph and $[\wt]$ is a gauge equivalence class of edge weights.

\subsubsection{The height function}\label{sec:height_function}

Let $\graphpl$ be a planar bipartite graph. Each dimer cover $M$ can be viewed as a white-to-black flow $\omega_M$ on $\graph$ that sends $1$ along each edge of $M$ and $0$ along all other edges. Fix a reference face $\f_0$ and a reference dimer cover $M_0$. Following Thurston~\cite{Thu90}, we define a function 
\[
h_M:F(\graph) \rightarrow \ZZ,
\]
called the \emph{height function} of $M$, as follows. For any face $\f$, let $\gamma^*$ be a path in the dual graph $\graphpl^*$ of $\graphpl$ from $\f_0$ to $\f$ and let 
\[
h_M(\f):=(M-M_0) \wedge \gamma^* = \sum_{\e = \wh \bl } (\e \wedge \gamma^*)(\mathbf 1_{ \{\e \in M\}} - \mathbf{1}_{ \{\e \in M_0 \}}),
\]  
where $\wedge$ denotes the intersection pairing in the plane with the convention that 
\begin{equation}\label{eq:intersection_convention}
     \e \wedge \gamma^* = \begin{cases} -1, &\text{if $\gamma^*$ crosses $\e$ from left to right},\\
1, &\text{if $\gamma^*$ crosses $\e$ from right to left}.
\end{cases}
\end{equation}
The height function allows us to interpret a random dimer cover of $\graphpl$ as a random discrete surface, which is sometimes a more suitable object for studying scaling limits.

\subsubsection{The Kasteleyn matrix}

The fundamental tool in the study of the dimer model is the Kasteleyn matrix which we now introduce. A \emph{Kasteleyn sign} is a function $\kappa: E(\graph) \rightarrow \CC$ such that $|\kappa(\e)|=1$ for all $\e \in E(\graph)$ and the alternating product of $\kappa$ around the boundary $\partial \f$ of each face $\f \in F(\graph)$ is $(-1)^{\frac{|\partial \f|}{2}+1}$. The \emph{Kasteleyn matrix} is the $W(\graph) \times B(\graph)$ weighted signed bipartite adjacency matrix defined by 
\[
\kast_{\wh,\bl} := \sum_{\e = \bl \wh} \wt(\e) \kappa(\e).
\]

\begin{theorem}[\cite{Kas61,TF61}]
If $\graph$ is finite, the partition function satisfies
\[
Z = \lvert \det \kast \rvert .
\]
\end{theorem}

A fundamental consequence is that the inverse Kasteleyn matrix encodes all local statistics of the dimer model.

\begin{theorem}[\cite{Ken97}] \label{thm:kenyon_inv_K}
    Let $\graph$ be finite and let $\{\e_i = \bl_i \wh_i\}_{i=1}^k \subset E(\graph)$ be a fixed collection of its edges. Then the probability that all $\e_i$ occur in a random dimer configuration $M$ is 
    \[
    \mathbb P(\e_1,\dots, \e_k \in M) = \prod_{i=1}^k \kast_{\wh_i, \bl_i} \det(\kast^{-1}_{\bl_i,\wh_j})_{1 \leq i,j \leq k}.
    \]
\end{theorem}

\subsubsection{Minimal periodic bipartite graphs} \label{sec:minimal_graphs}

Let $\graphpl \subset \RR^2$ be a $\ZZ^2$-periodic graph embedded in the plane such that every face is a topological disk. Let $\torus := \RR^2/\ZZ^2$ be the torus, $p: \RR^2 \rightarrow \torus$ be the quotient map and let $\graphtor:=\graphpl/\ZZ^2$ be the corresponding graph in $\torus$.

A \emph{zig-zag path} in $\graphpl$ is an oriented path that turns maximally right at white vertices and maximally left at black vertices. Equivalently, zig-zag paths can be viewed as certain paths in the medial graph. Recall that the (directed) medial graph $\graphpl^\times$ has a vertex at the midpoint of each edge of $\graphpl$, and an edge whenever two edges of $\graphpl$ are consecutive around a face. Thus we have a bijection
\[
V(\graphpl^\times) \cong E(\graphpl),
\]
and we will use this identification hereafter without comment. We orient the edges of $\graphpl^\times$ so that they go clockwise around black vertices and counterclockwise around white vertices. Every vertex of $\graphpl^\times$ has degree four, and the cyclic order of the incident edges determines two pairs of opposite edges. A \emph{(medial) strand} is a directed path in $\graphpl^\times$ with the property that whenever the path enters a vertex, it exits along the opposite edge (\emph{i.e.}, it goes straight through the vertex). The bijection with zig-zag paths is obtained as follows. Given a zig-zag path in $\graphpl$, record the sequence of edges it traverses; the corresponding strand is the path in $\graphpl^\times$ visiting the midpoints of these edges in order. The maximal-turn rule at black and white vertices is equivalent to the condition that the strand goes straight at each medial vertex. Zig-zag paths and strands in $\graphtor$ are defined in the same way. We say that a vertex \emph{lies on} a strand if it lies on the corresponding zig-zag path.

We say that $\graphpl$ is \emph{minimal} if:
\begin{enumerate}
    \item no strand is a closed loop,
    \item no strand has a self-intersection,
    \item there is no pair of strands $\alpha,\beta$ (or, rather, zig-zag paths) that have two edges $\e_1 \neq \e_2 \in E(\graphpl)$ in common such that both of them are oriented from $\e_1$ to $\e_2$.
\end{enumerate}
Minimal graphs have no multiple edges, leaves or isolated vertices. Informally, the strands in a minimal graph ``behave like straight lines''.

Every strand $\alpha$ in $\graphtor$ is a cycle and therefore, has a homology class $[\alpha] \in H_1(\torus,\ZZ) \cong \ZZ^2$. The unique convex lattice polygon $N \subset \RR^2$, defined modulo translation by $\ZZ^2$, whose sides consist of the vectors $[\alpha] \in \ZZ^2$ for all strands $\alpha$ in $\graphtor$ in counterclockwise cyclic order is called the \emph{Newton polygon} of $\graphpl$. This construction produces a closed polygon because the homology classes of all strands in $\graphtor$ sum to zero: a strand has the same homology class as the corresponding zig-zag path and every edge of the graph is traversed by two zig-zag paths with opposite orientations.

\begin{theorem}[{\cite[Theorem~2.5]{GK13}}]
For any convex lattice polygon $N \subset \RR^2$, there is a nonempty family of minimal graphs with Newton polygon $N$. Moreover, any two graphs in the same family are related by sequences of the two types of local moves shown in Figure~\ref{fig:moves:X}.
\end{theorem}

\begin{figure}
\centering

\begin{tikzpicture}[scale=0.6]

\begin{scope}[shift={(-7,0)}]

\begin{scope}[shift={(-3,0)},rotate=135]
  \def\r{2};

  \coordinate[wvert] (n1) at (0:\r);
  \coordinate[wvert] (n2) at (90:\r);
  \coordinate[wvert] (n3) at (180:\r);
  \coordinate[wvert] (n4) at (270:\r);
  \coordinate[bvert] (b1) at (0:0.5*\r);
  \coordinate[bvert] (b2) at (180:0.5*\r);

  \draw[-]
    (n1) -- (b1) -- (n2) -- (b2) -- (n4) -- (b1)
    (b2) -- (n3);
\end{scope}

\begin{scope}[shift={(3,0)},rotate=45]
  \def\r{2};

  \coordinate[wvert] (n1) at (0:\r);
  \coordinate[wvert] (n2) at (90:\r);
  \coordinate[wvert] (n3) at (180:\r);
  \coordinate[wvert] (n4) at (270:\r);
  \coordinate[bvert] (b1) at (0:0.5*\r);
  \coordinate[bvert] (b2) at (180:0.5*\r);

  \draw[-]
    (n1) -- (b1) -- (n2) -- (b2) -- (n4) -- (b1)
    (b2) -- (n3);
\end{scope}

\draw[<->] (-0.9,0) -- (0.9,0);

\end{scope}

\begin{scope}[shift={(7,0)}]

\begin{scope}[shift={(3,0)}, rotate=-45]
  \def\r{2};

  \coordinate[wvert] (n1) at (0:\r);
  \coordinate[wvert] (n2) at (90:\r);
  \coordinate[wvert] (n3) at (180:\r);
  \coordinate[wvert] (n4) at (270:\r);
  \coordinate[bvert] (n5) at (0,0);

  \draw[] (n5)--(n1) (n5)--(n2) (n5)--(n3) (n5)--(n4);
\end{scope}

\begin{scope}[shift={(-3,0)},rotate=-45]
  \def\r{2};

  \coordinate[wvert] (n1) at (0:\r);
  \coordinate[wvert] (n2) at (90:\r);
  \coordinate[wvert] (n3) at (180:\r);
  \coordinate[wvert] (n4) at (270:\r);
  \coordinate[wvert] (n5) at (0,0);
  \coordinate[bvert] (b1) at (-0.5,0.5);
  \coordinate[bvert] (b2) at (0.5,-0.5);

  \draw[-]
    (n1)--(b2)--(n4)
    (n2)--(b1)--(n3)
    (b1)--(n5)--(b2);
\end{scope}

\draw[<->] (-0.9,0) -- (0.9,0);

\end{scope}

\end{tikzpicture}

\caption{Local moves: spider move (left) and contraction-uncontraction move (right).}
\label{fig:moves:X}
\end{figure}
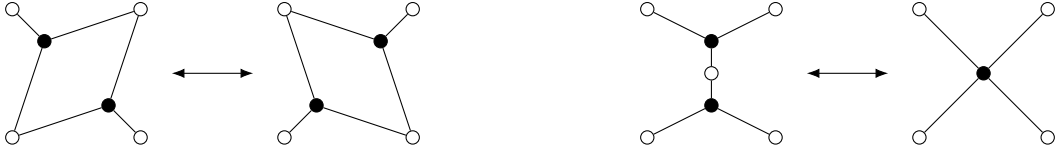

These local moves do not affect the statistical-mechanical properties of the model like the limit shape or local limits.

We denote by \(\zz\) (resp. \(\zz_\torus\)) the set of strands of \(\graphpl\) (resp. \(\graphtor\)). 
Each strand \(\alpha\in\zz\) projects to a strand \(p(\alpha)\in\zz_\torus\). 
To each strand \(\alpha\in\zz\), we associate the unique edge \(e(\alpha)\in E(N)\) containing the homology class \([p(\alpha)]\). 
We say that \(\alpha\) is \emph{parallel to} \(e(\alpha)\). 
For each edge \(e\in E(N)\), let
\[
\zz_e:=\{\alpha\in\zz : e(\alpha)=e\}
\]
be the set of strands parallel to \(e\). The counterclockwise orientation of $\partial N$ induces a partial cyclic order on $\zz$ with respect to which two members of $\zz_e$ for any $e$ are incomparable.

\begin{lemma}[{\cite[Lemma~8]{BCdTiso}}; see also {\cite[Lemma~3.4]{GK13}}]
\label{lem:cyclic_order_same}
Let $\graph$ be a minimal graph, and let $\x\in B(\graph)\sqcup W(\graph)$ be a vertex. Then the strands that $\x$ lies on are parallel to distinct edges of $N$. Moreover, the cyclic order of these strands (or the corresponding zig-zag paths) around $\x$ agrees with the cyclic order of the corresponding edges along $\partial N$ if $\x$ is white, and with the opposite cyclic order if $\x$ is black.
\end{lemma}

\subsection{M-curves} \label{sec:M-curves}
{
We briefly recall the basic material on M-curves that we will use, following the notation and conventions of \cite{BCT22}.
}

An \emph{$M$-curve} is a compact Riemann surface $\curve$ of genus $g$
together with an anti-holomorphic involution
$\sigma:\curve \to \curve$ such that the fixed-point set
\[
\curve(\RR) := \{ \zeta \in \curve : \sigma(\zeta)=\zeta \},
\]
called the \emph{real locus}, consists of $g+1$ connected components called \emph{ovals}, each homeomorphic to a circle. We denote these components by $A_0,\dots,A_g$. The real locus separates $\curve$ into two connected components
\[
\curve \setminus \curve(\RR) = \curve_+^\circ \sqcup \curve_-^\circ .
\]
Let
\(
\curve_\pm := \overline{\curve_\pm^\circ}
\) denote their closures. We fix the orientations of the ovals $A_j$ so that
\[
\partial \curve_+ = A_0 - \sum_{j=1}^g A_j .
\]

Let $B_1,\dots,B_g$ be dual cycles such that $\sigma_* B_j=-B_j$ and 
\[
A_j \wedge B_k = \delta_{jk}
\]
where $\wedge : H_1(\curve,\ZZ) \times H_1(\curve,\ZZ) \rightarrow \ZZ$ is the intersection pairing. Let $\omega_1,\dots,\omega_g$ be the dual basis of holomorphic 1-forms on $\curve$ such that 
\[
\int_{A_j} \omega_k = \delta_{jk}.
\]
It is known as the basis of normalized holomorphic $1$-forms.
\subsubsection{Abel map and Jacobian} \label{sec:abel_map_jacobian}

The $g \times g$-matrix $\Omega$ defined by $\Omega_{jk} = \int_{B_j} \omega_k$ is called the \emph{period matrix}. For an M-curve, the matrix $\Omega$ is purely imaginary. The complex torus 
\[
\jac(\curve) := \CC^g/(\ZZ^g \oplus \Omega \ZZ^g)
\]
is called the \emph{Jacobian variety} of $\curve$. 

The \emph{Abel map} with basepoint $\zeta_0 \in A_0$ is the map
\[
\mu : \curve \rightarrow \jac(\curve),
\qquad
\mu(\zeta) := \Big(\int_{\zeta_0}^{\zeta} \omega_j\Big)_{j=1}^g .
\]
The Abel map is independent of the path from $\zeta_0$ to $\zeta$ because the integrals of $\omega_j$ along cycles are contained in $\ZZ^g \oplus \Omega \ZZ^g$. 

A \emph{divisor} on $\curve$ is a finite formal $\ZZ$-linear combination of points of $\curve$,
\[
D = \sum_{j=1}^n a_j \zeta_j .
\]
Its \emph{degree} is
\[
\deg(D) := \sum_{j=1}^n a_j .
\]
If $s$ is a meromorphic section of a holomorphic line bundle on $\curve$, we define its divisor
\[
\divisor(s) := \sum_{\zeta \in \curve} \ord_\zeta(s) \zeta,
\]
where $\ord_\zeta$ denotes the order of vanishing at $\zeta$. 

We extend the Abel map to divisors by linearity,
\[
\mu(D) := \sum_{j=1}^n a_j  \mu(\zeta_j).
\]
{
By abuse of notation, we will henceforth suppress the Abel map from the notation, writing \(\zeta\) for \(\mu(\zeta)\) and similarly \(D\) for its image \(\mu(D)\in \jac(\curve)\).}

\subsubsection{Theta function and prime form}
The \emph{Riemann theta function} is the multi-valued function on $\jac(\curve)$ defined by
\[
\theta(\zeta):=\sum_{n \in \ZZ^g} e^{\ii \pi \langle n,\Omega n\rangle + 2 \pi \ii \langle n,\zeta\rangle}.
\]
Pulling back via the Abel map, we obtain a multi-valued function on \(\curve\). {By the same abuse of notation as in Section~\ref{sec:abel_map_jacobian}, we will usually suppress the Abel map and write \(\theta(\zeta)\) for \(\theta(\mu(\zeta))\).}

The \emph{prime form} $E(\zeta,\eta)$ is a holomorphic $(-\frac 1 2,-\frac 1 2)$-form on $\curve \times \curve$ that is antisymmetric, vanishes simply along the diagonal with the expansion
\[
E(\zeta,\eta) = \frac{\eta - \zeta}{\sqrt{d\zeta} \sqrt{d\eta}} (1+ O((\eta - \zeta)^2)),
\]
and has no other zeroes and poles.

\subsubsection{Meromorphic differentials of the third kind}\label{sec:third_kind_differentials}

We follow \cite[Chapter~I]{Fay73}. A \emph{meromorphic differential of the third kind} on \(\curve\) is a meromorphic \(1\)-form all of whose poles are simple. Although Fay~\cite{Fay73} states the construction for integral divisors, all formulas are linear in the coefficients, so they extend immediately to degree-zero divisors with real coefficients. If
\[
D=\sum_{j=1}^n a_j \zeta_j
\]
is a degree-zero real divisor, meaning that \(a_j\in \RR\) and
\[
\sum_{j=1}^n a_j=0,
\]
then, after fixing a basepoint \(\zeta_0\in \curve\), there exists a unique meromorphic \(1\)-form \(\omega_D\) with vanishing \(A\)-periods,
\[
\int_{A_k}\omega_D=0, \qquad 1\le k\le g,
\]
whose only singularities are simple poles at the points \(\zeta_j\), with residues
\[
\res_{\zeta_j}(\omega_D)=a_j .
\]
It is given by
\begin{equation}\label{eq:omega_D_def}
\omega_D(\zeta)
:=
d \sum_{j=1}^n a_j
\log \left(\frac{E(\zeta_j,\zeta)}{E(\zeta_0,\zeta)}\right);
\end{equation}
here $d$ is the differential. By the \emph{Riemann bilinear relations}, for \(1\le k\le g\),
\begin{equation}\label{eq:Riemann_bilinear}
\int_{B_k}\omega_D
=
2\pi \ii \sum_{j=1}^n a_j \int_{\zeta_0}^{\zeta_j}\omega_k .
\end{equation}

\subsubsection{Imaginary-normalized differentials}\label{sec:ind}

Instead of normalizing a meromorphic differential of the third kind to have vanishing $A$-periods, we may normalize it by requiring all of its periods to be purely imaginary. We say that a meromorphic $1$-form is \emph{imaginary-normalized} if 
\[
\Re \int_\gamma \omega=0
\]
for every cycle $\gamma$ in $\curve$.

\begin{lemma}[{\cite[Lemma~2.1]{Kri13}}]\label{lem:unique_ind}
Let $D=\sum_{j=1}^n a_j \zeta_j$ be a degree-$0$ real divisor on \(\curve\). Then there exists a unique imaginary-normalized meromorphic \(1\)-form on \(\curve\) whose only singularities are simple poles at the points \(\zeta_j\) with residues \(a_j\).
\end{lemma}

Recall the meromorphic $1$-forms $\omega_D$ associated to degree-$0$ divisors $D$ with real coefficients defined in Section~\ref{sec:third_kind_differentials}. We now show that, when the support of \(D\) lies on \(A_0\), this differential is also imaginary-normalized.

\begin{lemma}\label{lem:third_kind_ind}
For any degree-$0$ real divisor $D=\sum_{j=1}^n a_j \zeta_j$ such that $\zeta_j \in A_0$, the $1$-form $\omega_{D}$ is imaginary-normalized and satisfies $\overline{\sigma^* \omega_{D} } = \omega_{D}$; in particular, $\omega_D$ is real-valued on $\curve(\RR)$.
\end{lemma}
\begin{proof}
Since the $A$-periods of $\omega_D$ vanish by definition, it remains only to show that the $B$-periods are purely imaginary. By the Riemann bilinear relations~\eqref{eq:Riemann_bilinear},
\[
\int_{B_k}\omega_D
=
2\pi \ii \sum_{j=1}^n a_j \int_{\zeta_0}^{\zeta_j}\omega_k,
\]
which is purely imaginary since $\int_{\zeta_0}^{\zeta_j} \omega_k \in \RR$ for any $\zeta_j \in A_0$ by \cite[Lemma~14]{BCT22} as $\zeta_0 \in A_0$.

The $1$-form $\overline{\sigma^* \omega_{D}}$ is imaginary-normalized since 
\[
\Re \int_\gamma \overline{ \sigma^* \omega_{D}} = \Re \int_{\sigma_*(\gamma)} \omega_{D} = 0.
\]
Since $A_0$ is fixed by $\sigma$, $\overline{\sigma^*\omega_{D}}$ has the same poles as $\omega_{D}$. To see that the residues are also the same, let $C$ be a small counterclockwise-oriented circle around $\alpha$. Then $\sigma_* C = -C$, so
\[
\res_{\alpha}(\overline{ \sigma^*\omega_{D}}) = \frac{1}{2 \pi \ii} \int_C \overline{ \sigma^* \omega_{D}} = \frac{1}{2 \pi \ii} \overline{\int_{-C} \omega_{D}} = \res_{\alpha}(\omega_{D}).
\]
By Lemma~\ref{lem:unique_ind}, $\overline{\sigma^* \omega_{D}} = \omega_{D}$.
\end{proof}

Throughout the paper, we will count zeroes with multiplicity.

\begin{lemma}\label{lem:zeroes_of_omega_D}
Let \(D=\sum_{j=1}^n a_j \zeta_j\) be a degree-\(0\) real divisor such that $a_j\neq0$,
\(\zeta_j\in A_0\) for all \(j\), and index the points
\(\zeta_1,\dots,\zeta_n\) cyclically along \(A_0\). Then the differential
\(\omega_D\) satisfies:
\begin{enumerate}
    \item for each \(1\le k\le g\), the number of zeroes of \(\omega_D\) on \(A_k\) is even and at least \(2\),

    \item for each interval \((\zeta_j,\zeta_{j+1})\subset A_0\), the number of zeroes
    of \(\omega_D\) in \((\zeta_j,\zeta_{j+1})\) is odd
    if \(\res_{\zeta_j}(\omega_D)\) and \(\res_{\zeta_{j+1}}(\omega_D)\) have the same sign,
    and even if they have opposite signs. In particular, \(\omega_D\) has at least one
    zero in \((\zeta_j,\zeta_{j+1})\) whenever the residues at \(\zeta_j\) and
    \(\zeta_{j+1}\) have the same sign.
\end{enumerate}
\end{lemma}

\begin{proof}
Both assertions are consequences of the real-valuedness of \(\omega_D\) on
\(\curve(\RR)\) (Lemma~\ref{lem:third_kind_ind}).

Let \(1\le k\le g\). Since \(\omega_D\) is imaginary-normalized, we have
\[
\int_{A_k}\omega_D=0.
\]
Since \(\omega_D\) is not identically zero on \(A_k\), it cannot have constant sign on
\(A_k\). Thus, it must change sign, which forces the existence of at least two zeroes on
\(A_k\). Since sign changes correspond to zeroes of odd multiplicity and the total number of sign changes is even, the total number of zeroes counted with multiplicity is even.

Now consider an interval \((\zeta_j,\zeta_{j+1})\subset A_0\). Since \(\omega_D\) is
real-valued on this interval, the parity of the number of zeroes is determined by its
sign near the endpoints. Near \(\zeta_j\), the sign of \(\omega_D\) on the interval is
the sign of \(\res_{\zeta_j}(\omega_D)\), while near \(\zeta_{j+1}\), it is the opposite
of the sign of \(\res_{\zeta_{j+1}}(\omega_D)\). Therefore, these endpoint signs are
opposite precisely when \(\res_{\zeta_j}(\omega_D)\) and
\(\res_{\zeta_{j+1}}(\omega_D)\) have the same sign. In that case the number of zeroes
in \((\zeta_j,\zeta_{j+1})\) is odd, hence, at least one. If the residues have opposite
signs, then the endpoint signs agree, so the number of zeroes is even.
\end{proof}

\subsection{Fock's dimer model } \label{sec:focks_dimer_model}

In this subsection we recall Fock's construction of a dimer model from data associated to an M-curve following~\cite{Foc15, BCT22}.

\subsubsection{Angle functions}

Let $\curve$ be an M-curve. We call a function $\angle:\zz \rightarrow A_0$ an \emph{angle} function if it is compatible with the partial cyclic order on $\zz$ and the total cyclic order on $A_0$, and it maps non-parallel strands to distinct points. By a customary abuse of notation, we identify the angle $\nu(\alpha)$ with the strand $\alpha$.

\subsubsection{The discrete Abel map} 
\label{sec:d_abel_map}

Let \(\curve\) be an M-curve and let
\(
\nu
\)
be an angle function. Via \(\nu\), we henceforth identify \(\ZZ^{\zz}\) with the group of divisors supported on \(\nu(\zz)\subset A_0\). In particular, for \(\bm D\in \ZZ^{\zz}\), we define its \emph{degree} to be the degree of the corresponding divisor.

For \(\bm D\in \ZZ^{\zz}\) and \(e\in E(N)\), let
\[
\bm D_{\zz_e}\in \ZZ^{\zz_e}
\]
denote the restriction of \(\bm D\) to \(\zz_e\), so that
\[
\bm D=\sum_{e\in E(N)} \bm D_{\zz_e}.
\]
For \(\alpha\in\zz\), let
\[
\bm D_\alpha\in \ZZ^\alpha\cong \ZZ
\]
denote the restriction of \(\bm D\) to \(\alpha\); we identify this with the coefficient of \(\alpha\) in \(\bm D\).

The \emph{discrete Abel map} \[\dabel: W(\graphpl) \sqcup F(\graphpl) \sqcup B(\graphpl) \rightarrow \ZZ^{\zz}\] is defined as follows:
\begin{enumerate}
    \item Choose a reference face $\f_0$ and set $\dabel(\f_0) =0$.
    \item For a white vertex $\wh$ incident to a face $\f$ such that they are separated by a strand $\alpha$, we have $\dabel(\wh) = \dabel(\f)-\alpha$.
    \item For a black vertex $\bl$ incident to a face $\f$ such that they are separated by a strand $\alpha$, we have $\dabel(\bl) = \dabel(\f)+\alpha$.
\end{enumerate}
Equivalently, 
\begin{equation} \label{eq:dabel_intersection_version}
\dabel(\y) - \dabel(\x) = \sum_{\alpha \in \zz} (\gamma \wedge \alpha) \alpha,
\end{equation} 
where $\gamma$ is any path from $\x$ to $\y$ that intersects all strands transversely and $\wedge$ denotes the intersection pairing in the plane with the same convention as in~\eqref{eq:intersection_convention}. 
By definition, $\deg(\dabel (\x))$ is $-1$ if $\x$ is a white vertex, $0$ if $\x$ is a face and $1$ if $\x$ is a black vertex.

\subsubsection{Fock's Kasteleyn matrix}\label{sec:fock's_kast}

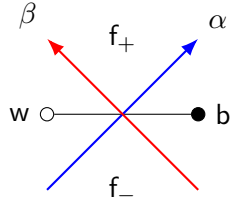
\begin{figure}
\centering
\begin{tikzpicture}
    \coordinate[wvert, label = left:$\wh$] (w) at (0,0);
    \coordinate[bvert, label = right:$\bl$] (b) at (2,0);
    \node at (1,1) {$\f_+$};
    \node at (1,-1) {$\f_-$};
    \draw[] (b) -- (w);
    \draw[blue, ->,thick] (0,-1) -- (2,1);
    \draw[red,->,thick] (2,-1) -- (0,1);
    \coordinate[label=45:{$\alpha$}] (no) at (2,1);
     \coordinate[label=135:{$\beta$}] (no) at (0,1);
\end{tikzpicture}
\caption{Local configuration near an edge in the definition of Fock's Kasteleyn matrix $\kast$.}\label{fig:fock_kast}
\end{figure}

Given an M-curve $\curve$, a point $t \in \RR^g/\ZZ^g \subset \jac(\curve)$ and an angle function $\nu$, \emph{Fock's Kasteleyn matrix} $\kast$ is defined by 
\begin{equation}\label{eq:kast_fock}
    \kast_{\wh, \bl} := \begin{cases}\dfrac{E(\alpha,\beta)}{\theta(t+\dabel(\f_-)) \theta(t+\dabel(\f_+))}, &\text{if $\e = \bl \wh$ is an edge},\\
0, & \text{otherwise},
\end{cases}
\end{equation}
where $\alpha,\beta$ are the two strands containing $\e$, and $\f_-,\f_+$ are the two faces incident to $\e$ (Figure~\ref{fig:fock_kast}).

\begin{proposition}[{\cite[Proposition~3]{BCT22}}]
    Fock's Kasteleyn matrix is the Kasteleyn matrix of a dimer model $(\graph,[\wt])$.
\end{proposition}

\subsubsection{Forms in the kernel of $\kast$} \label{sec:form_g}

Following~\cite{Foc15} (see also \cite{KO06}), we define some differential forms on $\curve$ that will be part of the integrand in the integral formula for $\kast^{-1}$. When $\mathsf x \in W(\graphpl) \sqcup F(\graphpl) \sqcup B(\graphpl)$, define $g_{\mathsf x}$ as follows: 
\begin{enumerate}
    \item Set $g_{\f_0}(\zeta) \equiv 1$.
    \item For a white vertex $\wh$ incident to a face $\f$ such that they are separated by a strand $\alpha$, we have \[
    g_{\wh}(\zeta) =  \frac{\theta(t+\zeta+\dabel(\wh))}{E(\alpha,\zeta)} g_{\f}(\zeta).
    \]
    \item For a black vertex $\bl$ incident to a face $\f$ such that they are separated by a strand $\alpha$, we have \[
    g_{\bl}(\zeta) =  \frac{E(\alpha,\zeta)}{\theta(-t+\zeta-\dabel(\bl))} g_{\f}(\zeta).
    \]
\end{enumerate}
More generally, for $\mathsf x, \mathsf y \in W(\graphpl) \sqcup F(\graphpl) \sqcup B(\graphpl)$, define
\begin{equation} \label{eq:g_prod}
g_{\mathsf x,\mathsf y} := \frac{g_{\mathsf y}}{g_{\mathsf x}}.
\end{equation}
In particular, $g_{\mathsf x} = g_{\f_0,\mathsf x}$.

\begin{lemma}[{\cite[Lemma~13]{BCT22}}; see also \cite{KO06}] \label{lem:ker_K}
 For any white vertex $\wh \in W(\graphpl)$, we have
    \begin{equation}
            \sum_{\bl \sim \wh} \kast_{\wh,\bl} g_{\bl,\x}=0
    \end{equation}
    for any $\x \neq \wh  \in W(\graphpl) \sqcup F(\graphpl) \sqcup B(\graphpl)$.
\end{lemma}

\begin{lemma}[{\cite[Proposition~2.10]{Fay73} and~\cite[Equation~(28)]{BCT22}}] \label{lem:fays}
Let $\wh\in W(\graphpl)$ and $\bl\in B(\graphpl)$ be adjacent white and black vertices. Then
\begin{equation}
\kast_{\wh,\bl} g_{\bl,\wh}=\omega_{\beta-\alpha}
+\sum_{k=1}^g\left(\frac{\partial \log \theta}{\partial z_k}(t+\dabel(\f_+))-\frac{\partial \log \theta}{\partial z_k}(t+\dabel(\f_-))\right)\omega_k,
\end{equation}
where $\alpha$ and $\beta$ are the strands containing the edge $\bl\wh$, and $\f_+$ and $\f_-$ are the adjacent faces as in Figure~\ref{fig:fock_kast}.  
\end{lemma}

We now describe the structure of the zeroes and poles of $g_{\x,\y}$ along $A_0$.

\begin{lemma}[{\cite[Section~3.4]{BCT22}}]\label{lem:div_g}
Let $\x,\y \in W(\graphpl) \sqcup F(\graphpl) \sqcup B(\graphpl)$. Then \[\divisor(g_{\x,\y})|_{A_0} = \dabel(\y)-\dabel(\x).\]
\end{lemma}

\begin{lemma} \label{lem:parallel_zz_same_sign}
Let $\x,\y \in W(\graphpl) \sqcup F(\graphpl) \sqcup B(\graphpl)$. Then $g_{\x,\y}$ cannot have both zeroes and poles at angles in $\zz_e$ for any $e \in E(N)$. 
\end{lemma}
\begin{proof}
By Lemma~\ref{lem:div_g}, we just need to check that for any $e \in E(N)$, there is a path $\gamma$ from $\x$ to $\y$ whose intersections with $\alpha \in \zz_e$ all have the same orientation. By minimality, the strands in $\zz_e$ are pairwise disjoint and hence, they divide the plane into an ordered sequence of
strips between consecutive strands. Among all paths from $\x$ to $\y$, choose $\gamma$ transverse to those strands and minimizing the number of intersections with
$\zz_e$. Then $\gamma$ cannot go back and forth along strips, so it crosses every strand $\alpha \in \zz_e$ between $\x$ and $\y$ with the same orientation exactly once.
\end{proof}

\begin{lemma}[{\cite[Lemma~20]{BCT20}; see also \cite[Lemma~34]{BCT22}}]\label{lemma:interlacing}
Let $\x,\y \in W(\graphpl) \sqcup F(\graphpl) \sqcup B(\graphpl)$.
As we move cyclically along $\partial N$, the edges corresponding to angles at which
$g_{\x,\y}$ has zeros and the edges corresponding to angles at which it has poles do not interlace.
\end{lemma}

\begin{definition}\label{def:null_sector}
Let $\bl \in B(\graphpl)$ and $\wh \in W(\graphpl)$. We define the \emph{null sector}
\[
\sector_{\bl,\wh}\subset \partial N
\]
as follows.

\begin{enumerate}
\item Suppose $g_{\bl,\wh}$ has at least one zero in $A_0$. By Lemma~\ref{lemma:interlacing}, the zeros and poles of $g_{\bl,\wh}$ along $\partial N$ do not interlace. We then define $\sector_{\bl,\wh}$ to be the maximal closed connected interval of $\partial N$ which contains every edge $e$ such that $g_{\bl,\wh}$ has a zero at some angle $\alpha\in \zz_e$, and contains no edge $e$ such that $g_{\bl,\wh}$ has a pole at some angle $\alpha\in \zz_e$.

\item Suppose $g_{\bl,\wh}$ has no zeros in $A_0$. Then
\[
\divisor(g_{\bl,\wh})|_{A_0}=-\alpha-\beta
\]
for some angles $\alpha,\beta\in\zz$. Then there are two complementary cyclic intervals in $\partial N \setminus \{e(\alpha),e(\beta)\}$. In this case, we define $\sector_{\bl,\wh}$ by ``continuity in $(\bl,\wh)$.'' 
Precisely, by contraction moves as in Figure~\ref{fig:moves:X}(right), we may assume that at least one of $\bl$ and $\wh$ has degree at least three. We then extend a path from $\bl$ to $\wh$ into an adjacent face of one of these vertices so as to cross a strand $\gamma\neq\alpha,\beta$. We define $\sector_{\bl,\wh}$ to be the interval in $\partial N \setminus \{e(\alpha),e(\beta)\}$ containing $e(\gamma)$.
\end{enumerate}
\end{definition}

\subsubsection{Whole-plane inverse Kasteleyn matrices}\label{sec:whole_plane_inverse_K}

We recall the construction of a two-parameter family of inverse Kasteleyn matrices for graph $\graphpl$ from~\cite{BCT22}. These whole-plane inverse Kasteleyn matrices will be a basic ingredient in the definition of inverse Kasteleyn matrices of finite subgraphs as well as in the description of local statistics.

{
 Since the angle function is compatible with the partial cyclic order on $\zz$, there is an open
interval in \(A_0\setminus\zz\) corresponding to each \(v= e_-\cap e_+\in V(N)\), namely the interval between the angles in
\(\zz_{e_-}\) and those in \(\zz_{e_+}\).
}

\begin{definition}\label{def:contours_whole_plane}
 For \(v\in V(N)\) and \(\zeta\in \Sigma_+^\circ\), let \(C_v^\zeta\) be a curve from \(\sigma(\zeta)\) to \(\zeta\) intersecting \(A_0\) exactly once in the component of \(A_0\setminus\zz\) corresponding to \(v\), and such that \(\sigma_*(C_v^\zeta)=-C_v^\zeta\). If \(\zeta\in \Sigma(\RR)\setminus\zz\), let \(C_v^\zeta\) denote the natural limit as \(\zeta\) approaches \(\Sigma(\RR)\).
\end{definition}

Given \(\bl\in B(\graphpl)\), \(\wh\in W(\graphpl)\) and \(\zeta\in \Sigma_+\setminus\zz\), define
\[
C_{\bl,\wh}^\zeta:=C_v^\zeta,
\]
where \(v\) is any vertex in the null sector \(\sector_{\bl,\wh}\).

\begin{theorem}[{\cite[Theorem~40]{BCT22}}]\label{thm:whole_plane_inverse_kast}
For any \(\zeta\in \Sigma_+\setminus\zz\), the matrix
\[
\mathsf A_{\bl,\wh}^{\zeta}
:=
\frac{1}{2\pi\ii}\int_{C_{\bl,\wh}^\zeta} g_{\bl,\wh}
\]
is an inverse of \(\kast\) in the sense that $\kast \mathsf A^{\zeta} = \mathsf A^{\zeta} \kast = \mathsf{I}$.
\end{theorem}

We now introduce some additional notation for the case \(\zeta\in \Sigma(\RR)\setminus\zz\) which will be used in Section~\ref{sec:inv_K}.

\begin{definition}\label{def:contours_from_intervals}
  Given an oriented interval $I \subset \partial N$, let $C(I) \subset \curve$ be a curve which is the union of small counterclockwise-oriented circles that enclose all the angles $\bm \alpha_e$ for $e \subset I$.
\end{definition}

If \(\zeta\in \Sigma(\RR)\setminus\zz\), then \(\mathsf A_{\bl,\wh}^{\zeta}\) depends only on the connected component of $\Sigma(\RR)\setminus\zz$ containing \(\zeta\). If this component corresponds to \(u\in V(N)\), then we can take
\[
C_{\bl,\wh}^\zeta=C(I_{\bl,\wh}^u),
\qquad I_{\bl,\wh}^u:=[u,v],
\]
where \(v\) is any vertex in \(\sector_{\bl,\wh}\), and we write \(\mathsf A_{\bl,\wh}^u:=\mathsf A_{\bl,\wh}^\zeta\).

\subsubsection{Gibbs measures}\label{sec:slope}

In this section we assume that the angle function $\nu:\zz\to A_0$ is \emph{periodic}, \emph{i.e.} 
\[
\nu(\alpha+(i,j))=\nu(\alpha) \qquad \text{for all}~\alpha\in \zz \text{ and } (i,j)\in \ZZ^2.
\]
Equivalently, we have a well-defined angle function 
\[
\nu:\zz_\torus \to A_0
\] 
on the quotient. As explained in~\cite[Section~4.2]{BCT22}, this setting is strictly more general than assuming that the weights are periodic; see~\cite[Proposition~43]{BCT22} for a characterization of the subset of periodic weights.

\begin{theorem}[{\cite[Theorem~62]{BCT22}}]\label{thm:gibbs_measure}
For any $\zeta\in \Sigma_+\setminus \zz$, the operator $\mathsf A^\zeta$ defines a Gibbs measure $\PP^\zeta$ on dimer covers of $\graphpl$. For any fixed collection of edges $\{\e_i=\bl_i\wh_i\}_{i=1}^k\subset E(\graphpl)$, the probability that all of the edges $\e_1,\dots,\e_k$ are present is
\[
\PP^\zeta(\e_1,\dots,\e_k\in M)
=
\prod_{i=1}^k \kast_{\wh_i,\bl_i}
\det\bigl(\mathsf A^\zeta_{\bl_i,\wh_j}\bigr)_{1\le i,j\le k}.
\]
\end{theorem}

The \emph{slope} $(s^\zeta,t^\zeta)$ of the Gibbs measure $\PP^\zeta$ is defined as the expected horizontal and vertical height change:
\begin{equation}
s^\zeta:=\EE^\zeta\left[h_M(\f_0+(1,0))-h_M(\f_0)\right], \quad t^\zeta:=\EE^\zeta\left[h_M(\f_0+(0,1))-h_M(\f_0)\right],
\end{equation}
where $\EE^\zeta$ denotes expectation with respect to the probability measure $\PP^\zeta$.

For \(u\in V(N)\), let \(\zeta\in A_0\setminus\zz\) be any point in the connected component corresponding to \(u\). Since \(\PP^\zeta\) depends only on this component, we set
\[
\PP^u:=\PP^\zeta,
\qquad
\EE^u:=\EE^\zeta,
\qquad
(s^u,t^u):=(s^\zeta,t^\zeta).
\]
By \cite[Remark~53]{BCT22}, the measure \(\PP^u\) is deterministic: more precisely,
\[
\PP^u(\e\in M)=\mathbf 1_{\e\in M_u},
\]
where \(M_u\) is the dimer cover consisting of all edges \(\e\) such that \(e(\alpha)<u<e(\beta)\) along $\partial N$, with \(\alpha,\beta\) the two strands containing \(\e\) as in Figure~\ref{fig:fock_kast}; we call $M_u$ the \emph{extremal dimer cover} corresponding to $u$.

Recall from Section~\ref{sec:conventions} that \(|e|_{\ZZ}\) is the lattice length of $e$ and \((a_e,b_e)=\vec e/|e|_{\ZZ}\in\ZZ^2\). Let $u_0 \in V(N)$ and let $M_0:=M_{u_0}$ be the reference dimer cover used to define the height function. Define
\begin{equation}\label{eq:z_and_w}
z(\zeta):=  \prod_{\alpha \in \zz_\torus } E(\alpha,\zeta)^{-b_{e(\alpha)}}, \qquad w(\zeta):= \prod_{\alpha \in \zz_\torus } E(\alpha,\zeta)^{a_{e(\alpha)}}.
\end{equation}
Then
\begin{equation}\label{eq:slope_integral}
(s^\zeta, t^\zeta)
=\left(-\frac{1}{2\pi\i}\int_{C_{u_0}^{\zeta}} d\log z , -\frac{1}{2\pi\i}\int_{C_{u_0}^{\zeta}}d \log w\right);
\end{equation}
see \cite[Equation 32]{BCT22}. 
The set of possible slopes is the Newton polygon $N$ rotated by $-\pi/2$; see~\cite[Corollary~56]{BCT22} {and \cite[Section~3.1.4]{KOS06}.}


\section{Astroidal zig-zag graphs} \label{sec:AZ_graphs}

In this section, we introduce a class of finite graphs that we will work with 
and study their basic properties.

\subsection{Definition of astroidal zig-zag graphs}\label{sec:def_AZ_graphs}

\begin{figure}
\centering

\begin{tikzpicture}
  \node[anchor=center] (A) at (0,0) {%
    \begin{tikzpicture}[baseline=(current bounding box.center)]
      \begin{scope}[scale = 0.4]
        \node[bvert] (1)  at (0,0) {};
        \node[bvert] (2)  at (2,0) {};
        \node[bvert] (3)  at (4,1) {};
        \node[bvert] (4)  at (5,3) {};
        \node[bvert] (5)  at (4,5) {};
        \node[bvert] (6)  at (3,6) {};
        \node[bvert] (7)  at (1,6) {};
        \node[bvert] (8)  at (-2,5) {};
        \node[bvert] (9)  at (-2,3) {};
        \node[bvert] (10) at (-1,1) {};

        \draw[hollowedge]  (1)  -- node[below]       {$e_1$}   (2);
        \draw[hollowedge] (2)  -- node[below right] {$e_2$}   (3);
        \draw[solidedge] (3)  -- node[right]       {$e_3$}   (4);
        \draw[solidedge] (4)  -- node[right]       {$e_4$}   (5);
        \draw[hollowedge] (5)  -- node[above right] {$e_5$}   (6);
        \draw[hollowedge] (6)  -- node[above]       {$e_6$}   (7);
        \draw[solidedge] (7)  -- node[above left]  {$e_7$}   (8);
        \draw[solidedge] (8)  -- node[left]        {$e_8$}   (9);
        \draw[solidedge] (9)  -- node[left]        {$e_9$}   (10);
        \draw[hollowedge] (10) -- node[below left]  {$e_{10}$} (1);
      \end{scope}
    \end{tikzpicture}
  };

  \node[anchor=center] (B) at ([xshift=6cm]A.east) {
    \begin{tikzpicture}[baseline=(current bounding box.center)]

    \begin{scope}[scale=0.80]

  \node[nvert] (v0) at (0,0)    {};
  \node[nvert] (v1) at (2,0)    {};
  \node[nvert] (v2) at (6,2)    {};
  \node[nvert] (v3) at (4,-2)   {};
  \node[nvert] (v4) at (6,-6)   {};
  \node[nvert] (v5) at (4,-4)   {};
  \node[nvert] (v6) at (-8,-4)  {};
  \node[nvert] (v7) at (-2,-2)  {};
  \node[nvert] (v8) at (-2,0)   {};
  \node[nvert] (v9) at (-4,4)   {};

  \draw[hollowedge] (v9) -- node[above right] {$\beta_{e_{10}}$} (v0);
  \draw[hollowedge] (v0) -- node[above]       {$\beta_{e_{1}}$}  (v1);
  \draw[hollowedge] (v1) -- node[above left]  {$\beta_{e_{2}}$}  (v2);
  \draw[solidedge]  (v2) -- node[right]       {$\beta_{e_{3}}$}  (v3);
  \draw[solidedge]  (v3) -- node[right]       {$\beta_{e_{4}}$}  (v4);

  \draw[hollowedge] (v4) -- node[left]        {$\beta_{e_{5}}$}  (v5);

  \draw[hollowedge] (v5) -- node[below]       {$\beta_{e_{6}}$}  (v6);
  \draw[solidedge]  (v6) -- node[above left]  {$\beta_{e_{7}}$}  (v7);
  \draw[solidedge]  (v7) -- node[left]        {$\beta_{e_{8}}$}  (v8);
  \draw[solidedge]  (v8) -- node[left]        {$\beta_{e_{9}}$}  (v9);

\end{scope}

    \end{tikzpicture}
  };
\end{tikzpicture}
\caption{A Newton polygon $N$ and an illustration of $\partial\graphbeta^\times$, along with the edge coloring. Note that the actual $\beta_{e_j}$'s are (pieces of) strands, but they have the same asymptotic direction as the vectors pictured on the right. 
} \label{fig:Gbeta}
\end{figure}
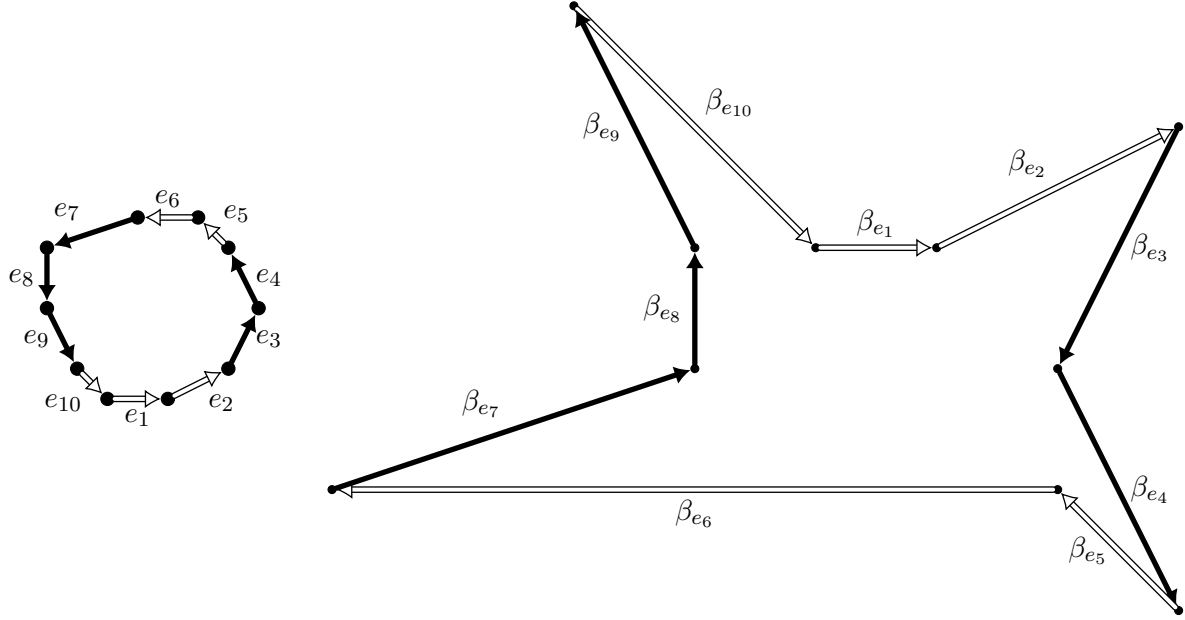

\begin{figure}
    \centering
\begin{tikzpicture}
    \begin{scope}
    \filldraw[fill=gray!30, draw=none]
    (1,1) -- (0,0) -- (-1,1) -- (-2,0) -- (-1,-1) -- (0,0) -- (1,-1) -- (-2,-1) -- (-2,1) -- (1,1);
 \coordinate[wvert,label=right:${\wh}$] (w) at (0,0) {};
    \coordinate[wvert] (w2) at (-2,0) {};
    \coordinate[bvert] (b1) at (1,1) {};
    \coordinate[bvert] (b2) at (1,-1) {};
    \coordinate[bvert] (b3) at (-1,1) {};
    \coordinate[bvert] (b4) at (-1,-1) {};
    \draw[line width=2pt] (w) edge (b1) edge (b2) edge (b3) edge (b4);
    \draw[line width=2pt] (w2) edge (b3) edge (b4);
    \draw[hollowedge] (1.5,0.5) -- (-1.5,0.5);
    \draw[hollowedge] (-1.5,-0.5) -- (-1.5,0.5);
     \draw[hollowedge] (-1.5,-0.5) -- (1.5,-0.5);
    \end{scope}

        \begin{scope}[shift={(7,0)}]
    \filldraw[fill=gray!30, draw=none]
    (90:1.5) -- (150:1.5) -- (210:1.5) -- (270:1.5) -- (330:1.5) -- (0,0) -- (90:1.5);

    \coordinate[wvert] (w) at (0,0) {};
 
    \coordinate[bvert] (b1) at (90:1.5) {};
    \coordinate[bvert] (b2) at (150:1.5) {};
       \coordinate[bvert] (b3) at (210:1.5) {};
    \coordinate[bvert] (b4) at (270:1.5) {};

       \coordinate[bvert] (b5) at (30:1.5) {};
    \coordinate[bvert] (b6) at (-30:1.5) {};
    \draw[line width=2pt] (w) edge (b1) edge (b2)  edge (b3) edge (b4) edge (b6);
    \draw[] (w) edge (b5) edge (b6);

\draw[hollowedge]
  
  ($ (90:0.75) - (150:0.75) + (90:0.75) $) -- (150:0.75) ; 
    \draw[hollowedge] (150:0.75) -- (210:0.75);
    \draw[hollowedge] (210:0.75) -- (270:0.75);
    \draw[hollowedge] (270:0.75) -- ($ (330:0.75) + (330:0.75) - (270:0.75) $);

    \end{scope}
\end{tikzpicture}
    \caption{The left panel shows an example illustrating the necessity of Condition~2 in Definition~\ref{def:azgraph}. The white arrows are a portion of $\partial \graphbeta^\times$, the thick edges are the corresponding portion of $\partial \graphbeta$ and the shaded region consists of faces of $\graphbeta$. The white vertex $\wh$ on the left-hand side is visited twice by non-consecutive zig-zag paths so $\partial \graphbeta$ has a self-intersection. The right panel shows the neighborhood of an outer white vertex that is allowed. }
    \label{fig:condition_2}
\end{figure}
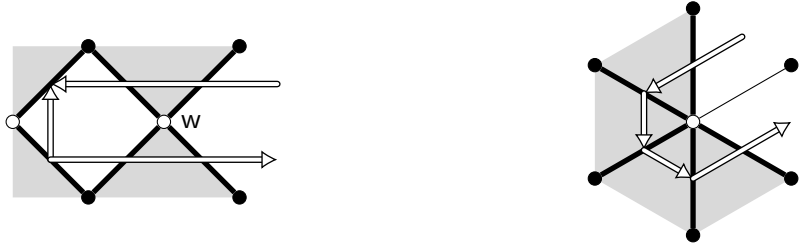

Let \(\bm \beta =(\beta_e)_{e \in E(N)}\) be a collection of strands such that \(\beta_e \in \bm \alpha_e\).
For each vertex \(v = e_- \cap e_+ \in V(N)\), choose
\[
\e_v \in \beta_{e_-} \cap \beta_{e_+} \in V(\graphpl^\times) \cong E(\graphpl),
\]
where \(e_-\) (resp. \(e_+\)) is the edge immediately clockwise (resp. counterclockwise) from \(v\). 

We first define the boundary cycle in the medial graph (see Figure~\ref{fig:Gbeta}). Let \(\partial \graphbeta^\times\) be the counterclockwise-oriented cycle in \(\graphpl^\times\) obtained as follows: it visits the medial vertices \(\e_v\) in the clockwise cyclic order of \(v \in V(N)\), and for each edge \(e=[v_-,v_+] \in E(N)\) it follows the strand \(\beta_e\) from \(\e_{v_+}\) to \(\e_{v_-}\). We assume throughout that \(\partial \graphbeta^\times\) is a simple cycle.

From this we obtain a subgraph \(\partial \graphbeta\) of \(\graphpl\) by defining
\[
\partial \graphbeta
:=
\bigcup_{\e \in V(\partial \graphbeta^\times)} \e
\subset \graphpl,
\]
\emph{i.e.} the union of the corresponding zig-zag path segments.

Even if the medial cycle \(\partial \graphbeta^\times\) is simple, the resulting subgraph \(\partial \graphbeta\) need not be: a vertex of \(\graphpl\) may lie on several strands \(\beta_e\). This is harmless provided such occurrences are consecutive along \(\partial \graphbeta^\times\). The problematic situation is when two visits to a vertex are separated by other parts of the boundary, which would produce a genuine self-intersection in \(\graphpl\) (see Figure~\ref{fig:condition_2}(left)). Throughout, we assume this does not happen; \emph{i.e.} we assume that whenever a vertex of \(\graphpl\) is visited more than once by \(\partial \graphbeta\), the corresponding strands appear consecutively in the cyclic order of \(\partial \graphbeta^\times\).

We then define the graph \(\graphbeta\) to be the finite portion of \(\graph\) contained weakly inside \(\partial \graphbeta\). The reader may find it helpful to look ahead at the examples in Section~\ref{sec:AZ_examples} before proceeding with the definitions below.

We next assign a color to each edge \(e \in E(N)\). We say that \(e\) is \emph{black} if its induced orientation agrees with the counterclockwise orientation of \(\partial\graphbeta^\times\), and \emph{white} otherwise (see Figure~\ref{fig:Gbeta}). We also define a sign function
\begin{equation} \label{eq:sign_convention}
\sign_{\bm \beta} : E(N)\rightarrow\{+,-\},
\qquad
\sign_{\bm\beta}(e) :=
\begin{cases}
+1, & \text{if \(e\) is black},\\
-1, & \text{if \(e\) is white}.
\end{cases}
\end{equation}
We will use the terminology ``black/white'' and ``sign \(\pm1\)'' interchangeably.

For each edge \(e \in E(N)\), the portion of the strand \(\beta_e\) between \(\e_{v_-}\) and \(\e_{v_+}\) separates the vertices of the corresponding zig-zag path into two sides. We call the vertices on the exterior of \(\partial\graphbeta^\times\) \emph{\(\beta_e\)-outer} vertices. A vertex of \(\graphbeta\) that is not \(\beta_e\)-outer for any \(e \in E(N)\) is called an \emph{inner} vertex. Note that a vertex can be \(\beta_e\)-outer for several \(e \in E(N)\). By Lemma~\ref{lem:cyclic_order_same}, such edges all have the same color and form a cyclic interval in \(\partial N\) (see Figure~\ref{fig:condition_2}(right)).

Informally, we think of \(\partial \graphbeta\) as defining a polygonal region whose ``vertices'' are the edges \(\e_v\) and whose ``sides'' run along the zig-zag paths corresponding to \(\beta_e\). An edge \(e\) has the color of the corresponding \(\beta_e\)-outer vertices.

\begin{definition}
   Given such a \(\bm \beta\), define an element \(\bm D_{\bm \beta} \in \ZZ^{\zz}\) by
\begin{equation} \label{eq:d_beta_defn}
     (\bm D_{\bm \beta})_{\zz_e} := -\dabel(\x)_{\zz_e}
\end{equation}
for all \(e \in E(N)\), where \(\x\) is any \(\beta_e\)-outer vertex. We say that \(\graphbeta\) is \emph{admissible} if
\begin{equation} \label{eq:adm_defn}
    \deg(\bm D_{\bm \beta})=0.
\end{equation} 
\end{definition}
Since
\[
\deg(\bm D_{\bm\beta})
=
\sum_{e\in E(N)} \deg\bigl((\bm D_{\bm\beta})_{\zz_e}\bigr),
\]
the admissibility condition~\eqref{eq:adm_defn} says that the total number of strands between the reference face and \(\beta_e\), summed over white edges \(e\), is equal to the corresponding total summed over black edges.

The role of admissibility will be clarified in Lemma~\ref{lem:kernel_is_a_form}.

\begin{definition} \label{def:azgraph}
We call \(\graphbeta\) an \emph{astroidal zig-zag graph} (AZ graph) if:
\begin{enumerate}
\item the medial cycle \(\partial \graphbeta^\times\) is simple,
\item whenever a vertex of \(\graphpl\) is visited more than once as we go around \(\partial \graphbeta\), all such visits occur consecutively along \(\partial \graphbeta^\times\),
\item \(\graphbeta\) is admissible.
\end{enumerate}
\end{definition}

The reason for the name is that the boundary \(\partial \graphbeta\) resembles an astroid curve (see Lemma~\ref{lem:four_convex_vertices}) and is made up of zig-zag paths. Since \(\graphpl\) is \(\ZZ^2\)-periodic and we have one constraint~\eqref{eq:adm_defn}, AZ graphs form a \((\# E(N)-3)\)-parameter family. {In concrete examples, AZ graphs are easy to construct directly. A general construction showing that the class is nonempty for every Newton polygon with at least four sides will be given in Remark~\ref{rem:existence_astroidal_domain}.}

\subsection{Examples of AZ graphs} \label{sec:AZ_examples}

\begin{example}
When $N$ is any triangle, the $\graph$ is the hexagonal lattice~\cite[Proposition~11.3]{IU}, we have a $0$-parameter family, and there are no AZ graphs.
\end{example}

\begin{figure}
  \centering
  \begin{tikzpicture}

    \node[anchor=center] (A) {
      \begin{tikzpicture}[scale=1]
       \begin{scope}
    \node[bvert](1) at (0,0){}; 
     \node[bvert](2) at (1,0){}; 
      \node[bvert](3) at (1,1){}; 
       \node[bvert](4) at (0,1){};
  \draw[solidedge] (1) -- 
       (2);
       \draw[solidedge] (3) --
       (4) ;
       \draw[hollowedge] (2) --
       (3);
       \draw[hollowedge] (4) -- 
       (1);

       \end{scope}
      \end{tikzpicture}
    };

    \node[anchor=center] (B) at ([xshift=5.0cm]A.east) {%
      \begin{tikzpicture}[scale=0.6]

        \foreach \y in {0,2,4,6} {
          \foreach \x in {1,3,5} {
            \node[bvert] (b-\x-\y) at (\x,\y) {};
          }
        }

        \foreach \y in {1,3,5} {
          \foreach \x in {0,2,4,6} {
            \node[wvert] (w-\x-\y) at (\x,\y) {};
          }
        }

        \foreach \y in {0,2,4,6} {
          \foreach \x in {1,3,5} {

            \ifcsname pgf@sh@ns@w-\the\numexpr\x-1\relax-\the\numexpr\y+1\relax\endcsname
              \draw[] (b-\x-\y) -- (w-\the\numexpr\x-1\relax-\the\numexpr\y+1\relax);
            \fi

            \ifcsname pgf@sh@ns@w-\the\numexpr\x+1\relax-\the\numexpr\y+1\relax\endcsname
              \draw[] (b-\x-\y) -- (w-\the\numexpr\x+1\relax-\the\numexpr\y+1\relax);
            \fi

            \ifcsname pgf@sh@ns@w-\the\numexpr\x-1\relax-\the\numexpr\y-1\relax\endcsname
              \draw[] (b-\x-\y) -- (w-\the\numexpr\x-1\relax-\the\numexpr\y-1\relax);
            \fi

            \ifcsname pgf@sh@ns@w-\the\numexpr\x+1\relax-\the\numexpr\y-1\relax\endcsname
              \draw[] (b-\x-\y) -- (w-\the\numexpr\x+1\relax-\the\numexpr\y-1\relax);
            \fi

          }
        }

      \end{tikzpicture}
    };

\node[anchor=center] (C) at ([xshift=4.2cm]B.east) {%
      \begin{tikzpicture}[scale=0.6]

        \foreach \y in {0,2,4,6} {
          \foreach \x in {1,3,5} {
            \node[bvert] (b-\x-\y) at (\x,\y) {};
          }
        }

        \foreach \y in {1,3,5} {
          \foreach \x in {0,2,4,6} {
            \node[wvert] (w-\x-\y) at (\x,\y) {};
          }
        }

        \foreach \y in {0,2,4,6} {
          \foreach \x in {1,3,5} {

            \ifcsname pgf@sh@ns@w-\the\numexpr\x-1\relax-\the\numexpr\y+1\relax\endcsname
              \draw[] (b-\x-\y) -- (w-\the\numexpr\x-1\relax-\the\numexpr\y+1\relax);
            \fi

            \ifcsname pgf@sh@ns@w-\the\numexpr\x+1\relax-\the\numexpr\y+1\relax\endcsname
              \draw[] (b-\x-\y) -- (w-\the\numexpr\x+1\relax-\the\numexpr\y+1\relax);
            \fi

            \ifcsname pgf@sh@ns@w-\the\numexpr\x-1\relax-\the\numexpr\y-1\relax\endcsname
              \draw[] (b-\x-\y) -- (w-\the\numexpr\x-1\relax-\the\numexpr\y-1\relax);
            \fi

            \ifcsname pgf@sh@ns@w-\the\numexpr\x+1\relax-\the\numexpr\y-1\relax\endcsname
              \draw[] (b-\x-\y) -- (w-\the\numexpr\x+1\relax-\the\numexpr\y-1\relax);
            \fi

          }
        }
\foreach \y in {1,3,5}
{
    \pgfmathtruncatemacro{\k}{(\y+1)/2}

    \ifnum\k=1
        \draw[solidedge] (0.5,\y-0.5) -- (5.5,\y-0.5);
    \else

    \fi

    \ifnum\k=3
        \draw[solidedge]  (5.5,\y+0.5) -- (0.5,\y+0.5) ;
    \else
 
    \fi

    \ifnum\k=1
        \draw[hollowedge]  (\y-0.5,0.5) -- (\y-0.5,5.5);
    \else
        
    \fi

    \ifnum\k=3
        \draw[hollowedge] (\y+0.5,5.5) --  (\y+0.5,0.5);
    \else
       
    \fi
 
}
      \end{tikzpicture}
};

  \end{tikzpicture}

  \caption{A square Newton polygon $N$ (left), the Aztec diamond $\graphbeta$ (middle) and its medial cycle $\partial \graphbeta^\times$ (right). 
  }
  \label{fig:square_aztec}
\end{figure}
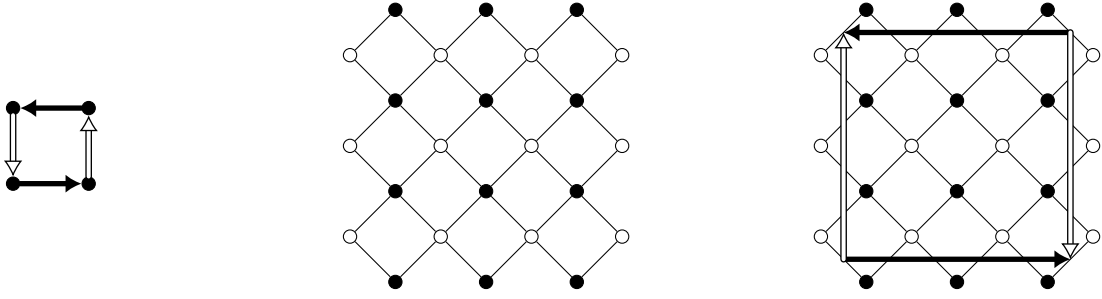

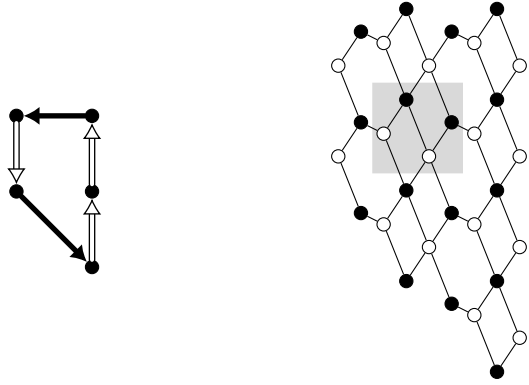
\begin{figure}
\centering

\begin{tikzpicture}

\node[anchor=center] (A) {
\begin{tikzpicture}[baseline=(current bounding box.center)]
\begin{scope}[rotate=0]
   \node[bvert] (1) at (0,0){}; 
    \node[bvert] (2) at (0,-1){}; 
    \node[bvert] (3) at (1,-2){}; 
    \node[bvert] (4) at (1,-1){};
    \node[bvert] (5) at (1,0){};

           \draw[solidedge] (2) -- (3);
           \draw[solidedge] (5) --  (1)
           ;

  \draw[hollowedge] (1) -- (2);
   \draw[hollowedge]   (3) --(4)
  ;
     \draw[hollowedge]   (4)-- (5)
  ;
  \end{scope}
\end{tikzpicture}
};

\node[anchor=center] (B) at ([xshift=4.2cm]A.east) {%
\begin{tikzpicture}[scale=1.2]
\begin{scope}[cm={0.25,-0.125,0,0.5,(0.625,0.1875)}]

 \fill[gray!30]
   (-2.5,-1) -- (-2.5,1) -- (1.5,2) -- (1.5,0) -- cycle;

  \node[wvert] (L1)  at (-4, 1) {};
  \node[wvert] (L0)  at (-4,-1) {};

  \node[bvert] (B1)  at (-3, 2) {};
  \node[bvert] (B0)  at (-3, 0) {};
  \node[bvert] (Bm1) at (-3,-2) {};

  \node[wvert] (W1)  at (-2, 2) {};
  \node[wvert] (W0)  at (-2, 0) {};
  \node[wvert] (Wm1) at (-2,-2) {};

  \node[bvert] (C1)  at (-1, 3) {};
  \node[bvert] (C0)  at (-1, 1) {};
  \node[bvert] (Cm1) at (-1,-1) {};
  \node[bvert] (Cm2) at (-1,-3) {};

  \node[wvert] (D1)  at ( 0, 2) {};
  \node[wvert] (D0)  at ( 0, 0) {};
  \node[wvert] (Dm1) at ( 0,-2) {};

  \node[bvert] (E1)  at ( 1, 3) {};
  \node[bvert] (E0)  at ( 1, 1) {};
  \node[bvert] (Em1) at ( 1,-1) {};
  \node[bvert] (Em2) at ( 1,-3) {};

  \node[wvert] (F1)  at ( 2, 3) {};
  \node[wvert] (F0)  at ( 2, 1) {};
  \node[wvert] (Fm1) at ( 2,-1) {};
  \node[wvert] (Fm2) at ( 2,-3) {};

  \node[bvert] (G1)  at ( 3, 4) {};
  \node[bvert] (G0)  at ( 3, 2) {};
  \node[bvert] (Gm1) at ( 3, 0) {};
  \node[bvert] (Gm2) at ( 3,-2) {};
  \node[bvert] (Gm3) at ( 3,-4) {};

  \node[wvert] (H1)  at ( 4, 3) {};
  \node[wvert] (H0)  at ( 4, 1) {};
  \node[wvert] (Hm1) at ( 4,-1) {};
  \node[wvert] (Hm2) at ( 4,-3) {};

  \draw
    (L1)--(B1) (L1)--(B0)
    (L0)--(B0) (L0)--(Bm1)
    (B1)--(W1) (B0)--(W0) (Bm1)--(Wm1)

    (W1)--(C1) (W1)--(C0)
    (W0)--(C0) (W0)--(Cm1)
    (Wm1)--(Cm1) (Wm1)--(Cm2)

    (C1)--(D1)
    (C0)--(D1) (C0)--(D0)
    (Cm1)--(D0) (Cm1)--(Dm1)
    (Cm2)--(Dm1)

    (D1)--(E1) (D1)--(E0)
    (D0)--(E0) (D0)--(Em1)
    (Dm1)--(Em1) (Dm1)--(Em2)

    (E1)--(F1) (E0)--(F0) (Em1)--(Fm1) (Em2)--(Fm2)

    (F1)--(G1) (F1)--(G0)
    (F0)--(G0) (F0)--(Gm1)
    (Fm1)--(Gm1) (Fm1)--(Gm2)
    (Fm2)--(Gm2) (Fm2)--(Gm3)

    (H1)--(G1) (H1)--(G0)
    (H0)--(G0) (H0)--(Gm1)
    (Hm1)--(Gm1) (Hm1)--(Gm2)
    (Hm2)--(Gm2) (Hm2)--(Gm3);

\end{scope}
\end{tikzpicture}

};

\end{tikzpicture}

\caption{A quadrilateral Newton polygon and an example of a corresponding AZ graph.}
\label{fig:quad}
\end{figure}

\begin{example}
When $N$ is a convex lattice quadrilateral, there is a one-parameter family of AZ graphs. If $N$ is the unit square, one choice of minimal graph $\graphpl$ is the square lattice. A general choice of $\bm\beta$ produces a rectangular subgraph, and admissibility forces this subgraph to be the Aztec diamond (Figure~\ref{fig:square_aztec}). Any other choice of minimal graph is related to the square lattice by local moves and has the same large-scale statistical-mechanical behavior. Thus, up to local moves, the Aztec diamonds are the only AZ graphs associated with the unit square Newton polygon. Figure~\ref{fig:quad} shows another quadrilateral Newton polygon together
with an associated AZ graph, called a \emph{tower graph}, which is the unique AZ graph up to size, {and it has been studied in~\cite{BF15,BNR23,Nicoletti}
and in the physics literature as the brane tiling/dimer model for the suspended pinch point singularity; see for example,~\cite{FHKVW06,FHMSVW06,Kennaway07,EF12}.} {Another example is given by the \emph{pinecone graphs} associated with Gale--Robinson sequences studied in \cite{BMPW09,JMZ13,CM26}.}
\end{example}

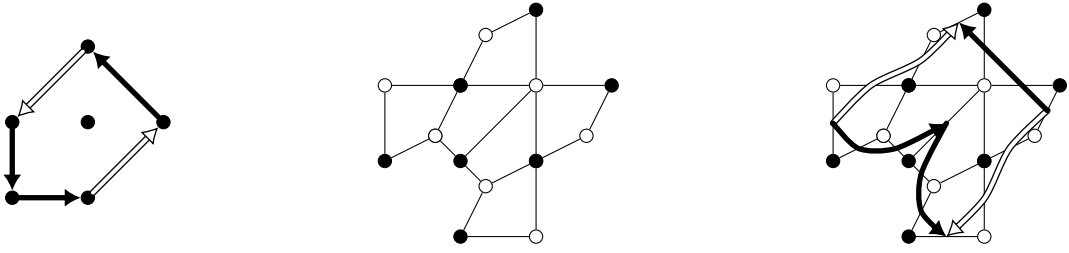
\begin{figure}
\centering

\begin{tikzpicture}

\node[anchor=center] (A) {%
\begin{tikzpicture}[baseline=(current bounding box.center)]
\begin{scope}[rotate=0]
   \node[bvert] (1) at (0,1){}; 
    \node[bvert] (2) at (1,0){}; 
    \node[bvert] (3) at (2,1){}; 
    \node[bvert] (4) at (1,1){};
    \node[bvert] (5) at (0,0){}; 
    \node[bvert] (6) at (1,2){};

           \draw[solidedge]  (5) -- (2);
           \draw[solidedge] (3) -- (6);
           \draw[solidedge] (1) --  (5)
           ;

  \draw[hollowedge] (2) -- (3);
   \draw[hollowedge]   (6) -- (1)
  ;
  \end{scope}
\end{tikzpicture}
};

\pgfmathsetmacro{\dx}{sqrt(3)}
\pgfmathsetmacro{\hx}{sqrt(3)/2}
\pgfmathsetmacro{\hy}{3/2}

\pgfmathsetmacro{\cmA}{1/sqrt(3)}
\pgfmathsetmacro{\cmB}{1/3}
\pgfmathsetmacro{\cmC}{2/3}

\node[anchor=center] (B) at ([xshift=4.2cm]A.east) {%
\begin{tikzpicture}[scale=2,cm={1,0,1,1,(0,0)}]

\coordinate (S1) at (3,0);
\coordinate (S2) at (2,0);
\coordinate (S3) at (4,-1);

\coordinate[wvert] (wh1) at (S1);
\coordinate[wvert] (wh2) at ($(S1)+(-2/3, 1/3)$);
\coordinate[wvert] (wh3) at (S2);
\coordinate[wvert] (wh4) at ($(S1)+(-1/3,-1/3)$);
\coordinate[wvert] (wh5) at ($(S1)+( 1/3,-2/3)$);
\coordinate[wvert] (wh6) at (S3);
\coordinate[wvert] (wh7) at ($(S1)+( 2/3,-1/3)$);

\coordinate (w1a) at ($(S1)+( 1/3, 1/3)$);
\coordinate (w2a) at ($(S1)+(-1/3, 2/3)$);

\coordinate (w1b) at ($(S2)+( 1/3, 1/3)$);
\coordinate (w5b) at ($(S2)+( 1/3,-2/3)$);
\coordinate[wvert] (w6b) at ($(S2)+( 2/3,-1/3)$);

\coordinate (w2c) at ($(S3)+(-1/3, 2/3)$);
\coordinate (w4c) at ($(S3)+(-1/3,-1/3)$);


\coordinate[bvert] (bl1) at ($(w2a)!0.5!(wh2)$);
\coordinate[bvert]  (bl2) at ($(wh2)!0.5!(wh4)$);
\coordinate[bvert] (bl3) at ($(w5b)!0.5!(w6b)$);
\coordinate[bvert](bl4) at ($(wh4)!0.5!(wh5)$);
\coordinate[bvert]  (bl5) at ($(wh5)!0.5!(w4c)$); 
\coordinate[bvert] (bl6) at ($(wh5)!0.5!(wh7)$);
\coordinate[bvert] (bl7) at ($(wh7)!0.5!(w1a)$);

\coordinate[bvert] (b61b) at ($(w6b)!0.5!(w1b)$);
\coordinate[bvert] (b23c) at ($(w2c)!0.5!(wh5)$); 

\draw (wh1) -- (bl1);
\draw (wh1) -- (bl2);
\draw (wh1) -- (bl4);
\draw (wh1) -- (bl6);
\draw (wh1) -- (bl7);

\draw (bl1) -- (wh2);

\draw (bl2) -- (wh2);
\draw (bl2) -- (wh4);

\draw (bl4) -- (wh4);
\draw (bl4) -- (wh5);

\draw (bl6) -- (wh5);
\draw (bl6) -- (wh7);

\draw (bl7) -- (wh7);

\draw (wh3) -- (bl3);
\draw (wh3) -- (b61b);
\draw (bl3) -- (w6b);

\draw (wh6) -- (b23c);
\draw (wh6) -- (bl5);
\draw (bl5) -- (wh5);

\midp{m_bl1_wh1}{bl1}{wh1}
\midp{m_bl2_wh1}{bl2}{wh1}
\midp{m_bl4_wh1}{bl4}{wh1}
\midp{m_bl6_wh1}{bl6}{wh1}
\midp{m_bl7_wh1}{bl7}{wh1}

\midp{m_bl1_wh2}{bl1}{wh2}
\midp{m_bl1_wh7}{bl1}{wh7}

\midp{m_bl2_wh1}{bl2}{wh1}
\midp{m_bl2_wh2}{bl2}{wh2}
\midp{m_bl2_wh3}{bl2}{wh3}
\midp{m_bl2_wh4}{bl2}{wh4}

\midp{m_bl3_wh3}{bl3}{wh3}
\midp{m_bl3_wh4}{bl3}{wh4}

\midp{m_bl4_wh1}{bl4}{wh1}
\midp{m_bl4_wh4}{bl4}{wh4}
\midp{m_bl4_wh5}{bl4}{wh5}

\midp{m_bl5_wh6}{bl5}{wh6}
\midp{m_bl5_wh5}{bl5}{wh5}

\midp{m_bl6_wh1}{bl6}{wh1}
\midp{m_bl6_wh5}{bl6}{wh5}
\midp{m_bl6_wh6}{bl6}{wh6}
\midp{m_bl6_wh7}{bl6}{wh7}

\midp{m_bl7_wh7}{bl7}{wh7}
\midp{m_bl7_wh6}{bl7}{wh6}

\end{tikzpicture}
};

\node[anchor=center] (C) at ([xshift=4.2cm]B.east) {%

\begin{tikzpicture}[scale=2,cm={1,0,1,1,(0,0)}]

\coordinate (S1) at (3,0);
\coordinate (S2) at (2,0);
\coordinate (S3) at (4,-1);

\coordinate[wvert] (wh1) at (S1);
\coordinate[wvert] (wh2) at ($(S1)+(-2/3, 1/3)$);
\coordinate[wvert] (wh3) at (S2);
\coordinate[wvert] (wh4) at ($(S1)+(-1/3,-1/3)$);
\coordinate[wvert] (wh5) at ($(S1)+( 1/3,-2/3)$);
\coordinate[wvert] (wh6) at (S3);
\coordinate[wvert] (wh7) at ($(S1)+( 2/3,-1/3)$);

\coordinate (w1a) at ($(S1)+( 1/3, 1/3)$);
\coordinate (w2a) at ($(S1)+(-1/3, 2/3)$);

\coordinate (w1b) at ($(S2)+( 1/3, 1/3)$);
\coordinate (w5b) at ($(S2)+( 1/3,-2/3)$);
\coordinate[wvert] (w6b) at ($(S2)+( 2/3,-1/3)$);

\coordinate (w2c) at ($(S3)+(-1/3, 2/3)$);
\coordinate (w4c) at ($(S3)+(-1/3,-1/3)$);


\coordinate[bvert] (bl1) at ($(w2a)!0.5!(wh2)$);
\coordinate[bvert]  (bl2) at ($(wh2)!0.5!(wh4)$);
\coordinate[bvert] (bl3) at ($(w5b)!0.5!(w6b)$);
\coordinate[bvert](bl4) at ($(wh4)!0.5!(wh5)$);
\coordinate[bvert]  (bl5) at ($(wh5)!0.5!(w4c)$); 
\coordinate[bvert] (bl6) at ($(wh5)!0.5!(wh7)$);
\coordinate[bvert] (bl7) at ($(wh7)!0.5!(w1a)$);

\coordinate[bvert] (b61b) at ($(w6b)!0.5!(w1b)$);
\coordinate[bvert] (b23c) at ($(w2c)!0.5!(wh5)$); 

\draw (wh1) -- (bl1);
\draw (wh1) -- (bl2);
\draw (wh1) -- (bl4);
\draw (wh1) -- (bl6);
\draw (wh1) -- (bl7);

\draw (bl1) -- (wh2);

\draw (bl2) -- (wh2);
\draw (bl2) -- (wh4);

\draw (bl4) -- (wh4);
\draw (bl4) -- (wh5);

\draw (bl6) -- (wh5);
\draw (bl6) -- (wh7);

\draw (bl7) -- (wh7);

\draw (wh3) -- (bl3);
\draw (wh3) -- (b61b);
\draw (bl3) -- (w6b);

\draw (wh6) -- (b23c);
\draw (wh6) -- (bl5);
\draw (bl5) -- (wh5);

\midp{m_bl1_wh1}{bl1}{wh1}
\midp{m_bl2_wh1}{bl2}{wh1}
\midp{m_bl4_wh1}{bl4}{wh1}
\midp{m_bl6_wh1}{bl6}{wh1}
\midp{m_bl7_wh1}{bl7}{wh1}

\midp{m_bl1_wh2}{bl1}{wh2}
\midp{m_bl1_wh7}{bl1}{wh7}

\midp{m_bl2_wh1}{bl2}{wh1}
\midp{m_bl2_wh2}{bl2}{wh2}
\midp{m_bl2_wh3}{bl2}{wh3}
\midp{m_bl2_wh4}{bl2}{wh4}

\midp{m_bl3_wh3}{bl3}{wh3}
\midp{m_bl3_wh4}{bl3}{wh4}

\midp{m_bl4_wh1}{bl4}{wh1}
\midp{m_bl4_wh4}{bl4}{wh4}
\midp{m_bl4_wh5}{bl4}{wh5}

\midp{m_bl5_wh6}{bl5}{wh6}
\midp{m_bl5_wh5}{bl5}{wh5}

\midp{m_bl6_wh1}{bl6}{wh1}
\midp{m_bl6_wh5}{bl6}{wh5}
\midp{m_bl6_wh6}{bl6}{wh6}
\midp{m_bl6_wh7}{bl6}{wh7}

\midp{m_bl7_wh7}{bl7}{wh7}
\midp{m_bl7_wh6}{bl7}{wh6}

\draw[hollowedge]
  plot[smooth] coordinates {
    (m_bl3_wh3)
      (m_bl2_wh3)
       (m_bl2_wh2)
      (m_bl1_wh2)
  };

\draw[hollowedge]
  plot[smooth] coordinates {
    (m_bl7_wh7)
      (m_bl6_wh7)
       (m_bl6_wh6)
      (m_bl5_wh6)
  };

  \draw[solidedge]
  plot[smooth] coordinates {
  (m_bl7_wh7)
    (m_bl1_wh7)
      (m_bl1_wh1)
       (m_bl1_wh2)
  };

   \draw[solidedge]
  plot[smooth] coordinates {
  (m_bl4_wh1)
    (m_bl4_wh5)
      (m_bl5_wh5)
       (m_bl5_wh6)
  };

\draw[solidedge]
  plot[smooth] coordinates {
  (m_bl3_wh3)
    (m_bl3_wh4)
      (m_bl4_wh4)
       (m_bl4_wh1)
  };

\end{tikzpicture}
};

\end{tikzpicture}
\caption{
An example of an AZ graph for the pentagonal Newton polygon in Figure~\ref{fig:pentagon_0} which we will use as our running example.
}
\label{fig:pentagon}
\end{figure}

\begin{example}
Figures~\ref{fig:pentagon_intro_1} and~\ref{fig:pentagon} show two examples of AZ graphs for the pentagonal Newton polygon in Figure~\ref{fig:pentagon_0}. For this choice of $\graphpl$, the space of AZ graphs forms a two-parameter family, which was studied {from the point of view of cluster algebras} in~\cite{GLVY}.

\end{example}

\begin{figure}
\centering
    
\raisebox{-0.5\height}{
\begin{tikzpicture}[shift={(-6,-1)}]
    \node[bvert](1) at (-1,1){}; 
     \node[bvert](2) at (1,0){}; 
      \node[bvert](3) at (1,1){}; 
       \node[bvert](4) at (0,1){};
           \node[bvert](5) at (0,0){}; 
           \node[bvert](6) at (-1,2){};
            \node[bvert](7) at (0,2){};

           \draw[solidedge]  (5) --
           (2);
           \draw[hollowedge] (3) --
           (7);

  \draw[solidedge] (2) --
  (3);
 \draw[solidedge] (7) --
 (6);
 \draw[solidedge] (6)--
 (1);
 \draw[hollowedge] (1)-- 
 (5)
  ;
       \end{tikzpicture} } \hspace{2cm} \raisebox{-0.5\height}{
\begin{tikzpicture}[scale = 1.5]
\begin{scope}[cm={1/sqrt(3),0,-1/3,2/3,(0,0)}]

  \pgfmathsetmacro{\dx}{sqrt(3)}      
  \pgfmathsetmacro{\hx}{sqrt(3)/2}    
  \pgfmathsetmacro{\hy}{3/2}

  \coordinate (C) at ({-1+0.35*\dx + -1*\hx},{-0.5*\hy});

  \fill[gray!30]
    ($(C)+(210:1)$) --
    ($(C)+(330:1)$) --
    ($(C)+($(330:1)+(90:1)-(210:1)$)$) --
    ($(C)+(90:1)$) --
    cycle;

\begin{scope}
   \coordinate[wvert] (w0) at (0,0);
      \coordinate[wvert] (w1) at (30:1);
      \coordinate[wvert] (w2) at (90:1);
      \coordinate[wvert] (w3) at (150:1);
      \coordinate[wvert] (w4) at (210:1);
      \coordinate[wvert] (w5) at (270:1);
      \coordinate[wvert] (w6) at (330:1);

      \coordinate[bvert] (b12) at ($(w1)!0.5!(w2)$);
      \coordinate[bvert] (b23) at ($(w2)!0.5!(w3)$);
      \coordinate[bvert] (b34) at ($(w3)!0.5!(w4)$);
      \coordinate[bvert] (b45) at ($(w4)!0.5!(w5)$);
      \coordinate[bvert] (b56) at ($(w5)!0.5!(w6)$);
      \coordinate[bvert] (b61) at ($(w6)!0.5!(w1)$);

      \draw (w0) -- (b12);
      \draw (w0) -- (b23);
      \draw (w0) -- (b34);
      \draw (w0) -- (b45);
      \draw (w0) -- (b56);
      \draw (w0) -- (b61);

      \draw (b12) -- (w1) (b12) -- (w2);
      \draw (b23) -- (w2) (b23) -- (w3);
      \draw (b34) -- (w3) (b34) -- (w4);
      \draw (b45) -- (w4) (b45) -- (w5);
      \draw (b56) -- (w5) (b56) -- (w6);
      \draw (b61) -- (w6) (b61) -- (w1);
\end{scope}

\begin{scope}[shift={(1*\dx +0*\hx,0*\hy)}]
   \coordinate[wvert] (w0) at (0,0);
      \coordinate[wvert] (w1) at (30:1);
      \coordinate[wvert] (w2) at (90:1);
      \coordinate[wvert] (w3) at (150:1);
      \coordinate[wvert] (w4) at (210:1);
      \coordinate[] (w5) at (270:1);
      \coordinate[wvert] (w6) at (330:1);

      \coordinate[bvert] (b12) at ($(w1)!0.5!(w2)$);
      \coordinate[bvert] (b23) at ($(w2)!0.5!(w3)$);
      \coordinate[bvert] (b34) at ($(w3)!0.5!(w4)$);
      \coordinate[bvert] (b45) at ($(w4)!0.5!(w5)$);
      \coordinate[bvert] (b56) at ($(w5)!0.5!(w6)$);
      \coordinate[bvert] (b61) at ($(w6)!0.5!(w1)$);

      \draw (w0) -- (b12);
      \draw (w0) -- (b23);
      \draw (w0) -- (b34);
      \draw (w0) -- (b45);
      \draw (w0) -- (b56);
      \draw (w0) -- (b61);

      \draw (b12) -- (w1) (b12) -- (w2);
      \draw (b23) -- (w2) (b23) -- (w3);
      \draw (b34) -- (w3) (b34) -- (w4);
      \draw (b45) -- (w4) ;
      \draw  (b56) -- (w6);
      \draw (b61) -- (w6) (b61) -- (w1);
\end{scope}

\begin{scope}[shift={(0*\dx +1*\hx,1*\hy)}]
   \coordinate[wvert] (w0) at (0,0);
      \coordinate[wvert] (w1) at (30:1);
      \coordinate[wvert] (w2) at (90:1);
      \coordinate[] (w3) at (150:1);
      \coordinate[wvert] (w4) at (210:1);
      \coordinate[wvert] (w5) at (270:1);
      \coordinate[wvert] (w6) at (330:1);

      \coordinate[bvert] (b12) at ($(w1)!0.5!(w2)$);
      \coordinate[bvert] (b23) at ($(w2)!0.5!(w3)$);
      \coordinate[bvert] (b34) at ($(w3)!0.5!(w4)$);
      \coordinate[bvert] (b45) at ($(w4)!0.5!(w5)$);
      \coordinate[bvert] (b56) at ($(w5)!0.5!(w6)$);
      \coordinate[bvert] (b61) at ($(w6)!0.5!(w1)$);

      \draw (w0) -- (b12);
      \draw (w0) -- (b23);
      \draw (w0) -- (b34);
      \draw (w0) -- (b45);
      \draw (w0) -- (b56);
      \draw (w0) -- (b61);

      \draw (b12) -- (w1) (b12) -- (w2);
      \draw (b23) -- (w2) ;
      \draw  (b34) -- (w4);
      \draw (b45) -- (w4) (b45) -- (w5);
      \draw (b56) -- (w5) (b56) -- (w6);
      \draw (b61) -- (w6) (b61) -- (w1);
\end{scope}

\begin{scope}[shift={(-1*\dx +1*\hx,-1*\hy)}]
   \coordinate[wvert] (w0) at (0,0);
      \coordinate[wvert] (w1) at (30:1);
      \coordinate[wvert] (w2) at (90:1);
      \coordinate[wvert] (w3) at (150:1);
      \coordinate[wvert] (w4) at (210:1);
      \coordinate[wvert] (w5) at (270:1);
      \coordinate[] (w6) at (330:1);

      \coordinate[bvert] (b12) at ($(w1)!0.5!(w2)$);
      \coordinate[bvert] (b23) at ($(w2)!0.5!(w3)$);
      \coordinate[bvert] (b34) at ($(w3)!0.5!(w4)$);
      \coordinate[bvert] (b45) at ($(w4)!0.5!(w5)$);
      \coordinate[bvert] (b56) at ($(w5)!0.5!(w6)$);
      \coordinate[bvert] (b61) at ($(w6)!0.5!(w1)$);

      \draw (w0) -- (b12);
      \draw (w0) -- (b23);
      \draw (w0) -- (b34);
      \draw (w0) -- (b45);
      \draw (w0) -- (b56);
      \draw (w0) -- (b61);

      \draw (b12) -- (w1) (b12) -- (w2);
      \draw (b23) -- (w2) (b23) -- (w3);
      \draw (b34) -- (w3) (b34) -- (w4);
      \draw (b45) -- (w4) (b45) -- (w5);
      \draw (b56) -- (w5) ;
      \draw  (b61) -- (w1);
\end{scope}

\begin{scope}[shift={(-1*\dx +0*\hx,0*\hy)}]
   \coordinate[wvert] (w0) at (0,0);
      \coordinate[wvert] (w1) at (30:1);
      \coordinate[] (w2) at (90:1);
      \coordinate[wvert] (w3) at (150:1);
      \coordinate[wvert] (w4) at (210:1);
      \coordinate[wvert] (w5) at (270:1);
      \coordinate[wvert] (w6) at (330:1);

      \coordinate[bvert] (b12) at ($(w1)!0.5!(w2)$);
      \coordinate[bvert] (b23) at ($(w2)!0.5!(w3)$);
      \coordinate[bvert] (b34) at ($(w3)!0.5!(w4)$);
      \coordinate[bvert] (b45) at ($(w4)!0.5!(w5)$);
      \coordinate[bvert] (b56) at ($(w5)!0.5!(w6)$);
      \coordinate[bvert] (b61) at ($(w6)!0.5!(w1)$);

      \draw (w0) -- (b12);
      \draw (w0) -- (b23);
      \draw (w0) -- (b34);
      \draw (w0) -- (b45);
      \draw (w0) -- (b56);
      \draw (w0) -- (b61);

      \draw (b12) -- (w1);
      \draw (b23) -- (w3);
      \draw (b34) -- (w3) (b34) -- (w4);
      \draw (b45) -- (w4) (b45) -- (w5);
      \draw (b56) -- (w5) (b56) -- (w6);
      \draw (b61) -- (w6) (b61) -- (w1);
\end{scope}

\begin{scope}[shift={(-2*\dx +0*\hx,0*\hy)}]
   \coordinate[wvert] (w0) at (0,0);
      \coordinate[wvert] (w1) at (30:1);
      \coordinate[] (w2) at (90:1);
      \coordinate[] (w3) at (150:1);
      \coordinate[] (w4) at (210:1);
      \coordinate[] (w5) at (270:1);
      \coordinate[] (w6) at (330:1);

      \coordinate[bvert] (b12) at ($(w1)!0.5!(w2)$);
     
      \coordinate[bvert] (b61) at ($(w6)!0.5!(w1)$);

      \draw (w0) -- (b12);
     
      \draw (w0) -- (b61);

      \draw (b12) -- (w1) ;

      \draw (b61) -- (w1);
\end{scope}

\begin{scope}[shift={(-2*\dx +1*\hx,-1*\hy)}]
   \coordinate[wvert] (w0) at (0,0);
      \coordinate[wvert] (w1) at (30:1);
      \coordinate[] (w2) at (90:1);
      \coordinate[] (w3) at (150:1);
      \coordinate[] (w4) at (210:1);
      \coordinate[] (w5) at (270:1);
      \coordinate[] (w6) at (330:1);

      \coordinate[bvert] (b12) at ($(w1)!0.5!(w2)$);
     
      \coordinate[bvert] (b61) at ($(w6)!0.5!(w1)$);

      \draw (w0) -- (b12);
     
      \draw (w0) -- (b61);

      \draw (b12) -- (w1) ;

      \draw (b61) -- (w1);
\end{scope}

\begin{scope}[shift={(-2*\dx +2*\hx,-2*\hy)}]
   \coordinate[wvert] (w0) at (0,0);
      \coordinate[wvert] (w1) at (30:1);
      \coordinate[] (w2) at (90:1);
      \coordinate[] (w3) at (150:1);
      \coordinate[] (w4) at (210:1);
      \coordinate[] (w5) at (270:1);
      \coordinate[] (w6) at (330:1);

      \coordinate[bvert] (b12) at ($(w1)!0.5!(w2)$);
     
      \coordinate[bvert] (b61) at ($(w6)!0.5!(w1)$);

      \draw (w0) -- (b12);
     
      \draw (w0) -- (b61);

      \draw (b12) -- (w1) ;

      \draw (b61) -- (w1);
\end{scope}

\begin{scope}[shift={(2*\dx +0*\hx,0*\hy)}]
   \coordinate[wvert] (w0) at (0,0);
      \coordinate[] (w1) at (30:1);
      \coordinate[] (w2) at (90:1);
      \coordinate[] (w3) at (150:1);
      \coordinate[wvert] (w4) at (210:1);
      \coordinate[] (w5) at (270:1);
      \coordinate[] (w6) at (330:1);

      \coordinate[bvert] (b34) at ($(w3)!0.5!(w4)$);
      \coordinate[bvert] (b45) at ($(w4)!0.5!(w5)$);
 

      \draw (w0) -- (b34);
      \draw (w0) -- (b45);

      \draw  (b34) -- (w4);
      \draw (b45) -- (w4) ;
 
\end{scope}

\begin{scope}[shift={(1*\dx +1*\hx,1*\hy)}]
   \coordinate[wvert] (w0) at (0,0);
      \coordinate[] (w1) at (30:1);
      \coordinate[] (w2) at (90:1);
      \coordinate[] (w3) at (150:1);
      \coordinate[wvert] (w4) at (210:1);
      \coordinate[] (w5) at (270:1);
      \coordinate[] (w6) at (330:1);

      \coordinate[bvert] (b34) at ($(w3)!0.5!(w4)$);
      \coordinate[bvert] (b45) at ($(w4)!0.5!(w5)$);
 

      \draw (w0) -- (b34);
      \draw (w0) -- (b45);

      \draw  (b34) -- (w4);
      \draw (b45) -- (w4) ;
 
\end{scope}

\begin{scope}[shift={(0*\dx +2*\hx,2*\hy)}]
   \coordinate[wvert] (w0) at (0,0);
      \coordinate[] (w1) at (30:1);
      \coordinate[] (w2) at (90:1);
      \coordinate[] (w3) at (150:1);
      \coordinate[wvert] (w4) at (210:1);
      \coordinate[] (w5) at (270:1);
      \coordinate[] (w6) at (330:1);

      \coordinate[bvert] (b34) at ($(w3)!0.5!(w4)$);
      \coordinate[bvert] (b45) at ($(w4)!0.5!(w5)$);
 

      \draw (w0) -- (b34);
      \draw (w0) -- (b45);

      \draw  (b34) -- (w4);
      \draw (b45) -- (w4) ;
 
\end{scope}

\end{scope}
\end{tikzpicture}}
\caption{A hexagonal Newton polygon and an example of an AZ graph. }\label{fig:hexagon}
\end{figure}
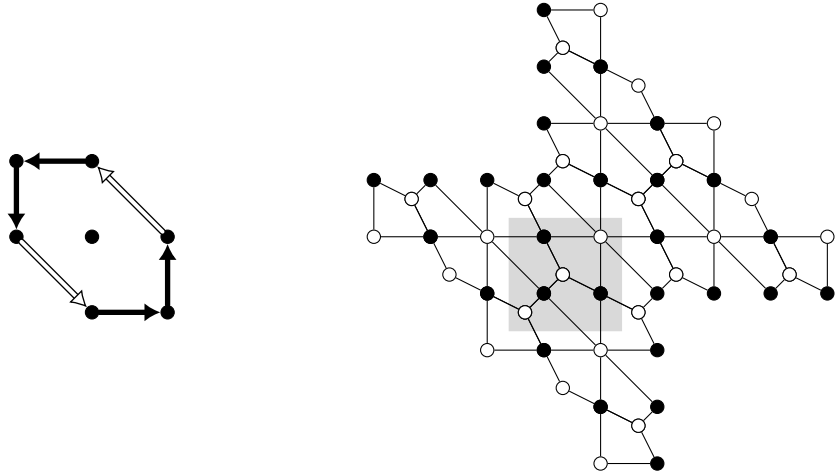

\begin{example}
Let $N$ be the hexagon shown in Figure~\ref{fig:hexagon}(left). Figure~\ref{fig:hexagon}(right) shows a choice of fundamental domain for $\graphtor$ and an example of an AZ graph. The class of AZ graphs for this $\graphpl$ has been studied extensively {from the cluster algebra point of view} in~\cite{CY10,dP315,dP32014,dP32017,dP324}.
\end{example}

\newcommand{\DrawMyTileFig}{%
  \coordinate (LT) at (0,4);
  \coordinate (LU) at (0,3);
  \coordinate (LM) at (0,2);
  \coordinate (LB) at (0,0);

  \coordinate (TL) at (1,4);
  \coordinate (TR) at (3,4);
  \coordinate (RT) at (4,4);
  \coordinate (RU) at (4,3);
  \coordinate (RM) at (4,2);
  \coordinate (RB) at (4,0);

  \coordinate (BL) at (1,0);
  \coordinate (BR) at (3,0);

  \coordinate (aL) at (1.32,3.35);
  \coordinate (bL) at (1.30,2.32);
  \coordinate (cL) at (0.55,2.00);
  \coordinate (dL) at (1.30,1.55);
  \coordinate (eL) at (1.32,0.52);

  \coordinate (ct) at (2,2.95);
  \coordinate (cm) at (2,2.00);
  \coordinate (cb) at (2,0.95);

  \coordinate (aR) at (2.68,3.35);
  \coordinate (bR) at (2.70,2.32);
  \coordinate (cR) at (3.45,2.00);
  \coordinate (dR) at (2.70,1.55);
  \coordinate (eR) at (2.68,0.52);

  \draw (LT) -- (TL);
  \draw (TR) -- (RT);
  \draw (RU) -- (RM) -- (RB);
  \draw (LU) -- (LM) -- (LB);
  \draw (LB) -- (BL);
  \draw (BR) -- (RB);

  \draw (LT) -- (bL);
  \draw (LU) -- (bL);
  \draw (LM) -- (cL);

  \draw (TL) -- (aL);
  \draw (aL) -- (ct);
  \draw (aL) -- (bL);

  \draw (bL) -- (cL);
  \draw (cL) -- (dL);
  \draw (bL) -- (cm);
  \draw (dL) -- (cm);

  \draw (dL) -- (eL);
  \draw (BL) -- (eL);
  \draw (eL) -- (cb);

  \draw (ct) -- (cm) -- (cb);

  \draw (ct) -- (aR);
  \draw (aR) -- (TR);
  \draw (aR) -- (bR);

  \draw (RT) -- (bR);
  \draw (RU) -- (bR);
  \draw (RM) -- (cR);

  \draw (bR) -- (cR);
  \draw (cR) -- (dR);
  \draw (bR) -- (cm);
  \draw (dR) -- (cm);

  \draw (dR) -- (eR);
  \draw (BR) -- (eR);
  \draw (eR) -- (cb);

  \foreach \p in {LM,RM,TL,TR,BL,BR,bL,dL,ct,cb,bR,dR}
    \node[bvert] at (\p) {};

  \foreach \p in {LT,LU,LB,RT,RU,RB,aL,cL,eL,cm,aR,cR,eR}
    \node[wvert] at (\p) {};
}

\newcommand{\OctCoordsFig}{%
  \coordinate (LT) at (0,4);
  \coordinate (LU) at (0,3);
  \coordinate (LM) at (0,2);
  \coordinate (LB) at (0,0);

  \coordinate (TL) at (1,4);
  \coordinate (TR) at (3,4);
  \coordinate (RT) at (4,4);
  \coordinate (RU) at (4,3);
  \coordinate (RM) at (4,2);
  \coordinate (RB) at (4,0);

  \coordinate (BL) at (1,0);
  \coordinate (BR) at (3,0);

  \coordinate (aL) at (1.32,3.35);
  \coordinate (bL) at (1.30,2.32);
  \coordinate (cL) at (0.55,2.00);
  \coordinate (dL) at (1.30,1.55);
  \coordinate (eL) at (1.32,0.52);

  \coordinate (ct) at (2,2.95);
  \coordinate (cm) at (2,2.00);
  \coordinate (cb) at (2,0.95);

  \coordinate (aR) at (2.68,3.35);
  \coordinate (bR) at (2.70,2.32);
  \coordinate (cR) at (3.45,2.00);
  \coordinate (dR) at (2.70,1.55);
  \coordinate (eR) at (2.68,0.52);
}

\newcommand{\OctFullFig}{%
  \OctCoordsFig
  \draw (LT) -- (TL);
  \draw (RM) -- (RB);
  \draw (LU) -- (LM) -- (LB);
  \draw (LB) -- (BL);
  \draw (BR) -- (RB);

  \draw (LT) -- (bL);
  \draw (LU) -- (bL);
  \draw (LM) -- (cL);

  \draw (TL) -- (aL);
  \draw (aL) -- (ct);
  \draw (aL) -- (bL);

  \draw (bL) -- (cL);
  \draw (cL) -- (dL);
  \draw (bL) -- (cm);
  \draw (dL) -- (cm);

  \draw (dL) -- (eL);
  \draw (BL) -- (eL);
  \draw (eL) -- (cb);

  \draw (ct) -- (cm) -- (cb);

  \draw (ct) -- (aR);
  \draw (aR) -- (TR);
  \draw (aR) -- (bR);

  \draw (RM) -- (cR);

  \draw (bR) -- (cR);
  \draw (cR) -- (dR);
  \draw (bR) -- (cm);
  \draw (dR) -- (cm);

  \draw (dR) -- (eR);
  \draw (BR) -- (eR);
  \draw (eR) -- (cb);

  \foreach \p in {LM,RM,TL,TR,BL,BR,bL,dL,ct,cb,bR,dR}
    \node[bvert] at (\p) {};
  \foreach \p in {LT,LU,LB,RB,aL,cL,eL,cm,aR,cR,eR}
    \node[wvert] at (\p) {};
}

\newcommand{\OctFullPFig}{%
  \OctCoordsFig
  \draw (LB) -- (BL) -- (eL) -- (cb) -- (eR) -- (BR)
        (cb) -- (cm) -- (dR) -- (eR);
  \foreach \p in {BL,cb,BR,dR}
    \node[bvert] at (\p) {};
  \foreach \p in {LB,eL,eR,cm}
    \node[wvert] at (\p) {};
}

\newcommand{\OctWFig}{%
  \OctCoordsFig
  \draw (bR) -- (RU)
        (bR) -- (RT);
  \foreach \p in {RM,bR}
    \node[bvert] at (\p) {};
  \foreach \p in {RT,RU}
    \node[wvert] at (\p) {};
}

\newcommand{\OctOOFig}{%
  \OctCoordsFig
  \draw (RU) -- (RM) -- (RB);
  \draw (LU) -- (LM);
  \draw (BR) -- (RB);

  \draw (LT) -- (bL);
  \draw (LU) -- (bL);
  \draw (LM) -- (cL);

  \draw (aL) -- (ct);
  \draw (aL) -- (bL);

  \draw (bL) -- (cL);
  \draw (cL) -- (dL);
  \draw (bL) -- (cm);
  \draw (dL) -- (cm);

  \draw (dL) -- (eL);
  \draw (BL) -- (eL);
  \draw (eL) -- (cb);

  \draw (ct) -- (cm) -- (cb);

  \draw (ct) -- (aR);
  \draw (aR) -- (bR);

  \draw (RT) -- (bR);
  \draw (RU) -- (bR);
  \draw (RM) -- (cR);

  \draw (bR) -- (cR);
  \draw (cR) -- (dR);
  \draw (bR) -- (cm);
  \draw (dR) -- (cm);

  \draw (dR) -- (eR);
  \draw (BR) -- (eR);
  \draw (eR) -- (cb);

  \foreach \p in {LM,RM,BL,BR,bL,dL,ct,cb,bR,dR}
    \node[bvert] at (\p) {};
  \foreach \p in {LT,LU,RT,RU,RB,aL,cL,eL,cm,aR,cR,eR}
    \node[wvert] at (\p) {};
}

\newcommand{\OctOMFig}{%
  \OctCoordsFig
  \draw (TR) -- (RT);

  \draw (TL) -- (aL);
  \draw (aL) -- (ct);
  \draw (aL) -- (bL);

  \draw (bL) -- (cL);
  \draw (cL) -- (dL);
  \draw (bL) -- (cm);
  \draw (dL) -- (cm);

  \draw (dL) -- (eL);
  \draw (eL) -- (cb);

  \draw (ct) -- (cm) -- (cb);

  \draw (ct) -- (aR);
  \draw (aR) -- (TR);
  \draw (aR) -- (bR);

  \draw (RT) -- (bR);
  \draw (RU) -- (bR);

  \draw (bR) -- (cR);
  \draw (cR) -- (dR);
  \draw (bR) -- (cm);
  \draw (dR) -- (cm);

  \draw (dR) -- (eR);
  \draw (eR) -- (cb);

  \foreach \p in {TL,TR,bL,dL,ct,cb,bR,dR}
    \node[bvert] at (\p) {};
  \foreach \p in {RT,RU,aL,cL,eL,cm,aR,cR,eR}
    \node[wvert] at (\p) {};
}

\newcommand{\OctSEFig}{%
  \OctCoordsFig
  \draw (LT) -- (TL);
  \draw (TR) -- (RT);

  \draw (LT) -- (bL);
  \draw (LU) -- (bL);

  \draw (TL) -- (aL);
  \draw (aL) -- (ct);
  \draw (aL) -- (bL);

  \draw (bL) -- (cL);
  \draw (cL) -- (dL);
  \draw (bL) -- (cm);
  \draw (dL) -- (cm);

  \draw (dL) -- (eL);
  \draw (eL) -- (cb);

  \draw (ct) -- (cm) -- (cb);

  \draw (ct) -- (aR);
  \draw (aR) -- (TR);
  \draw (aR) -- (bR);

  \draw (RT) -- (bR);
  \draw (RU) -- (bR);

  \draw (bR) -- (cR);
  \draw (cR) -- (dR);
  \draw (bR) -- (cm);
  \draw (dR) -- (cm);

  \draw (dR) -- (eR);
  \draw (eR) -- (cb);

  \foreach \p in {TL,TR,bL,dL,ct,cb,bR,dR}
    \node[bvert] at (\p) {};
  \foreach \p in {LT,LU,RT,RU,aL,cL,eL,cm,aR,cR,eR}
    \node[wvert] at (\p) {};
}

\newcommand{\OctSEEFig}{%
  \OctCoordsFig
  \draw (LT) -- (TL);

  \draw (LT) -- (bL);
  \draw (LU) -- (bL);

  \draw (TL) -- (aL);
  \draw (aL) -- (ct);
  \draw (aL) -- (bL);

  \draw (bL) -- (cL);
  \draw (cL) -- (dL);
  \draw (bL) -- (cm);
  \draw (dL) -- (cm);

  \draw (dL) -- (eL);
  \draw (eL) -- (cb);

  \draw (ct) -- (cm) -- (cb);

  \draw (bR) -- (cR);
  \draw (cR) -- (dR);
  \draw (bR) -- (cm);
  \draw (dR) -- (cm);

  \draw (dR) -- (eR);
  \draw (eR) -- (cb);

  \foreach \p in {TL,bL,dL,ct,cb,bR,dR}
    \node[bvert] at (\p) {};
  \foreach \p in {LT,LU,aL,cL,eL,cm,cR,eR}
    \node[wvert] at (\p) {};
}

\begin{figure}
\centering

\begin{subfigure}{0.2\textwidth}
\begin{tikzpicture}[baseline=(current bounding box.center), line cap=round, line join=round, scale=0.72]
  \node[bvert] (1) at (0,0) {};
  \node[bvert] (2) at (1,0) {};
  \node[bvert] (3) at (2,1) {};
  \node[bvert] (4) at (2,2) {};
  \node[bvert] (5) at (1,3) {};
  \node[bvert] (6) at (0,3) {};
  \node[bvert] (7) at (-1,2) {};
  \node[bvert] (8) at (-1,1) {};
  \node[bvert] at (0,1) {};
  \node[bvert] at (0,2) {};
  \node[bvert] at (1,1) {};
  \node[bvert] at (1,2) {};

  \draw[hollowedge] (1) -- (2);
  \draw[hollowedge] (6) -- (7);
  \draw[hollowedge] (5) -- (6);
  \draw[solidedge]  (8) -- (1);
  \draw[solidedge]  (7) -- (8);
    \draw[solidedge]  (3) -- (4);
     \draw[solidedge]  (4) -- (5);
       \draw[solidedge]  (2) -- (3);
\end{tikzpicture}
\end{subfigure}
\begin{subfigure}{0.33\textwidth}
\centering
\begin{tikzpicture}[baseline=(current bounding box.center), line cap=round, line join=round, scale=0.55]
  \def\dx{1.65}
  \def\dy{-1.14}
  \fill[gray!30](\dx,\dy) rectangle ++(4,4);

  \begin{scope}[shift={(0,0)}]
    \OctOOFig
  \end{scope}
  \begin{scope}[shift={(4,0)}]
    \OctFullFig
  \end{scope}
  \begin{scope}[shift={(0,-4)}]
    \OctOMFig
  \end{scope}
  \begin{scope}[shift={(-4,0)}]
    \OctWFig
  \end{scope}
  \begin{scope}[shift={(4,4)}]
    \OctFullPFig
  \end{scope}
  \begin{scope}[shift={(4,-4)}]
    \OctSEFig
  \end{scope}
  \begin{scope}[shift={(8,-4)}]
    \OctSEEFig
  \end{scope}
\end{tikzpicture}

\end{subfigure}




\caption{An octagonal Newton polygon \(N\) and an AZ graph.}
\label{fig:octagon-fd-graph}
\end{figure}
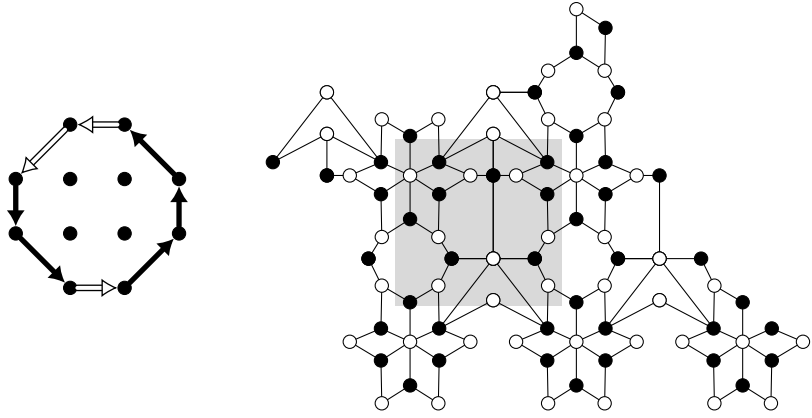

\begin{example}
    Figure~\ref{fig:octagon-fd-graph} shows an example of an AZ graph for an octagonal Newton polygon. The corresponding torus graph was studied in~\cite{octagon}.
\end{example}

\subsection{Criterion for admissibility}

For any $e \in E(N)$, we label the strands $\bm \alpha_e = \{\alpha_e^i: i \in \ZZ +\frac 1 2\}$ parallel to $e$ from left to right so that the reference face $\f_0$ lies in between $\alpha_e^{-\frac 1 2}$ and $\alpha_e^{\frac 1 2}$. Define 
\[
c_{\bm \beta} : E(N) \rightarrow \ZZ+\tfrac 1 2
\]
by $\alpha_e^{c_{\bm \beta}(e)} = \beta_e$.

\begin{lemma}\label{lem:c_admissibility}
For each \(e\in E(N)\),
\[
\deg\bigl((\bm D_{\bm\beta})_{\zz_e}\bigr)
=
-c_{\bm\beta}(e)+
\begin{cases}
-\frac12,& \text{if \(e\) is black},\\
\phantom{-}\frac12,& \text{if \(e\) is white}.
\end{cases}
\]
In particular, \(\graphbeta\) is admissible if and only if
\begin{equation}\label{eq:c_admissible}
\sum_{e\in E(N)} c_{\bm\beta}(e)
=
\frac12\Bigl(\#\{e\in E(N): e\text{ is white}\}-\#\{e\in E(N): e\text{ is black}\}\Bigr).
\end{equation}
\end{lemma}

\begin{proof}
We have, for a $\beta_e$-outer vertex $\bl$,
\[
\dabel(\bl)_{\zz_e} = \begin{cases}
-\alpha^{c_{\bm \beta}(e)+1}_{e}- \cdots - \alpha_e^{-\frac 1 2},   &\text{if $c_{\bm \beta}(e) <0$},\\
 \alpha_e^{\frac 1 2}+ \cdots +\alpha_e^{c_{\bm \beta}(e)},  &\text{if $c_{\bm \beta}(e) > 0$},\\
\end{cases}
\]
if $e$ is black, so $\deg ((\bm D_{\bm \beta})_{\zz_e})=-c_{\bm \beta}(e)-\frac 1 2$.
Similarly, for a $\beta_e$-outer vertex $\wh$,
\[
\dabel(\wh)_{\zz_e} = \begin{cases}
-\alpha^{c_{\bm \beta}(e)}_e- \cdots - \alpha_e^{-\frac 1 2},   &\text{if $c_{\bm \beta}(e) <0$},\\
 \alpha_e^{\frac 1 2}+ \cdots +\alpha_e^{c_{\bm \beta}(e)-1},  &\text{if $c_{\bm \beta}(e) > 0$},\\
\end{cases}
\]
if $e$ is white, so $\deg ((\bm D_{\bm \beta})_{\zz_e}) = -c_{\bm \beta}(e)+\frac 1 2$. Summing over all edges gives~\eqref{eq:c_admissible}.
\end{proof}

\begin{figure}
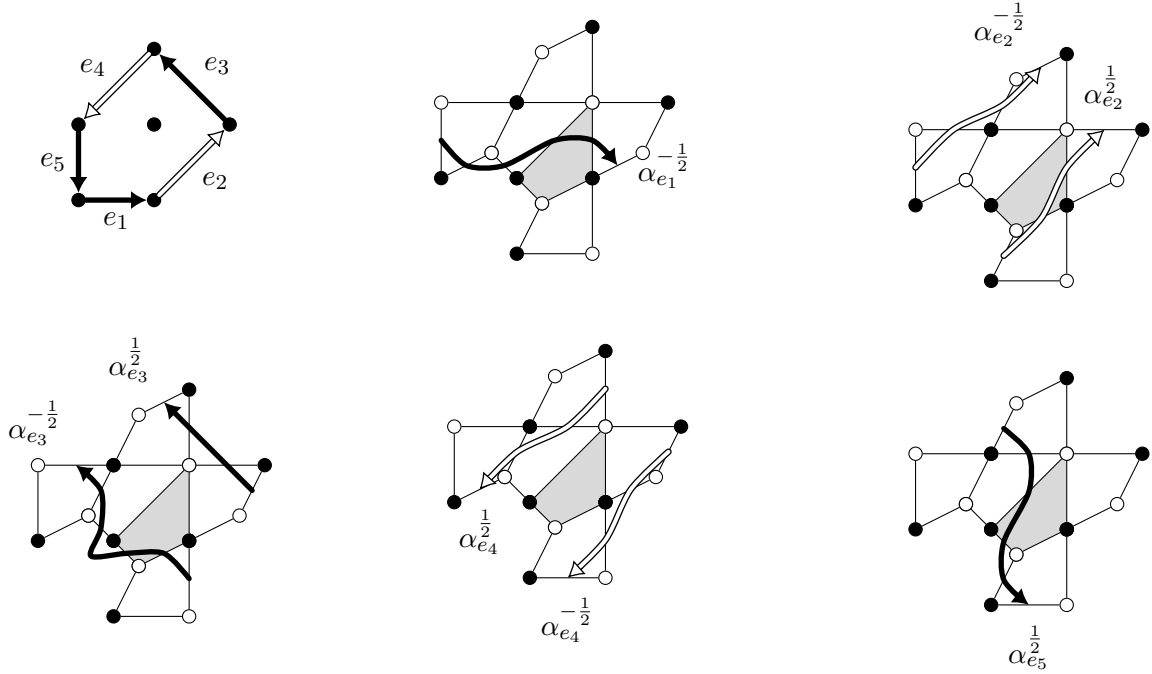

\centering


};

\end{tikzpicture}

\caption{Classes of parallel strands for the graph in Figure~\ref{fig:pentagon}. {The shaded face is the reference face $\f_0$.}}

\label{fig:pentagon_strands}
\end{figure}

\begin{example}
Let us verify that the graph in Figure~\ref{fig:pentagon} is admissible.
The strands are shown in Figure~\ref{fig:pentagon_strands}, from which we read
\[
c_{\bm \beta}(e_1) = -\frac12,\quad
c_{\bm \beta}(e_2) = -\frac12,\quad
c_{\bm \beta}(e_3) = \frac12,\quad
c_{\bm \beta}(e_4) = -\frac12,\quad
c_{\bm \beta}(e_5) = \frac12.
\]
Summing gives
\[
\sum_{i=1}^5 c_{\bm \beta}(e_i)
= -\frac12
= \frac12(2-3),
\]
so $\graphbeta$ is admissible by Lemma~\ref{lem:c_admissibility}.
Alternatively, from Figure~\ref{fig:pentagon} and Figure~\ref{fig:pentagon_strands} we see that
\begin{equation} \label{eq:bigD}
\bm D_{\bm \beta} = \alpha_{e_2}^{-\frac12}-\alpha_{e_3}^{\frac12}+\alpha_{e_4}^{-\frac12}-\alpha_{e_5}^{\frac12},
\end{equation}
hence $\deg(\bm D_{\bm \beta}) = 0$, which again shows admissibility.
\end{example}

\subsection{Characterization of AZ graphs}


\begin{figure}
\centering

\begin{tikzpicture}
    \filldraw[fill=gray!30, draw=none]
(-1,-1) -- (-1,5) -- (1,5) -- (0,4) -- (1,3) -- (0,2) -- (1,1) -- (2,0) -- (3,1) -- (4,0) -- (5,1) -- (5,-1) -- (-1,-1);

    \coordinate[bvert] (b0) at (0,0) {};
     \coordinate[bvert] (b1) at (2,0) {};
      \coordinate[bvert] (b2) at (4,0) {};
      \coordinate[bvert] (b3) at (0,2) {};
     \coordinate[bvert] (b4) at (0,4) {};

     \coordinate[wvert] (w0) at (1,1) {};
     \coordinate[wvert] (w1) at (3,1) {};
      \coordinate[wvert] (w2) at (5,1) {};
      \coordinate[wvert] (w3) at (1,3) {};
     \coordinate[wvert] (w4) at (1,5) {};

\draw[] (w4) -- (b4) -- (w3) -- (b3) -- (w0) -- (b1) -- (w1) -- (b2) -- (w2)
(b0) -- node[above left]{$\e_v$} (w0)
;

\draw[] (w1) edge (3.5,1.75) edge (2.5,1.75);
\draw[] (w3) edge (1.75,3.5) edge (1.75,2.5);
\draw[] (w0) edge (1.18,1.88) edge (1.88,1.18);

\draw[] (b1) edge (2.5,-0.75) edge (1.5,-0.75);
\draw[] (b2) edge (4.5,-0.75) edge (3.5,-0.75);

\draw[] (b3) edge (-0.75,2.5) edge (-0.75,1.5);
\draw[] (b4) edge (-0.75,4.5) edge (-0.75,3.5);

\draw[] (b0) edge (-0.18,-0.88) edge (-0.88,-0.18);

\tikzset{
  xvert/.style={
    cross out,
    draw,
    line width=0.6pt,
    minimum size=5pt,
    inner sep=0pt,
    outer sep=0pt
  }
}

\draw[line width = 0.7mm, red] 
(b0) -- (w0)
(b1) -- (w1)
(b2) -- (w2)
($(b3)+(0.03,-0.03)$) -- ($(w3)+(0.03,-0.03)$)
(b4) -- (w4)
;

\draw[line width = 0.7mm, blue] 
(b0) -- (-0.18,-0.88)
(b1) -- (w0)
(b2) -- (w1)
($(b3)+(-0.03,0.03)$) -- ($(w3)+(-0.03,0.03)$)
(b4) -- (-0.75,4.5)
;

   \coordinate[bvert] (b0) at (0,0) {};
     \coordinate[bvert] (b1) at (2,0) {};
      \coordinate[bvert] (b2) at (4,0) {};
      \coordinate[bvert] (b3) at (0,2) {};
     \coordinate[bvert] (b4) at (0,4) {};

     \coordinate[wvert] (w0) at (1,1) {};
     \coordinate[wvert] (w1) at (3,1) {};
      \coordinate[wvert] (w2) at (5,1) {};
      \coordinate[wvert] (w3) at (1,3) {};
     \coordinate[wvert] (w4) at (1,5) {};

     \coordinate[xvert] (f0) at (1.5,1.5) {};
     \coordinate[xvert] (f1) at (2,.5) {};
     \coordinate[xvert] (f2) at (3,1.5) {};
     \coordinate[xvert] (f3) at (4,.5) {};
     \coordinate[xvert] (f5) at (.5,2) {};
     \coordinate[xvert] (f6) at (1.5,3) {};
     \coordinate[xvert] (f7) at (.5,4) {};

     \draw[dotted, very thick] (f3) -- (f2) -- (f1) -- (f0) -- (f5);

\end{tikzpicture}

\caption{
The faces marked by $\times$ are the \(\beta_{e_-}\)- and \(\beta_{e_+}\)-outer faces associated with the vertex \(v=e_-\cap e_+\in V(N)\). The shaded region indicates \(\graphbeta\). The blue edges form a dimer cover \(M\) of \(\graphbeta\), while the red edges form the extremal dimer cover \(M_v\). Any dual path between two marked outer faces, such as the dotted path shown, crosses no edges of either \(M\) or \(M_v\). Hence, the difference \(h_{M_v}(\f)-\partial h(\f)\) in~\eqref{eq:const_height_diff} is constant on these outer faces.
}
\label{fig:char_fig} 
\end{figure}
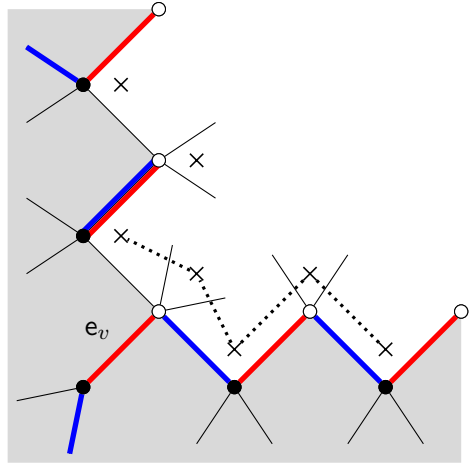

In this section we prove that, among graphs whose boundary consists of zig-zag paths in the prescribed order, the only ones that admit dimer covers are the AZ graphs. The converse, namely that every AZ graph admits a dimer cover, is proved later in Corollary~\ref{cor:dimer_cover}.

Let \(\partial \graphbeta^\times\), \(\partial \graphbeta\), and \(\graphbeta\) be as in Section~\ref{sec:def_AZ_graphs}. {We say that a face \( \f \in F(\graphpl) \setminus F(\graphbeta) \) is \emph{outer} if it is adjacent to a vertex of \(\graphbeta\), and \emph{\(\beta_e\)-outer} for $e \in E(N)$ if it is adjacent to a vertex in the zig-zag path corresponding to the strand \(\beta_e\).}

\begin{proposition}\label{prop:AZ_graphs_height_function}
Assume that Items~$1$ and~$2$ in Definition~\ref{def:azgraph} hold. If~$\graphbeta$ admits a dimer cover, then it is an AZ graph.
\end{proposition}
\begin{proof}
We prove that if $\graphbeta$ admits a dimer cover, then Item 3 in Definition~\ref{def:azgraph}, \emph{i.e.} the admissibility condition, is satisfied. 

Fix a vertex~$u_0\in V(N)$ and let $M_0:=M_{u_0}$ denote the extremal dimer cover from Section~\ref{sec:slope}. Let~$M$ be a dimer cover of~$\graphbeta$ and let~$h_M$ be the associated height function on~$F(\graphbeta)$ with reference dimer cover $M_{0}$ and reference face $\f_0$. {The union of two dimer covers $M$ and $M'$ of \(\graphbeta\) forms a collection of loops. The nontrivial loops can be interpreted as level lines of the height difference \(h_M - h_{M'}\). In particular, this difference vanishes on the outer faces, so \(h_M\) and \(h_{M'}\) induce the same height function on the outer faces, \emph{i.e.} the boundary values of \(h_M\) are independent of \(M\).}
We denote this common boundary height function by \(\partial h\).


Recall that for \(v\in V(N)\), the extremal dimer cover \(M_v\) is, by definition, the dimer cover consisting of all edges \(\e\) such that \(e(\alpha)<v<e(\beta)\) along \(\partial N\), where \(\alpha\) and \(\beta\) are the two strands passing though \(\e\) as in Figure~\ref{fig:fock_kast}. In particular, if $e \in E(N)$ is an edge containing $v$, then every second edge of $\beta_e$ is contained in $M_v$. Hence, 
\begin{equation}\label{eq:const_height_diff}
h_{M_v}(\f) - \partial h(\f)
\end{equation}
is constant for all $\beta_{e_-}$- and $\beta_{e_+}$-outer faces $\f$, where $v = e_- \cap e_+$; see Figure~\ref{fig:char_fig}. We denote this constant value by $d_v$. 

Now let $\f$ be any $\beta_e$-outer face for $e = [v_-,v_+]$. Then
\begin{equation}\label{eq:dv_dv}
d_{v_+}-d_{v_-}=\sum_{\e\in E(\graphpl)}(\e \wedge \gamma^*)\EE[\one_{\e\in M_{v_+}}-\one_{\e\in M_{v_-}}],
\end{equation}
where $\gamma^*$ is any dual path from the reference face $\f_0$ to $\f$. The right-hand side of~\eqref{eq:dv_dv} can be expressed in terms of the whole-plane inverse Kasteleyn matrices from Section~\ref{sec:whole_plane_inverse_K} as
\begin{align*}
\sum_{\e\in E(\graphpl)}(\e \wedge \gamma^*)\EE[\one_{\e\in M_{v_+}}-\one_{\e\in M_{v_-}}]
&=\sum_{\e=\bl\wh\in E(\graphpl)}(\e \wedge \gamma^*)\left(\mathsf{A}_{\bl,\wh}^{v_+}\kast_{\wh,\bl}-\mathsf{A}_{\bl,\wh}^{v_-}\kast_{\wh,\bl}\right) \\
&=-\sum_{\e\in E(\graphpl)}(\e \wedge \gamma^*)\frac{1}{2\pi\i}\int_{C([v_-,v_+])}g_{\bl,\wh}\cdot\kast_{\wh,\bl}.
\end{align*}
If~$\e=\bl \wh \in E(\graphpl)$, and~$\f_+,\f_-\in F(\graphpl)$ are adjacent to~$\e$ as in Figure~\ref{fig:fock_kast}, then by Lemma~\ref{lem:fays},
\begin{equation}
\frac{1}{2\pi\i}\int_{C([v_+,v_-])}g_{\bl,\wh}\cdot\kast_{\wh,\bl}=\frac{1}{2\pi\i}\int_{C([v_+,v_-])}\omega_{\dabel(\f_+)-\dabel(\f_-)}=\deg((\dabel(\f_+)-\dabel(\f_-))_{\zz_e}).
\end{equation}
Summing over the edges crossing~$\gamma^*$ yields
\begin{equation}\label{eq:difference_constant_terms}
d_{v_+}-d_{v_-}=-\deg(\dabel(\f)_{\zz_e}).
\end{equation}

Finally, summing~\eqref{eq:difference_constant_terms} over all~$e = [v_-,v_+]\in E(N)$, and using the definition of~$\bm D_{\bm\beta}$ in~\eqref{eq:d_beta_defn}, we obtain
\begin{equation}
0=\sum_{e=[v_-,v_+]\in E(N)}(d_{v_+}-d_{v_-})
=\deg (\bm D_{\bm\beta}),
\end{equation}
which is the admissibility condition.
\end{proof}

\begin{remark}
The proof of Proposition~\ref{prop:AZ_graphs_height_function} shows more generally that, if Items~$1$ and~$2$ of Definition~\ref{def:azgraph} hold, then~$\graphbeta$ is an AZ graph whenever the boundary height function is single-valued on the~$\beta_e$-outer faces; equivalently, the height increments along the boundary close up.
\end{remark}

\subsection{Convex and concave vertices}

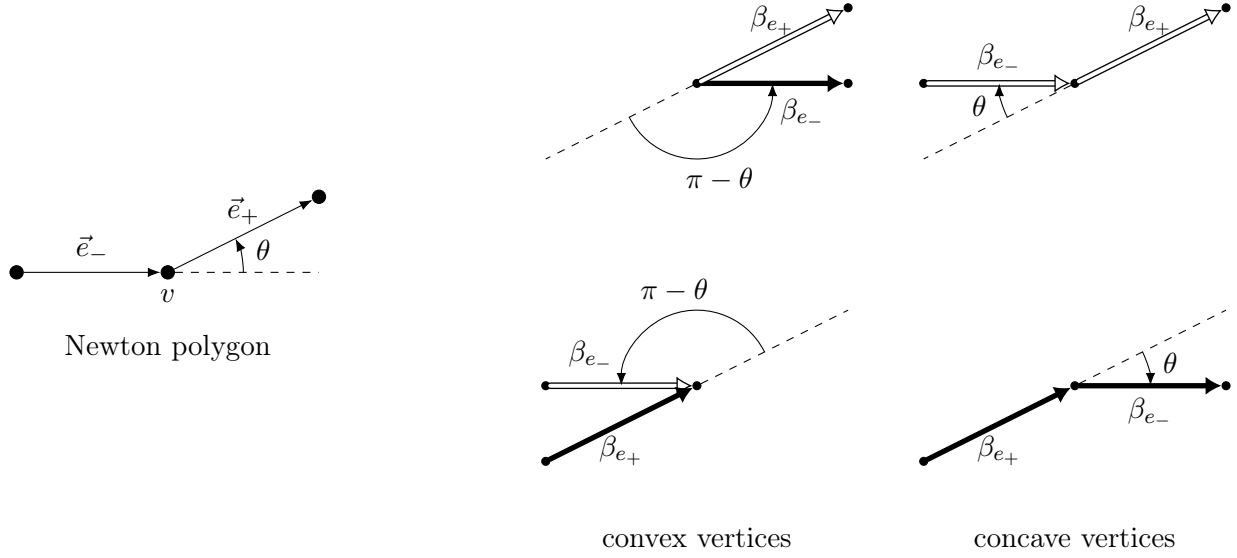
\begin{figure}
\centering
\begin{tikzpicture}
\begin{scope}[shift={(-2,-2.5)},scale = 2]
    \coordinate[bvert] (1) at (0,0);
    \coordinate[bvert, label=below:$v$] (2) at (1,0);

    \coordinate[bvert] (3) at (2,0.5);
    \coordinate (d) at (2,0); 

    \draw[->] (1) -- node[above]{$\vec e_-$} (2);
    \draw[->] (2) -- node[above]{$\vec e_+$} (3);
    \draw[dashed] (2) -- (d);

    \pic[draw, angle radius=10mm, angle eccentricity=1.3, "$\theta$", ->] {angle = d--2--3};
    \node at (1,-.5) {Newton polygon};
\end{scope}

\begin{scope}[shift={(7,0)}, scale = 2]
    \node[] (no) at (0.7,-0.2) {$\beta_{e_-}$};
    \coordinate[nvert] (1) at (0,0);
    \coordinate[nvert] (2) at (1,0);
    \coordinate[nvert] (3) at (1,0.5);
     \coordinate (d) at (-1,-0.5);
   \draw[dashed] (1) -- (d);

    \draw[solidedge] (1) --  (2);
    \draw[hollowedge] (1) -- node[above]{$\beta_{e_+}$} (3);

\pic[draw, angle radius=10mm, angle eccentricity=1.3, "$\pi-\theta$",->] {angle = d--1--2};
  
\end{scope}

\begin{scope}[shift={(10,0)}, scale = 2]
    \coordinate[nvert] (1) at (0,0);
    \coordinate[nvert] (2) at (1,0);
    \coordinate[nvert] (3) at (2,0.5);
    \coordinate (d) at (0,-0.5); 

    \draw[hollowedge] (1) -- node[above]{$\beta_{e_-}$} (2);
    \draw[hollowedge] (2) -- node[above]{$\beta_{e_+}$} (3);
    \draw[dashed] (2) -- (d);

\pic[draw, angle radius=10mm, angle eccentricity=1.3, "$\theta$",<-] {angle = 1--2--d};

\end{scope}

\begin{scope}[shift={(5,-5)}, scale = 2]
    \coordinate[nvert] (1) at (0,0);
    \coordinate[nvert] (2) at (0,0.5);
    \coordinate[nvert] (3) at (1,0.5);
     \coordinate (d) at (2,1);
   \draw[dashed] (3) -- (d);
    \node[] (no) at (1-0.7,0.5+0.2) {$\beta_{e_-}$};

    \draw[hollowedge] (2) --  (3);
    \draw[solidedge] (1) -- node[below]{$\beta_{e_+}$} (3);

\pic[draw, angle radius=10mm, angle eccentricity=1.3, "$\pi-\theta$",->] {angle = d--3--2};
    \node at (1,-.5) {convex vertices};
\end{scope}

\begin{scope}[shift={(10,-5)}, scale = 2]
    \coordinate[nvert] (1) at (0,0);
    \coordinate[nvert] (2) at (2,0.5);
    \coordinate[nvert] (3) at (1,0.5);
    \coordinate (d) at (2,1); 

    \draw[solidedge] (3) -- node[below]{$\beta_{e_-}$} (2);
    \draw[solidedge] (1) -- node[below]{$\beta_{e_+}$} (3);
    \draw[dashed] (3) -- (d);

\pic[draw, angle radius=10mm, angle eccentricity=1.3, "$\theta$",<-] {angle = 2--3--d};
    \node at (1,-.5) {concave vertices};
\end{scope}

    \end{tikzpicture}
\caption{Change of Gauss map $\gamma$ at a vertex $v = e_- \cap e_+$, where $\theta = \arg(\vec e_+) - \arg(\vec{e}_-)$. 
} \label{fig:argument}
\end{figure}

\begin{definition}
    We call a vertex $v=e_-\cap e_+\in V(N)$ \emph{convex} if $e_-$ and $e_+$ have opposite colors (equivalently, $\sign_{\bm\beta}(e_-)\neq \sign_{\bm\beta}(e_+)$), and \emph{concave} otherwise.
\end{definition}

\begin{lemma}\label{lem:four_convex_vertices}
There are exactly four convex vertices in $V(N)$.
\end{lemma}

\begin{proof}
Consider the ``Gauss map''
\[
\gamma: V(\partial \graphbeta^\times)\setminus\{\e_v\}_{v\in V(N)} \rightarrow \RR/2 \pi \ZZ,
\qquad
\gamma(\e)
:=
\arg(\sign_{\bm\beta}(e) \vec{e})
\quad\text{if }\e\in V(\beta_e),
\]
sending a medial vertex $\e\in V(\partial\graphbeta^\times)$ to the oriented ``tangent direction'' when $\partial \graphbeta^\times$ is oriented counterclockwise.
At $v=e_-\cap e_+$, the map $\gamma$ changes by (see Figure~\ref{fig:argument})
\[
\pi-(\arg(\vec{e}_+)-\arg(\vec{e}_-)) \quad\text{if $v$ is convex,}
\qquad
-(\arg(\vec{e}_+)-\arg(\vec{e}_-)) \quad\text{if $v$ is concave}. 
\]
Since $\partial\graphbeta^\times$ is simple, the total change of $\gamma$ along $\partial\graphbeta^\times$ is $2\pi$. Therefore,
\[
\pi\cdot \#\{\text{convex vertices}\}
-\sum_{v=e_-\cap e_+\in V(N)}(\arg(\vec{e}_+)-\arg(\vec{e}_-))= 2\pi.
\]
On the other hand, since $\partial N$ is a simple closed curve,
\[
\sum_{v=e_-\cap e_+\in V(N)}(\arg(\vec{e}_+)-\arg(\vec{e}_-)) = 2\pi.
\]
Substituting gives $\pi\cdot \#\{\text{convex vertices}\}=4\pi$, hence, there are exactly four convex vertices.
\end{proof}

Thus, the convex vertices partition $\partial N$ and $\partial \graphbeta^\times$ into four intervals that alternate in color (see Figure~\ref{fig:Gbeta}).

\subsection{Chambers}\label{sec:chambers_and_sign_changes}

Unlike the case of the Aztec diamond, the relative position of vertices with respect to strands need not be uniform. Thus, we introduce ``{local} sign functions'' in addition to the global sign function $\sign_{\bm\beta}$ from~\eqref{eq:sign_convention}.

The collection of strands $\bm\beta$ partitions the plane into several (open) \emph{chambers}. Precisely, the chambers are the connected components of 
\[
\RR^2 \setminus \bigcup_{e \in E(N)} \beta_e.
\]
We call the chambers contained inside (resp. outside) $\partial \graphbeta^\times$ \emph{inner} (resp. \emph{outer}) chambers. Each chamber $\graphregion$ corresponds to a \emph{sign function}
\begin{equation}\label{eq:local_sign}
\sign_{\graphregion} : E(N)\rightarrow\{+,-\},
\qquad
\sign_{\graphregion}(e) :=
\begin{cases}
+,& \text{if ${\graphregion}$ lies to the left of $\beta_e$,}\\
-,& \text{if ${\graphregion}$ lies to the right of $\beta_e$.}
\end{cases}
\end{equation}
Recall that we identify signs with colors, and therefore, $\sign_{\graphregion}$ equivalently defines a coloring of $E(N)$ for each chamber.

\begin{definition}
Let $v=e_-\cap e_+\in V(N)$ be a vertex of the Newton polygon. We say that a chamber ${\graphregion}$ has a \emph{sign (or color) change} at $v$ if
\[
\sign_{\graphregion}(e_-)\neq \sign_{\graphregion}(e_+).
\]
We write $\SC({\graphregion})\subset V(N)$ for the set of sign changes of ${\graphregion}$.
\end{definition}

\begin{lemma}\label{lem:four_signs_locally}
An inner chamber has exactly four sign changes, while an outer chamber has exactly two.
\end{lemma}

\begin{proof}
The argument is identical to that of Lemma~\ref{lem:four_convex_vertices}, with the Gauss map replaced by 
\[
\e \mapsto \arg(\sign_{\graphregion}(e) \vec{e})
\quad\text{if }\e\in V(\beta_e),
\]
\emph{i.e.} the tangent direction as seen from ${\graphregion}$ and using that the total winding is $2\pi$ for an inner chamber and $0$ for an outer chamber.
\end{proof}

\begin{figure}
\centering
\begin{tikzpicture}[scale=0.5, rotate=45]
\filldraw[fill=gray!30, draw=none]
 (-0.5,0.5) --  (5.5,0.5) --(-0.5,6.5)-- (-0.5,0.5) ;
\filldraw[fill=gray!30, draw=none]
 (-0.5,0.5) -- (-6.5,0.5) --(-0.5,-5.5) -- (-0.5,0.5) ;

\draw[blue,->,thick] (-6.5,0.5) -- (5.5,0.5);
\draw[red,->,thick] (-0.5,-5.5) -- (-0.5,6.5);

\coordinate[bvert] (b-3) at (-6,0) {};
\coordinate[wvert] (w-3) at (-5,1) {};
\coordinate[bvert] (b-2) at (-4,0) {};
\coordinate[wvert] (w-2) at (-3,1) {};
\coordinate[bvert] (b-1) at (-2,0) {};
\coordinate[wvert] (w-1) at (-1,1) {}; 
\coordinate[bvert] (b0)  at (0,0) {};  
\coordinate[wvert] (w1)  at (1,1) {};
\coordinate[bvert] (b1)  at (2,0) {};
\coordinate[wvert] (w2)  at (3,1) {};
\coordinate[bvert] (b2)  at (4,0) {};
\coordinate[wvert] (w3)  at (5,1) {};

\coordinate[bvert] (vb1) at (0,2) {};
\coordinate[wvert] (vw2) at (-1,3) {};
\coordinate[bvert] (vb2) at (0,4) {};
\coordinate[wvert] (vw3) at (-1,5) {};
\coordinate[bvert] (vb3) at (0,6) {};

\coordinate[wvert] (vw-1) at (-1,-1) {};
\coordinate[bvert] (vb-1) at (0,-2) {};
\coordinate[wvert] (vw-2) at (-1,-3) {};
\coordinate[bvert] (vb-2) at (0,-4) {};
\coordinate[wvert] (vw-3) at (-1,-5) {};

\draw (b-3)--(w-3)--(b-2)--(w-2)--(b-1)--(w-1)--(b0)--(w1)--(b1)--(w2)--(b2)--(w3);
\draw (vw-3)--(vb-2)--(vw-2)--(vb-1)--(vw-1)--(b0)--(w-1)--(vb1)--(vw2)--(vb2)--(vw3)--(vb3);

\end{tikzpicture}

\caption{The four conical regions. The shaded regions are the inconsistently oriented cones.}\label{fig:four_conical_regions}
\end{figure}
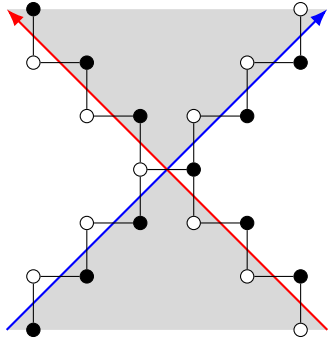

We now describe the sign changes geometrically. Let $v=e_-\cap e_+\in V(N)$. The strands $\beta_{e_-}$ and $\beta_{e_+}$ subdivide $\RR^2$ into four conical regions (Figure~\ref{fig:four_conical_regions}). Two of these cones have boundaries that are consistently oriented and two have boundaries that are inconsistently oriented. By definition, a chamber ${\graphregion}$ has a sign change at $v$ if and only if ${\graphregion}$ is contained in one of the two inconsistently oriented cones determined by $\beta_{e_-}$ and $\beta_{e_+}$ at $v$.

{Consider the adjacency graph $\mathcal R$ of inner chambers: its vertices are the inner chambers, and two vertices are joined by an edge if and only if the boundaries of the corresponding chambers share a nontrivial segment of a strand $\beta_e$. In particular, chambers meeting only at a crossing, such as the two gray regions in Figure~\ref{fig:four_conical_regions}, are not considered adjacent.}


\begin{lemma}\label{lem:sign_change_cts}
Across each edge of $\mathcal R$, exactly one of the four sign changes moves one step clockwise or counterclockwise along $\partial N$.
\end{lemma}

\begin{proof}
Let ${\graphregion}$ and ${\graphregion}'$ be adjacent chambers, separated by a segment of a strand $\beta_e$. Then $\sign_{{\graphregion}'}(e)=-\sign_{\graphregion}(e)$, while $\sign_{{\graphregion}'}(e')=\sign_{\graphregion}(e')$ for all $e'\neq e$.

If $e$ is not adjacent to the boundary of one of the four sign intervals for $\sign_{\graphregion}$ (equivalently, if $e$ is not adjacent along $\partial N$ to a vertex in $\SC({\graphregion})$), then changing only the sign at $e$ would create two additional sign changes (one at each endpoint of $e$ along $\partial N$), while the previous four sign changes persist. This would give $\#\SC({\graphregion}')\ge 6$, contradicting Lemma~\ref{lem:four_signs_locally}. Hence, $e$ must lie at the boundary of exactly one sign interval, so exactly one sign change shifts by one step along $\partial N$, and the other three remain fixed.
\end{proof}

In particular, for every edge ${\graphregion}{\graphregion}'$ of $\mathcal R$ there is a canonical bijection
\begin{equation} \label{eq:connection}
\SC({\graphregion}) \cong \SC({\graphregion}')
\end{equation}
which is the identity on three of the four elements. Consider the discrete bundle over $\mathcal R$ whose fiber over an inner chamber ${\graphregion}$ is the four-element set $\SC({\graphregion})$, equipped with the connection determined by the bijections~\eqref{eq:connection} along edges.

\begin{lemma}\label{lem:gamma_trivial}
The connection thus defined is flat, \emph{i.e.} the parallel transport around any closed loop in $\mathcal R$ is the identity.
\end{lemma}

\begin{proof}
Since $\mathcal R$ is planar, it suffices to check triviality around each face. Each face is a $4$-cycle of inner chambers surrounding an intersection point of two strands $\beta_{e_1}$ and $\beta_{e_2}$ in $\graphpl$. Label the chambers cyclically by
\[
{\graphregion} \xrightarrow{\beta_{e_1}} {\graphregion}_1 \xrightarrow{\beta_{e_2}} {\graphregion}_{12}
\xrightarrow{\beta_{e_1}} {\graphregion}_2 \xrightarrow{\beta_{e_2}} {\graphregion}.
\]
Crossing $\beta_{e_i}$ flips the sign at $e_i$ and, by Lemma~\ref{lem:sign_change_cts}, moves exactly one sign change by one step along $\partial N$, fixing the other three.
Since the intersection $\beta_{e_1}\cap \beta_{e_2}$ is not on $\partial\graphbeta$, the edges $e_1$ and $e_2$ are not adjacent along $\partial N$, hence, the moved sign change for $\beta_{e_1}$ is different from the moved sign change for $\beta_{e_2}$.
Therefore, going once around the $4$-cycle moves the $e_1$-sign change one step and then moves it back, and similarly for the $e_2$-sign change; all other sign changes remain fixed throughout.
Thus, the composed bijection around the face is the identity, and the connection is flat.
\end{proof}

The discrete bundle of sign changes has four canonical sections which we now define. Each convex vertex $v\in V(N)$ determines a section
\[
s_v:\{\text{inner chambers}\}\rightarrow V(N),
\qquad
s_v({\graphregion})\in \SC({\graphregion}),
\]
as follows. Let ${\graphregion}_v$ be the inner chamber adjacent to $v$. Then $v\in \SC({\graphregion}_v)$, and we set $s_v({\graphregion}_v)=v$.
For any other inner chamber ${\graphregion}$, define $s_v({\graphregion})$ by parallel transporting $v$ along any path in $\mathcal R$ from ${\graphregion}_v$ to ${\graphregion}$.
By Lemma~\ref{lem:gamma_trivial}, this is well-defined.

\begin{figure}
\centering
\begin{tikzpicture}
  \filldraw[fill=gray!30, draw=none]
 (-3,3) -- (-1,-1) -- (-5,-1) -- (-3,3);
 \coordinate[nvert, label=above:{$v_1$}]  (1) at (-3,3) {};
 \coordinate[nvert]  (2) at (-2,1) {};
    \coordinate[nvert]  (3) at (-1,0) {};
    \coordinate[nvert]  (4) at (1,0) {};
    \coordinate[nvert]  (5) at (2,1) {};
    \coordinate[nvert, label=above:{$v_2 = s_{v_2}({\graphregion}_{v_2})$}] (6) at (3,3) {};
    \coordinate[label = 10:${s_{v_3}({\graphregion}_{v_2})}$] (no) at (4,1) {};
 \coordinate[label = 170:${s_{v_1}({\graphregion}_{v_2})}$] (no) at (2,1) {};
 
    \draw[hollowedge] (1) -- (2);
    \draw[hollowedge] (2) -- (3);
    \draw[hollowedge] (3) -- (4);
    \draw[hollowedge] (4) -- (5);
    \draw[hollowedge] (5) -- (6);
    \draw[solidedge] (1) -- (-4,1) ;
     \draw[solidedge] (4,1) -- (6);
          \draw[solidedge] (5,0) -- (4,1) ;
          \draw[solidedge]  (-4,1) -- (-5,0) ;

     \draw[dashed] (2) -- (-1,-1)
     (3) -- (0,-1)
     (4) -- (3,0)
     (5) -- (4,3)
     (-4,1) -- (-5,-1)
     ;

\end{tikzpicture}
\caption{The chosen inconsistently oriented cones in $[v_1,v_2)$ whose boundary orientations match the boundary orientation of the reference cone at $v_1$ (shaded).} \label{fig:fig_opp_cones}
\end{figure}
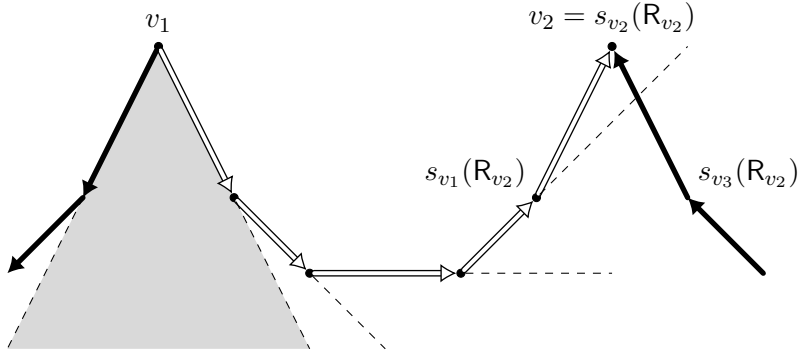

\begin{lemma}\label{lem:sign_changes_sections}
For every inner chamber ${\graphregion}$, the four sections associated to the four convex vertices are pairwise distinct at ${\graphregion}$, together exhaust $\SC({\graphregion})$, and appear in the same cyclic order along $\partial N$ as the convex vertices.
\end{lemma}

\begin{proof}
Let $v_1,v_2,v_3,v_4$ be the convex vertices in cyclic order along $\partial N$.
Consider any edge ${\graphregion}{\graphregion}'$ of $\mathcal R$. By Lemma~\ref{lem:sign_change_cts}, under the
connection bijection $\SC({\graphregion})\cong\SC({\graphregion}')$ exactly one element of $\SC({\graphregion})$ moves by
one step along $\partial N$, while the other three are fixed.
In particular, no two transported elements can collide or cross: a collision would
force $\#\SC({\graphregion}')<4$, contradicting Lemma~\ref{lem:four_signs_locally}, and a crossing
would require two elements to pass each other even though at most one moves.
Consequently, if $s_{v_i}({\graphregion})\neq s_{v_j}({\graphregion})$ for $i\neq j$ at some chamber ${\graphregion}$, then
$s_{v_i}$ and $s_{v_j}$ remain distinct over all chambers, and their cyclic order
along $\partial N$ is preserved throughout $\mathcal R$.

It therefore suffices to verify the distinctness and the cyclic ordering in a single chamber. Without loss of generality, we check the cyclic order at the inner chamber ${\graphregion}_{v_2}$ adjacent to $v_2$. Consider a path in $\mathcal R$ that moves clockwise along the boundary $\partial\graphbeta$ from the inner chamber ${\graphregion}_{v_1}$ adjacent to $v_1$ to ${\graphregion}_{v_2}$. Initially, we have $s_{v_1}({\graphregion}_{v_1})=v_1$. Along this path, the section $s_{v_1}$ evolves as follows.
For each vertex $u \in [v_1,v_2)$ there are two inconsistently oriented cones. Among them, choose the one whose boundary orientation agrees with that of the
reference cone at $v_1$ containing $\graphbeta$ (\emph{i.e.}, in both cones the strands point either toward or away from the apex). Then for every chamber ${\graphregion}$ encountered along the path, we set $s_{v_1}({\graphregion})=u$, where $u$ is the unique vertex in $[v_1,v_2)$ such that ${\graphregion}$ lies in this chosen
cone at $u$ (see Figure~\ref{fig:fig_opp_cones}). In particular, when the path reaches ${\graphregion}_{v_2}$, the section $s_{v_1}$ has advanced monotonically to the vertex immediately counterclockwise of $v_2$. Applying the same argument with $v_3$ in place of $v_1$, we get that $s_{v_3}({\graphregion}_{v_2})$ is the vertex immediately clockwise of $v_2$. Hence, the sections $s_{v_1},s_{v_2},s_{v_3}$ are distinct and appear in the correct cyclic order at ${\graphregion}_{v_2}$. 

Finally, since $\#\SC({\graphregion})=4$ for every chamber ${\graphregion}$, the four sections exhaust $\SC({\graphregion})$ for all ${\graphregion}$.
\end{proof}

\section{Exact inverse Kasteleyn formula for AZ graphs} \label{sec:inv_K}

In this section, we define $\kast^{-1}$ by a double integral plus single-integral correction term.

\subsection{Formula for $\kast^{-1}$}\label{sec:def_inverse}

Let $\bl \in B(\graphbeta)$ and $\wh \in W(\graphbeta)$. Define the kernel 
\begin{equation} \label{eq:kernel}
\omega_{\bl,\wh}(\zeta,\eta)
:= \frac{1}{(2\pi \ii)^2}
\frac{1}{E(\zeta,\eta)}
\frac{\theta(-t+\bm D_{\bm \beta}+\eta-\zeta)}
     {\theta(-t+\bm D_{\bm \beta})}
\frac{g_{\wh}(\eta) \Psi_{\bm \beta}(\eta)}
     {g_{\bl}(\zeta) \Psi_{\bm \beta}(\zeta)},
\end{equation}
where
\[
\Psi_{\bm \beta}(\zeta)
:=\prod_{\alpha\in\bm\alpha} E(\alpha,\zeta)^{(\bm D_{\bm\beta})_\alpha}
\]
encodes the dependence on the boundary data. As in Section~\ref{sec:abel_map_jacobian}, we identify $\bm D_{\bm \beta}+\eta-\zeta$ with its image $\mu(\bm D_{\bm \beta}+\eta-\zeta)$. The parameter $t$ here is as in Section~\ref{sec:fock's_kast}.

\begin{lemma}\label{lem:kernel_is_a_form}
The kernel $\omega_{\bl,\wh}(\zeta,\eta)$ is a meromorphic $(1,1)$-form on $\curve \times \curve$.
\end{lemma}
\begin{proof}
This is standard; see for instance \cite[Lemma~32]{BCT22}.
Fix $\eta$ and view $\omega_{\bl,\wh}(\zeta,\eta)$ as a function of $\zeta$.
To say it is a meromorphic $1$-form in $\zeta$ is equivalent to saying that its divisor is linearly equivalent to the canonical divisor:
\[
\divisor_\zeta(\omega_{\bl,\wh}) \sim K_\curve .
\]
By Abel's theorem, this holds provided
\[
\deg(\divisor_\zeta(\omega_{\bl,\wh}))=\deg(K_\curve)=2g-2, \qquad \mu(\divisor_\zeta(\omega_{\bl,\wh}))=K_\curve\in \jac(\curve).
\]
By Riemann's theorem~\cite[Theorem~1.1]{Fay73}, for any $e\in \CC^g$ such that $\theta(e+\zeta)$ is not identically $0$,
\[
\deg(\divisor_\zeta (\theta(e+\zeta)))=g,
\qquad
\mu(\divisor_\zeta (\theta(e+\zeta)))=-e+\Delta,
\]
where $\Delta$ is the vector of Riemann constants, satisfying $2\Delta=K_\curve$.

As a function of $\zeta$,
\begin{itemize}
\item $E(\zeta,\eta)^{-1}$ has a simple pole at $\zeta=\eta$, hence degree $-1$, and $\mu(\divisor_\zeta( E(\zeta,\eta)))=\eta$ (recall that we often suppress the Abel map; the $\eta$ on the right-hand side stands for $\mu(\eta)$).
\item $\theta(-t+\bm D_{\bm\beta}+\eta-\zeta)$ contributes degree $g$, and $\mu(\divisor_\zeta(\theta(-t+\bm D_{\bm\beta}+\eta-\zeta))) = -t+\bm D_{\bm\beta}+\eta+\Delta$. 
\item $(g_{\bl}\Psi_{\bm\beta})^{-1}$ contributes degree $g-1+\deg(\divisor(\Psi_{\bm\beta}))$, and  

\[
\mu(\divisor (g_{\bl}))=-t-\Delta, \qquad \mu(\divisor(\Psi_{\bm\beta})) = \bm D_{\bm\beta}.
\]
\end{itemize}
By admissibility, $\deg(\divisor(\Psi_{\bm\beta}))=\deg(\bm D_{\bm\beta})=0$. Thus,
\[
\deg(\divisor_\zeta(\omega_{\bl,\wh}))
= -1+g + (g-1 + \deg(\divisor(\Psi_{\bm\beta}))) = 2g-2.
\]
Moreover, 
\begin{align*}
\mu(\divisor_\zeta(\omega_{\bl,\wh}))
&= - \mu(\divisor_\zeta (E(\zeta,\eta))) + \mu(\divisor_\zeta (\theta(-t+\bm D_{\bm\beta}+\eta-\zeta)))
   -\mu(\divisor_\zeta (g_{\bl}(\zeta)))
   -\mu(\divisor_\zeta (\Psi_{\bm\beta}(\zeta))) \\
&= -\eta+(-t+\bm D_{\bm\beta}+\eta+\Delta)  + (t+\Delta) -\bm D_{\bm\beta} \\
&= 2\Delta = K_\curve.
\end{align*}
The same argument in the $\eta$-variable shows that $\omega_{\bl,\wh}(\zeta,\eta)$ is a meromorphic $(1,1)$-form on $\curve\times\curve$.
\end{proof}

The residue of $\omega_{\bl,\wh}$, viewed as a meromorphic $1$-form in $\zeta$, along the diagonal $\zeta = \eta$, is 
\[
\res_{\zeta=\eta}(\omega_{\bl,\wh})
=-\frac{1}{2\pi \ii} g_{\bl,\wh}.
\]
For an interval $I = [\ell,r] \subset \partial N$, define its boundary to be the $0$-chain 
\[
\partial I:=r-\ell.
\]
Recall from Section~\ref{sec:whole_plane_inverse_K} the definitions of the contour \(C(I_{\bl,\wh}^u)\) and the whole-plane inverse \(\mathsf A_{\bl,\wh}^u\). By integration, we obtain a linear functional on \(0\)-chains defined by
\begin{equation} \label{eq:functional}
u \mapsto \langle u,\mathsf A_{\bl,\wh}\rangle
:= \mathsf A_{\bl,\wh}^u
= \frac{1}{2\pi \ii}\int_{C(I_{\bl,\wh}^u)} g_{\bl,\wh},
\end{equation}
and extended to arbitrary \(0\)-chains by linearity.

Let $v$ be a convex vertex and let \(\bl \in B(\graphbeta)\) and \(\wh \in W(\graphbeta)\). For $\{\x,\y\}=\{\bl,\wh\}$, define
\begin{align}
\kast^{-1}_{\bl,\wh} &:= \sign(v)\left( \iint_{C(I_{\y}) \prec  C(I_{\x})} \omega_{\bl,\wh} - \col(\x) \langle \partial I_\x \cap I_\y,\mathsf A_{\bl,\wh} \rangle \right) \label{eq:formula_inverse_K}\\
&= \sign(v)\left( \iint_{C(I_{\y}) \prec  C(I_{\x})} \omega_{\bl,\wh} - \col(\x) \Bigl(\mathbf 1_{\{r_\x \in I_{\y}\}} {\mathsf A}^{r_\x}_{\bl,\wh} -\mathbf 1_{\{\ell_\x \in I_{\y}\} } {\mathsf A}^{\ell_\x}_{\bl,\wh} \Bigr) \right),\label{eq:intersection_concrete}
\end{align}
where:
\begin{enumerate}
\item We set $\sign(v)=+1$ if the color changes from black to white at the section $s_v$, and $\sign(v)=-1$ otherwise.
\item $\col(\x)=+1$ (resp. $-1$) if $\x$ is black (resp. white) in consistency with our sign/color convention~\eqref{eq:sign_convention}.
\item For each \(\x \in B(\graphbeta)\sqcup W(\graphbeta)\), let \({\graphregion}_\x\) denote the chamber containing \(\x\). We associate to \(\x\) an inner chamber \(\Rin_\x\) by setting
\[
\Rin_\x :=
\begin{cases}
{\graphregion}_\x, & \text{if \(\x\) is inner},\\[4pt]
\text{the inner chamber adjacent to \(\x\) across \(\beta_e\)}, & \text{if \(\x\) is \(\beta_e\)-outer}.
\end{cases}
\]
If \(\x\) is \(\beta_e\)-outer for more than one \(e \in E(N)\), we choose one such \(e\) arbitrarily. The resulting constructions are independent of this choice; see Proposition~\ref{prop:invariance}.

The interval $I_\x=[\ell_\x,r_\x]\subset\partial N$ is defined as the sign interval of $\sign_{\Rin_\x}$ whose color is opposite to that of $\x$, and which is adjacent to $s_v(\Rin_\x)$; the interval $I_\y$ is defined analogously. The contours $C(I_\x)$ and $C(I_\y)$ are then defined by Definition~\ref{def:contours_from_intervals}.

Equivalently, instead of specifying the convex vertex $v$, we can choose $I_\x$ to be either of the sign intervals for $\sign_{\Rin_\x}$ with opposite color as $\x$, and similarly for $I_\y$. The convex vertex is then uniquely determined as the common endpoint of both intervals when they are parallel transported to the same chamber.

\item {Throughout, $\zeta$ is integrated over $C(I_\bl)$ and $\eta$ is integrated over $C(I_\wh)$.} The notation 
\[C(I_{\y}) \prec  C(I_{\x})\]
denotes the product $C(I_{\bl}) \times  C(I_{\wh})$ with the convention that if both contours $C(I_\x)$ and $C(I_\y)$ contain small circles around the same angle $\alpha$, we choose them so that the circle in $C(I_\y)$ has smaller radius than the one in $C(I_\x)$.
\end{enumerate}

\begin{remark}\label{rem:inverse_kasteleyn_signs}
Since there are four choices of convex vertices and two choices of $(\x,\y) = (\bl,\wh)$ or $(\wh,\bl)$, the definition~\eqref{eq:formula_inverse_K} gives eight different expressions. We will prove in Proposition~\ref{prop:invariance} that the resulting quantities for all of these choices are equal. If we take $\x=\bl$ and choose the intervals $I_\bl$ and $I_\wh$ so that $v$ is a convex vertex at which $\sign_{\bm\beta}$ changes from black to white, then the signs $\sign(v)$ and $\varepsilon(\x)$ in \eqref{eq:formula_inverse_K}--\eqref{eq:intersection_concrete} are $+1$. However, the freedom to vary these choices will be important in our proofs.
\end{remark}

Our main result of this section is:
\begin{theorem}\label{thm:main}
The matrix $\kast^{-1}$ defined by~\eqref{eq:formula_inverse_K} is the two-sided inverse of Fock's Kasteleyn matrix $\kast$ on $\graphbeta$. 
\end{theorem}

\begin{corollary}\label{cor:dimer_cover}
The graph $\graphbeta$ has an equal number of black and white vertices, and admits a dimer cover.
\end{corollary}

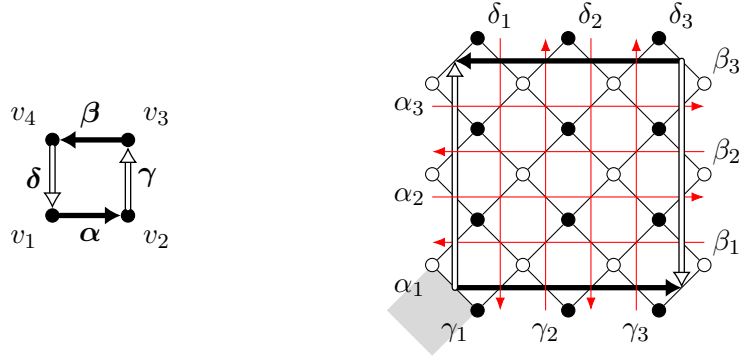
\begin{figure}
  \centering
  \begin{tikzpicture}

    \node[anchor=center] (A) {
      \begin{tikzpicture}[scale=1]
       \begin{scope}
    \coordinate[bvert, label=225:${v_1}$](1) at (0,0){}; 
     \coordinate[bvert, label=-45:${v_2}$](2) at (1,0){}; 
     \coordinate[bvert, label=45:${v_3}$](3) at (1,1){}; 
      \coordinate[bvert, label=135:${v_4}$](4) at (0,1){};
       \draw[solidedge] (1) -- node[below]{${\bm \alpha}$} 
       (2);
       \draw[solidedge] (3) -- node[above]{${\bm \beta}$} 
       (4) ;
       \draw[hollowedge] (2) -- node[right]{${\bm \gamma}$} 
       (3);
       \draw[hollowedge] (4) -- node[left]{${\bm \delta}$} 
       (1);
       
       \end{scope}
      \end{tikzpicture}
    };

    \node[anchor=center] (B) at ([xshift=5.0cm]A.east) {%
      \begin{tikzpicture}[scale=0.6]

\fill[gray!30]
(1,0) -- (0,1) -- (-1,0) -- (0,-1) -- cycle;
      
        \foreach \y in {0,2,4,6} {
          \foreach \x in {1,3,5} {
            \node[bvert] (b-\x-\y) at (\x,\y) {};
          }
        }

        \foreach \y in {1,3,5} {
          \foreach \x in {0,2,4,6} {
            \node[wvert] (w-\x-\y) at (\x,\y) {};
          }
        }

        \foreach \y in {0,2,4,6} {
          \foreach \x in {1,3,5} {

            \ifcsname pgf@sh@ns@w-\the\numexpr\x-1\relax-\the\numexpr\y+1\relax\endcsname
              \draw[] (b-\x-\y) -- (w-\the\numexpr\x-1\relax-\the\numexpr\y+1\relax);
            \fi

            \ifcsname pgf@sh@ns@w-\the\numexpr\x+1\relax-\the\numexpr\y+1\relax\endcsname
              \draw[] (b-\x-\y) -- (w-\the\numexpr\x+1\relax-\the\numexpr\y+1\relax);
            \fi

            \ifcsname pgf@sh@ns@w-\the\numexpr\x-1\relax-\the\numexpr\y-1\relax\endcsname
              \draw[] (b-\x-\y) -- (w-\the\numexpr\x-1\relax-\the\numexpr\y-1\relax);
            \fi

            \ifcsname pgf@sh@ns@w-\the\numexpr\x+1\relax-\the\numexpr\y-1\relax\endcsname
              \draw[] (b-\x-\y) -- (w-\the\numexpr\x+1\relax-\the\numexpr\y-1\relax);
            \fi

          }
        }
\foreach \y in {1,3,5}
{
    \pgfmathtruncatemacro{\k}{(\y+1)/2}

    \ifnum\k=1
        \draw[solidedge] (0.5,\y-0.5) -- (5.5,\y-0.5);
    \else
        \draw[->,red] (0,\y-0.5) -- (6,\y-0.5);
    \fi
    \node at (-.5,\y-0.5) {$\alpha_{\k}$};

    \ifnum\k=3
        \draw[solidedge] (5.5,\y+0.5) -- (0.5,\y+0.5) ;
    \else
        \draw[<- ,red] (0,\y+0.5) -- (6,\y+0.5);
    \fi
    \node at (6.5,\y+0.5) {$\beta_{\k}$};

    \ifnum\k=1
        \draw[hollowedge] (\y-0.5,0.5) -- (\y-0.5,5.5);
    \else
        \draw[->,red] (\y-0.5,0) -- (\y-0.5,6);
    \fi
    \node at (\y-0.5,-.5) {$\gamma_{\k}$};

    \ifnum\k=3
        \draw[hollowedge] (\y+0.5,5.5) -- (\y+0.5,0.5) ;
    \else
        \draw[<- ,red] (\y+0.5,0) -- (\y+0.5,6);
    \fi
    \node at (\y+0.5,6.5) {$\delta_{\k}$};
}

      \end{tikzpicture}
    };

  \end{tikzpicture}

  \caption{The sign intervals (left) and strands (right) in the Aztec diamond. {The shaded face is the reference face $\f_0$.} 
  }
  \label{fig:square_aztec_inv_kast}
\end{figure}

\subsection{Example: Aztec diamond} \label{sec:Aztec_diamond}
Consider the Aztec diamond with reference face $\f_0$ and strands labeled as shown in Figure~\ref{fig:square_aztec_inv_kast}(right). We switch to the notation in~\cite{BdT24} and denote the four families of parallel strands by $\bm \alpha,\bm \beta,\bm \gamma$ and $\bm \delta$ as shown in Figure~\ref{fig:square_aztec_inv_kast}(left). We check that~\eqref{eq:formula_inverse_K} recovers the formula
\begin{align}\label{eq:aztec_inv_kast_cbdt}
K^{-1}_{\bl,\wh}
&=
\frac{1}{(2\pi \ii)^2}
\int_{C_\wh }\int_ {C_\bl}
\frac{1}{E(\zeta,\eta)}
\frac{\theta(-t+\bm D_{\bm \beta}+\eta-\zeta)}
     {\theta(-t+\bm D_{\bm \beta})}
\frac{g_{\wh}(\eta)}
     {g_{\bl}(\zeta)}
\prod_{i=1}^n
\frac{E(\delta_i,\eta)}{E(\beta_i,\eta)}
\frac{E(\beta_i,\zeta)}{E(\delta_i,\zeta)}
\\
&\qquad
-\mathbf{1}_{\{\bl \text{ right of } \wh\}}
\frac{1}{2 \pi \ii}
\int_{C_\wh} g_{\bl,\wh}
\end{align}
in \cite[(17)]{BdT24}, where $C_\bl$ (resp. $C_\wh$) is a counterclockwise oriented curve that encloses only the angles $\bm \gamma$ (resp. $\bm \alpha$).

In this example,
\[
\bm D_{\bm\beta} = \sum_{i=1}^n \delta_i - \sum_{i=1}^n \beta_i,
\qquad 
\Psi_{\bm \beta}(\zeta) = \prod_{i=1}^n \frac{E(\delta_i,\zeta)}{E(\beta_i,\zeta)},
\]
so the integrand inside the double integral is the same as in~\eqref{eq:formula_inverse_K}. There is only one inner chamber with sign intervals as shown in Figure~\ref{fig:square_aztec_inv_kast}(left). Thus, we can take 
\[
I_\bl = [v_2,v_3],\qquad I_\wh = [v_1,v_2],\qquad v=v_2, \qquad \sign(v)=1, 
\]
so that $C_\bl = C(I_\bl)$ and $C_\wh = C(I_\wh)$ have no common angles. Therefore, the double integral terms are identical.

Our single integral term is 
\[
 - \langle \partial I_\bl \cap I_\wh,\mathsf A_{\bl,\wh} \rangle = A_{\bl,\wh}^{v_2}.
\]
If $\bl$ is on the left of (resp. above) $\wh$, $g_{\bl,\wh}$ has no poles at angles in $\bm \gamma$ (resp. $\bm \alpha$), so
\[
-\mathbf{1}_{\{\bl \text{ right of } \wh\}}
\frac{1}{2 \pi \ii}
\int_{C_\wh} g_{\bl,\wh} =0.
\]
On the other hand, in both cases, $v_2 \in \sector_{\bl,\wh}$ so $A_{\bl,\wh}^{v_2}=0$ as well. If $\bl$ is below and to the right of $\wh$, then $g_{\bl,\wh}$ has no poles at angles in $\bm \beta$ and $\bm \delta$, so $v_1 \in \sector_{\bl,\wh}$, so we can take $I_{\bl,\wh}^{v_2} = [v_2,v_1]$. Since the total sum of residues of $g_{\bl,\wh}$ is $0$,
\[
\frac{1}{2 \pi \ii}
\int_{C([v_1,v_2])} g_{\bl,\wh} + \frac{1}{2 \pi \ii}
\int_{C([v_2,v_1])} g_{\bl,\wh} =0.
\]
Hence,
\[
A_{\bl,\wh}^{v_2} = - 
\frac{1}{2 \pi \ii}
\int_{C_\wh} g_{\bl,\wh}. 
\]
Therefore, the single integral terms agree in all cases, and the two formulas for $K^{-1}_{\bl,\wh}$ coincide.

\begin{remark}
In this example, the eight \emph{a priori} different choices reduce to four distinct
formulas: when we interchange the roles of $\bl$ and $\wh$, the double integral is unchanged since $C(I_\bl)$ and $C(I_\wh)$ do not enclose any common angles, and the single integral is unchanged since $I_\bl \cap I_\wh = \{v\}$ and so changing $\col(\bl)$ to $\col(\wh)$ cancels the sign coming from changing $\partial I_\bl$ to $\partial I_\wh$. The resulting expressions, indexed by the choice of convex vertex $v$, are
\begin{align}
v_1: \qquad  &- \int_{C([v_1,v_2])} \int_{C([v_4,v_1])} \omega_{\bl,\wh}
               + \mathsf A_{\bl,\wh}^{v_1}, \label{eq:v1}\\
v_2: \qquad  & \int_{C([v_1,v_2])} \int_{C([v_2,v_3])} \omega_{\bl,\wh}
               + \mathsf A_{\bl,\wh}^{v_2},\label{eq:v2}\\
v_3: \qquad  &- \int_{C([v_3,v_4])} \int_{C([v_2,v_3])} \omega_{\bl,\wh}
               + \mathsf A_{\bl,\wh}^{v_3},\nonumber \\
v_4: \qquad  & \int_{C([v_3,v_4])} \int_{C([v_4,v_1])} \omega_{\bl,\wh}
               + \mathsf A_{\bl,\wh}^{v_4}.\nonumber
\end{align}
{
Let us check that the expressions \eqref{eq:v1} and \eqref{eq:v2}, corresponding to
\(v_1\) and \(v_2\), agree.
Keep \(C(I_{\wh})=C([v_1,v_2])\) fixed and deform \(C(I_{\bl})\) from \(C([v_4,v_1])\) to
\(-C([v_2,v_3])\). During this deformation, the only pole of \(\omega_{\bl,\wh}\) crossed is along the diagonal \(\zeta=\eta\), so the change in the
double integral is the residue contribution
\[
- \frac{1}{2 \pi \ii} \int_{C([v_1,v_2])} g_{\bl,\wh}
= \mathsf A_{\bl,\wh}^{v_2}-\mathsf A_{\bl,\wh}^{v_1}.
\]
Finally, replacing the inner contour \(-C([v_2,v_3])\) by \(C([v_2,v_3])\) reverses the sign of the double integral. Therefore, the formulas
\eqref{eq:v1} and \eqref{eq:v2} coincide. The general case, proved in Proposition~\ref{prop:invariance}, uses
the same underlying idea.
}
\end{remark}

\begin{remark}
    It is worth emphasizing that the indicators that appear in our formula~\eqref{eq:intersection_concrete} are of a different nature than the indicator in \eqref{eq:aztec_inv_kast_cbdt}; indeed, as we have seen above, the latter indicator term is simply the whole-plane inverse $A_{\bl,\wh}^{v_2}$.
\end{remark}

\subsection{Example: a pentagonal Newton polygon with isoradial weights} \label{sec:pentagon_example}

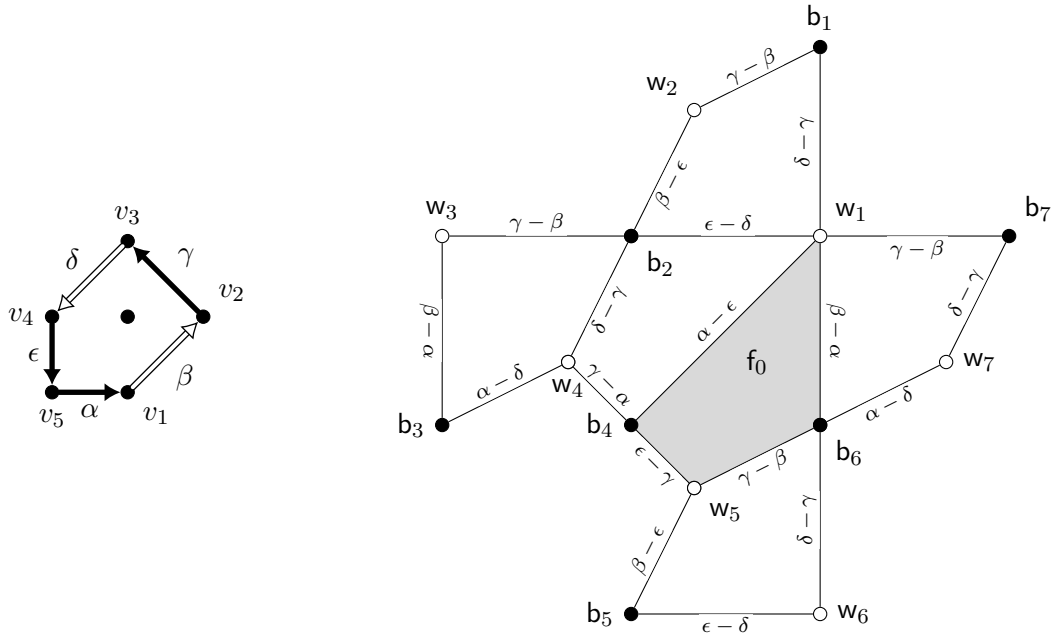
\begin{figure}
    \centering
        
\raisebox{-0.5\height}{

\begin{tikzpicture}

\begin{scope}[shift={(0,1.5)},rotate=0]

 \coordinate[bvert, label = left:$v_4$](1) at (0,1){}; 
    \coordinate[bvert, label = -45:$v_1$](2) at (1,0){}; 
     \coordinate[bvert, label = 45:$v_2$](3) at (2,1){}; 
      \coordinate[bvert](4) at (1,1){};
          \coordinate[bvert, label = below:$v_5$](5) at (0,0){}; 
          \coordinate[bvert, label = above:$v_3$](6) at (1,2){};

           \draw[solidedge]  (5) --node[below]{$\alpha$} (2);
           \draw[solidedge] (3) --node[above right]{$\gamma$} (6);
           \draw[solidedge] (1) -- node[left]{$\epsilon$} (5)
           ;

  \draw[hollowedge] (2) --node[below right]{$\beta$} (3);
   \draw[hollowedge]   (6) --node[above left]{$\delta$} (1)
  ;
       \end{scope}
\end{tikzpicture}}\hspace{1.5cm}    
\raisebox{-0.5\height}{

\begin{tikzpicture}[scale=5,
  elab/.style={midway, sloped, font=\scriptsize, inner sep=1.2pt,
               fill=white, fill opacity=.85, text opacity=1}
]

\coordinate (S1) at (3,0);
\coordinate (S2) at (2,0);
\coordinate (S3) at (3,-1);

\coordinate[] (wh1) at (S1);
\coordinate[] (wh5) at ({ 8/3},{-2/3});
\coordinate[](bl4) at ({5/2},{-1/2});
\coordinate[] (bl6) at (3,{-1/2});
\midp{f0}{wh5}{wh1}
\def\t{0.0}
\fill[gray!30]
  ($(wh1)!\t!(f0)$) --
  ($(bl4)!\t!(f0)$) --
  ($(wh5)!\t!(f0)$) --
  ($(bl6)!\t!(f0)$) -- cycle;
\node at (f0) {$\f_0$};

\coordinate[wvert,label=45:${\wh_1}$]      (w0a)  at (3,0);
\coordinate                               (w1a)  at ({11/3},{ 1/3});
\coordinate                               (w2a)  at ({10/3},{ 2/3});
\coordinate[wvert,label=135:{$\wh_2$}]    (w3a)  at ({ 8/3},{ 1/3});
\coordinate[wvert,label=below:{$\wh_4$}]  (w4a)  at ({ 7/3},{-1/3});
\coordinate[wvert,label=-45:{$\wh_5$}]    (w5a)  at ({ 8/3},{-2/3});
\coordinate[wvert,label=right:{$\wh_7$}]  (w6a)  at ({10/3},{-1/3});

\coordinate[bvert,label=above:{$\bl_1$}]  (b23a) at (3,{ 1/2});
\coordinate[bvert,label=-45:{$\bl_2$}]     (b34a) at ({5/2},0);
\coordinate[bvert,label=left:{$\bl_4$}]  (b45a) at ({5/2},{-1/2});
\coordinate[bvert,label=below right:{$\bl_6$}]  (b56a) at (3,{-1/2});
\coordinate[bvert,label=above right:{$\bl_7$}]  (b61a) at ({7/2},0);

\draw (w0a) -- node[elab, above, yshift=1pt] {$\delta-\gamma$} (b23a);
\draw (w0a) -- node[elab, above, xshift=1pt] {$\epsilon-\delta$} (b34a);
\draw (w0a) -- node[elab, above, yshift=1pt] {$\alpha-\epsilon$} (b45a);
\draw (w0a) -- node[elab, above, yshift=1pt] {$\beta-\alpha$} (b56a);
\draw (w0a) -- node[elab, below, xshift=1pt] {$\gamma-\beta$} (b61a);

\draw (b23a) -- node[elab, above] {$\gamma-\beta$} (w3a);

\draw (b34a) -- node[elab, below, xshift=-2pt] {$\beta-\epsilon$} (w3a);
\draw (b34a) -- node[elab, below, xshift=-2pt] {$\delta-\gamma$} (w4a);

\draw (b45a) -- node[elab, above] {$\gamma-\alpha$} (w4a);
\draw (b45a) -- node[elab, below] {$\epsilon-\gamma$} (w5a);

\draw (b56a) -- node[elab, below] {$\gamma-\beta$} (w5a);
\draw (b56a) -- node[elab, below] {$\alpha-\delta$} (w6a);

\draw (b61a) -- node[elab, above] {$\delta-\gamma$} (w6a);

\coordinate[wvert,label=above:{$\wh_3$}] (w0b) at (2,0);
\coordinate                              (w1b) at ({ 8/3},{ 1/3});
\coordinate                              (w5b) at ({ 5/3},{-2/3});
\coordinate[wvert]                       (w6b) at ({ 7/3},{-1/3});

\coordinate[bvert,label=left:{$\bl_3$}]  (b56b) at (2,{-1/2});
\coordinate[bvert]                       (b61b) at ({5/2},0);

\draw (w0b) -- node[elab, below] {$\beta-\alpha$} (b56b);
\draw (w0b) -- node[elab, above] {$\gamma-\beta$} (b61b);
\draw (b56b) -- node[elab, above] {$\alpha-\delta$} (w6b);

\coordinate[wvert,label=right:{$\wh_6$}] (w0c) at (3,-1);
\coordinate                              (w2c) at ({10/3},{-1/3});
\coordinate[wvert]                       (w3c) at ({ 8/3},{-2/3});
\coordinate                              (w4c) at ({ 7/3},{-4/3});

\coordinate[bvert]                       (b23c) at (3,{-1/2});
\coordinate[bvert,label=left:{$\bl_5$}]  (b34c) at ({5/2},{-1});

\draw (w0c) -- node[elab, above] {$\delta-\gamma$} (b23c);
\draw (w0c) -- node[elab, below] {$\epsilon-\delta$} (b34c);
\draw (b34c) -- node[elab, above, xshift=-2pt] {$\beta-\epsilon$} (w3c);

\end{tikzpicture}
}

    \caption{A pentagonal Newton polygon with edges labeled by angles (right) and an Aztec graph with isoradial weights (left).}
    \label{fig:example_graph}
\end{figure}

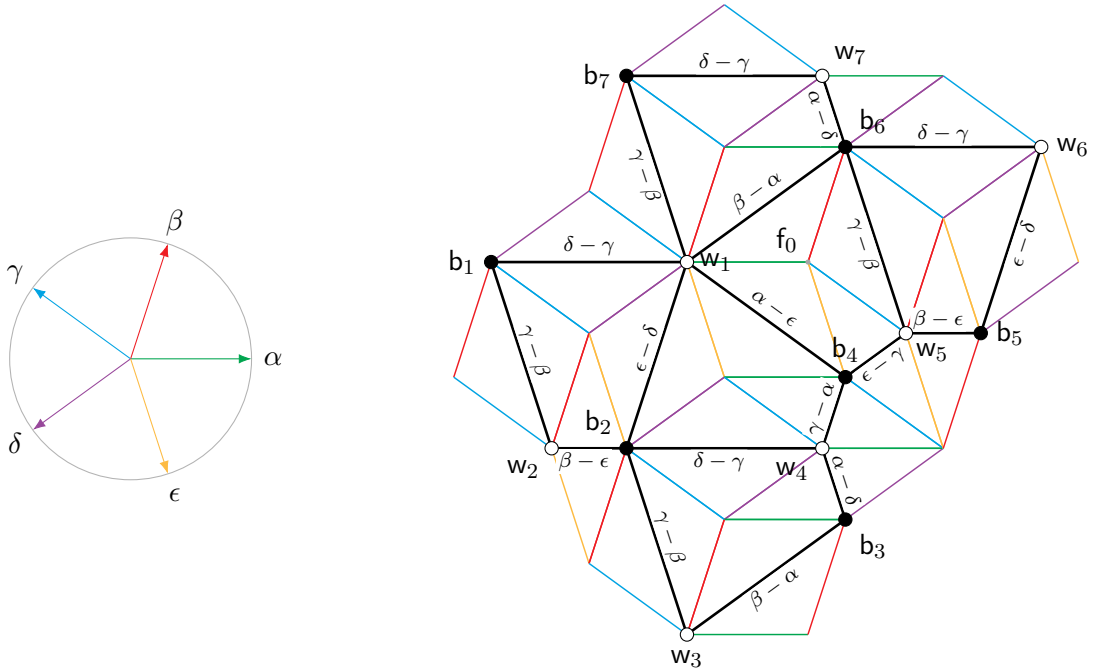
\begin{figure}
\centering
\tikzset{
  bip/.style={draw=black,line width=1pt},
  elab/.style={font=\scriptsize, inner sep=1pt, fill=white, text=black},
}

\colorlet{colalpha}{Green}
\colorlet{colbeta}{Red}
\colorlet{colgamma}{Cerulean}
\colorlet{coldelta}{Purple}
\colorlet{colepsilon}{Dandelion}

\newcommand{\Rh}[6]{%
  \draw[line width=.55pt,draw=#3] (#1) -- ($(#1)+(#2)$);
  \draw[line width=.55pt,draw=#5] ($(#1)+(#2)$) -- (#6);
  \draw[line width=.55pt,draw=#3] (#6) -- ($(#1)+(#4)$);
  \draw[line width=.55pt,draw=#5] ($(#1)+(#4)$) -- (#1);
}

\begin{tikzpicture}[scale=1.6]

  \def\R{1.0}
  \coordinate (va) at (\R,0);
  \coordinate (vb) at ({\R*cos(72)},{\R*sin(72)});
  \coordinate (vg) at ({\R*cos(144)},{\R*sin(144)});
  \coordinate (vd) at ({\R*cos(216)},{\R*sin(216)});
  \coordinate (ve) at ({\R*cos(288)},{\R*sin(288)});

  \coordinate (W1) at (0,0);

  \coordinate (B1) at ($(W1)+(vg)+(vd)$); 
  \coordinate (B2) at ($(W1)+(vd)+(ve)$); 
  \coordinate (B4) at ($(W1)+(ve)+(va)$); 
  \coordinate (B6) at ($(W1)+(va)+(vb)$); 
  \coordinate (B7) at ($(W1)+(vb)+(vg)$); 

  \coordinate (W2) at ($(B1)-(vg)-(vb)$); 
  \coordinate (W4) at ($(B4)-(vg)-(va)$); 
  \coordinate (W5) at ($(B4)-(ve)-(vg)$); 
  \coordinate (W7) at ($(B6)-(vd)-(va)$); 
  \coordinate (W3) at ($(B2)-(vg)-(vb)$); 
  \coordinate (W6) at ($(B6)-(vg)-(vd)$); 

  \coordinate (B3) at ($(W3)+(va)+(vb)$); 
  \coordinate (B5) at ($(W6)+(vd)+(ve)$); 

  \coordinate (F0) at ($(W1)+(va)$);

  \Rh{W1}{vg}{colgamma}{vd}{coldelta}{B1}
  \Rh{W1}{vd}{coldelta}{ve}{colepsilon}{B2}
  \Rh{W1}{ve}{colepsilon}{va}{colalpha}{B4}
  \Rh{W1}{va}{colalpha}{vb}{colbeta}{B6}
  \Rh{W1}{vb}{colbeta}{vg}{colgamma}{B7}

  \Rh{W2}{vb}{colbeta}{vg}{colgamma}{B1}
  \Rh{W2}{ve}{colepsilon}{vb}{colbeta}{B2}

  \Rh{W4}{vg}{colgamma}{vd}{coldelta}{B2}
  \Rh{W4}{va}{colalpha}{vg}{colgamma}{B4}
  \Rh{W4}{vd}{coldelta}{va}{colalpha}{B3}

  \Rh{W5}{vg}{colgamma}{ve}{colepsilon}{B4}
  \Rh{W5}{vb}{colbeta}{vg}{colgamma}{B6}
  \Rh{W5}{ve}{colepsilon}{vb}{colbeta}{B5}

  \Rh{W7}{vd}{coldelta}{va}{colalpha}{B6}
  \Rh{W7}{vg}{colgamma}{vd}{coldelta}{B7}

  \Rh{W3}{va}{colalpha}{vb}{colbeta}{B3}
  \Rh{W3}{vb}{colbeta}{vg}{colgamma}{B2}

  \Rh{W6}{vg}{colgamma}{vd}{coldelta}{B6}
  \Rh{W6}{vd}{coldelta}{ve}{colepsilon}{B5}

  \draw[bip] (W1) -- node[elab,sloped,above,yshift=1pt] {$\delta-\gamma$} (B1);
  \draw[bip] (W1) -- node[elab,sloped,above,xshift=1pt] {$\epsilon-\delta$} (B2);
  \draw[bip] (W1) -- node[elab,sloped,above,yshift=1pt] {$\alpha-\epsilon$} (B4);
  \draw[bip] (W1) -- node[elab,sloped,above,yshift=1pt] {$\beta-\alpha$} (B6);
  \draw[bip] (W1) -- node[elab,sloped,below,xshift=1pt] {$\gamma-\beta$} (B7);

  \draw[bip] (B1) -- node[elab,sloped,above] {$\gamma-\beta$} (W2);

  \draw[bip] (B2) -- node[elab,sloped,below,xshift=-2pt] {$\beta-\epsilon$} (W2);
  \draw[bip] (B2) -- node[elab,sloped,below,xshift=-2pt] {$\delta-\gamma$} (W4);

  \draw[bip] (B4) -- node[elab,sloped,above] {$\gamma-\alpha$} (W4);
  \draw[bip] (B4) -- node[elab,sloped,below] {$\epsilon-\gamma$} (W5);

  \draw[bip] (B6) -- node[elab,sloped,below] {$\gamma-\beta$} (W5);
  \draw[bip] (B6) -- node[elab,sloped,below] {$\alpha-\delta$} (W7);

  \draw[bip] (B7) -- node[elab,sloped,above] {$\delta-\gamma$} (W7);

  \draw[bip] (W3) -- node[elab,sloped,below] {$\beta-\alpha$} (B3);
  \draw[bip] (W3) -- node[elab,sloped,above] {$\gamma-\beta$} (B2);

  \draw[bip] (B3) -- node[elab,sloped,above] {$\alpha-\delta$} (W4);

  \draw[bip] (W6) -- node[elab,sloped,above] {$\delta-\gamma$} (B6);
  \draw[bip] (W6) -- node[elab,sloped,below] {$\epsilon-\delta$} (B5);
  \draw[bip] (B5) -- node[elab,sloped,above,xshift=-2pt] {$\beta-\epsilon$} (W5);

  \foreach \P in {W1,W2,W3,W4,W5,W6,W7}
    \node[wvert] at (\P) {};
  \foreach \P in {B1,B2,B3,B4,B5,B6,B7}
    \node[bvert] at (\P) {};

  \node[right =1pt] at (W1) {$\wh_1$};
  \node[below left =1pt] at (W2) {$\wh_2$};
  \node[below      =2pt] at (W3) {$\wh_3$};
  \node[below left     =2pt] at (W4) {$\wh_4$};
  \node[below right] at (W5) {$\wh_5$};
  \node[right      =2pt] at (W6) {$\wh_6$};
  \node[above right=1pt] at (W7) {$\wh_7$};

  \node[left       =2pt] at (B1) {$\bl_1$};
  \node[above left =2pt] at (B2) {$\bl_2$};
  \node[below right=2pt] at (B3) {$\bl_3$};
  \node[above      =2pt] at (B4) {$\bl_4$};
  \node[right      =2pt] at (B5) {$\bl_5$};
  \node[above right      =2pt] at (B6) {$\bl_6$};
  \node[left       =2pt] at (B7) {$\bl_7$};

  \fill[gray!70] (F0) circle (0.6pt);
  \node[above left] at (F0) {$\f_0$};

  \begin{scope}[shift={(-4.6,-0.8)}]
    \draw[gray!60] (0,0) circle (1);

    \draw[->,thin,draw=colalpha] (0,0) -- (0:1);
    \node[] at (0:1.18) {$\alpha$};

    \draw[->,thin,draw=colbeta] (0,0) -- (72:1);
    \node[] at (72:1.18) {$\beta$};

    \draw[->,thin,draw=colgamma] (0,0) -- (144:1);
    \node[] at (144:1.18) {$\gamma$};

    \draw[->,thin,draw=coldelta] (0,0) -- (216:1);
    \node[] at (216:1.18) {$\delta$};

    \draw[->,thin,draw=colepsilon] (0,0) -- (288:1);
    \node[] at (288:1.18) {$\epsilon$};
  \end{scope}

\end{tikzpicture}
\caption{
A choice of uniformly spaced angles (left) and the corresponding isoradial embedding (right) of the graph in Figure~\ref{fig:example_graph}. The colored edges form the quad graph $\graph^\diamond$.}\label{fig:isoradial_pentagon}
\end{figure}

For simplicity, we consider the setting of \emph{isoradial weights} from~\cite{Kenyonisoradial}. Namely, we take $\curve=\CC\mathbb P^1$ and let
\(
A_0=\{\zeta\in\CC\mathbb P^1:\ |\zeta|=1\}
\)
be the unit circle. In this case, the theta function is identically equal to $1$, and the prime form is
\[
E(\zeta,\eta)=\frac{\eta-\zeta}{\sqrt{d\zeta}\sqrt{d\eta}}.
\]
After performing a gauge transformation that removes the factor
$\frac{1}{\sqrt{d\zeta}\sqrt{d\eta}}$, Fock’s Kasteleyn matrix~\eqref{eq:kast_fock}
simplifies to
\begin{equation} \label{eq:isoradial_weight}
\kast_{\wh,\bl}=\beta-\alpha.
\end{equation}
{
These weights have the following geometric interpretation. We associate to $\graph$ its \emph{quad graph} $\graph^\diamond$, whose vertices are the vertices and faces of $\graph$, and whose edges are the pairs $(\x,\f)$ with
\(
\x \in B(\graph)\sqcup W(\graph), \f\in F(\graph),
\)
such that $\x$ lies on the boundary of $\f$. Thus, every face of $\graph^\diamond$ is a quadrilateral, and each edge of $\graph$ appears as one of its diagonals. The choice of angles determines an embedding of $\graph^\diamond$ in which every quadrilateral is a rhombus; see Figure~\ref{fig:isoradial_pentagon}(left). Precisely, if an edge $\e=\bl\wh$ is crossed by the strands $\alpha,\beta$, then the corresponding rhombus has opposite sides in the directions given by the unit complex numbers $\alpha$ and $\beta$. In this
embedding, the Kasteleyn matrix entry~\eqref{eq:isoradial_weight} is the complex vector representing the dual diagonal of the rhombus. In particular, the isoradial weight is its modulus, \emph{i.e.} the Euclidean length of the dual edge.
} 

{For small graphs, the formula~\eqref{eq:formula_inverse_K} can be evaluated explicitly using residues.} 

Consider the graph shown in Figure~\ref{fig:example_graph}. For each edge
$e\in E(N)$, we assume that all angles in $\zz_e$ are
equal. To simplify notation, we denote these angles by
$\alpha,\beta,\gamma,\delta,\epsilon$, and we label each edge of $N$ by its
associated angle. Then, Fock’s Kasteleyn matrix is
\[
\kast=
\begin{blockarray}{cccccccc}
& \bl_1 & \bl_2 & \bl_3 & \bl_4 & \bl_5 & \bl_6 & \bl_7 \\
\begin{block}{c[ccccccc]}
\wh_1 & \delta-\gamma & \epsilon-\delta & 0 & \alpha-\epsilon & 0 & \beta-\alpha & \gamma-\beta \\
\wh_2 & \gamma-\beta  & \beta-\epsilon & 0 & 0 & 0 & 0 & 0 \\
\wh_3 & 0 & \gamma-\beta & \beta-\alpha & 0 & 0 & 0 & 0 \\
\wh_4 & 0 & \delta-\gamma & \alpha-\delta & \gamma-\alpha & 0 & 0 & 0 \\
\wh_5 & 0 & 0 & 0 & \epsilon-\gamma & \beta-\epsilon & \gamma-\beta & 0 \\
\wh_6 &0 & 0 & 0 & 0 & \epsilon-\delta &  \delta-\gamma & 0 \\
\wh_7 & 0 & 0 & 0 & 0 & 0  &\alpha-\delta &  \delta-\gamma\\
\end{block}
\end{blockarray}.
\]
Directly computing the inverse, we find, for instance, 
\begin{equation} \label{eq:K11}
    \kast_{\bl_1,\wh_1}^{-1} = -\frac{(\alpha -\beta ) (\beta -\epsilon ) (\gamma -\delta )}{(\alpha -\gamma ) (\beta -\delta )^2 (\gamma
   -\epsilon )}, \qquad  \kast_{\bl_5,\wh_1}^{-1}= \frac{(\beta -\gamma ) (\gamma -\delta )^2}{(\alpha-\gamma)
   (\beta -\delta )^2 (\gamma -\epsilon )}.
\end{equation}
We now verify that these expressions agree with the evaluation of
\eqref{eq:formula_inverse_K}. From Figure~\ref{fig:pentagon_strands}, we compute the discrete Abel map:
\[
\begin{array}{c|c}
& \dabel \\ \hline
\wh_1 & -\alpha \\
\wh_2 & -\alpha-\beta+\delta \\
\wh_3 & -\alpha-\beta-\gamma+\delta+\epsilon \\
\wh_4 & -\alpha+\gamma+\epsilon \\
\wh_5 & -\gamma \\
\wh_6 & \beta-\gamma-\delta \\
\wh_7 & -\alpha+\beta-\delta
\end{array}
\qquad
\begin{array}{c|c}
 & \dabel \\ \hline
\bl_1 & -\alpha+\gamma+\delta \\
\bl_2 & -\alpha+\delta+\epsilon \\
\bl_3 & -\gamma+\delta+\epsilon \\
\bl_4 & \epsilon \\
\bl_5 & \beta-\gamma+\epsilon \\
\bl_6 & \beta \\
\bl_7 & -\alpha+\beta+\gamma
\end{array}
\]
With our notation,~\eqref{eq:bigD} is
\(
\bm D_{\bm \beta} = \beta-\gamma+\delta-\epsilon,
\)
so the boundary term
\[
\Psi_{\bm \beta}(\zeta) = \frac{(\zeta-\beta) (\zeta-\delta)}{(\zeta-\gamma) (\zeta-\epsilon)}.
\]
We first consider $\kast^{-1}_{\bl_1,\wh_1}$. We have
\[
\sign_{{\graphregion}_{\bl_1}}=\sign_{{\graphregion}_{\wh_1}}
=
\vcenter{
\hbox{
\begin{tikzpicture}[scale=0.9]
  \begin{scope}[rotate=0]
  \coordinate[bvert, label = left:$v_4$](1) at (0,1){}; 
    \coordinate[bvert, label = -45:$v_1$](2) at (1,0){}; 
     \coordinate[bvert, label = 45:$v_2$](3) at (2,1){}; 
      \coordinate[bvert](4) at (1,1){};
          \coordinate[bvert, label = below:$v_5$](5) at (0,0){}; 
          \coordinate[bvert, label = above:$v_3$](6) at (1,2){};
          \draw[solidedge]  (5) --node[below]{$\alpha$} (2);
           \draw[solidedge] (3) --node[above right]{$\gamma$} (6);
           \draw[solidedge] (1) -- node[left]{$\epsilon$} (5)
           ;

  \draw[hollowedge] (2) --node[below right]{$\beta$} (3);
   \draw[hollowedge]   (6) --node[above left]{$\delta$} (1)
  ;
    
  \end{scope}
\end{tikzpicture}
}}
\]
so we can take
\[
I_{\bl_1} = [v_1,v_2], \qquad I_{\wh_1} = [v_4,v_1], \qquad v = v_1, \qquad \sign(v)=1.
\]
The double integral is
\begin{align*}
&\frac{1}{(2 \pi \ii)^2} \int_{C(I_{\wh_1})} \int_{C(I_{\bl_1})} \frac{1}{(\eta-\zeta)} \frac{(\zeta-\alpha)(\zeta-\epsilon)} {(\zeta-\beta) (\zeta-\delta)^2} \frac{(\eta-\beta) (\eta-\delta)}{(\eta-\alpha)(\eta-\gamma) (\eta-\epsilon)}  d\zeta d\eta \\
&=-\frac{(\alpha-\beta) (\beta - \epsilon)}{(\beta-\delta)^2}        \frac{1}{2 \pi \ii} \int_{C(I_{\wh_1})}  \frac{ (\eta-\delta)}{(\eta-\alpha)(\eta-\gamma) (\eta-\epsilon)}   d\eta \\
&= -\frac{(\alpha-\beta)(\beta - \epsilon)}{(\beta-\delta)^2} \left(\frac{(\alpha-\delta)}{(\alpha-\gamma)(\alpha-\epsilon)}-\frac{(\delta-\epsilon)}{(\alpha-\epsilon)(\gamma-\epsilon)} \right)\\
&=  -\frac{(\alpha -\beta ) (\beta -\epsilon ) (\gamma -\delta )}{(\alpha -\gamma ) (\beta -\delta )^2 (\gamma
   -\epsilon )}.
\end{align*}
Since $\dabel(\wh_1)-\dabel(\bl_1) = -\gamma-\delta$, the null sector $\sector_{\bl_1,\wh_1} = [v_4,v_2]$. Since \[\partial I_\bl \cap I_\wh = (v_2-v_1) \cap [v_4,v_1] = -v_1, \]the single integral term is
\[
- \langle \partial I_\bl \cap I_\wh,\mathsf A_{\bl,\wh} \rangle = \mathsf A_{\bl,\wh}^{v_1} = 0
\]
since $v_1 \in \sector_{\bl_1,\wh_1}$. Thus, the total agrees with~\eqref{eq:K11}.

We now consider $\kast^{-1}_{\bl_5,\wh_1}$. We have
\[
\sign_{{\graphregion}_{\bl_5}}
=
\vcenter{
\hbox{
\begin{tikzpicture}[scale=0.9]
  \begin{scope}[rotate=0]
 \coordinate[bvert, label = left:$v_4$](1) at (0,1){}; 
    \coordinate[bvert, label = -45:$v_1$](2) at (1,0){}; 
     \coordinate[bvert, label = 45:$v_2$](3) at (2,1){}; 
      \coordinate[bvert](4) at (1,1){};
          \coordinate[bvert, label = below:$v_5$](5) at (0,0){}; 
          \coordinate[bvert, label = above:$v_3$](6) at (1,2){};

         \draw[hollowedge]  (5) --node[below]{$\alpha$} (2);
           \draw[solidedge] (3) --node[above right]{$\gamma$} (6);
           \draw[solidedge] (1) -- node[left]{$\epsilon$} (5)
           ;

  \draw[hollowedge] (2) --node[below right]{$\beta$} (3);
   \draw[hollowedge]   (6) --node[above left]{$\delta$} (1)
  ;

  \end{scope}
\end{tikzpicture}
}}
\]
so we can take
\[
I_{\bl_5} = [v_3,v_4], \qquad I_{\wh_1} = [v_2,v_3], \qquad v = v_3, \qquad \sign(v)=1.
\]
The double integral is 
\begin{align*}
&\frac{1}{(2 \pi \ii)^2} \int_{C(I_{\wh_1})} \int_{C(I_{\bl_5})} \frac{1}{(\eta-\zeta)}  \frac{(\zeta-\gamma)^2}{(\zeta-\beta)^2(\zeta-\delta)} \frac{(\eta-\beta)(\eta-\delta)}{(\eta-\alpha)(\eta-\gamma)(\eta-\epsilon)} d\zeta d\eta \\
&=\frac{(\gamma-\delta)^2}{(\beta-\delta)^2}\frac{1}{2 \pi \ii} \int_{C(I_{\wh_1})}   \frac{(\eta-\beta)}{(\eta-\alpha)(\eta-\gamma)(\eta-\epsilon)}  d\eta \\
&=\frac{(\beta -\gamma ) (\gamma -\delta )^2}{(\alpha-\gamma)
   (\beta -\delta )^2 (\gamma -\epsilon )},
\end{align*}
Since $\dabel(\wh_1) -\dabel(\bl_5) =-\alpha- \beta+\gamma-\epsilon$, the null sector $\sector_{\bl_5,\wh_1} = [v_2,v_4]$. Since \[ \partial I_\bl \cap I_\wh = (v_4-v_3) \cap [v_2,v_3] = -v_3,\] the single integral term is
\[
- \langle \partial I_\bl \cap I_\wh,\mathsf A_{\bl,\wh} \rangle = \mathsf A_{\bl,\wh}^{v_3} = 0
\] 
since $v_3 \in \sector_{\bl_5,\wh_1}$. Thus, the total again agrees with~\eqref{eq:K11}.

\subsection{Correctness of the definition of $\kast^{-1}$}

\begin{figure}
    \centering
\begin{tikzpicture}
    \begin{scope}
    \coordinate[bvert, label = left:$\bl$] (w) at (0,0) {};

    \coordinate[wvert] (b2) at (1,-1) {};
    \coordinate[wvert] (b4) at (-1,-1) {};
    \draw[] (w) edge (b2)  edge (b4);
    \draw[solidedge] (1.5,-0.5) -- (-1.5,-0.5);
    \node at (2,-0.5) {$\beta_e$};
    \coordinate[label=below:${\x}$] (xx) at (-1,-2) {};
    \draw[->] (xx) to[out=70, in=-90] (w);
    \end{scope}

        \begin{scope}[shift={(7,0)}]
    \coordinate[bvert, label = left:$\bl$] (w) at (0,0) {};

    \coordinate[wvert] (b2) at (1,-1) {};
    \coordinate[wvert] (b4) at (-1,-1) {};
    \draw[] (w) edge (b2)  edge (b4);
    \draw[solidedge] (1.5,-0.5) -- (-1.5,-0.5);
    \node at (2,-0.5) {$\beta_e$};
 \coordinate[label=above:${\x}$] (xx) at (-1,2) {};
     \draw[->] (xx) to[out=-70, in=90] (w);
  \coordinate[label=below:${\f}$] (f) at (0,-1.3) {};
  \draw[->] (w) -- (f);
 
    \end{scope}
\end{tikzpicture}
    \caption{The two cases in the proof of Lemma~\ref{lem:sign_poles}: $\sign_{{\graphregion}_\x}(e) = +$ (left) and $\sign_{{\graphregion}_\x}(e)=-$ (right).} 
    \label{fig:sign_poles}
\end{figure}

\begin{lemma} \label{lem:sign_poles}
Let $e \in E(N)$. If $\x \in V(\graph)$ is a vertex and $\sign_{{\graphregion}_{\x}}(e)=+$ (resp. $-$), then
    \[
        \ord_{\alpha}\left(\frac{1}{g_{\x}\Psi_{\bm \beta}}\right) \geq 0
        \quad (\text{resp. } \leq 0)
    \]
    for all $\alpha \in \zz_e$.  
\end{lemma}
\begin{proof}
By~\eqref{eq:d_beta_defn} and Lemma~\ref{lem:div_g}, we have 
\begin{equation} \label{eq:order_at_alpha}
     \ord_{\alpha}\left(\frac{1}{g_{\x}\Psi_{\bm \beta}}\right)= (\dabel(\y)-\dabel(\x))_\alpha
\end{equation}
where $\y$ is any $\beta_e$-outer vertex. Assume first that $e$ is black so that $\y = \bl$ is a black vertex. We have two cases (see Figure~\ref{fig:sign_poles}):
\begin{enumerate}
    \item $\sign_{{\graphregion}_\x}(e) = +$: Any path from $\x$ to $\bl$ crosses $\beta_e$ positively, \emph{i.e.} 
    $(\dabel(\bl)-\dabel(\x))_{\beta_e}=1$. By Lemma~\ref{lem:parallel_zz_same_sign}, $(\dabel(\bl)-\dabel(\x))_\alpha \geq 0$ for all $\alpha \in \zz_e$.
    \item $\sign_{{\graphregion}_\x}(e) = -$: Let $\f$ be the face separated from $\bl$ by $\beta_e$. Any path from $\x$ to $\f$ crosses $\beta_e$ negatively, \emph{i.e.} $(\dabel(\f)-\dabel(\x))_{\beta_e}=-1$. By Lemma~\ref{lem:parallel_zz_same_sign}, $(\dabel(\f)-\dabel(\x))_\alpha \leq 0$ for all $\alpha \in \zz_e$. Since $\dabel(\bl)-\dabel(\x) = \dabel(\f)-\dabel(\x) + \beta_e$, we get $(\dabel(\bl)-\dabel(\x))_\alpha \leq 0$ for all $\alpha \in \zz_e$.
\end{enumerate}
The argument for white $e$ is identical, with $\bl$ replaced by $\wh$.
\end{proof}







Recall Definition~\ref{def:null_sector} of the null sector $\sector_{\bl,\wh}$.

\begin{lemma}\label{lem:sector}
Let $e\in E(N)$, $\bl\in B(\graphbeta)$ and $\wh\in W(\graphbeta)$.
Assume that one of the following holds:
\begin{enumerate}
\item $\sign_{{\graphregion}_\bl}(e)=+$ and $\sign_{{\graphregion}_\wh}(e)=-$.
\item $\sign_{{\graphregion}_\wh}(e)=-$ and $\bl$ lies on $\beta_e$.
\item $\sign_{{\graphregion}_\bl}(e)=+$ and $\wh$ lies on $\beta_e$.
\end{enumerate}
Then $e\subset \sector_{\bl,\wh}$.
\end{lemma}

\begin{proof}
We consider each of the three cases:

\begin{enumerate}
\item Suppose $\sign_{{\graphregion}_\bl}(e)=+$ and $\sign_{{\graphregion}_\wh}(e)=-$. Then
\[
(\dabel(\wh)-\dabel(\bl))_{\beta_e}=1,
\]
so Lemma~\ref{lem:div_g} shows that $g_{\bl,\wh}$ has a zero at $\beta_e$. The conclusion therefore follows from the definition of $\sector_{\bl,\wh}$.

\item Suppose $\sign_{{\graphregion}_\wh}(e)=-$ and $\bl$ lies on $\beta_e$. Let $\f$ be the face on the other side of $\beta_e$ from $\bl$, as in Figure~\ref{fig:sign_poles}(right). Then
\[
\dabel(\wh)-\dabel(\f)
=
\dabel(\wh)-\dabel(\bl)+\beta_e.
\]
Thus, \(g_{\f,\wh}\) has the same zeros and poles in \(A_0\) as \(g_{\bl,\wh}\), except for one additional zero at \(\beta_e\).

We now distinguish the two cases in~Definition~\ref{def:null_sector} of \(\sector_{\bl,\wh}\).

\begin{enumerate}
\item {If $g_{\bl,\wh}$ has at least one zero, then \(e\not\subset \sector_{\bl,\wh}\) would imply that the edges carrying zeros and poles of \(g_{\f,\wh}\) interlace along \(\partial N\), contradicting Lemma~\ref{lemma:interlacing}.}


\item {Suppose $g_{\bl,\wh}$ has no zeros in $A_0$ and
\[
\divisor(g_{\bl,\wh})|_{A_0}=-\alpha-\beta,
\]
where $\alpha,\beta\in\zz$ are as in Definition~\ref{def:null_sector}. 
Since \(\sign_{{\graphregion}_\wh}(e)=-\), the vertex \(\wh\) does not lie on the other side of \(\beta_e\) from \(\bl\), so \(\alpha\) and \(\beta\) are distinct from \(\beta_e\). 
Choosing $\gamma=\beta_e$ in Case~2 of Definition~\ref{def:null_sector}, we get \(e\subset \sector_{\bl,\wh}\).}
\end{enumerate}

\item This is symmetric to Case~2.
\end{enumerate}
\end{proof}

We say that a cyclic interval $[a,b] \subset \partial N$ is \emph{nontrivial} if $a\neq b$; equivalently, $[a,b]$ is neither a single point nor all of $\partial N$.

\begin{lemma}\label{lem:boundary_intersection_indicator}
Let \(J=[a,b]\) and \(K=[c,d]\) be nontrivial cyclic closed intervals in \(\partial N\) with endpoints in $V(N)$. Then, as \(0\)-chains,
\begin{equation}
\partial(J\cap K)
=
\mathbf 1_{\{b\in K\}} b
-\mathbf 1_{\{a\in K\}} a
+\mathbf 1_{\{d\in J\}} d
-\mathbf 1_{\{c\in J\}} c
+\mathbf 1_{\{a=c\}} a
-\mathbf 1_{\{b=d\}} b .
\label{eq:star_indicator}
\end{equation}
\end{lemma}
\begin{proof}
If \(J\) and \(K\) have no endpoints in common, then the formula is immediate: the boundary of \(J\cap K\) consists of those endpoints of \(J\) that lie in \(K\), together with those endpoints of \(K\) that lie in \(J\), with the usual signs. This gives
\[
\partial(J\cap K)
=
\mathbf 1_{\{b\in K\}} b
-\mathbf 1_{\{a\in K\}} a
+\mathbf 1_{\{d\in J\}} d
-\mathbf 1_{\{c\in J\}} c.
\]
In general, the same expression remains correct except when two endpoints of the same type coincide. If \(a=c\), then the point \(a\) is counted twice as a left endpoint, so one copy must be added back; this produces the term \(\mathbf 1_{\{a=c\}} a\). Similarly, if \(b=d\), then the point \(b\) is counted twice as a right endpoint, so one copy must be subtracted; this produces the term \(-\mathbf 1_{\{b=d\}} b\). By contrast, if \(a=d\) or \(b=c\), the two contributions occur with opposite signs and therefore cancel automatically, so no further correction is needed. This proves \eqref{eq:star_indicator}.
\end{proof}

\begin{corollary}\label{cor:arc_chain_identity}
Let $\bl \in B(\graphbeta)$ and $\wh \in W(\graphbeta)$. Let $J=[a,b]$ and $K=[c,d]$ be any two nontrivial closed intervals in $\partial N$ with endpoints in $V(N)$. Then
\begin{equation}\label{eq:arc_chain_identity_integral}
\begin{aligned}
\frac{1}{2 \pi \ii} \int_{C(J \cap K)} g_{\bl,\wh}
= -\langle \partial J \cap K,\mathsf A_{\bl,\wh} \rangle
-\langle \partial K \cap J,\mathsf A_{\bl,\wh} \rangle 
-\mathbf 1_{\{a=c\}} \mathsf A_{\bl,\wh}^{a}
+\mathbf 1_{\{b=d\}} \mathsf A_{\bl,\wh}^{b}.
\end{aligned}
\end{equation}
\end{corollary}

\begin{proof}
We apply the linear functional $\langle \cdot, \mathsf A_{\bl,\wh} \rangle$ from~\eqref{eq:functional} to~\eqref{eq:star_indicator}; see also Definition~\ref{def:contours_from_intervals}.
\end{proof}

\begin{proposition}\label{prop:invariance}
The formula \eqref{eq:formula_inverse_K} is:
\begin{enumerate}
\item Invariant under swapping $\x$ and $\y$.
\item Invariant under fixing $I_\wh$ and expanding $I_\bl$ across any edge $e$ 
such that:
\begin{itemize}
    \item $\sign_{{\graphregion}_\bl}(e)=+$, or
    \item $\bl$ lies on $\beta_e$ and $e$ is black (in particular, this holds whenever $\bl$ is $\beta_e$-outer).
\end{itemize}
Precisely, if $I_\bl' = [\ell_{\bl}',r_\bl']$ is obtained from $I_\bl$ by a sequence of such expansions, then
\begin{equation}\label{eq:Kinv_expansion}
\kast^{-1}_{\bl,\wh}
=\sign(v)\left(
\iint_{C(I_{\wh}) \prec C(I_{\bl}')} \omega_{\bl,\wh}
-
\langle \partial I_\bl' \cap I_\wh,\mathsf A_{\bl,\wh} \rangle
\right).
\end{equation}
Similarly, we can fix $I_\bl$ and expand $I_\wh$ across any edge such that:
\begin{itemize}
    \item $\sign_{{\graphregion}_\wh}(e)=-$, or
    \item $\wh$ lies on $\beta_e$ and $e$ is white (in particular, this holds whenever $\wh$ is $\beta_e$-outer).
\end{itemize}
\item Independent of the choice of convex vertex $v$.
\item Invariant under changing the choice of $\Rin_\x$ for an outer vertex $\x$ whenever it is not unique.
\end{enumerate}
\end{proposition}
\begin{proof}
\begin{enumerate}
\item 
It suffices to take $(\x,\y) = (\bl,\wh)$. Exchanging the order of integration requires deforming contours so that, at common angles, the circles in $C(I_\wh)$ lie strictly inside those in $C(I_\bl)$. The only new contribution arises from crossing the simple pole on the diagonal $\zeta=\eta$, giving
\begin{align*}
\kast^{-1}_{\bl,\wh} &:= \sign(v)\left( \iint_{C(I_{\wh})\prec C(I_{\bl})} \omega_{\bl,\wh} - \langle \partial I_\bl \cap I_\wh,\mathsf A_{\bl,\wh} \rangle \right)\\
&= \sign(v)\left( \iint_{C(I_\bl) \prec C(I_\wh)}\omega_{\bl,\wh}
-\frac{1}{2 \pi \ii}\int_{C(I_\bl\cap I_\wh)} g_{\bl,\wh} - \langle \partial I_\bl \cap I_\wh,\mathsf A_{\bl,\wh} \rangle \right).
\end{align*}
Applying Corollary~\ref{cor:arc_chain_identity} to $I_\bl=[\ell_\bl,r_\bl]$ and $I_\wh=[\ell_\wh,r_\wh]$, we get 
\begin{align*}
-\frac{1}{2 \pi \ii}\int_{C(I_\bl\cap I_\wh)} g_{\bl,\wh} - \langle \partial I_\bl \cap I_\wh,\mathsf A_{\bl,\wh} \rangle = \langle \partial I_\wh \cap I_\bl,\mathsf A_{\bl,\wh} \rangle
+\mathbf 1_{\{\ell_\bl=\ell_\wh\}}\mathsf A_{\bl,\wh}^{\ell_\bl}
-\mathbf 1_{\{r_\bl=r_\wh\}}\mathsf A_{\bl,\wh}^{r_\bl}.
\end{align*}
It remains to show that the two terms with indicators vanish. Consider the term $\mathbf 1_{\{\ell_\bl=\ell_\wh\}}\mathsf A_{\bl,\wh}^{\ell_\bl}$. If $\ell_\bl=\ell_\wh$, let $e$ be the boundary edge immediately clockwise from the common point $\ell_\bl=\ell_\wh$. Then, by definition of $I_\bl$ and $I_\wh$,
\[
\sign_{\Rin_\bl}(e)=+,
\qquad
\sign_{\Rin_\wh}(e)=-.
\]
We have the following cases:
\begin{itemize}
\item $\sign_{\Rin_\bl}(e) = \sign_{{\graphregion}_\bl}(e)$ and $\sign_{\Rin_\wh}(e) = \sign_{{\graphregion}_\wh}(e)$: Lemma~\ref{lem:sector} implies that $\ell_\bl=\ell_\wh \in \sector_{\bl,\wh}$ and therefore, \(\mathsf A_{\bl,\wh}^{\ell_\bl}
 =0.\)
\item $\sign_{\Rin_\bl}(e) \neq \sign_{{\graphregion}_\bl}(e)$ and $\sign_{\Rin_\wh}(e) = \sign_{{\graphregion}_\wh}(e)$: Then $\bl$ is $\beta_e$-outer; in particular it lies on $\beta_e$. Therefore, Lemma~\ref{lem:sector} again implies that $\ell_\bl=\ell_\wh \in \sector_{\bl,\wh}$ and therefore, \(\mathsf A_{\bl,\wh}^{\ell_\bl}
 =0.\)
\item  $\sign_{\Rin_\bl}(e) = \sign_{{\graphregion}_\bl}(e)$ and $\sign_{\Rin_\wh}(e) \neq \sign_{{\graphregion}_\wh}(e)$: This is symmetric to the previous case.
\item $\sign_{\Rin_\bl}(e) \neq \sign_{{\graphregion}_\bl}(e)$ and $\sign_{\Rin_\wh}(e) \neq \sign_{{\graphregion}_\wh}(e)$: This case cannot happen since it would mean that $\bl$ and $\wh$ are both $\beta_e$-outer but $\beta_e$-outer vertices must all have the same color.
\end{itemize}
 
Similarly, $\mathbf 1_{\{r_\bl=r_\wh\}} \mathsf A_{\bl,\wh}^{r_\bl} =0$.

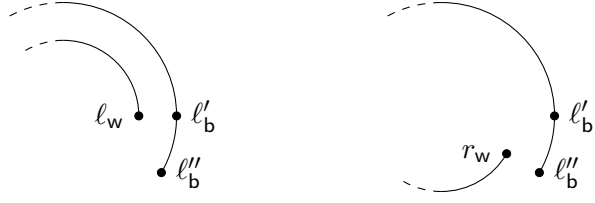
\begin{figure}
    \centering
   \begin{tikzpicture}
   \begin{scope}
  \coordinate (O) at (0,0);
  \draw[] ($(O)+(0:1.5)$) arc[start angle=0,end angle=90,radius=1.5];
  \draw[dashed] ($(O)+(90:1.5)$) arc[start angle=90,end angle=120,radius=1.5];

  \draw ($(O)+(0:1.)$) arc[start angle=0,end angle=90,radius=1.];
  \draw[dashed] ($(O)+(90:1.)$) arc[start angle=90,end angle=120,radius=1.];

  \draw ($(O)+(0:1.5)$) arc[start angle=0,end angle=-30,radius=1.5];

  \coordinate[nvert,label = left: {$\ell_\wh$}] (1) at (0:1) {};
    \coordinate[nvert,label = right: {$\ell_\bl'$}] (2) at (0:1.5) {};
      \coordinate[nvert,label = right: {$\ell_\bl''$}] (3) at (-30:1.5) {};
    \end{scope}

     \begin{scope}[shift={(5,0)}]
  \coordinate (O) at (0,0);
  \draw ($(O)+(0:1.5)$) arc[start angle=0,end angle=90,radius=1.5];
  \draw[dashed] ($(O)+(90:1.5)$) arc[start angle=90,end angle=120,radius=1.5];

  \draw ($(O)+(-30:1.)$) arc[start angle=-30,end angle=-90,radius=1.];
  \draw[dashed] ($(O)+(-90:1.)$) arc[start angle=-90,end angle=-120,radius=1.];
   \draw  ($(O)+(0:1.5)$) arc[start angle=0,end angle=-30,radius=1.5];
 \coordinate[nvert,label = left: {$r_\wh$}] (1) at (-30:1) {};
    \coordinate[nvert,label = right: {$\ell_\bl'$}] (2) at (0:1.5) {};
      \coordinate[nvert,label = right: {$\ell_\bl''$}] (3) at (-30:1.5) {};
    \end{scope}
\end{tikzpicture}
\caption{The two possible configurations in the case 
\(\sign_{\Rin_\wh}(e)=\sign_{{\graphregion}_\wh}(e)\) considered in the proof of Item~2. The inner (resp. outer) cyclic interval is $I_\wh$ (resp. $I_{\bl}''$). The left panel shows the case \(\ell'_\bl=\ell_\wh\), and the right panel the case \(\ell''_\bl=r_\wh\). In both cases \(\sign_{{\graphregion}_\wh}(e)=-\) by definition of $I_\wh$.}
    \label{fig:item_2_case_1}
\end{figure}

\item By Item~1, it suffices to consider the case where $(\x,\y) = (\bl,\wh)$ and we expand $I_\bl$. Moreover, {since moving the right endpoint is analogous}, it suffices to show that~\eqref{eq:Kinv_expansion} is invariant under moving the left endpoint 
of $I_\bl'$ one step, \emph{i.e.} let
\[
 I''_\bl=[\ell''_\bl,r_\bl'],
\qquad e=[\ell''_\bl,\ell_\bl'] \in E(N).
\]
We deform the contour $C(I'_\bl)$ to $C(I''_\bl)$, keeping $C(I_\wh)$ fixed. Since $\zz_e$ does not contain any poles of the integrand (by Lemma~\ref{lem:sign_poles} when $\sign_{{\graphregion}_\bl}(e)=+$, or because $\eqref{eq:order_at_alpha}=0$ if $\bl$ lies on $\beta_e$ and $e$ is black), the only poles that the deformation crosses are along the diagonal $\zeta=\eta$. Therefore,
\begin{align*}
&\sign(v)\left( \iint_{C(I_{\wh}) \prec C(I_{\bl}')} \omega_{\bl,\wh} - \langle \partial I_\bl' \cap I_\wh,\mathsf A_{\bl,\wh} \rangle \right)\\
&=\sign(v)\left( \iint_{C(I_\wh)\prec C(I_\bl'')}\omega_{\bl,\wh}
+\frac{1}{2 \pi \ii}\int_{C(e\cap I_\wh)} g_{\bl,\wh} - \langle \partial I_\bl' \cap I_\wh,\mathsf A_{\bl,\wh} \rangle \right).
\end{align*}
Thus, it suffices to show
\[
\frac{1}{2 \pi \ii}\int_{C(e\cap I_\wh)} g_{\bl,\wh} -\langle \partial I_\bl' \cap I_\wh,\mathsf A_{\bl,\wh} \rangle + \langle \partial I_\bl'' \cap I_\wh,\mathsf A_{\bl,\wh} \rangle =\frac{1}{2 \pi \ii}\int_{C(e\cap I_\wh)} g_{\bl,\wh} + \langle \partial e \cap I_\wh,\mathsf A_{\bl,\wh} \rangle =0.
\]
Applying Corollary~\ref{cor:arc_chain_identity} to the pair of intervals
$e$ and $I_\wh=[\ell_\wh,r_\wh]$, 
\begin{align} \label{eq:two_brackets}
&\frac{1}{2 \pi \ii}\int_{C(e\cap I_\wh)} g_{\bl,\wh} + \langle \partial e \cap I_\wh,\mathsf A_{\bl,\wh} \rangle \\ &= -\langle \partial I_\wh \cap e,\mathsf A_{\bl,\wh} \rangle  -\mathbf 1_{\{\ell''_\bl=\ell_\wh\}}\mathsf A^{\ell''_\bl}_{\bl,\wh} +\mathbf 1_{\{\ell'_\bl=r_\wh\}}\mathsf A^{\ell'_\bl}_{\bl,\wh}\\&=
\Big(\mathbf 1_{\{\ell_\wh\in e\}}\mathsf A^{\ell_\wh}_{\bl,\wh}
-\mathbf 1_{\{\ell''_\bl=\ell_\wh\}}\mathsf A^{\ell''_\bl}_{\bl,\wh}\Big)
 + 
\Big(-\mathbf 1_{\{r_\wh\in e\}}\mathsf A^{r_\wh}_{\bl,\wh}
+\mathbf 1_{\{\ell'_\bl=r_\wh\}}\mathsf A^{\ell'_\bl}_{\bl,\wh}\Big).
\end{align}
Since $e=[\ell''_\bl,\ell'_\bl]$ is an edge,
\[
\mathbf 1_{\{\ell_\wh\in e\}} = \mathbf 1_{\{\ell_\wh = \ell_\bl'\}} + \mathbf 1_{\{\ell_\wh = \ell_\bl''\}}, \qquad \mathbf 1_{\{r_\wh\in e\}} = \mathbf 1_{\{r_\wh = \ell_\bl'\}} + \mathbf 1_{\{r_\wh = \ell_\bl''\}}.
\]
Thus,~\eqref{eq:two_brackets} becomes
\begin{equation} \label{eq:two_indicator_terms}
\mathbf 1_{\{\ell'_\bl=\ell_\wh\}}\mathsf A^{\ell'_\bl}_{\bl,\wh} - \mathbf 1_{\{\ell''_\bl=r_\wh\}}\mathsf A^{\ell''_\bl}_{\bl,\wh}.
\end{equation}
We now have two cases:
\begin{itemize}
    \item $\sign_{\Rin_\wh}(e) = \sign_{{\graphregion}_\wh}(e)$: If $\ell'_\bl=\ell_\wh$ or if $\ell''_\bl=r_\wh$, then $e$ is adjacent to an endpoint of $I_\wh$, hence $\sign_{{\graphregion}_\wh}(e) = -$ by definition of $I_\wh$ (see Figure~\ref{fig:item_2_case_1}). Thus, the common endpoint lies in $\sector_{\bl,\wh}$ by Cases~1 and 2 of Lemma~\ref{lem:sector}, corresponding exactly to the two cases in the hypothesis of Item~2.
    \item $\sign_{\Rin_\wh}(e) \neq \sign_{{\graphregion}_\wh}(e)$: Then $\wh$ is $\beta_e$-outer; in particular, $e$ is white so the second hypothesis case that $\bl$ lies on $\beta_e$ and $e$ is black cannot occur. In the first hypothesis case where $\sign_{{\graphregion}_\bl}(e) = +$, any common endpoints are again, as in the previous bullet point, in $\sector_{\bl,\wh}$ by Case~1 of Lemma~\ref{lem:sector}. 
\end{itemize}
Thus, in both cases, both terms in~\eqref{eq:two_indicator_terms} are $0$, completing the proof of Item~2. 

\item
By Item~1, it suffices to show that the formula is invariant under replacing $I_\bl$ with the other white sign interval $I_\bl'$ for $\Rin_\bl$. We can enlarge $I_\bl$ until it becomes the complement of $I_\bl'$. For every edge $e$ that we absorb along the way, either 
\begin{itemize}
\item $\sign_{{\graphregion}_\bl}(e)= \sign_{\Rin_\bl}(e) = + $, or
\item $\sign_{{\graphregion}_\bl}(e) \neq \sign_{\Rin_\bl}(e)$ so $\bl$ is $\beta_e$-outer.
\end{itemize}
Thus,~\eqref{eq:Kinv_expansion} is unchanged by Item~2. We then reverse the orientation, which complements the interval and flips $\sign(v)$.

\item By Item~1, we may assume that $\x=\bl$ is a black vertex.
By Condition~2 of Definition~\ref{def:azgraph}, the edges
$e\in E(N)$ for which $\bl$ is $\beta_e$-outer form a consecutive
sequence
\[
e_1 < e_2 < \cdots < e_k
\]
along $\partial N$; see Figure~\ref{fig:condition_2}(right). Let $v_1<v_2<v_3<v_4$ be the convex vertices of $N$
such that $e_1,\dots,e_k \subset [v_2,v_3]$.

If we choose $e_i$ to define $\Rin_\bl$, then 
\[
\sign_{\Rin_\bl}(e_i)=+,
\qquad
\sign_{\Rin_\bl}(e_j)=- \quad (j\neq i).
\]
Considering the inconsistently oriented cones as in the proof of Lemma~\ref{lem:sign_changes_sections} and Figure~\ref{fig:fig_opp_cones}, we obtain that $s_{v_2}(\Rin_\bl)$ is the left endpoint of $e_i$ and
$s_{v_3}(\Rin_\bl)$ is the right endpoint of $e_i$. If $v\in\{v_1,v_2\}$, then
\[
I_\bl = [s_{v_1}(\Rin_\bl),  s_{v_2}(\Rin_\bl)].
\]
As $e_i$ varies, the right endpoint of $I_\bl$ moves across the edges $e_1,\dots,e_k$. By Item~2, expanding or shrinking $I_\bl$ across these edges does not change the value of~\eqref{eq:formula_inverse_K}. The case $v\in\{v_3,v_4\}$ is analogous.

\end{enumerate}
\end{proof}

\subsection{Proof of Theorem~\ref{thm:main}}

Our proof is based on the same general ideas as those used for the Aztec diamond in~\cite{BdT24}. Proposition~\ref{prop:invariance} allows us to simplify parts of the argument from that setting. The added complexity of our proof comes from features of the general setting that do not arise in the Aztec diamond case.

It suffices to check that $\kast \kast^{-1} = \mathsf{I}$; then $\kast^{-1} \kast = \mathsf{I}$ (or, rather, its transpose) follows by Item~1 of Proposition~\ref{prop:invariance}. Let $\wh, \wh' \in W(\graphbeta)$. We need to show that 
\begin{equation} \label{eq:inverse_defn}
\sum_{\bl \sim \wh'}\kast_{\wh',\bl} \kast^{-1}_{\bl,\wh}=\mathbf{1}_{\{\wh = \wh' \}}.
\end{equation}

We distinguish two cases:
\begin{enumerate}
    \item $\wh'$ is an inner vertex: Suppose first that $\Rin_\bl= \Rin_{\wh'}$ for all $\bl \sim \wh'$ so that $I_\bl$ is the same interval $J = [\ell,r]$ for all of them. Then, we get  
\begin{equation}\label{eq:kast_cancel}
    \sum_{\bl \sim \wh'}\kast_{\wh',\bl} \left( \iint_{C(I_{\wh}) \prec C(I_{\bl})} \omega_{\bl,\wh} \right) = \iint_{C(I_{\wh}) \prec C(J)} \left( \sum_{\bl \sim \wh'}\kast_{\wh',\bl}  \omega_{\bl,\wh} \right)=0
\end{equation}
by Lemma~\ref{lem:ker_K} applied to $\x=\f_0$ so that $\frac{1}{g_\bl(\zeta)} = g_{\bl,\f_0}(\zeta)$. Moreover, by~\eqref{eq:intersection_concrete} and Theorem~\ref{thm:whole_plane_inverse_kast}, 
\[
-\sum_{\bl \sim \wh'}\kast_{\wh',\bl} \sign(v) \langle \partial I_\bl \cap I_\wh,\mathsf A_{\bl,\wh} \rangle = \sign(v)\big(\mathbf 1_{\{\ell \in I_{\wh}\}}
-\mathbf 1_{\{r \in I_{\wh}\}}\big)\mathbf{1}_{\{\wh = \wh' \}}.
\]
In order for the right-hand side to be nonzero, we need $\wh=\wh'$. In that case, we show that
\begin{equation} \label{eq:sing_cancels}
    \mathbf 1_{\{\ell \in I_{\wh}\}}
-\mathbf 1_{\{r \in I_{\wh}\}}
=\sign(v).
\end{equation}

Indeed, if $\sign(v)=+1$, then
$J$ (resp. $I_\wh$) is the interval immediately counterclockwise (resp. clockwise)
of $s_v(\Rin_\bl)=s_v(\Rin_\wh)$, so $\ell \in I_\wh$ and $r \notin I_\wh$, hence, the left-hand side equals $1$. If instead $\sign(v)=-1$, then $J$ (resp. $I_\wh$) is the interval immediately clockwise (resp. counterclockwise) of $s_v(\Rin_\bl)$, so $r \in I_\wh$ and $\ell \notin I_\wh$, hence, the left-hand side equals $-1$.

Now we consider the general case where the intervals $I_\bl$ may depend on $\bl \sim \wh'$. We first claim that we can replace $I_\bl$ for all $\bl \sim \wh'$ by the single uniform choice of interval 
\[
 J:=\bigcup_{\bl \sim \wh'} I_\bl 
\] 
without changing the value of~\eqref{eq:formula_inverse_K}.

Let $I$ denote the $-$ sign interval for $\sign_{{\graphregion}_{\wh'}}$ that is adjacent to $s_v({\graphregion}_{\wh'})$. Let $e_-$ (resp. $e_+$) denote the edge immediately clockwise (resp. counterclockwise) of $I$.

\begin{lemma} \label{lem:explicit_Ib}
For every $\bl \sim \wh'$, the interval $I_\bl$ is obtained from $I$
by adjoining a (possibly empty) subset of $\{e_-,e_+\}$.
\end{lemma}

\begin{proof}
We have two cases.
\begin{itemize}
\item {$\bl$ is an inner vertex}:
As we move from ${\graphregion}_{\wh'}$ to ${\graphregion}_\bl$, we cross exactly the two strands
passing through the edge $\bl\wh'$. If one of the strands is $\beta_e$ for some $e \in E(N)$, then $\sign{R_{\wh'}}(e) = +$ changes to $\sign{R_{\bl}}(e) = -$. In particular, if we cross $\beta_e$ for $e \in \{e_-,e_+\}$, $I_\bl$ is obtained from $I$ by shifting the corresponding endpoint outwards. 

We claim that an endpoint cannot shift by more than one step.
If it did, then the two crossed strands would correspond to consecutive
edges $e$ and $e'$ along $\partial N$.
In that case the edge $\bl\wh'$ would lie on both the zig-zag
paths $\beta_e$ and $\beta_{e'}$, hence would lie on
$\partial \graphbeta$. This contradicts the assumption that both $\wh'$ and $\bl$ are inner. Therefore, each endpoint moves by at most one step.

\item {$\bl$ is an outer vertex}:
Since $\wh'$ is inner by assumption and $\bl \sim \wh'$, the edge $\bl\wh'$ lies on
at least one strand $\beta_e$ for which $\bl$ is $\beta_e$-outer.
We may use such an $e$ to define $\Rin_\bl$. Then, moving from ${\graphregion}_{\wh'}$ to $\Rin_\bl$ crosses exactly one strand, namely
the strand through $\bl\wh'$ other than $\beta_e$. Hence, at most one endpoint of $I$ shifts and by at most one step.
\end{itemize}
\end{proof}

We claim that every such adjoining step is allowed by Item~2 of
Proposition~\ref{prop:invariance}.

\begin{lemma} \label{lem:endpoint_shift_occurs}
Let $\bl \sim \wh'$ and let $e \in \{e_-,e_+\}$. If $\bl\wh'$ lies on $\beta_e$ and $e$ is white, then $e \subset I_\bl$.
\end{lemma}

\begin{proof}
Since $\bl$ lies on $\beta_e$, $\sign_{{\graphregion}_\bl}(e) = -$. Since $e$ is white, $\bl$ is not $\beta_e$-outer, so $\sign_{\Rin_\bl}(e) = \sign_{{\graphregion}_\bl}(e) = -$. On the other hand, $e$ is one of the two boundary edges adjacent to the interval
$I$. By Lemma~\ref{lem:explicit_Ib}, the only way to obtain $I_\bl$ from $I$ is
by adjoining some subset of $\{e_-,e_+\}$. Since $\sign_{\Rin_\bl}(e)=-$, the edge $e$ must be adjoined. Therefore, $e \subset I_\bl$.
\end{proof}

Let $e\in\{e_-,e_+\}$ with $e\subset J\setminus I_\bl$. Since $e$ is adjacent to the white sign interval $I$ for $\sign_{{\graphregion}_{\wh'}}$, we have $\sign_{{\graphregion}_{\wh'}}(e)=+$.
If $\bl$ does not lie on $\beta_e$, then passing from ${\graphregion}_{\wh'}$ to ${\graphregion}_\bl$
does not cross $\beta_e$, so
\[
\sign_{{\graphregion}_\bl}(e)=\sign_{{\graphregion}_{\wh'}}(e)=+,
\]
and the first hypothesis of Item~2 of Proposition~\ref{prop:invariance} applies. If instead $\bl$ lies on
$\beta_e$ and $e$ is black, then the second hypothesis of Item~2 of Proposition~\ref{prop:invariance} applies.
Finally, if $\bl$ lies on $\beta_e$ and $e$ is white, then
Lemma~\ref{lem:endpoint_shift_occurs} implies that $e\subset I_\bl$ already,
contrary to $e\subset J\setminus I_\bl$. Thus, this case cannot occur.

Hence, every edge needed to enlarge $I_\bl$ to $J$ can be adjoined using
Item~2 of Proposition~\ref{prop:invariance}, and therefore, we may replace
every $I_\bl$ by the common interval $J$ without changing the value of
$\kast^{-1}_{\bl,\wh}$. 

To complete the proof for inner $\wh'$, it remains to verify~\eqref{eq:sing_cancels} when $\wh = \wh'$. If $\sign(v)=+1$, then $I$ (resp. $I_\wh$) is the interval immediately counterclockwise (resp. clockwise) of $s_v(\Rin_\wh)$. 
Regardless of whether the endpoints of $J = [\ell,r]$ 
are the same as $I$ or pushed out one step, $\ell \in I_\wh$ and $r$ is contained in the other black sign interval, so $r \notin I_\wh$ and therefore, $\mathbf 1_{\{\ell \in I_{\wh}\}}
-\mathbf 1_{\{r \in I_{\wh}\}} =1$. The case when $\sign(v)=-1$ is analogous.

\begin{figure}
    \centering
\begin{tikzpicture}

        \begin{scope}[shift={(7,0)}]
    \filldraw[fill=gray!30, draw=none]
    (90:1.5) -- (150:1.5) -- (210:1.5) -- (270:1.5) -- (90:1.5);

    \coordinate[wvert, label = right:${\wh'}$] (w) at (0,0) {};
 
    \coordinate[bvert, label=above:$\bl_1$] (b1) at (90:1.5) {};
    \coordinate[bvert, label=left:$\bl_2$] (b2) at (150:1.5) {};
       \coordinate[bvert, label=left:$\bl_3$] (b3) at (210:1.5) {};
    \coordinate[bvert, label=below:$\bl_4$] (b4) at (270:1.5) {};

       \coordinate[bvert] (b5) at (30:1.5) {};
    \coordinate[bvert] (b6) at (-30:1.5) {};
    \draw[] (w) edge (b1) edge (b2)  edge (b3) edge (b4);
    \draw[dashed] (w) edge (b5) edge (b6);

\draw[hollowedge]
  
  ($ (90:0.75) - (150:0.75) + (90:0.75) $) -- (150:0.75) ; 
    \draw[hollowedge] (150:0.75) -- (210:0.75);
    \draw[hollowedge] (210:0.75) --  ($ (270:0.75) + (270:0.75) - (210:0.75) $);

    \coordinate[label=right:$\beta_{e_1}$] (no) at  ($ (90:0.75) - (150:0.75) + (90:0.75) $) {};

    \coordinate[label=right:$\beta_{e_3}$] (no) at  ($ (270:0.75) + (270:0.75) - (210:0.75) $) {};

    \node[] (no) at (-1.1,0) {${\beta_{e_2}}$};
  
    \end{scope}
\end{tikzpicture}
    \caption{The labeling of black vertices in a neighborhood of an outer white vertex $\wh'$ when $k=3$. The shaded region represents faces of $\graphbeta$. Solid edges belong to $E(\graphbeta)$ while dashed edges do not. 
    }
    \label{fig:nhd_outer_white}
\end{figure}
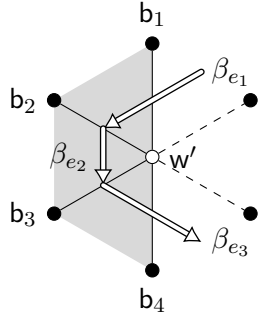

\item $\wh'$ is an outer vertex:
By Condition~2 of Definition~\ref{def:azgraph}, the edges
\(e\in E(N)\) for which \(\wh'\) is \(\beta_e\)-outer form a consecutive block
\[
e_1 < e_2 < \cdots < e_k
\]
along \(\partial N\). Let \(v\) be the nearest convex vertex to this block such that \(\sign(v)=1\), and set
\begin{equation}\label{eq:J}
   J:=\bigcup_{i=1}^k e_i. 
\end{equation}
Let \(\bl_1,\dots,\bl_{k+1}\) be the black vertices adjacent to \(\wh'\) that lie on at least one of the strands \(\beta_{e_i}\), \(1\le i\le k\), listed in counterclockwise order around \(\wh'\), as in Figure~\ref{fig:nhd_outer_white}. By considering the inconsistently oriented cones, exactly as in the proof of Lemma~\ref{lem:sign_changes_sections} and in Figure~\ref{fig:fig_opp_cones}, we obtain
\begin{equation}\label{eq:intervalsa}
I_{\bl_1}=e_1,\quad
I_{\bl_2}=e_1\cup e_2,\quad
\dots,\quad
I_{\bl_k}=e_{k-1}\cup e_k,\quad
I_{\bl_{k+1}}=e_k.
\end{equation}

Equivalently, one can argue as follows. Passing from \(\wh'\) to \(\bl_i\), the number of sign changes increases from two to four by Lemma~\ref{lem:four_signs_locally}. The only signs that change are those on the intervals appearing in~\eqref{eq:intervalsa}, and they change from \(+\) to \(-\). These intervals must therefore be precisely the negative sign intervals.

Now suppose that, in enlarging \(I_{\bl_i}\) to \(J\), one has to adjoin an edge \(e_j\). Figure~\ref{fig:nhd_outer_white} shows that \(\bl_i\) lies to the left of \(\beta_{e_j}\), hence, \(\sign_{{\graphregion}_{\bl_i}}(e_j)=+\). We may therefore apply Case~1 of Item~2 in Proposition~\ref{prop:invariance}.

To use the same calculation as in~\eqref{eq:kast_cancel}, we extend the definition of $\kast^{-1}_{\bl,\wh}$ to the black vertices $\bl \sim \wh'$ such that $\bl \in B(\graphpl) \setminus B(\graphbeta)$ (the black vertices joined to \(\wh'\) by dashed edges in Figure~\ref{fig:nhd_outer_white}) by setting
\begin{equation} \label{eq:bdy_k_inv}
\kast^{-1}_{\bl,\wh}
:= \sign(v)\left(
\iint_{C(J)\prec C(I_{\wh})}\omega_{\bl,\wh}
+
\langle \partial I_\wh \cap J, \mathsf A_{\bl,\wh}\rangle
\right).
\end{equation}
We claim that~\eqref{eq:bdy_k_inv} vanishes.

Indeed, since $\bl$ lies to the left of each strand $\beta_{e_i}$, the sign function
$\sign_{{\graphregion}_\bl}$ is equal to $+$ on $J$. Hence, by Lemma~\ref{lem:sign_poles},
the contour $C(J)$ encloses no poles of $\omega_{\bl,\wh}$, and therefore,
\[
\iint_{C(J)\prec C(I_{\wh})}\omega_{\bl,\wh}=0.
\]

For the remaining term, we need to know that the sign stays positive also on the two boundary edges adjacent to $J$.

\begin{lemma}\label{lem:sign_plus_both_sides} Assume $\bl \sim \wh'$ such that $\bl \in B(\graphpl) \setminus B(\graphbeta)$. Let $e_-$ (resp. $e_+$) denote the edge immediately clockwise (resp. counterclockwise) of $J$. Then
\[
\sign_{{\graphregion}_\bl}(e_-)=\sign_{{\graphregion}_\bl}(e_+)=+.
\]
\end{lemma}

\begin{proof}
As we move from $\bl_1$ to $\wh'$ to $\bl$, Condition~2 of
Definition~\ref{def:azgraph} implies that the only boundary strand crossed is
$\beta_{e_1}$. Therefore,
\[
\sign_{{\graphregion}_\bl}(e_-)=\sign_{{\graphregion}_{\bl_1}}(e_-)=+.
\]
Applying the same argument to $\bl_{k+1}$ gives
\[
\sign_{{\graphregion}_\bl}(e_+)=\sign_{{\graphregion}_{\bl_{k+1}}}(e_+)=+.
\]
\end{proof}
Write \(I_\wh=[\ell_\wh,r_\wh]\). Then,
\[
- \langle \partial I_\wh \cap J, \mathsf A_{\bl,\wh}\rangle
=
\mathbf 1_{\{\ell_\wh \in J\}} \mathsf A_{\bl,\wh}^{\ell_\wh}
-
\mathbf 1_{\{r_\wh \in J\}} \mathsf A_{\bl,\wh}^{r_\wh}.
\]
We show that both terms on the right-hand side vanish. Consider first the term
\(\mathbf 1_{\{\ell_\wh \in J\}} \mathsf A_{\bl,\wh}^{\ell_\wh}\), and assume that \(\ell_\wh\in J\). Let \(e'\) be the edge immediately clockwise from \(I_\wh\). Since \(e'\) is adjacent to \(I_\wh\), we have
\(
\sign_{{\graphregion}_\wh}(e')=-.
\)
On the other hand, \(\sign_{{\graphregion}_\bl}\) is positive on \(J\) and on the two edges adjacent to \(J\), so
\(
\sign_{{\graphregion}_\bl}(e')=+.
\)
It follows from Case~1 of Lemma~\ref{lem:sector} that \(\ell_\wh\in \sector_{\bl,\wh}\), and therefore,
\[
\mathbf 1_{\{\ell_\wh \in J\}} \mathsf A_{\bl,\wh}^{\ell_\wh}=0.
\]
A similar argument shows that 
\[
\mathbf 1_{\{r_\wh \in J \} } {\mathsf A}^{r_\wh}_{\bl,\wh} =0.
\]
  
We have thus shown that both terms in~\eqref{eq:bdy_k_inv} are $0$. Therefore,
\[
\sum_{ \bl \wh' \in E(\graphbeta)}\kast_{\wh',\bl} \kast^{-1}_{\bl,\wh} = \sum_{ \bl \wh' \in E(\graphpl)}\kast_{\wh',\bl} \kast^{-1}_{\bl,\wh} = \sign(v)\big(\mathbf 1_{\{\ell \in I_{\wh}\}}
-\mathbf 1_{\{r\in I_{\wh}\}}\big)\mathbf{1}_{\{\wh = \wh' \}}.
\]
as in~\eqref{eq:kast_cancel}. 

Finally, we need to verify~\eqref{eq:sing_cancels} when $\wh = \wh'$. Recall the definition of $v$ above~\eqref{eq:J}. Suppose we use $e_1$ to define $\Rin_{\wh'}$, so that 
\[
\Rin_{\wh'} = {\graphregion}_{\bl_1}.
\]
Then $I_{\wh}$ and $I_{\bl_1} = e_1$ are adjacent sign intervals that meet at the vertex $s_v({\graphregion}_{\bl}) = \ell$, so 
\[\mathbf 1_{\{\ell \in I_{\wh}\}}
-\mathbf 1_{\{r\in I_{\wh}\}}=1 = \sign(v).\]
\end{enumerate}
The proof of Theorem~\ref{thm:main} is complete.
\qed

\section{Scaling limits and astroidal domains} \label{sec:scaling_limit_astroidal_domain}

We now turn to the asymptotic consequences of the exact inverse Kasteleyn formula~\eqref{eq:formula_inverse_K}. Throughout the remainder of the paper, we assume that the angle function is periodic as in Section~\ref{sec:slope}.

\subsection{Astroidal domains}\label{sec:def_Astroidal_domains}

In this section, we consider scaling limits of AZ graphs. Recall from Section~\ref{sec:conventions} that \(|e|_{\ZZ}\) is the lattice length of $e \in E(N)$ and \((a_e,b_e)=\vec e/|e|_{\ZZ}\in\ZZ^2\). Let
\[
c:E(N) \to \RR
\]
be a function. For each \(e \in E(N)\), consider the line
\begin{equation}\label{eq:line_l_e}
    \line_e :=\{(x,y) \in \RR^2: -b_e x + a_e y + c(e) =0\}.
\end{equation}
For a vertex \(v=e_- \cap e_+ \in V(N)\), let
\[
\point_v := \line_{e_-} \cap \line_{e_+}
\]
denote the point of intersection.

We first define a cycle in \(\RR^2\) analogous to the medial boundary cycle construction. Let \(\partial \mathcal D_c\) be the counterclockwise-oriented cycle obtained as follows: it visits the points \(\point_v\) in the clockwise cyclic order of \(v \in V(N)\), and for each edge \(e=[v_-,v_+] \in E(N)\), it goes from \(\point_{v_+}\) to \(\point_{v_-}\) along \(\line_e\). We do not assume that \(\point_{v_+}\neq \point_{v_-}\).

Assume that \(\partial \mathcal D_c\) is a simple cycle, and let \(\mathcal D_c \subset \RR^2\) be the closed set bounded by \(\partial \mathcal D_c\). We say that $\mathcal D_c$ is \emph{admissible} if 
\begin{equation} \label{eq:sum_c_zero}
    \sum_{e \in E(N)} |e|_\ZZ c(e)=0.
\end{equation}

\begin{definition}\label{def:astroidal_domain}
    We call $\mathcal D_c$ an \emph{astroidal domain} if:
    \begin{enumerate}
        \item the cycle $\partial \mathcal D_c$ is simple,
        \item $\mathcal D_c$ is admissible.
    \end{enumerate}
\end{definition}
Astroidal domains naturally arise as scaling limits of AZ graphs. Let \(
\graphbetan, n=1,2,\dots,
\) be a family of AZ graphs such that for each \(e \in E(N)\),
\begin{equation}\label{eq:scaling_limit_AZ}
c_{\bm \beta^n}(e) = n |e|_\ZZ c(e) + O(1) \qquad \text{as $n \to \infty$}.
\end{equation}
By~Lemma~\ref{lem:c_admissibility}, $\mathcal D_c$ is admissible. If $\partial \mathcal D_c$ is also simple, then $\mathcal D_c$ is an astroidal domain.
\begin{remark}\label{rem:existence_astroidal_domain}
{
For every Newton polygon \(N\), one can construct an astroidal domain. Choose two consecutive edges \(e_1,e_2\in E(N)\), and write
\[
    e_1=[v_0,v_1], \qquad e_2=[v_1,v_2],
\]
where \(v_0,v_1,v_2\in V(N)\) occur consecutively in the cyclic order on
\(\partial N\). We first choose \(c:E(N)\to\RR\) so that
\(\partial\mathcal D_c\) is simple and so that the three vertices of \(\partial\mathcal D_c\) corresponding to \(v_0,v_1,v_2\) are three consecutive convex vertices. To arrange this, first place the oriented lines \(\line_{e_1}\) and
\(\line_{e_2}\). Among the four cones determined by these two oriented lines at
\(
    \point_{v_1}=\line_{e_1}\cap \line_{e_2},
\)
choose one of the inconsistently oriented cones. Then choose the
remaining lines \(\line_e\), \(e\neq e_1,e_2\), in their cyclic order,
so that the rest of \(\partial\mathcal D_c\) is a simple polygonal
chain contained in this cone.
}
{
We color the sides as in the definition of AZ graphs: an edge
\(e\in E(N)\) is black if the orientation of \(\line_e\) induced by
\(e\) agrees with the orientation of the corresponding side of
\(\partial\mathcal D_c\), and white otherwise. The color changes occur precisely at the convex vertices of \(\mathcal D_c\). Hence, the two sides corresponding to \(e_1\) and \(e_2\) have opposite colors. Moreover, either of them can be translated arbitrarily far outward while keeping \(\partial\mathcal D_c\) simple. Replacing \(c(e)\) by \(c(e)+t\) for a black side, or by
\(c(e)-t\) for a white side, translates that side outward. Thus, by translating $e_1$ and $e_2$ outward, we can make the domain admissible.
}

{
Astroidal domains can also be obtained from the tropical limit of the
Aztec diamond, as in Appendix~\ref{appendix:tropical}. 
}

{
Once
an astroidal domain \(\mathcal D_c\) is fixed, one can choose boundary
zig-zag paths so as to obtain a sequence of AZ graphs \(\graphbetan\)
satisfying~\eqref{eq:scaling_limit_AZ}.
}

{
We expect the inverse Kasteleyn formula and the asymptotic analysis
developed in this paper to extend to more general domains, including domains whose boundary cycle is not simple. For this reason, we do not pursue a more systematic construction of astroidal domains here beyond the above argument showing that the class is nonempty.
}
\end{remark}

\subsection{Cell decomposition and sign changes}

The collection of lines $\{\line_e : e \in E(N)\}$ determines a cell decomposition of $\RR^2$; see Figure~\ref{fig:decagon_arctic_curve}. To each (0-, 1- or 2-) cell $R$, we associate a \emph{sign function}
\[
\sign_R : E(N) \rightarrow \{+,-,0\},
\]
defined exactly as in~\eqref{eq:local_sign}, except that we use the oriented line $\line_e$ (oriented in the direction of $\vec e$) in place of the strand $\beta_e$. Explicitly, for $e \in E(N)$,
\[
\sign_R(e)=
\begin{cases}
+, & \text{if $R$ lies on the left-hand side of $\line_e$,} \\
-, & \text{if $R$ lies on the right-hand side of $\line_e$,} \\
0, & \text{if $R \subset \line_e$.}
\end{cases}
\]
The \emph{chambers} are precisely those cells for which $\sign_R(e)\neq 0$ for all $e \in E(N)$. We say that a cell $R$ has a \emph{sign change} 
\begin{enumerate}
\item at a vertex \(v \in V(N)\) if \(\sign_R\) takes opposite values on the two edges of \(\partial N\) incident to \(v\).
\item along a closed interval \([v_-,v_+]\subset \partial N\) with $v_-,v_+ \in V(N)$ if \(\sign_R=0\) on \([v_-,v_+]\) and \(\sign_R\) takes opposite values on the two adjacent boundary intervals.
\end{enumerate}

With this convention, all definitions and results of Section~\ref{sec:chambers_and_sign_changes} admit direct continuous analogues. In particular, we have the following continuous analogue of Lemma~\ref{lem:four_signs_locally}.

\begin{lemma}\label{lem:sign_R_continuous}
For each cell $R \subset \mathcal D_c^\circ$ (resp. $R  \subset \RR^2 \setminus \mathcal D_c^\circ$), $\sign_R$ has four (resp. two) sign changes. 
\end{lemma}

\subsection{Characterization of astroidal domains}\label{sec:char_astroidal_domains}

The limit shape offers a natural perspective on the astroidal domains, which we discuss in this section. Throughout this discussion, we assume that the points~$\point_v$ are pairwise distinct for all~$v\in V(N)$. 

Let~$N^{\perp}$ be the polygon obtained (up to translation) by rotating~$N$ by~$-\frac{\pi}{2}$. For any~$v\in V(N)$ we denote the corresponding vertex of~$N^\perp$ by~$(s^v,t^v)\in V(N^\perp)$.
\begin{proposition}
Let~$\partial \mathcal D\subset \RR^2$ be a simple polygon and assume there is an order-preserving bijection~$v\mapsto \point_v$ from~$V(N)$ to~$V(\partial \mathcal D)$. Assume further that there is a continuous piecewise linear function~$\ls_0:\partial \mathcal D\to \RR$ such that, for all~$v\in V(N)$,
\begin{equation}\label{eq:height_boundary}
\ls_0(x,y)=(x,y)\cdot(s^v,t^v)-d_v
\end{equation}
along the two line segments of $\partial \mathcal D$ adjacent to~$\point_v$, for some constant~$d_v$. Then the enclosed region~$\mathcal D$ is an astroidal domain~$\domain$ with
\begin{equation}
c(e)=\frac{d_{v_+}-d_{v_-}}{|e|_{\ZZ}}
\end{equation}
for~$e=[v_-,v_+]\in E(N)$.
\end{proposition}
\begin{remark}
We do not assume that the edges of~$\partial \mathcal D$ are parallel to the edges of $N$, nor that the map~$v \mapsto \point_v$ reverses orientation. Both properties follow from~\eqref{eq:height_boundary}.
\end{remark}
\begin{remark}
In Section~\ref{sec:limit_shape} below, we obtain the limit shape for astroidal domains, and that limit shape indeed satisfies the boundary condition~\eqref{eq:height_boundary}. 
\end{remark}
\begin{proof}
Recall the conventions from Section~\ref{sec:conventions}. We divide the proof into three steps. First, we show that each edge of~$\partial \mathcal D$ is parallel to its corresponding edge of $N$. Next, we show that the orientation of~$\partial \mathcal D$ is reversed compared to that of~$\partial N$. Finally, we verify admissibility.
\begin{enumerate}
\item Let~$e=[v_-,v_+] \in E(N)$ and let~$\line_e$ be the line through~$\point_{v_-}$ and~$\point_{v_+}$. For any~$(x,y)\in \line_e$, \eqref{eq:height_boundary} holds with~$v=v_-$ and with~$v=v_+$. Subtracting the two expressions, we find that the equation for the line~$\line_e$ is 
\begin{equation}
(x,y)\cdot(s^{v_+}-s^{v_-},t^{v_+}-t^{v_-})-d_{v_+}+d_{v_-}=0.
\end{equation}
By the definition of~$(s^v,t^v)\in V(N^\perp)$,
\begin{equation}
(s^{v_+}-s^{v_-},t^{v_+}-t^{v_+})=|e|_{\ZZ}(b_e,-a_e),
\end{equation}
and therefore,
\begin{equation}\label{eq:line_general_domain}
(x,y)\cdot|e|_{\ZZ}(b_e,-a_e)-d_{v_+}+d_{v_-}=0.
\end{equation}
In particular,~$\vec e=|e|_{\ZZ}(a_e,b_e)$ and~$\line_e$ are parallel. 
\item For any~$e=[v_-,v_+]\in E(N)$, let
\begin{equation}
\vec{u}_e:=(s^{v_+}-s^{v_-},t^{v_+}-t^{v_-}), \quad \text{and} \quad \vec{w}_e:=\point_{v_+}-\point_{v_-}.
\end{equation}
For any~$v=e_-\cap e_+\in V(N)$, let~$\theta_v {\in (-\pi,\pi)}$ be the angle between~$\vec u_{e_-}$ and~$\vec u_{e_+}$ and let~$\varphi_v {\in (-\pi,\pi)}$ be the angle between~$\vec w_{e_-}$ and~$\vec w_{e_+}$. {Since $N$ is convex and oriented counterclockwise, $\theta_v\in (0,\pi)$. It follows, using~$\vec{u}_e \perp \vec{w}_e$, that}
\begin{equation}
\varphi_v=\theta_v \quad \text{or} \quad \varphi_v=\theta_v-\pi.
\end{equation}
We say that~$\vec w_e$ is \emph{reversed} at~$v$ if~$\varphi_v=\theta_v-\pi$. If~$\vec w_e$ is reversed at~$v$, it means that~$\vec w_{e}$ changes from being an inward pointing normal to an outward pointing normal of~$N^{\perp}$, or the other way around. 

Since~$\partial N$ is simple and oriented counterclockwise,
\begin{equation}
\sum_{v\in V(N)}\theta_v=2\pi.
\end{equation}
So, by summing over the angles~$\varphi_v$, we get
\begin{equation}
\sum_{v\in V(N)}\theta_v-\pi \#\{v\in V(N):\vec w_e\text{ is reversed at $v$}\}=2\pi-\pi \#\{v\in V(N):\vec w_e\text{ is reversed at $v$}\}.
\end{equation}
Since~$\partial D$ is simple, the above sum is equal to~$2\pi$ or~$-2\pi$ according to whether~$\partial D$ is oriented counterclockwise or clockwise. We therefore only need to show that~$\vec w_e$ is reversed at least once. 

By the assumption~\eqref{eq:height_boundary},
\begin{equation}
\ls_0(\point_{v_+})-\ls_0(\point_{v_-})=\vec w_e \cdot (s^{v_+},t^{v_+}).
\end{equation}
Summing over all~$e\in E(N)$ yields
\begin{equation}
\sum_{e=[v_-,v_+]\in E(N)}\vec w_e \cdot (s^{v_+},t^{v_+})=0.
\end{equation}
Now fix~$(s,t)\in N^{\perp}$. Since~$\sum_{e\in E(N)}\vec w_e=0$, we have
\begin{equation}
\sum_{e=[v_-,v_+]\in E(N)}\vec w_e \cdot (s^{v_+}-s,t^{v_+}-t)=0.
\end{equation}
Since $N^\perp$ is convex, the quantity 
\begin{equation}
\vec w_e \cdot (s^{v_+}-s,t^{v_+}-t)
\end{equation}
changes sign at~$v\in V(N)$ only if~$\vec w_e$ is reversed at $v$. Hence,~$\vec w_e$ is reversed at least once and so $\partial \mathcal D$ is oriented opposite to $\partial N$.
\item Let~$e=[v_-,v_+]$. With~$c(e)=(d_{v_+}-d_{v_-})/|e|_{\ZZ}$ as in the statement, the line~$\line_e$ given by the equation~\eqref{eq:line_general_domain} coincides with the corresponding line given before Definition \ref{def:astroidal_domain}. The domain is admissible since
\begin{equation}
\sum_{e\in E(N)}|e|_{\ZZ}c(e)=\sum_{e=[v_-,v_+]\in E(N)}(d_{v_+}-d_{v_-})=0.
\end{equation}
\end{enumerate}
\end{proof}

\begin{remark}
The previous proposition shows that, in the language of~\cite{ADPZ20}, astroidal domains are natural domains with natural boundary values.
\end{remark}

\section{Phase regions and the arctic curve}\label{sec:phase_regions}

In this section, we study the geometry of the phase regions and the arctic curve, building on the methods developed in{~\cite{Ber21, BB23, BB24, CJ16, DK21} in the context of dimer models. See also \cite{Kri13} for similar arguments in a different context.} {Recall that we assume that the angle function is periodic as in Section~\ref{sec:slope}.}

\subsection{Action \(1\)-form and action function}\label{sec:action_one_form}

Fix an astroidal domain $\mathcal D_c$ and a sequence of AZ graphs
\(
\graphbetan, n = 1,2,\dots,
\)
such that
\[
c_{\bm\beta^n}(e)=n|e|_\ZZ c(e)+O(1),
\qquad e\in E(N).
\]
Recall the notion of imaginary-normalized differentials from Section~\ref{sec:ind}. We also recall the (multi-valued) functions
\[
z(\zeta):=  \prod_{\alpha \in \zz_\torus} E(\alpha,\zeta)^{-b_{e(\alpha)}},
\qquad
w(\zeta):= \prod_{\alpha \in \zz_\torus} E(\alpha,\zeta)^{a_{e(\alpha)}},
\]
introduced in~\eqref{eq:z_and_w}. By Lemma~\ref{lem:third_kind_ind}, the logarithmic differentials \(d\log z\) and \(d\log w\) are
imaginary-normalized meromorphic \(1\)-forms with residues
\[
\res_\alpha(d\log z)=-b_{e(\alpha)},
\qquad
\res_\alpha(d\log w)=a_{e(\alpha)}.
\]

\begin{definition}\label{def:lambda_c}
For \((x,y)\in\RR^2\), we define the \emph{action \(1\)-form} \(\lambda_c(x,y)\) to be the unique
imaginary-normalized meromorphic \(1\)-form such that
\[
\res_\alpha\bigl(\lambda_c(x,y)\bigr)
=
-b_{e(\alpha)}x+a_{e(\alpha)}y+c(e(\alpha))
\qquad\text{for all }\alpha\in\zz.
\]
\end{definition}

Let
\begin{equation}\label{eq:div_c}
\bm D_c:=\sum_{\alpha\in\zz} c(e(\alpha)) \alpha.
\end{equation}
Then the differential
\[
x d\log z+y d\log w+\omega_{\bm D_c}
\]
is imaginary-normalized by Lemma~\ref{lem:third_kind_ind} and it has residue at \(\alpha\) equal to
\[
-b_{e(\alpha)}x+a_{e(\alpha)}y+c(e(\alpha)).
\]
Hence, by Lemma~\ref{lem:unique_ind},
\begin{equation}\label{eq:action_zw}
\lambda_c(x,y)=x d\log z+y d\log w+\omega_{\bm D_c}.
\end{equation}

The reason for introducing \(\lambda_c(x,y)\) is that it governs the leading exponential factor in the inverse Kasteleyn matrix formula~\eqref{eq:formula_inverse_K}. To make this explicit, we introduce its multi-valued primitive.

\begin{definition}\label{def:action_function}
The \emph{action function} is the multi-valued function on \(\curve\) defined by
\begin{equation}\label{eq:action_function}
\Lambda_c(\zeta;x,y)
:=
\sum_{\alpha\in\zz} \bigl(-b_{e(\alpha)} x+a_{e(\alpha)} y+c(e(\alpha))\bigr)\log E(\alpha,\zeta).
\end{equation}
\end{definition}

By construction,
\[
d\Lambda_c(\cdot;x,y)=\lambda_c(x,y).
\]

\begin{lemma}\label{lem:divisor_coefficients_asymptotic}
Let \((x,y)\in\RR^2\) and let \((j_n,k_n)\in\ZZ^2\) be such that
\[
(j_n,k_n)=n(x,y)+O(1).
\]
For every \(\alpha\in\zz\),
\begin{align*}
\frac1n\Bigl(-\dabel(\f_0+(j_n,k_n))-\bm D_{\bm\beta^n}\Bigr)_\alpha
&=
-b_{e(\alpha)}x+a_{e(\alpha)}y+c(e(\alpha))
+O\left(\frac1n\right)\\
&=\res_\alpha\bigl(\lambda_c(x,y)\bigr)+O\left(\frac1n\right).
\end{align*}
\end{lemma}

\begin{proof}
Lemma~\ref{lem:c_admissibility} implies that
\[
(\bm D_{\bm\beta^n})_\alpha
=
-\frac{c_{\bm\beta^n}(e(\alpha))}{|e(\alpha)|_\ZZ}+O(1).
\]
On the other hand,
\[
\dabel(\x+(j,k))
=
\dabel(\x)+\sum_{\alpha\in\zz}\bigl(b_{e(\alpha)}j-a_{e(\alpha)}k\bigr)\alpha,
\]
so
\[
\Bigl(-\dabel(\f_0+(j_n,k_n))-\bm D_{\bm\beta^n}\Bigr)_\alpha
=
-b_{e(\alpha)}j_n+a_{e(\alpha)}k_n
+\frac{c_{\bm\beta^n}(e(\alpha))}{|e(\alpha)|_\ZZ}
+O(1).
\]
Using
\[
(j_n,k_n)=n(x,y)+O(1),
\qquad
c_{\bm\beta^n}(e)=n|e|_\ZZ c(e)+O(1),
\]
we obtain
\[
\frac1n\Bigl(-\dabel(\f_0+(j_n,k_n))-\bm D_{\bm\beta^n}\Bigr)_\alpha
=
-b_{e(\alpha)}x+a_{e(\alpha)}y+c(e(\alpha))
+O\left(\frac1n\right).
\]
\end{proof}

The previous lemma identifies the asymptotic exponent of the prime-form factor in the kernel~\eqref{eq:kernel}. More precisely, if \(\bl_n\) and \(\wh_n\) are vertices whose macroscopic location is \((x,y)\), then
\[
(\dabel(\bl_n)+\bm D_{\bm\beta^n})_\alpha
=
n \res_\alpha(\lambda_c(x,y))+O(1),
\qquad
(\dabel(\wh_n)+\bm D_{\bm\beta^n})_\alpha
=
n \res_\alpha(\lambda_c(x,y))+O(1),
\]
and therefore,
\begin{equation}\label{eq:prime_form_action_asymptotic}
\frac{
\prod_{\alpha\in\zz}E(\alpha,\eta)^{(\dabel(\wh_n)+\bm D_{\bm\beta^n})_\alpha}
}{
\prod_{\alpha\in\zz}E(\alpha,\zeta)^{(\dabel(\bl_n)+\bm D_{\bm\beta^n})_\alpha}
}
=
\exp\Bigl(
n\bigl(\Lambda_c(\eta;x,y)-\Lambda_c(\zeta;x,y)\bigr)
+O(1)\Bigr).
\end{equation}
In other words, \(\Lambda_c\) is the action governing the exponential part of the kernel.

\subsection{Phase regions and boundary strata of the limit shape}

The continuous analogue of Lemma~\ref{lem:sign_poles} is:

\begin{lemma} \label{lem:sign_poles_cts}
Let $e \in E(N)$ and let $R$ be a cell. Then, for all $\alpha \in \zz$ and $(x,y) \in R$,
    \[
      \sign(\res_{\alpha}(\lambda_c(x,y)))  =\sign_R(e(\alpha)).
    \]
    
\end{lemma}

\begin{proof}
By~Definition~\ref{def:lambda_c}, for every $\alpha\in\zz_e$ we have
\[
\res_\alpha(\lambda_c(x,y))=-b_{e(\alpha)} x+a_{e(\alpha)} y+c(e(\alpha)),
\]
which is the defining linear function of \(\line_{e(\alpha)}\). Since \(\line_e\) is oriented along \((a_{e(\alpha)},b_{e(\alpha)})\), this linear function is positive on the left-hand side of \(\line_{e(\alpha)}\), negative on the right-hand side, and zero on \(\line_{e(\alpha)}\). Therefore,
\[
\sign\bigl(\res_\alpha(\lambda_c(x,y))\bigr)=\sign_R({e(\alpha)})
\]
for all \((x,y)\in R\).
\end{proof}

Recall that we always count zeroes with multiplicity.

\begin{lemma}\label{lem:forced_zeroes_and_remaining}
Let \((x,y)\in \RR^2\) and let \(R\) be the cell containing \((x,y)\). Then: 
\begin{itemize}
   \item For each \(1\le j\le g\), the number of zeroes of \(\lambda_c(x,y)\) in \(A_j\) is even and at least two.
    \item For consecutive poles $\alpha,\beta \in A_0$ with \(e(\alpha)\neq
    e(\beta)\), \(\lambda_c(x,y)\) has an odd (resp. even) number of zeroes in \((\alpha,\beta)\subset A_0\) when the residues at \(\alpha\) and
    \(\beta\) have the same (resp. opposite) sign. In particular, there is at least one zero in \((\alpha,\beta)\) when its residues have {the same sign.} 
    \item For consecutive poles $\alpha,\beta \in A_0$ with $e(\alpha) = e(\beta)$ {and \(\res_\alpha(\lambda_c(x,y)) = \res_\beta(\lambda_c(x,y)) \neq 0\)}, \(\lambda_c(x,y)\) has an odd number of zeroes in $(\alpha,\beta) \subset A_0$. In particular, there is at least one zero.
\end{itemize}
The above accounts for all but possibly two zeroes of $\lambda_c(x,y)$. If \(R\subset \mathcal D_c^\circ\), there are exactly two remaining zeroes, while if \(R\subset \RR^2\setminus \mathcal D_c^\circ\), then all zeroes are accounted for.
\end{lemma}
\begin{proof} 
By Lemma~\ref{lem:zeroes_of_omega_D}, together with
Lemma~\ref{lem:sign_poles_cts}, the zeroes forced by the above conditions account for
\[
2g+\sum_{e\in E(N):\sign_R(e)\neq 0} |\zz_e|
-\#\{\text{sign changes of }\sign_R\}
\]
zeroes of \(\lambda_c(x,y)\). Since \(\deg(\lambda_c(x,y))=2g-2\) and \(\lambda_c(x,y)\) has
\[
\sum_{e \in E(N):\sign_R(e)\neq 0} |\zz_e|
\]
simple poles, its total number of zeroes is
\[
2g+\sum_{e \in E(N):\sign_R(e)\neq 0} |\zz_e|-2.
\]
Hence, the number of zeroes of \(\lambda_c(x,y)\) not accounted for by the above
conditions is exactly
\[
\#\{\text{sign changes of $\sign_R$}\}-2.
\]
The final claim now follows from Lemma~\ref{lem:sign_R_continuous}, which gives four sign changes when \(R\subset \mathcal D_c^\circ\) and two sign changes when \(R\subset \RR^2\setminus \mathcal D_c^\circ\).
\end{proof}

\begin{lemma}
Let \((x,y)\in \mathcal D_c^\circ\). With the conclusion of Lemma~\ref{lem:forced_zeroes_and_remaining}, exactly one of the following holds:
\begin{enumerate}
    \item The two remaining zeroes form a \(\sigma\)-conjugate pair
    \[
\{\zeta,\sigma(\zeta)\},
\qquad
\zeta\in \curve_+^\circ,\ \sigma(\zeta)\in \curve_-^\circ.
\]
    \item The two remaining zeroes are real and lie in the same connected
    component of \[\curve(\RR)\setminus \{\text{poles of $\lambda_c(x,y)$}\}.\]
\end{enumerate}
\end{lemma}

\begin{proof}

By Lemma~\ref{lem:third_kind_ind}, the zero set of \(\lambda_c(x,y)\) is invariant under \(\sigma\). Hence, these two remaining zeroes either both lie on \(\curve(\RR)\) or else they form a \(\sigma\)-conjugate pair.

Assume they lie on \(\curve(\RR)\). By Lemma~\ref{lem:forced_zeroes_and_remaining}, on each connected
component of $\curve(\RR)\setminus \{\text{poles of $\lambda_c(x,y)$}\}$, the number of zeroes of
\(\lambda_c(x,y)\) has a prescribed parity, and the parity of the number of zeroes already forced is correct, so the two remaining zeroes must lie in the same connected component of \(\curve(\RR)\setminus \{\text{poles of $\lambda_c(x,y)$}\}\).
\end{proof}

We now give a partial classification of regions in $\mathcal D_c^\circ$ according to the locations of these two remaining zeroes. For some of these regions, we also define a map $\Omega$ to $\curve_+$. The reader is encouraged to look ahead to Section~\ref{sec:tangent_arctic_curve} where we give a geometric reformulation.


\begin{definition} \label{def:regions_in_limit_shape}

We say that \((x,y)\in \mathcal D_c^\circ\):

\begin{enumerate}
\item lies in the \emph{rough region} \(\liq\) if \(\lambda_c(x,y)\) has a
\(\sigma\)-conjugate pair of zeroes. In this case, there is a unique zero
\(\zeta\in \curve_+^\circ\), and we define
\[
\Omega:\liq \to \curve_+^\circ,
\qquad
\Omega(x,y):=\zeta.
\]

\item lies in the \emph{frozen region} \(\mathcal F_{(\alpha,\beta)}\), where
\((\alpha,\beta)\subset A_0\) is the open interval between two consecutive
angles \(\alpha,\beta\) with \(e(\alpha)\neq e(\beta)\), such that
\(\lambda_c(x,y)\) has exactly two simple zeroes in \((\alpha,\beta)\) when
the vertex \(e(\alpha)\cap e(\beta)\in V(N)\) is convex, and {two or} three
simple zeroes when it is concave.

\item lies in the \emph{quasi-frozen region} \(\mathcal Q_{(\alpha,\beta)}\),
where \((\alpha,\beta)\subset A_0\) is the open interval between two consecutive
angles \(\alpha,\beta\) with \(e(\alpha)=e(\beta)\), if \(\lambda_c(x,y)\) has
three simple zeroes in \((\alpha,\beta)\).

\item lies in the \emph{smooth region} \(\mathcal G_j\), for \(1\le j\le g\), if
\(\lambda_c(x,y)\) has four simple zeroes in \(A_j\).

\item lies on the \emph{frozen-rough boundary}
\(\mathscr{A}_{(\alpha,\beta)}\), where
\((\alpha,\beta)\subset A_0\) is the open interval between two consecutive
angles \(\alpha,\beta\) with \(e(\alpha)\neq e(\beta)\), if \(\lambda_c(x,y)\)
has a zero \(\zeta\) of order two or three in \((\alpha,\beta)\). In this case,
\(\zeta\) is unique, and we define
\[
\Omega:\mathscr{A}_{(\alpha,\beta)}\to (\alpha,\beta),
\qquad
\Omega(x,y):=\zeta.
\]
If \(\zeta\) is a zero of order three, which can occur only when the vertex
\(e(\alpha)\cap e(\beta)\in V(N)\) is concave, we call \((x,y)\) a
\emph{frozen cusp}.

\item lies on the \emph{quasi-frozen-rough boundary}
\(\mathscr{A}_{(\alpha,\beta)}\), where
\((\alpha,\beta)\subset A_0\) is the open interval between two consecutive
angles \(\alpha,\beta\) with \(e(\alpha)=e(\beta)\), if \(\lambda_c(x,y)\) has
a zero \(\zeta\) of order two or three in \((\alpha,\beta)\). In this case,
\(\zeta\) is unique, and we define
\[
\Omega:\mathscr{A}_{(\alpha,\beta)}\to (\alpha,\beta),
\qquad
\Omega(x,y):=\zeta.
\]
If \(\zeta\) is a zero of order three, we call \((x,y)\) a
\emph{quasi-frozen cusp}.

\item lies on the \emph{smooth-rough boundary} \(\mathscr{A}_{A_j}\), for
\(1\le j\le g\), if \(\lambda_c(x,y)\) has a total of four zeroes in $A_j$ (counted with multiplicity), one of which is a zero \(\zeta\) of order two or three. In this case, \(\zeta\) is unique: if \(\lambda_c(x,y)\) had
two double zeroes on \(A_j\), then it would have constant sign on \(A_j\),
forcing
\[
\int_{A_j}\lambda_c(x,y)\neq 0,
\]
contradicting imaginary-normalization. We therefore define
\[
\Omega:\mathscr{A}_{A_j}\to A_j,
\qquad
\Omega(x,y):=\zeta.
\]
If \(\zeta\) is a zero of order three, we call \((x,y)\) a
\emph{smooth cusp}; an order four zero is excluded by the same argument that rules out two double zeroes.
\end{enumerate}

\end{definition}

\begin{remark}
The above list is only a partial classification of the various singularities that can occur and does not in general exhaust $\mathcal D_c^\circ$; for example, it does not include the frozen--frozen boundary in Figure~\ref{fig:tropical_astroidal_arctic_curve}. We do not attempt a complete classification of all possibilities in this paper.
\end{remark}

We set
\[
\mathscr{A}_{A_0\setminus\zz}
:=
\bigsqcup_{(\alpha,\beta)\subset A_0}\mathscr{A}_{(\alpha,\beta)},
\]
where the union ranges over all open intervals \((\alpha,\beta)\) between
consecutive angles on \(A_0\). We also set
\[
\mathscr{A}^\circ
:=
\left(\bigsqcup_{j=1}^g \mathscr{A}_{A_j}\right)
\sqcup
\mathscr{A}_{A_0\setminus\zz}.
\]

The next few subsections describe the geometry of the map \(\Omega\).
We first show in Section~\ref{sec:rough_region} that \(\Omega\) is a diffeomorphism from the rough region \(\liq\) onto \(\curve_+^\circ\). We then extend the inverse map $\Omega^{-1}$ to
\(\curve(\RR)\setminus \zz\) in Section~\ref{sec:extension_bdy}. Finally, in Section~\ref{sec:turning_points}, we extend $\Omega^{-1}$ to the points of $\zz$, whose images will be denoted by $\mathscr{A}_{\zz}$; these points are called \emph{turning points}. Once $\mathscr{A}_{\zz}$ has been defined, we define the \emph{arctic curve}
\[
\mathscr{A}:=\mathscr{A}^\circ\sqcup \mathscr{A}_{\zz}.
\]

We then show in Proposition~\ref{prop:omega_inverse} that these constructions glue together to give a smooth map
\[
\Omega^{-1}:\Sigma_+\to \overline{\liq}.
\]
In particular, we will prove that the arctic curve is precisely the boundary of the rough region, \emph{i.e.}
\[
\mathscr{A} = \partial \liq,
\]
and moreover that the restriction
\[
\Omega^{-1}\big|_{\curve(\RR)}:\curve(\RR)\to \mathscr{A}
\]
is a smooth parametrization of $\mathscr{A}$. {By a smooth parameterization, we mean that the components of $\Omega^{-1}$ are smooth functions on $\curve(\RR)$. The image $\mathscr{A}$ need not be smooth;
indeed, it typically has cusps at the critical values of $\Omega^{-1}$.}

\begin{remark}
{By general theory~\cite{ADPZ20}, the arctic curve is algebraic. The characterization of the arctic curve via the zeros of \(\lambda_c\) fits naturally into the framework developed in~\cite{Pio24}. It may provide an alternative proof that the arctic curve is algebraic while simultaneously determining the degree of the curve. We will not pursue this argument here.}
\end{remark}

\subsection{The rough region and the interior parametrization} \label{sec:rough_region}

Recall from~\eqref{eq:z_and_w} the definition of the functions $z(\zeta)$ and $w(\zeta)$. While these functions may be multi-valued, the logarithmic differentials $d\log z$ and $d\log w$ are nevertheless well-defined meromorphic 1-forms.

\begin{lemma}\label{lem:zeroes_xdlogzydlogw}
Let \((x,y)\in\mathbb R^2\setminus\{(0,0)\}\). All zeroes of the
\(1\)-form
\[
    x d\log z+y d\log w
\]
lie in {\(\Sigma(\mathbb R)\setminus \zz.\)} Moreover, all zeroes that lie in the oval \(A_0\) are simple.
\end{lemma}

\begin{proof}
The $1$-form has residues 
\[
\res_\alpha(x d\log z +y d \log w) = - b_{e(\alpha)} x +  a_{e(\alpha)} y = \frac{1}{|{e(\alpha)}|_\ZZ} \langle {e(\alpha)} , (y,-x) \rangle. 
\]
Since $N$ is convex, the residues have exactly two sign changes, and therefore, Lemma~\ref{lem:zeroes_of_omega_D} accounts for all of its zeroes. {In particular, there are no zeroes in any interval of} 
\begin{equation}\label{eq:intervals}
A_0 \setminus \{ \alpha \in \zz: -b_{e(\alpha)}x+a_{e(\alpha)} y \neq 0\}
\end{equation}
{
at whose endpoints the residues of $x d\log z +y d \log w$ have opposite signs. It remains only to exclude zeroes at angles \(\alpha\in\bm\alpha\) where the residue vanishes. Such an angle satisfies
\[
    -b_{e(\alpha)}x+a_{e(\alpha)}y=0,
\]
so the edge \(e(\alpha)\) is parallel to \((x,y)\). By convexity of \(N\), $\alpha$ lies in an interval of~\eqref{eq:intervals} at whose endpoints the residues have opposite sign, and these intervals contain no zeroes as shown above.
}
\end{proof}

Given a local coordinate \(\zeta\) on \(\curve^\circ_+\), write
\begin{equation} \label{eq:3_forms}
d\log z = f_1(\zeta) d\zeta,\qquad
d\log w = f_2(\zeta) d\zeta,\qquad
\omega_{\bm{D}_c}
= f_3(\zeta) d\zeta.
\end{equation}

\begin{proposition}\label{prop:diff}
The map $\Omega:\liq \rightarrow \curve_+^\circ$ is a diffeomorphism with inverse 
$\Omega^{-1}: \curve_+^\circ \rightarrow \liq$ given, for $\zeta\in \curve_+^\circ$, by
\begin{equation}\label{eq:imff}
x=
\frac{\Im(\overline{f_2}f_3)}{\Im(\overline{f_1}f_2)},
\qquad
y=
\frac{\Im(f_1\overline{f_3})}{\Im(\overline{f_1}f_2)},
\end{equation}
where all functions are evaluated at $\zeta$.
\end{proposition}
\begin{proof}
Let $\zeta \in \curve_+^\circ$. The vanishing of $\lambda_c(x,y)$ at $\zeta$ is equivalent to the system of linear equations
\begin{equation} \label{eq:linear_system_1}
\begin{bmatrix}
 \Re f_1(\zeta)  & \Re f_2(\zeta)   \\
  \Im f_1(\zeta) & \Im f_2(\zeta)   
\end{bmatrix}\begin{bmatrix}x \\ y
\end{bmatrix} = -\begin{bmatrix} \Re f_3(\zeta)\\\Im f_3(\zeta)
\end{bmatrix}.
\end{equation}
The system has a unique solution $(x,y) \in \RR^2$ provided the determinant
\begin{equation} \label{eq:matrix}
    \det \begin{bmatrix}
 \Re f_1(\zeta)  & \Re f_2(\zeta)   \\
  \Im f_1(\zeta) & \Im f_2(\zeta)   
\end{bmatrix} = \Im (\overline{f_1(\zeta)} f_2(\zeta))  \neq 0.
\end{equation}

If this determinant vanishes, then 
\[
r:=\frac{f_1(\zeta)}{f_2(\zeta)}  \in \RR.
\]
By Lemma~\ref{lem:zeroes_xdlogzydlogw}, the $1$-form 
\[
d\log z - r d \log w
\]
cannot vanish at $\zeta$, a contradiction. 

Thus, the determinant is nonzero, so for every $\zeta\in \curve_+^\circ$
the system~\eqref{eq:linear_system_1} has a unique solution $(x,y)\in\mathbb R^2$, which is given by~\eqref{eq:imff}. 

We now show that $(x,y) \in \mathcal D_c^\circ$, and therefore, $(x,y) \in \liq$. If $(x,y) \notin \mathcal D_c^\circ$, then there are only two sign changes in the residues of $\lambda_c(x,y)$ and so Lemma~\ref{lem:forced_zeroes_and_remaining} accounts for all the zeroes which must therefore all be in $\curve(\RR)$, contradicting that \(\lambda_c(x,y)\) has a
\(\sigma\)-conjugate pair of zeroes. Thus, $(x,y) \in \mathcal D_c^\circ$.

Hence, we obtain a map
\[
\Omega^{-1}:\curve_+^\circ\to\liq
\]
sending $\zeta$ to the unique solution~\eqref{eq:imff}. By construction, $\Omega^{-1}$ is a right inverse to $\Omega$. Moreover, uniqueness of the solution to~\eqref{eq:linear_system_1} implies that $\Omega$ is injective, so $\Omega$ is a bijection with inverse $\Omega^{-1}$.

Let 
\begin{equation}\label{eq:Fxyzeta}
    F(x,y;\zeta) := \frac{\lambda_c(x,y)}{d\zeta} =x f_1(\zeta)+yf_2(\zeta)+f_3(\zeta).
\end{equation}
Let $\zeta = u+\ii v$ and let 
\[
G(x,y;u,v):= \begin{bmatrix} \Re F(x,y;u+\ii v)\\\Im F(x,y;u+\ii v)
\end{bmatrix}.
\]
Implicitly differentiating $G(x,y;\Omega(x,y))=0$, we obtain
\[
\mathscr{J}_{(x,y)} G + \mathscr{J}_{(u,v)}G \mathscr{J} \Omega =0,
\]
where \(\mathscr{J}\) denotes the Jacobian matrix. Hence,
\[
\mathscr{J} \Omega = - (\mathscr{J}_{(u,v)}G)^{-1}\mathscr{J}_{(x,y)} G,
\]
and, therefore,
\[
\det \mathscr{J}\Omega = \frac{\det \mathscr{J}_{(x,y)}G} { \det \mathscr{J}_{(u,v)} G}.
\]
The numerator is~\eqref{eq:matrix} and the denominator is \[\left|\frac{d F(x,y;\zeta)}{d \zeta}\right|^2\] since \(F\) is holomorphic in \(\zeta\). Therefore,
\begin{equation}\label{eq:jacobian_Omega}
\det \mathscr{J}\Omega
=
\frac{\Im(\overline{f_1}f_2)}
{|x f_1'+y f_2'+f_3'|^2}.
\end{equation}
We have already seen that \(\Im(\overline{f_1}f_2)\neq 0\), and the denominator is nonzero because \(\zeta\) is a simple zero of \(\lambda_c(x,y)\). Therefore, \(\Omega\) is a local diffeomorphism. Since it is a bijection, it is a diffeomorphism.
\end{proof}

\subsection{Boundary parametrization away from angles} \label{sec:extension_bdy}

For $\zeta_0\in \curve(\RR)\setminus \zz$, choose a local coordinate $\zeta = u+\ii v$ centered at $\zeta_0$ such that $\sigma(\zeta)=\overline{\zeta}$ and such that \(\Sigma_+\) is locally given by \(v\ge 0\). By Lemma~\ref{lem:third_kind_ind}, the \(1\)-forms in~\eqref{eq:3_forms} are \(\sigma\)-real, so \(f_j(\bar\zeta)=\overline{f_j(\zeta)},\) and hence, \[f_j(u)\in\RR\] for \(v=0\).

\begin{proposition}\label{prop:bdy_diffeomorphisms}
The map
\(
\Omega:\mathscr{A}^\circ \to \curve(\RR)\setminus \zz
\)
is a homeomorphism with inverse 
\(
\Omega^{-1}:\curve(\RR)\setminus \zz \to \mathscr{A}^\circ
\)
given, for $\zeta=u\in \curve(\RR)\setminus\zz$, by
\begin{equation}\label{eq:fsol}
    x=
\frac{f_2f_3'-f_2'f_3}{f_1f_2'-f_1'f_2},
\qquad
y=
\frac{f_1'f_3-f_1f_3'}{f_1f_2'-f_1'f_2},
\end{equation}
where all functions are evaluated at $u$.
\end{proposition}
\begin{proof}
The condition \(\Omega(x,y)=u\) means that \(\lambda_c(x,y)\) vanishes at \(u\) to order at least two, which is equivalent to the linear system
\begin{equation}\label{eq:linear_system_2}
\begin{bmatrix}
f_1(u) & f_2(u)\\
f_1'(u) & f_2'(u)
\end{bmatrix}
\begin{bmatrix}
x\\ y
\end{bmatrix}
=
-\begin{bmatrix}
f_3(u)\\ f_3'(u)
\end{bmatrix}.
\end{equation}

We claim that the matrix is invertible. Indeed, if its determinant
were zero, there would exist a nonzero vector \((a,b)\in \RR^2\) such that
\[
\begin{bmatrix}
f_1(u) & f_2(u)\\
f_1'(u) & f_2'(u)
\end{bmatrix}
\begin{bmatrix}
a\\ b
\end{bmatrix}
=0.
\]
Equivalently, the \(1\)-form
\[
a d\log z+b d\log w
\]
would vanish at \(u\) to order at least two, contradicting
Lemma~\ref{lem:zeroes_xdlogzydlogw}. Thus, the matrix in
\eqref{eq:linear_system_2} is invertible, and the system has a unique solution which is given by~\eqref{eq:fsol}.

If \((x(u),y(u)) \notin \mathcal D_c^\circ\), then there are only two sign changes in the residues of \(\lambda_c(x(u),y(u))\), and so Lemma~\ref{lem:forced_zeroes_and_remaining} accounts for all the zeroes. Hence:
\begin{itemize}
    \item if $u \in A_0$, then $u$ must be a simple zero, contradicting that it has order at least two.
    \item if $u \in A_j$ and is a zero of order two, then since there are no other zeroes in $A_j$, $\lambda_c(x,y)$ would have constant sign on \(A_j\), forcing
\[
\int_{A_j}\lambda_c(x,y)\neq 0,
\]
contradicting imaginary-normalization.
\end{itemize} 
Thus, \((x(u),y(u)) \in \mathcal D_c^\circ\), and hence, \((x(u),y(u)) \in \mathscr{A}^\circ\).

Hence, we obtain a map
\[
\Omega^{-1}:\curve(\RR)\setminus \zz\rightarrow
\mathscr{A}^\circ
\]
sending \(u\) to this unique solution. By construction, $\Omega^{-1}$ is a right inverse to $\Omega$. Moreover, uniqueness of the solution to \eqref{eq:linear_system_2} implies that \(\Omega\) is injective. Therefore, \(\Omega\) is a homeomorphism with inverse \(\Omega^{-1}\).
\end{proof}

\subsection{Turning points}\label{sec:turning_points}

\begin{definition}\label{def:turning_point}
    We say that $(x,y) \in \domain$ is a \emph{turning point corresponding to $\alpha \in \zz$} if $\lambda_c(x,y)$ has a zero at $\alpha$. We denote by $\mathscr{A}_{\zz}$ the set of turning points.
\end{definition}

 
For \(\alpha\in\zz\), choose a local coordinate \(\zeta\) centered at \(\alpha\) such that $\sigma(\zeta)=\overline{\zeta}$, and
write
\[
f_j(\zeta)=\frac{r_j}{\zeta}+s_j+O(|\zeta|),
\qquad r_j,s_j\in\RR,
\qquad j=1,2,3,
\]
where, by definitions~\eqref{eq:z_and_w} of $z$ and $w$ and~\eqref{eq:div_c} of $\bm{D}_c$,  
\begin{equation} \label{eq:r_explicit}
    r_1 = -b_{e(\alpha)}, \qquad r_2 = a_{e(\alpha)}, \qquad r_3 = c({e(\alpha)}).
\end{equation}
We then define 
\[
\Omega^{-1}:\zz\to \mathscr{A}_{\zz}
\]
by sending \(\alpha\) to 
\begin{equation}\label{eq:turning_point_formula}
x=\frac{r_2s_3-s_2r_3}{r_1s_2-s_1r_2},
\qquad
y=\frac{s_1r_3-r_1s_3}{r_1s_2-s_1r_2}.
\end{equation}

\begin{proposition}\label{prop:turning_points_unique}
For each \(\alpha\in\zz\), the point \((x,y) = \Omega^{-1}(\alpha)\) is the unique turning point corresponding to $\alpha$, and it lies on \(\line_e\). 
\end{proposition}

\begin{proof}
The \(1\)-form \(\lambda_c(x,y)\) has a zero at \(\alpha\) if and only if both the pole
term and the constant term vanish. Equivalently, \((x,y)\) satisfies
\begin{equation}\label{eq:linear_system_turning}
\begin{bmatrix}
r_1 & r_2\\
s_1 & s_2
\end{bmatrix}
\begin{bmatrix}
x\\ y
\end{bmatrix}
=
-
\begin{bmatrix}
r_3\\ s_3
\end{bmatrix}.
\end{equation}
Solving~\eqref{eq:linear_system_turning} gives~\eqref{eq:turning_point_formula}, provided
the matrix is invertible.

We claim that the matrix in~\eqref{eq:linear_system_turning} is invertible.
Indeed, if its determinant were zero, the $1$-form
\[
a_{e(\alpha)} d\log z+ b_{e(\alpha)} d\log w
\]
would have a zero at \(\alpha\), contradicting Lemma~\ref{lem:zeroes_xdlogzydlogw}. Therefore, \eqref{eq:linear_system_turning} has a unique solution.

The point $(x,y)$ is contained in $\line_{e(\alpha)}$ since $\res_\alpha\bigl(\lambda_c(x,y)\bigr)=-b_{e(\alpha)} x+a_{e(\alpha)} y+c({e(\alpha)}) =0,$ and this is exactly the linear function defining $\line_{e(\alpha)}$.

Finally, we show that \((x,y)\in \domain\). Suppose that \((x,y)\notin \domain\), and let \(R\) be the cell containing \((x,y)\). Since
\((x,y)\in \line_{e(\alpha)}\), we have \(\sign_R(e(\alpha))=0\). Let
\(I\subset \partial N\) be the maximal closed interval containing \(e(\alpha)\)
on which \(\sign_R=0\), and let \(e_-\) and \(e_+\) be the nearest edges on the
two sides of \(I\) such that
\[
\sign_R(e_-)\neq 0,
\qquad
\sign_R(e_+)\neq 0.
\]
Since \(R\subset \RR^2\setminus \mathcal D_c^\circ\), Lemma~\ref{lem:sign_R_continuous}
implies that \(\sign_R\) has exactly two sign changes. Crossing the line
\(\line_{e(\alpha)}\) through \((x,y)\) changes the sign at $e(\alpha)$ and leaves all other signs unchanged. The two adjacent cells are still contained
in \(\RR^2\setminus \mathcal D_c^\circ\), and hence their sign functions also
have exactly two sign changes. Therefore, the number of sign changes cannot
change under this crossing, which is possible only if
\[
\sign_R(e_-)=-\sign_R(e_+).
\]
Thus, \(\sign_R\) has a sign change along the closed interval \(I\).

By Lemma~\ref{lem:forced_zeroes_and_remaining}, all zeroes of
\(\lambda_c(x,y)\) are accounted for. Moreover, the connected component of
\[
A_0\setminus \{\text{poles of }\lambda_c(x,y)\}
\]
corresponding to \(I\) has endpoints with opposite residues, and therefore
contains no zeroes of \(\lambda_c(x,y)\). In particular,
\(\lambda_c(x,y)\) has no zero at \(\alpha\), contradicting the definition of
\((x,y)\). Hence \((x,y)\in \domain\).
\end{proof}

\begin{lemma}\label{lem:turning_point_multiplicity}
Let \(\alpha\in \zz_e\), and let \((x,y)=\Omega^{-1}(\alpha)\in \mathscr A_{\zz}\)
be the turning point corresponding to \(\alpha\). Then:
\begin{enumerate}
\item If \((x,y)\in \line_{e(\alpha)}\cap \partial \mathcal D_c\), then \(\alpha\) is a simple zero of
\(\lambda_c(x,y)\).

\item If \((x,y)\in \line_{e(\alpha)}\cap \mathcal D_c^\circ\), then \(\alpha\) is a zero of \(\lambda_c(x,y)\) of order at
most two.
\end{enumerate}
\end{lemma}

\begin{proof}
If \((x,y)\in \line_{e(\alpha)} \cap \partial \mathcal D_c\), then \((x,y)\) lies on the boundary of \(\mathcal D_c\). By Lemma~\ref{lem:forced_zeroes_and_remaining}, \(\alpha\) is simple.

Now suppose \((x,y)\in \line_{e(\alpha)}\cap \mathcal D_c^\circ\) and let $R$ be the cell containing $(x,y)$. Let \(e_-\) (resp. \(e_+\)) denote the edges of \(N\)
immediately clockwise (resp. counterclockwise) of \(e\). Then
\[
\sign_R(e)=0,
\qquad
\sign_R(e_-)\neq \sign_R(e_+).
\]
By Lemma~\ref{lem:forced_zeroes_and_remaining}, no zero in the connected component of 
\[
A_0 \setminus \{\text{poles of }\lambda_c(x,y)\}
\]
containing $\alpha$ is forced, so there are either none or two zeroes in that component. In particular, $\alpha$ is a zero of order at most two.
\end{proof}

\begin{definition}\label{def:tp_cusp}
For \(\alpha\in \zz\), we call the turning point \((x,y)=\Omega^{-1}(\alpha)\in \mathscr A_{\zz}\) a \emph{turning-point cusp} if \(\alpha\) is a zero of order two of \(\lambda_c(x,y)\).
\end{definition}

\subsection{Global extension of \texorpdfstring{$\Omega^{-1}$}{Omega_inverse} and the arctic curve}

Let
\(
\overline{\liq}
\)
denote the closure of the rough region \(\liq\) in \(\RR^2\) and let
\(
\partial \liq:=\overline{\liq}\setminus \liq
\)
denote the boundary of \(\liq\).

\begin{proposition}\label{prop:omega_inverse}
The maps
\[
\Omega^{-1}:\curve_+^\circ \to \liq,\qquad
\Omega^{-1}:\curve(\RR)\setminus \zz \to \mathscr{A}^\circ,\qquad
\Omega^{-1}:\zz\to \mathscr{A}_{\zz}
\]
glue together to a smooth map
\[
\Omega^{-1}:\Sigma_+ \to \overline{\liq}.
\]
Moreover,
\[
\Omega^{-1}\bigl(\curve(\RR)\bigr)=\partial \liq=\mathscr{A}.
\]
\end{proposition}

\begin{proof}
First let \(\zeta_0\in \curve(\RR)\setminus \zz\). As in Section~\ref{sec:extension_bdy}, choose a local coordinate \(\zeta=u+\ii v\) centered at \(\zeta_0\) such that \(\sigma(\zeta)=\overline{\zeta}\) and such that \(\Sigma_+\) is locally given by \(v\ge 0\). Since \(f_j(u+\ii v)\in\RR\) when \(v=0\), the functions
\[
\Im (\overline{f_j(\zeta)} f_k(\zeta))
\]
vanish on \(v=0\). By Taylor's theorem, there exist smooth functions
\(B_{jk}(u,v)\) such that locally
\[
\Im(\overline{f_j(\zeta)}f_k(\zeta))=v B_{jk}(u,v).
\]
Hence, for \(v>0\),~\eqref{eq:imff} becomes
\[
x=\frac{B_{23}}{B_{12}},\qquad
y=\frac{B_{31}}{B_{12}}.
\]
Locally expanding \(f_j(\zeta)\) and \(f_k(\zeta)\), we obtain
\begin{equation}\label{eq:b_local}
    B_{jk}(u,v) = f_j(\zeta_0) f_k'(\zeta_0)-f_j'(\zeta_0) f_k(\zeta_0) + O(|\zeta|).
\end{equation}
Thus,
\[
B_{12}(0,0)=f_1(\zeta_0)f_2'(\zeta_0)-f_1'(\zeta_0)f_2(\zeta_0)\neq 0
\]
by Proposition~\ref{prop:bdy_diffeomorphisms}, so the above
formulas define a smooth extension of \(\Omega^{-1}\) across \(\zeta_0\), and its
value at \(\zeta_0\) is exactly~\eqref{eq:fsol}.

Now let \(\alpha\in\zz\), and choose a local coordinate \(\zeta=u+\ii v\) centered
at \(\alpha\) such that \(\sigma(\zeta)=\bar\zeta\). Write
\[
f_j(\zeta)=\frac{r_j}{\zeta}+s_j+O(|\zeta|),
\qquad r_j,s_j\in\RR,
\]
and let
\[
h_j(\zeta):=\zeta f_j(\zeta)=r_j+s_j\zeta+O(|\zeta|^2).
\]
Then each \(h_j\) is holomorphic and satisfies
\[
h_j(\bar\zeta)=\overline{h_j(\zeta)},
\]
so \(h_j(u)\in\RR\) for \(v=0\). Since \(|\zeta|^2\) is real, the formula~\eqref{eq:imff} can be rewritten for \(v>0\) as
\[
x=\frac{\Im(\overline{h_2}h_3)}{\Im(\overline{h_1}h_2)},
\qquad
y=\frac{\Im(h_1 \overline{h_3})}{\Im(\overline{h_1}h_2)}.
\]
Again by Taylor's theorem, locally
\[
\Im (\overline{h_j(\zeta)} h_k(\zeta))=v \widetilde B_{jk}(u,v)
\]
for smooth \(\widetilde B_{jk}\). Hence, 
\[
x=\frac{\widetilde B_{23}}{\widetilde B_{12}},
\qquad
y=\frac{\widetilde B_{31}}{\widetilde B_{12}}.
\]
At \((u,v)=(0,0)\),
\[
\widetilde B_{12}(0,0)=r_1s_2-s_1r_2\neq 0
\]
by Proposition~\ref{prop:turning_points_unique}. Therefore, these formulas extend
smoothly across \(\alpha\), and their value at \(\alpha\) is exactly given by~\eqref{eq:turning_point_formula}. 
We have therefore obtained a smooth map
\[
\Omega^{-1}:\Sigma_+\to \overline{\liq}
\]
whose restrictions to \(\curve_+^\circ\), \(\curve(\RR)\setminus\zz\), and
\(\zz\) are the three maps above.

Since \(\Omega^{-1}(\curve(\RR))\subset \mathscr{A}^\circ\sqcup\mathscr{A}_{\zz}\),
it is disjoint from \(\liq\), and hence,
\[
\Omega^{-1}(\curve(\RR))\subset \partial\liq.
\]

Conversely, let \((x,y)\in \partial\liq\), and choose \((x_n,y_n)\in \liq\) with
\((x_n,y_n)\to (x,y)\). Set \(\zeta_n:=\Omega(x_n,y_n)\in \curve_+^\circ\).
By compactness of \(\Sigma_+\), after passing to a subsequence, we may assume
\(\zeta_n\to \zeta\in \Sigma_+\). By continuity of \(\Omega^{-1}\),
\[
(x,y)=\lim_{n\to\infty}(x_n,y_n)
=\lim_{n\to\infty}\Omega^{-1}(\zeta_n)
=\Omega^{-1}(\zeta).
\]
If \(\zeta\in \curve_+^\circ\), then \((x,y)=\Omega^{-1}(\zeta)\in \liq\), a
contradiction. Hence, \(\zeta\in \curve(\RR)\), and therefore
\[
\partial\liq\subset \Omega^{-1}(\curve(\RR)).
\]
Thus,
\[
\partial\liq=\Omega^{-1}(\curve(\RR)).
\]
Since
\[
\Omega^{-1}(\curve(\RR)\setminus\zz)=\mathscr{A}^\circ,
\qquad
\Omega^{-1}(\zz)=\mathscr{A}_{\zz},
\]
we conclude that
\[
\partial\liq=\mathscr{A}^\circ\sqcup \mathscr{A}_{\zz}=\mathscr{A}.
\]
\end{proof}

\subsection{Geometry of the arctic curve}

The map
\[
\gamma:\curve \to \mathbb{CP}^1,
\qquad
\zeta \mapsto \frac{d\log w}{d\log z}(\zeta)
\]
is called the \emph{logarithmic Gauss map}.

By Proposition~\ref{prop:omega_inverse}, the restriction
\[
\Omega^{-1}\big|_{\curve(\RR)}:\curve(\RR)\to \mathbb R^2
\]
is a smooth parametrization of the arctic curve \(\mathscr{A}\). We now identify its critical points and compute tangent lines to $\mathscr{A}$. {We phrase the statement about tangent lines in terms of branches since we have not yet established injectivity of the parameterization; we will show this in Corollary~\ref{cor:turning_points_injective}.}

\begin{proposition}\label{prop:tangent_curvature}
Let \(\zeta_0\in \curve(\RR)\). Then:
\begin{enumerate}
    \item The point \(\zeta_0\) is a critical point of
    \(\Omega^{-1}\big|_{\curve(\RR)}\) if and only if
    \(\Omega^{-1}(\zeta_0)\) is one of the four types of cusps in Definition~\ref{def:regions_in_limit_shape} and Definition~\ref{def:tp_cusp}. 
    
    \item Each of these cusps is an ordinary cusp in the geometric sense, \emph{i.e.} after a smooth change of coordinates, it is locally given by $y^2=x^3$.
    
    \item The tangent line to the branch of \(\mathscr{A}\) at \( \Omega^{-1}(\zeta_0) \)
    determined by \(\zeta_0\) has slope
    \[
    -\frac{1}{\gamma(\zeta_0)}=-\frac{f_1(\zeta_0)}{f_2(\zeta_0)}.
    \]
    In particular, if $\zeta_0 = \alpha \in \zz$, then the tangent line is $\line_{e(\alpha)}$.
\end{enumerate}
\end{proposition}

\begin{proof}
First, we consider the case \(\zeta_0\in \curve(\RR)\setminus \zz\). Choose a local coordinate
\(\zeta=u+\ii v\) centered at \(\zeta_0\) such that
\(\sigma(\zeta)=\overline{\zeta}\). Write
\[
(x(u),y(u)):=\Omega^{-1}(u),
\]
and set
\begin{equation}\label{eq:delta_and_H}
W:=f_1f_2'-f_1'f_2,
\qquad
H:=x f_1''+y f_2''+f_3'',
\end{equation}
evaluated at \(\zeta_0\). Since \(\lambda_c(x(u),y(u))\) has a zero of order at least two at \(u\), we have
\[
x(u)f_1(u)+y(u)f_2(u)+f_3(u)=0,
\qquad
x(u)f_1'(u)+y(u)f_2'(u)+f_3'(u)=0.
\]
Differentiating with respect to \(u\) gives
\begin{align*}
f_1x'+f_2y'&=0, \\
f_1'x'+f_2'y'&=-H. 
\end{align*}
Solving this linear system yields
\begin{equation}\label{eq:slope_tangent}
(x',y')=\frac{H}{W}(f_2,-f_1).
\end{equation}
By Proposition~\ref{prop:bdy_diffeomorphisms}, we have \(W\neq 0\) and
\((f_2,-f_1)\neq 0\). Hence, \((x',y')=0\) if and only if \(H=0\), \emph{i.e.} if and only if \(\lambda_c(x,y)\) has a zero of order at least three at \(\zeta_0\). This proves Item~1 in this case.

A standard criterion for an ordinary cusp (see, for example,
\cite[Definition~5.3.9]{Duistermaat}) is that
\[
(x',y')=0
\qquad\text{and}\qquad
\det \begin{bmatrix}
x''&y''\\
x'''&y'''\end{bmatrix} \neq 0;
\]
Suppose \(\Omega^{-1}(\zeta_0)\) is one of the cusp types from
Definition~\ref{def:regions_in_limit_shape}. Since \(\lambda_c(x,y)\) has a zero of exactly order three at \(\zeta_0\),
\[
H=0,
\qquad
x f_1'''+y f_2'''+f_3''' \neq 0.
\]
Therefore,
\begin{equation} \label{eq:x'y'}
(x',y') = \frac{H}{W}(f_2,-f_1) = 0.
\end{equation}
Differentiating~\eqref{eq:slope_tangent}, we get
\begin{align}
(x'',y'') &= \left(\frac{H}{W}\right)'(f_2,-f_1) + \frac{H}{W}(f_2',-f_1'), \label{eq:x''}\\
(x''',y''') &= \left(\frac{H}{W}\right)'' (f_2,-f_1) + 2 \left(\frac{H}{W}\right)'(f_2',-f_1') + \frac{H}{W}(f_2'',-f_1''). \nonumber
\end{align}
By~\eqref{eq:x'y'},
\[
H' =x' f_1''+y'f_2''+ x f_1'''+y f_2'''+f_3''' = x f_1'''+y f_2'''+f_3'''.
\]
Using $H=0$ and simplifying, we get
\[
\det
\begin{bmatrix}
x'' & y''\\
x''' & y'''
\end{bmatrix} = \frac{2 (x f_1'''+y f_2'''+f_3''')^2}{W} \neq 0,
\]
so \(\Omega^{-1}(\zeta_0)\) is an ordinary cusp in the geometric sense, proving Item~2 in this case.

If \(\zeta_0\) is not a critical point, then the tangent line to the branch of \(\mathscr{A}\) at \((x,y)\) determined by \(\zeta_0\) has slope
\[
\frac{y'}{x'}=-\frac{f_1}{f_2}=-\frac{1}{\gamma(\zeta_0)}.
\]
When $\zeta_0$ is a critical point, the ordinary cusp $\Omega^{-1}(\zeta_0)$ still has a well-defined tangent line spanned by 
\[
(x'',y'') = \frac{H'}{W}(f_2,-f_1)
\]
which gives the same slope $-{1}/{\gamma(\zeta_0)}$, completing the proof of Item~3 in this case.

Now we consider the case \(\zeta_0=\alpha\in \zz\). Define
\[
h_j(\zeta):=\zeta f_j(\zeta).
\]
All items now follow exactly as in the previous case with \(h_j\) in place of \(f_j\).

Finally, the tangent vector at $\Omega^{-1}(\alpha)$ is 
\[
(h_2,-h_1) = (a_{e(\alpha)},b_{e(\alpha)}),
\]
which is the vector along $\line_{e(\alpha)}$. Since $\Omega^{-1}(\alpha) \in \line_{e(\alpha)}$, the tangent line is $\line_{e(\alpha)}$.
\end{proof}

\begin{corollary}\label{cor:orientation_reversing}
The map \(\Omega:\curve_+^\circ \to \liq\) is orientation reversing.
\end{corollary}

\begin{proof}
Choose \(\zeta_0\in A_0\setminus \zz\) and local coordinates
\(\zeta=u+\ii v\) with \(v\ge 0\) on \(\Sigma_+\). Since $f_j$ is real, $\Omega^{-1}(u + \ii v)$ given by~\eqref{eq:imff} is an even function of $v$. Therefore, locally we have
\[
(x(u+\ii v), y(u+ \ii v)) = \Omega^{-1}(u + \ii v) = (x(u),y(u)) + \frac 1 2 v^2 (\partial_{vv}x(u),\partial_{vv}y(u))+O(v^4). 
\]
Expanding 
\[
F(x(u+\ii v), y(u+\ii v); u+\ii v)=0 
\]
in $v$, with $F(x,y;\zeta)$ as in~\eqref{eq:Fxyzeta}, and using the vanishing of the $v^2$-term, we obtain
\[
(f_1(u),f_2(u)) \cdot (\partial_{vv}x(u),\partial_{vv}y(u)) = H \neq 0
\]
where $H$ is as in~\eqref{eq:delta_and_H}. In particular, 
\[
\partial_v \Omega^{-1} (u+\ii v) = v (\partial_{vv}x(u),\partial_{vv}y(u)) + O(v^3) \neq 0.
\]
Differentiating with respect to $u$ and using \eqref{eq:slope_tangent}, we also get 
\[
\partial_u \Omega^{-1}(u+\ii v) = (\partial_u x(u), \partial_u y(u)) + O(v^2) = \frac{H}{W}(f_2(u),-f_1(u)) + O(v^2) \neq 0.
\]
It follows that
\[
\det(\partial_u\Omega^{-1}(u+\ii v), \partial_v\Omega^{-1}(u+\ii v))=v\frac{H^2}{W}+\Ordo(v^3).
\] 
In particular, \(\Omega^{-1}\) is orientation reversing if~$W<0$.

We use the associated \emph{amoeba} to show that \(W<0\). Define
\[
X(\zeta):=\Re\left(\int_{\zeta_0}^{\zeta} d\log z\right),
\qquad
Y(\zeta):=\Re\left(\int_{\zeta_0}^{\zeta} d\log w\right),
\]
where \(\zeta_0\) is a fixed point on \(A_0\). Since \(d\log z\) and \(d\log w\) are imaginary-normalized differentials,~\(X\) and~\(Y\) are well-defined single-valued functions on \(\Sigma\). The amoeba is the image of the map \( \curve \ni \zeta \mapsto (X(\zeta),Y(\zeta))\in\mathbb{R}^2\). Under this map, \(\Sigma(\mathbb{R})\setminus \zz\) is mapped to the boundary of the amoeba, while \(\Sigma_+^\circ\) is mapped to its interior. In the algebraic (periodic) setting, this agrees with the classical amoeba of the spectral curve; see~\cite{Kri13} for a proper introduction.

The connected components of the complement of the amoeba are convex~\cite[Proposition 2.7]{Kri13}. This convexity implies that, along the boundary, the tangent vector turns to the right. Consequently, for \(\zeta\in \Sigma(\mathbb{R})\), the vector \((-Y'(\zeta),X'(\zeta))\) is an inward-pointing normal, and
\[
(-Y'(\zeta),X'(\zeta))\cdot (X''(\zeta),Y''(\zeta))<0.
\]
Locally we may write
\[
X'(\zeta)=f_1(\zeta), \qquad Y'(\zeta)=f_2(\zeta),
\]
so that
\[
(-Y'(\zeta),X'(\zeta))\cdot (X''(\zeta),Y''(\zeta))
=(-f_2(\zeta),f_1(\zeta))\cdot (f_1'(\zeta),f_2'(\zeta))
= W.
\]
This proves that \(W<0\) and, hence, that $\Omega^{-1}$ is orientation reversing.
\end{proof}

\begin{proposition}\label{prop:local_conv}
    If $(x,y) = \Omega^{-1}(\zeta_0)$ is a smooth point of $\mathscr{A}$, then the curvature $\kappa(x,y) < 0$. In particular, $\liq$ is locally convex at $(x,y)$.
\end{proposition}
\begin{proof}
Let $\zeta_0 \in \curve(\RR) \setminus \zz$. Choose a local coordinate
\(\zeta=u+\ii v\) centered at \(\zeta_0\) such that \(\sigma(\zeta)=\bar\zeta\),
so that \(\Sigma_+\) is locally given by \(v\ge 0\). Let $W$ and $H$ be as in~\eqref{eq:delta_and_H}. Using~\eqref{eq:slope_tangent} and~\eqref{eq:x''}, we get that the curvature is
\[
\kappa(x,y)
=
\frac{\det
\begin{bmatrix}
x' & y'\\
x'' & y''
\end{bmatrix}}{|(x',y')|^3}
=
\frac{|W|^3}{ |H|(f_1^2+f_2^2)^{3/2}} \frac{1}{W}.
\]
Note that $H \neq 0$ since $(x,y)$ is assumed to be a smooth point and therefore, not a cusp. {We saw in the proof of Corollary~\ref{cor:orientation_reversing} that~\(W<0\), which proves that $\kappa(x,y)<0$.}


Since \(\Omega^{-1}\) is orientation reversing, the boundary is traversed clockwise, and therefore, negative curvature means that \(\mathscr{A}\) bends towards the right, \emph{i.e.} towards the rough region.

The argument for $\zeta_0 = \alpha \in \zz$ is similar with $h_j(\zeta) = \zeta f_j(\zeta)$ in place of $f_j(\zeta)$ as in Proposition~\ref{prop:tangent_curvature}.
\end{proof}

\begin{corollary}\label{cor:turning_points_injective}
The map
\[
\Omega^{-1}:\curve(\RR)\to \mathscr A
\]
is a bijection.
\end{corollary}

\begin{proof}
Surjectivity follows from Proposition~\ref{prop:omega_inverse}. It remains to
show that \(\Omega^{-1}\) is injective.

By Proposition~\ref{prop:tangent_curvature}, every point of \(\mathscr A\) is
locally either a smooth point of a branch of $\mathscr{A}$ or an ordinary cusp. In particular,
\(\Omega^{-1}\big|_{\curve(\RR)}\) is locally injective. Therefore, if
\(\Omega^{-1}\) were not injective, then \(\mathscr A\) would have a
self-intersection.

We first claim that no self-intersection can occur at a point of \(\partial \mathcal D_c\).
Indeed, there are no cusps on \(\partial \mathcal D_c\), so such a point would have to be
a smooth point of two distinct branches of \(\mathscr A\) with a common tangent
line. Since \(\overline{\liq}\) is locally convex by
Proposition~\ref{prop:local_conv} and the map $\Omega^{-1}\big|_{\curve_+^\circ}$ is a diffeomorphism, the two branches would have to lie on opposite
sides of this common tangent line. This is impossible because \(\partial \mathcal D_c\) is
simple.

Therefore, any self-intersection \((x,y)\) of \(\mathscr A\) would have to lie
in the interior \(\mathcal D_c^\circ\). But then each branch through \((x,y)\) imposes a
different condition on the location of the two remaining zeroes of
\(\lambda_c(x,y)\), which is impossible.

Thus,
\[
\Omega^{-1}:\curve(\RR)\to \mathscr A
\]
is globally injective, and therefore, a bijection.
\end{proof}



\subsection{Zeroes of \texorpdfstring{\(\lambda_c(x,y)\)}{lambda} and tangent lines to the arctic curve}\label{sec:tangent_arctic_curve}

\begin{figure}[t]
\centering

\begin{subfigure}{0.48\textwidth}
\centering
\begin{tikzpicture}[scale=0.1]
    \node  (0) at (0, 31) {};
    \node  (1) at (0, 0) {};
    \node  (2) at (37, 0) {};
    \node[bluevert]  (3) at (0, 27) {};
    \node[bluevert]  (4) at (32.75, 0) {};
    \node  (5) at (0, 24.5) {};

    \draw (2.center) -- (1.center);
    \draw (1.center) -- (0.center);
    \draw [blue] (4.center) to[out=180,in=270] (3.center);

    \node [nvert] (7) at (9, 5) {};
    \node [redvert] (8) at (20.75, 1.75) {};
    \node [redvert] (9) at (3.5, 14.25) {};
    \draw (7) -- (9);
    \draw (7) -- (8);
\end{tikzpicture}
\caption*{(A)}
\end{subfigure}
\hfill
\begin{subfigure}{0.48\textwidth}
\centering
\begin{tikzpicture}
    \clip (-1,-1.5) rectangle (3.7,2.7);

    \node  (0) at (0, -1.5) {};
    \node  (1) at (0, 0) {};
    \node  (2) at (3.7, 0) {};
    \node  (3) at (0, 2.7) {};
    \node[bluevert]  (4) at (3.275, 0) {};
    \node  (5) at (0, 2.45) {};
    \node [nvert] (7) at (0.3, 0.7) {};
    \node  (10) at (0, 1.8) {};
    \node[greenvert]  (11) at (-0.6, 1.6) {};
    \node[bluevert]  (12) at (0, 0.6) {};
    \node  (13) at (-2.8, -1.5) {};

    \draw (2.center) -- (1.center);
    \draw (1.center) -- (0.center);
    \draw[dashed,gray] (1.center) -- (10.center);
    \draw[blue] (4.center) to[out=180,in=-45] (11.center);
    \draw[blue] (12.center) to[out=90,in=-45] (11.center);
    \draw[blue] (12.center) to[out=-90,in=0,looseness=0.75] (13.center);
    \node[greenvert]  (11) at (-0.6, 1.6) {};
    \node[redvert] (8) at (1.225, 0.325) {};
    \node[redvert] (9) at (-0.275, -0.15) {};
    \node[redvert] (14) at (-0.175, 1.2) {};
    \draw (7) -- (9);
    \draw (7) -- (8);
    \draw (7) -- (14);
\end{tikzpicture}
\caption*{(B)}
\end{subfigure}

\vspace{1em}

\begin{subfigure}{0.48\textwidth}
\centering
\begin{tikzpicture}[scale=1.3]
    \node (1) at (0, 0) {};
    \node (2) at (3.7, 0) {};
    \node (3) at (0, 2.7) {};
    \node[bluevert] (4) at (3.275, 0) {};
    \node (5) at (0, 2.45) {};
    \node [nvert] (7) at (2.05, 0.6) {};
    \node[greenvert] (11) at (1.9, 1.5) {};
    \node[bluevert] (12) at (0.4, 0) {};
    \node (14) at (2.425, 0.3) {};
    \node (15) at (1.95, 1.075) {};

    \draw (2) -- (1);
    \draw [blue] (4.center) to[out=180,in=-90] (11.center);
    \draw [blue] (12.center) to[out=0,in=-90,looseness=0.75] (11.center);
    \node[greenvert] (11) at (1.9, 1.5) {};
    \node[redvert] (8) at (0.875, 0.125) {};
    \node[redvert] (14r) at (2.425, 0.3) {};
    \node[redvert] (15r) at (1.95, 1.075) {};
    \draw (7) -- (8);
    \draw (7) -- (14r);
    \draw (7) -- (15r);
\end{tikzpicture}
\caption*{(C)}
\end{subfigure}
\hfill
\begin{subfigure}{0.48\textwidth}
\centering
\begin{tikzpicture}[scale=1.3]
    \node[greenvert] (4) at (3.275, 0) {};
    \node [nvert] (7) at (1.75, 0.35) {};
    \node[greenvert] (11) at (2, 1.525) {};
    \node[greenvert] (13) at (0.3, 0) {};
    \node[greenvert] (15) at (2.1, -1) {};

    \draw [blue] (4.center) to[out=180,in=-90] (11.center);
    \draw (13.center) to[out=0,in=-90,looseness=1.25] (11.center);
    \draw [blue] (15.center) to[out=105,in=180,looseness=1.25] (4.center);
    \draw [blue] (15.center) to[out=105,in=0] (13.center);
    \node[greenvert] (4) at (3.275, 0) {};
    \node[greenvert] (11) at (2, 1.525) {};
    \node[greenvert] (13) at (0.3, 0) {};
    \node[greenvert] (15) at (2.1, -1) {};

    \node[redvert] (8) at (1.275, 0.15) {};
    \node[redvert] (9) at (2.925, 0.05) {};
    \node[redvert] (14) at (1.9, 0.825) {};
    \node[redvert] (16) at (2.05, -0.85) {};

    \draw (7) -- (9);
    \draw (7) -- (8);
    \draw (7) -- (14);
    \draw (7) -- (16);
\end{tikzpicture}
\caption*{(D)}
\end{subfigure}

\caption{Local tangent configurations associated with the zeroes of
\(\lambda_c(x,y)\). (A)--(D) correspond respectively to the frozen region
at a convex vertex, the frozen region at a concave vertex, the quasi-frozen
region and the smooth region. {In (B), the number of simple zeroes in \((\alpha,\beta)\) is either two or three, depending on which side of the
dashed line \((x,y)\) lies; here the frozen region is \(\mathcal F_{(\alpha,\beta)}\).} Orange points denote zeroes of \(\lambda_c(x,y)\) under the identification \(\Omega^{-1}\bigl|_{\curve(\RR)}\).
}
\label{fig:local-configurations}

\end{figure}
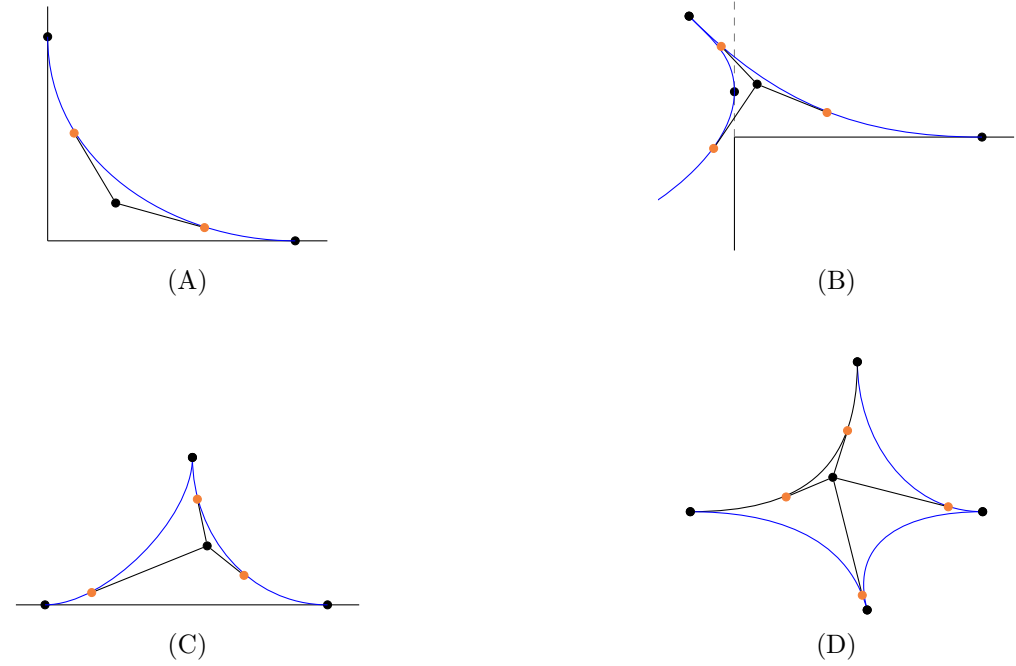

The zeroes of \(\lambda_c(x,y)\) admit a simple geometric interpretation in terms of tangent lines to the arctic curve. This gives a convenient dictionary between the phase classification from Definition~\ref{def:regions_in_limit_shape} and the geometry of \(\mathscr A\), generalizing the picture of Kenyon and Okounkov~\cite[Section~1.8]{KO07}; see also~\cite[Proposition~15]{BB24}.

\begin{proposition}\label{prop:tangent_zero_iff}
Let \((x_0,y_0)\in \mathscr{A}\) and let
\[
\zeta_0:=\Omega(x_0,y_0).
\]
For \((x,y)\in \RR^2\), the following are equivalent:
\begin{enumerate}
\item \((x,y)\) lies on the tangent line to \(\mathscr{A}\) at \((x_0,y_0)\).
\item \(\zeta_0\) is a zero of \(\lambda_c(x,y)\).
\end{enumerate}
\end{proposition}

\begin{proof}
By Proposition~\ref{prop:tangent_curvature}, \((x,y)\) lies on the tangent line through
\((x_0,y_0)\) if and only if
\begin{equation} \label{eq:tangent_condition}
    (x-x_0)f_1(\zeta_0)+(y-y_0)f_2(\zeta_0)=0.
\end{equation}
Since $\zeta_0=\Omega(x_0,y_0)$, we have
\[
x_0f_1(\zeta_0)+y_0f_2(\zeta_0)+f_3(\zeta_0)=0.
\]
Thus,~\eqref{eq:tangent_condition} is equivalent to
\[
x f_1(\zeta_0)+y f_2(\zeta_0)+f_3(\zeta_0)=0,
\]
\emph{i.e.} to \(\zeta_0\) being a zero of \(\lambda_c(x,y)\).
\end{proof}

Proposition~\ref{prop:tangent_zero_iff} translates the conditions on the zeroes of \(\lambda_c(x,y)\) in Definition~\ref{def:regions_in_limit_shape} into statements about tangent lines to \(\mathscr A\). The possible degenerations as we approach $\mathscr{A}$ may be summarized as follows.

\begin{enumerate}
\item As \((x,y)\) approaches the arctic curve from outside, two real tangents merge into a double tangent. When it crosses into the rough region, this double tangent splits into a \(\sigma\)-conjugate pair.

\item As \((x,y)\) approaches a cusp of \(\mathscr A\), three tangents merge into a triple tangent.

\item As $(x,y)$ approaches a turning point corresponding to \(\alpha\in\zz\) that is not a turning-point cusp, two tangents collide with the pole at \(\alpha\) and cancel, producing a simple zero.

\item At a turning-point cusp, a triple tangent collides with the pole, producing a double zero.
\end{enumerate}

\section{Asymptotics: local statistics and the limit shape}  \label{sec:loc_stats_limit_shape}

We now turn to the asymptotic analysis of the exact inverse Kasteleyn formula~\eqref{eq:formula_inverse_K}. The argument follows the general steepest-descent framework developed in~\cite{BB23}; see also~\cite{BN25}, with the necessary modifications for our general setting.

\subsection{Local statistics}\label{sec:steepest}

Fix an astroidal domain $\mathcal D_c$ and a sequence of AZ graphs
\(
\graphbetan, n = 1,2,\dots,
\)
such that
\[
c_{\bm\beta^n}(e)=n|e|_\ZZ c(e)+O(1),
\qquad e\in E(N),
\]
see Sections~\ref{sec:def_AZ_graphs} and~\ref{sec:def_Astroidal_domains} for definitions.
Let~$\compact\subset \graphpl$ be a finite subgraph, and let~$(j_n,k_n)\in \ZZ^2$ satisfy
 \[(j_n,k_n)=n(x,y)+\Ordo(1)\] for some~$(x,y) \in \domain^\circ$. Set~\[\compact_n:=(j_n,k_n)+\compact.\] Recall the phase regions~$\liq$,~$\mathcal F_{(\alpha,\beta)}$,~$\mathcal Q_{(\alpha,\beta)}$ and~$\mathcal G_j$ from Definition~\ref{def:regions_in_limit_shape}. We define \(A_{(x,y)}\) by
\[
A_{(x,y)}:=
\begin{cases}
(\alpha,\beta), & \text{if \((x,y)\) lies in a frozen or quasi-frozen region corresponding to }(\alpha,\beta)\subset A_0,\\
A_j, & \text{if \((x,y)\in \mathcal G_j\)},
\end{cases}
\]
where
\[
\curve(\RR) = A_0 \sqcup \left(\bigsqcup_{j=1}^g A_j\right)
\]
as in Section~\ref{sec:M-curves}.
\begin{definition}\label{def:zero}
We say that a point~$(x,y)\in \mathcal D_c^\circ$ lies in a \emph{phase region} if it belongs to a rough, smooth, frozen, or quasi-frozen region. For such a point~$(x,y)$, let
\begin{equation} \label{eq:zeta_*}
\zero=\zero(x,y) := \begin{cases} \Omega(x,y), &\text{if $(x,y) \in \liq$},\\
\text{any point of $A_{(x,y)}$}, &\text{otherwise}.
\end{cases}
\end{equation}
\end{definition}
Convergence of local correlations is captured by the following statement.
\begin{theorem}\label{thm:inverse_convergence}
{Assume that the angle function is periodic.} Let~$(x,y)\in \mathcal D_c^\circ$ lie in a phase region and let~$\zero$ be as in Definition~\ref{def:zero}. If \((x,y)\in \mathcal F_{(\alpha,\beta)}\), assume also that
\[
(x,y)\notin \line_{e(\alpha)}\cup \line_{e(\beta)}.
\] 
For a black vertex $\bl$ and a white vertex $\wh$ of $\compact$, let
\[
\bl_n:=\bl+(j_n,k_n),
\qquad
\wh_n:=\wh+(j_n,k_n)
\]
denote their translates in $\compact_n$. Then
\[
\kast^{-1}_{\bl_n,\wh_n}\rightarrow \mathsf A^{\zero}_{\bl,\wh}
\qquad\text{as }n\to\infty,
\]
uniformly over $\bl$ and $\wh$.
\end{theorem}

\begin{remark}
An equivalent way of stating Theorem~\ref{thm:inverse_convergence} without formulas is that the local statistics are described by the Gibbs measure of the corresponding slope; see Corollary~\ref{cor:local_fluctuations} below for a precise statement.
\end{remark}

\begin{remark}
If~$(x,y) \in \mathcal F_{(\alpha,\beta)} \cap (\line_{e(\alpha)} \cup \line_{e(\beta)})$, then we expect that $(x,y)$ lies either in the frozen region or on a frozen-frozen boundary. The analysis of these points requires additional technical arguments which we do not pursue here {(see Remark~\ref{rem:technical_lines} below)}; we expect that it does not lead to interesting new phenomena.
\end{remark}

The rest of the section is dedicated to the proof of Theorem~\ref{thm:inverse_convergence}. We begin by introducing notation of the relevant sign intervals. For $(x,y) \in \domain^\circ$, let $I_{(x,y),+}$ (resp. $I_{(x,y),-}$) denote the union of the two $+$ (resp. $-$) sign intervals for $\sign_{R}$ where $R$ is the cell containing $(x,y)$ as in Section~\ref{sec:action_one_form}.
Let~$A_{\cR,+}\subset A_0$ (resp. $A_{\cR,-}\subset A_0$) be the union of the two smallest intervals in~$A_0$ containing all angles~$\bm \alpha_e$ for all~$e\in I_{(x,y),+}$ (resp. $e\in I_{(x,y),-}$) and any $(x,y)\in \cR$. Note that if~\((x,y)\in \line_e\) for some edge \(e\), then some signs vanish, and hence,~$I_{(x,y),+}\cup I_{(x,y),-} \neq \partial N$. 

Similarly, for $\x \in B(\graphbeta) \sqcup W(\graphbeta)$, let $I_{{\x},+}$ (resp. $I_{\x,-}$) denote the union of the two $+$ (resp. $-$) sign intervals for $\sign_{\mathsf{R}_\x}$ as in Section~\ref{sec:chambers_and_sign_changes}.

\begin{lemma}\label{lem:asymp_intervals}
Let \(\cR\subset \domain^\circ\) be a cell, and let
\[
\x_n=\x+(j_n,k_n)\in B(\compact_n)\sqcup W(\compact_n),\qquad \x \in B(\compact)\sqcup W(\compact),
\]
be a sequence such that
\(
(j_n,k_n)=n(x,y)+O(1)
\) for some \((x,y)\in \cR\). Then, for any \(e\in E(N)\) such that $\sign_{\cR}(e)\neq 0$, we have
\(
\sign_{\dR_{\x_n}}(e)=\sign_{\cR}(e)
\) for all sufficiently large \(n\). In particular,
\[
I_{(x,y),+}\subseteq I_{{\x_n},+},
\qquad
I_{(x,y),-}\subseteq I_{{\x_n},-}
\]
for all sufficiently large \(n\).
\end{lemma}

\begin{proof}
Since \(\compact\) is a fixed finite subgraph,
\[
(\dabel(\x)-\dabel(\f_0))_\alpha=O(1), \quad \text{as } n\to \infty,
\]
for every \(\alpha\in\zz\). Hence,
\[
\bigl(-\dabel(\x_n)-\bm D_{\bm\beta^n}\bigr)_\alpha
=
\bigl(-\dabel(\f_0+(j_n,k_n))-\bm D_{\bm\beta^n}\bigr)_\alpha
-(\dabel(\x)-\dabel(\f_0))_\alpha.
\]
Dividing by \(n\) and using Lemma~\ref{lem:divisor_coefficients_asymptotic}, we obtain
\[
\frac1n\bigl(-\dabel(\x_n)-\bm D_{\bm\beta^n}\bigr)_\alpha
=
-b_e x+a_e y+c(e)+O\left(\frac1n\right)
\qquad\text{for all }\alpha\in\zz_e.
\]
Therefore, {recalling Lemma~\ref{lem:sign_poles_cts}}, {for all~$e\in E(N)$ such that} 
\[
{\sign_{\cR}(e)=\sign\bigr(-b_e x+a_e y+c(e)\bigl)\neq 0}
\]
{and} all sufficiently large \(n\),
\[
\sign\bigl((-\dabel(\x_n)-\bm D_{\bm\beta^n})_\alpha\bigr)
=
\sign\bigl(-b_e x+a_e y+c(e)\bigr)
=
\sign_{\cR}(e)
\qquad\text{for all }\alpha\in\zz_e.
\]
As in the proof of Lemma~\ref{lem:parallel_zz_same_sign}, this implies that \(\x_n\) lies in one of the strips on the corresponding side of \(\beta_e\), and hence,
\[
\sign_{\dR_{\x_n}}(e)=\sign_{\cR}(e)
\]
for all sufficiently large \(n\). The inclusions
\[
I_{(x,y),+}\subseteq I_{\x_n,+},
\qquad
I_{(x,y),-}\subseteq I_{\x_n,-}
\]
follow immediately.
\end{proof}

Fix~$(x,y)\in \domain^\circ$. We call a curve~$\gamma\subset \Sigma$ a \emph{steepest descent} (resp. \emph{ascent}) \emph{curve} if: 
\begin{itemize}
\item it passes through a zero~$\zeta_0$ of~$\lambda_c(x,y)$,
\item for~$\zeta\in \gamma$, we have\[\Im \int_{\zeta_0}^\zeta\lambda_c(x,y)=0,\] where the integral is taken along~$\gamma$,
\item the function~\[\gamma\ni \zeta\mapsto \Re \int_{\zeta_0}^\zeta\lambda_c(x,y)\] is locally decreasing (resp. locally increasing) as~$\zeta$ leaves~$\zeta_0$ along~$\gamma$. 
\end{itemize}

By Lemma~\ref{lem:third_kind_ind} a steepest descent/ascent curve is invariant (up to orientation) under the involution~$\sigma$. We therefore first describe the portion of the curve lying in~$\Sigma_+$ and postpone specifying its orientation until Lemma~\ref{lem:descent_curves} below. 

We say that~$\zeta_0\in \Sigma(\RR)$ is a \emph{local minimum} (resp. \emph{maximum}) if~$\zeta_0$ is a simple zero of~$\lambda_c(x,y)$ and a local minimum (resp. maximum) of the function \[\Sigma(\RR)\ni \zeta \mapsto \Re\int_{\zeta_0}^\zeta \lambda_c(x,y).\] 
Such a point is a saddle for~$\zeta \mapsto \Re\int_{\zeta_0}^\zeta \lambda_c(x,y)$: along the curve going perpendicular to $\Sigma(\RR)$, it is a local maximum (resp. minimum). Our convention is that the term local minimum (resp. maximum) refers to the behavior along $\curve(\RR)$.

\begin{lemma}\label{lem:steepest_descent_curves}
Let~$(x,y)\in \cR\subset \domain^\circ$ and assume that~$(x,y)$ lies in a phase region.
\begin{enumerate}
    \item If~$(x,y)$ is in the frozen, quasi-frozen, or smooth region, and~$\zeta_0\in A_{(x,y)}$ is a local minimum (resp. maximum), there exists a simple curve $C_{\zeta_0,\mathrm{desc}}^+\subset \Sigma_+$ (resp. $C_{\zeta_0,\mathrm{asc}}^+\subset \Sigma_+$) going between~$\zeta_0$ and~$A_{\cR,+}$ (resp. $A_{\cR,-}$), such that~$\Im\int_{\zeta_0}^\zeta\lambda_c(x,y)=0$ and~$\Re\int_{\zeta_0}^\zeta\lambda_c(x,y)$ is decreasing (resp. increasing) along the curve.
\item If~$(x,y) \in \liq$ and~$\zeta_0=\Omega(x,y)\in \Sigma_+^{\circ}$, there exists a simple curve~$C_{\zeta_0,\mathrm{desc}}^+\subset \Sigma_+$ (resp. $C_{\zeta_0,\mathrm{asc}}^+\subset \Sigma_+$) with both endpoints in \(A_{\cR,+}\) (resp. \(A_{\cR,-}\)), passing through \(\zeta_0\), such that~$\Im\int_{\zeta_0}^\zeta\lambda_c(x,y)=0$ and~$\Re\int_{\zeta_0}^\zeta\lambda_c(x,y)$ is decreasing (resp. increasing) along the curve as we move away from $\zeta_0$.
\end{enumerate}
Moreover, if \(\zeta_0\in A_{(x,y)}\) is a local minimum and \(\zeta_0'\in A_{(x,y)}\) is a local maximum, then
\[
C_{\zeta_0,\mathrm{desc}}^+\cap C_{\zeta_0',\mathrm{asc}}^+=\varnothing.
\]
If instead~$(x,y) \in \liq$ and~\(\zeta_0=\Omega(x,y)\in \Sigma_+^\circ\), then
\[
C_{\zeta_0,\mathrm{desc}}^+\cap C_{\zeta_0,\mathrm{asc}}^+=\{\zeta_0\}.
\]
\end{lemma}
See Figure~\ref{fig:ascent/descent_contours} below for examples of the curves.

\begin{proof}
We first consider the case~$\zeta_0\in \Sigma(\RR)$, that is,~$(x,y)$ is in the frozen, quasi-frozen, or smooth region. Since~$(x,y)\not \in \arctic$, the zero~$\zeta_0$ of~$\lambda_c(x,y)$ is simple. Hence, there is one level line of
\[
\Im \int_{\zeta_0}^{\zeta}\lambda_c(x,y)=0
\]
along~$\Sigma(\RR)$ and one level line entering~$\Sigma_+^{\circ}$. If~$\zeta_0$ is a local minimum (resp. maximum), the function
\begin{equation}\label{eq:real_part_of_action}
\zeta \mapsto \Re \int_{\zeta_0}^{\zeta}\lambda_c(x,y)
\end{equation}
is decreasing (resp. increasing) along the level line entering~$\Sigma_+^{\circ}$. Following that level line, we eventually reach either another zero of $\lambda_c(x,y)$ which has to be in $\Sigma(\RR)$, since $(x,y)\notin \liq$, and is a local maximum (resp. minimum) (recall that the maximum (resp. minimum) refers to the behavior along~$\Sigma(\RR)$)
or a point where~\eqref{eq:real_part_of_action} tends to~$-\infty$ (resp.~$+\infty$). In the latter case, the endpoint must be an angle~$\alpha\in \bm \alpha_e$ with~$e\in I_{\cR,+}$ (resp. $e\in I_{\cR,-}$), which means that~$\alpha\in A_{\cR,+}$ (resp. $\alpha\in A_{\cR,-}$). 

We define the curve~$C_{\zeta_0,\mathrm{desc}}^+\subset \Sigma_+$ (resp. $C_{\zeta_0,\mathrm{asc}}^+\subset \Sigma_+$) to be this level line, extended if necessary until it hits~$A_{\cR,+}$ (resp. $A_{\cR,-}$); see Figure~\ref{fig:ascent/descent_contours}(B)--(D) for an illustration. If it first hits another zero~$\zeta_0'$ of~$\lambda_c(x,y)$, then either:
\begin{enumerate}
\item $\zeta_0'\in A_{\cR,+}$ (resp. $\zeta_0'\in A_{\cR,-}$), and we are done.
\item $\zeta_0' \in A_{(x,y)}$, which is impossible{. Indeed,~$\zeta_0'$ would be a local maximum (resp. minimum) and since the types of extrema alternate along~$A_{(x,y)}$,~$\zeta_0$ and~$\zeta_0'$ would be joined by a steepest ascent (resp. descent) curve along $\Sigma(\RR)$ with no zero between them. This is a contradiction, since they are also joined by a steepest descent (resp. ascent) curve in $\Sigma_+^\circ$.}
\item $\zeta_0'\in A_j \neq A_{(x,y)}$, for some~$j=1,\dots,g$, in which case we continue the curve along~$A_j$ in either direction until it hits the other zero of~$\lambda_c(x,y)$~in~$A_j$ which is a local minimum (resp. maximum), and then continue again into~$\Sigma_+^{\circ}$. 
\end{enumerate}

Since a curve constructed by this procedure cannot return to the same~$A_j$ twice, it eventually has to end up at an angle~$\alpha\in A_{\cR,+}$ (resp. $\alpha\in A_{\cR,-}$) or a local maximum (resp. minimum) which lies in~$A_{\cR,+}$ (resp. $A_{\cR,-}$).

The resulting curve is simple because the function~\eqref{eq:real_part_of_action} is monotone along it (recall that this function is well-defined since $\lambda_c(x,y)$ is imaginary-normalized).
If~$\zeta_0\in A_{(x,y)}$ is a local minimum and~$\zeta_0'\in A_{(x,y)}$ is a local maximum, there cannot be another local minimum or maximum in $A_{(x,y)}$ between $\zeta_0$ and $\zeta_0'$. Thus, the same monotonicity argument shows that
\[
C_{\zeta_0,\mathrm{desc}}^+\cap C_{\zeta_0',\mathrm{asc}}^+=\varnothing.
\]

Now we consider the case~$\zeta_0\in \Sigma_+^{\circ}$, \emph{i.e.} $(x,y) \in \liq$. The local level set~$\Im\int_{\zeta_0}^\zeta \lambda_c(x,y)=0$ consists of two curves going through~$\zeta_0$ orthogonally. One of those curves is used to define~$C_{\zeta_0,\mathrm{desc}}^+\subset \Sigma_+$ and the other is used to define~$C_{\zeta_0,\mathrm{asc}}^+\subset \Sigma_+$. To construct~$C_{\zeta_0,\mathrm{desc}}^+\subset \Sigma_+$ (resp. $C_{\zeta_0,\mathrm{asc}}^+\subset \Sigma_+$), we follow the level line for which~$\zeta \mapsto \Re\int_{\zeta_0}^\zeta\lambda_c(x,y)$ is decreasing (resp. increasing); see Figure~\ref{fig:ascent/descent_contours}(A) for an illustration. The same argument as in the first case now proves the statement for this second case. 
\end{proof}

\begin{figure}
\centering

\tikzset{
  hollowedgem/.style={
    draw=black,
    line cap=round,
    line join=round,
    line width=0.5pt,
    double=white,
    double distance=1.5pt
  },
  solidedgem/.style={
    draw=black,
    line cap=round,
    line join=round,
    line width=2pt
  }
}

\begin{subfigure}{0.48\textwidth}
\centering

\begin{tikzpicture}[scale=0.17]
	
		\node  (1) at (12, 0) {};
		\node  (2) at (0, 12) {};
		\node  (3) at (0, -12) {};
		\node  (4) at (-12, 0) {};
		\node  (6) at (-2, 4) {};
		\node  (7) at (-5, 7) {};
		\node  (8) at (-5, 1) {};
		\node  (9) at (-8, 4) {};
		\node  (10) at (8, -3) {};
		\node  (11) at (3, 2) {};
		\node  (12) at (3, -8) {};
		\node  (13) at (-2, -3) {};
		\coordinate (14) at (2, 7.5);
		\coordinate (15) at (-3.25, 6.5);
		\coordinate (18) at (-6, 1.25);
		\coordinate (19) at (3, 2);
		\coordinate (22) at (0.5, -7.5);
		\coordinate (23) at (6, 10.5);
		\coordinate (24) at (0, -12);
		\coordinate (25) at (11.25, 4.25);
		\coordinate (26) at (-12, 0);

\draw (0,0) circle[radius=12];
\draw[solidedgem] (55:12) arc(55:80:12);
\draw[hollowedgem] (100:12) arc(100:215:12);
\draw[solidedgem] (235:12) arc(235:305:12);
\draw[hollowedgem] (325:12) arc(325:395:12);

		\draw [in=270, out=0] (8.center) to (6.center);
		\draw [in=0, out=90] (6.center) to (7.center);
		\draw [in=90, out=-180] (7.center) to (9.center);
		\draw [in=-180, out=-90] (9.center) to (8.center);
		\draw [in=270, out=0] (12.center) to (10.center);
		\draw [in=0, out=90] (10.center) to (11.center);
		\draw [in=90, out=-180] (11.center) to (13.center);
		\draw [in=-180, out=-90] (13.center) to (12.center);
		\draw [in=30, out=180, Orange, very thick] (14) to (15);
		\draw [in=-150, out=0, Orange, very thick] (14) to (25);
		\draw [in=-120, out=90, looseness=0.75, very thick, Cerulean] (14) to (23);
		\draw [in=90, out=-90, very thick, Cerulean] (14) to (19);
		\draw [in=0, out=-120, looseness=1.25, Orange, very thick] (18) to (26);
		\draw [in=90, out=-120, very thick, Cerulean] (22) to (24);

		\node[nvert, label=135:$\zeta_0$] at (14) {};
		\node[Greenvert, label={[font=\scriptsize]above:min}] at (15) {};
		\node[Greenvert, label={[font=\scriptsize]-70:max}] at (18) {};
		\node[Greenvert, label={[font=\scriptsize]45:max}] at (19) {};
		\node[Greenvert, label={[font=\scriptsize]45:min}] at (22) {};
		\node[nvert,label=above:$A_{R,+}$] at (23) {};
		\node[nvert,label=below:$A_{R,+}$] at (24) {};
		\node[nvert,label=right:$A_{R,-}$] at (25) {};
		\node[nvert,label=left:$A_{R,-}$] at (26) {};

\end{tikzpicture}

\caption*{(A)}
\end{subfigure}
\hfill
\begin{subfigure}{0.48\textwidth}
\centering

\begin{tikzpicture}[scale=0.17]

		\node  (1) at (12, 0) {};
		\node  (2) at (0, 12) {};
		\node  (3) at (0, -12) {};
		\node  (4) at (-12, 0) {};
		\node  (6) at (-2, 4) {};
		\node  (7) at (-5, 7) {};
		\node  (8) at (-5, 1) {};
		\node  (9) at (-8, 4) {};
		\node  (10) at (8, -3) {};
		\node  (11) at (3, 2) {};
		\node  (12) at (3, -8) {};
		\node  (13) at (-2, -3) {};
		\coordinate (16) at (-8, 4);
		\coordinate (17) at (-2.5, 5.75);
		\coordinate (23) at (9, 8);
		\coordinate (26) at (-11.25, -4);
		\node  (27) at (9, -6) {};
		\coordinate (28) at (6, 10.5);
		\coordinate (29) at (-11, 5);

\draw (0,0) circle[radius=12];
\draw[solidedgem] (90:12) arc(90:120:12);
\draw[hollowedgem] (140:12) arc(140:170:12);
\draw[solidedgem] (190:12) arc(190:295:12);
\draw[hollowedgem] (315:12) arc(315:360:12);

		\draw [in=270, out=0] (8.center) to (6.center);
		\draw [in=0, out=90] (6.center) to (7.center);
		\draw [in=90, out=-180] (7.center) to (9.center);
		\draw [in=-180, out=-90] (9.center) to (8.center);
		\draw [in=270, out=0] (12.center) to (10.center);
		\draw [in=0, out=90] (10.center) to (11.center);
		\draw [in=90, out=-180] (11.center) to (13.center);
		\draw [in=-180, out=-90] (13.center) to (12.center);
		\draw [in=30, out=-135, looseness=0.75, very thick, Cerulean] (23) to (26);
		\draw [in=30, out=-120, very thick, Orange] (28) to (17);
		\draw [in=345, out=180, very thick, Orange] (16) to (29.center);

		\node[Greenvert, label={[font=\scriptsize]-70:max}] at (16) {};
		\node[Greenvert, label={[font=\scriptsize]above:min}] at (17) {};
		\node[Greenvert, label={[font=\scriptsize]45:min}] at (23) {};
		\node[nvert,label=left:$A_{R,+}$] at (26) {};
		\node[Greenvert, label={[font=\scriptsize]45:max}] at (28) {};
		\node[nvert,label=left:$A_{R,-}$]at (29) {};
        \node[nvert,label=above:{$\beta$}] at (90:12) {};
        \node[nvert,label=right:{$\alpha$}] at (0:12) {};
	\coordinate (21) at (7.25, -0.25);
		\coordinate (22) at (0.5, -7.5);
        \node[Greenvert, label={[font=\scriptsize]10:min}] at (21) {};
		\node[Greenvert, label={[font=\scriptsize]45:max}] at (22) {};
\end{tikzpicture}

\caption*{(B)}
\end{subfigure}

\vspace{1em}

\begin{subfigure}{0.48\textwidth}
\centering

\begin{tikzpicture}[scale=0.17]
	
		\node  (1) at (12, 0) {};
		\node  (2) at (0, 12) {};
		\node  (3) at (0, -12) {};
		\node  (4) at (-12, 0) {};
		\node  (6) at (-2, 4) {};
		\node  (7) at (-5, 7) {};
		\node  (8) at (-5, 1) {};
		\node  (9) at (-8, 4) {};
		\node  (10) at (8, -3) {};
		\node  (11) at (3, 2) {};
		\node  (12) at (3, -8) {};
		\node  (13) at (-2, -3) {};
		\coordinate (16) at (-8, 4);
		\coordinate (17) at (-2.75, 6);
		\coordinate (21) at (7.25, -0.25);
		\coordinate (22) at (0.5, -7.5);
		\coordinate (23) at (9, 8);
		\coordinate (24) at (0, -12);
		\coordinate (25) at (11.25, 4.25);
		\coordinate (26) at (-12, 0);
		\node  (27) at (9, -6) {};
		\coordinate (28) at (6, 10.5);
		\coordinate (29) at (-11, 5);
	
\draw (0,0) circle[radius=12];
\draw[solidedgem] (-20:12) arc(-20:110:12);
\draw[hollowedgem] (140:12) arc(140:165:12);
\draw[solidedgem] (175:12) arc(175:215:12);
\draw[hollowedgem] (235:12) arc(235:305:12);

		\draw [in=270, out=0] (8.center) to (6.center);
		\draw [in=0, out=90] (6.center) to (7.center);
		\draw [in=90, out=-180] (7.center) to (9.center);
		\draw [in=-180, out=-90] (9.center) to (8.center);
		\draw [in=270, out=0] (12.center) to (10.center);
		\draw [in=0, out=90] (10.center) to (11.center);
		\draw [in=90, out=-180] (11.center) to (13.center);
		\draw [in=-180, out=-90] (13.center) to (12.center);
		\draw [in=90, out=-120, Orange, very thick]  (22) to (24);
		\draw [in=0, out=-135, Cerulean, very thick]  (23) to (26);
		\draw [in=45, out=-120, looseness=0.75, Orange, very thick] (28) to (17);
		\draw [in=345, out=180, Orange, very thick]  (16) to (29);
		\draw [in=30, out=-150, looseness=1.25, Orange, very thick]  (25) to (21);

		\node[Greenvert, label={[font=\scriptsize]-20:max}] at (16) {};
		\node[Greenvert, label={[font=\scriptsize]right:min}] at (17) {};
		\node[Greenvert, label={[font=\scriptsize]-20:min}] at (21) {};
		\node[Greenvert, label={[font=\scriptsize]45:max}] at (22) {};
		\node[Greenvert, label={[font=\scriptsize]45:min}] at (23) {};
		\node[nvert,label=below:$A_{R,-}$] at (24) {};
		\node[Greenvert, label={[font=\scriptsize]45:max}] at (25) {};
		\node[nvert,label=left:$A_{R,+}$] at (26) {};
		\node[Greenvert, label={[font=\scriptsize]45:max}] at (28) {};
		\node[nvert,label=left:$A_{R,-}$] at (29) {};

        \node[nvert,label=above:{$\beta$}] at (90:12) {};
        \node[nvert,label=right:{$\alpha$}] at (0:12) {};

\end{tikzpicture}

\caption*{(C)}
\end{subfigure}
\hfill
\begin{subfigure}{0.48\textwidth}
\centering

\begin{tikzpicture}[scale=0.17]
	
		\node  (1) at (12, 0) {};
		\node  (2) at (0, 12) {};
		\node  (3) at (0, -12) {};
		\node  (4) at (-12, 0) {};
		\node  (6) at (-2, 4) {};
		\node  (7) at (-5, 7) {};
		\node  (8) at (-5, 1) {};
		\node  (9) at (-8, 4) {};
		\node  (10) at (8, -3) {};
		\node  (11) at (3, 2) {};
		\node  (12) at (3, -8) {};
		\node  (13) at (-2, -3) {};
		\coordinate (15) at (-2.25, 3);
		\node  (16) at (-8, 4) {};
		\node  (17) at (-5, 7) {};
		\coordinate (18) at (-6, 1.25);
		\coordinate (19) at (-1, 0);
		\coordinate (20) at (7, 0.25);
		\coordinate (21) at (3.75, 2);
		\coordinate (22) at (0.5, -7.5);
		\coordinate (23) at (6, 10.5);
		\coordinate (24) at (0, -12);
		\coordinate (25) at (11.25, 4.25);
		\coordinate (26) at (-12, 0);
		\node  (27) at (9, -6) {};
	
\draw (0,0) circle[radius=12];
\draw[solidedgem] (55:12) arc(55:80:12);
\draw[hollowedgem] (100:12) arc(100:215:12);
\draw[solidedgem] (235:12) arc(235:305:12);
\draw[hollowedgem] (325:12) arc(325:395:12);

		\draw [in=270, out=0] (8.center) to (6.center);
		\draw [in=0, out=90] (6.center) to (7.center);
		\draw [in=90, out=-180] (7.center) to (9.center);
		\draw [in=-180, out=-90] (9.center) to (8.center);
		\draw [in=270, out=0] (12.center) to (10.center);
		\draw [in=0, out=90] (10.center) to (11.center);
		\draw [in=90, out=-180] (11.center) to (13.center);
		\draw [in=-180, out=-90] (13.center) to (12.center);
		\draw [in=0, out=-120, looseness=1.25, Orange, very thick] (18) to (26);
		\draw [in=90, out=-120, Cerulean, very thick] (22) to (24);
		\draw [in=-30, out=135, Orange, very thick] (19) to (15);
		\draw [in=-120, out=90, Cerulean, very thick] (21) to (23);
		\draw [in=-150, out=30, Orange, very thick] (20) to (25);

		\node[Greenvert, label={[font=\scriptsize]135:min}] at (15) {};
		\node[Greenvert, label={[font=\scriptsize]left:max}] at (18) {};
		\node[Greenvert, label={[font=\scriptsize]225:max}] at (19) {};
		\node[Greenvert, label={[font=\scriptsize]-50:max}] at (20) {};
		\node[Greenvert, label={[font=\scriptsize]below:min}] at (21) {};
		\node[Greenvert, label={[font=\scriptsize]45:min}] at (22) {};
		\node[nvert,label=above:$A_{R,+}$] at (23) {};
		\node[nvert,label=below:$A_{R,+}$] at (24) {};
		\node[nvert,label=right:$A_{R,-}$] at (25) {};
		\node[nvert,label=left:$A_{R,-}$] at (26) {};

\end{tikzpicture}

\caption*{(D)}
\end{subfigure}

\caption{
Steepest descent (blue) and steepest ascent (orange) curves in $\curve_+$ for a genus-$2$ example. The outer circle is $A_0$, while the inner circles are $A_1$ and $A_2$. {Green points indicate zeros of $\lambda_c(x,y)$. (A)--(D) correspond respectively to $(x,y)$ lying in the rough region, frozen region $\mathcal F_{(\alpha,\beta)}$ such that $\lambda_c(x,y)$ has two simple zeroes in $(\alpha,\beta)$, quasi-frozen region (when $e(\alpha) = e(\beta)$) or frozen region $\mathcal F_{(\alpha,\beta)}$ such that $\lambda_c(x,y)$ has three simple zeroes in $(\alpha,\beta)$ (when $e(\alpha)\neq e(\beta)$), and smooth region.} In keeping with our sign/color convention, the black and white intervals in $A_0$ correspond to $A_{R,+}$ and $A_{R,-}$ respectively. 
}
\label{fig:ascent/descent_contours}

\end{figure}
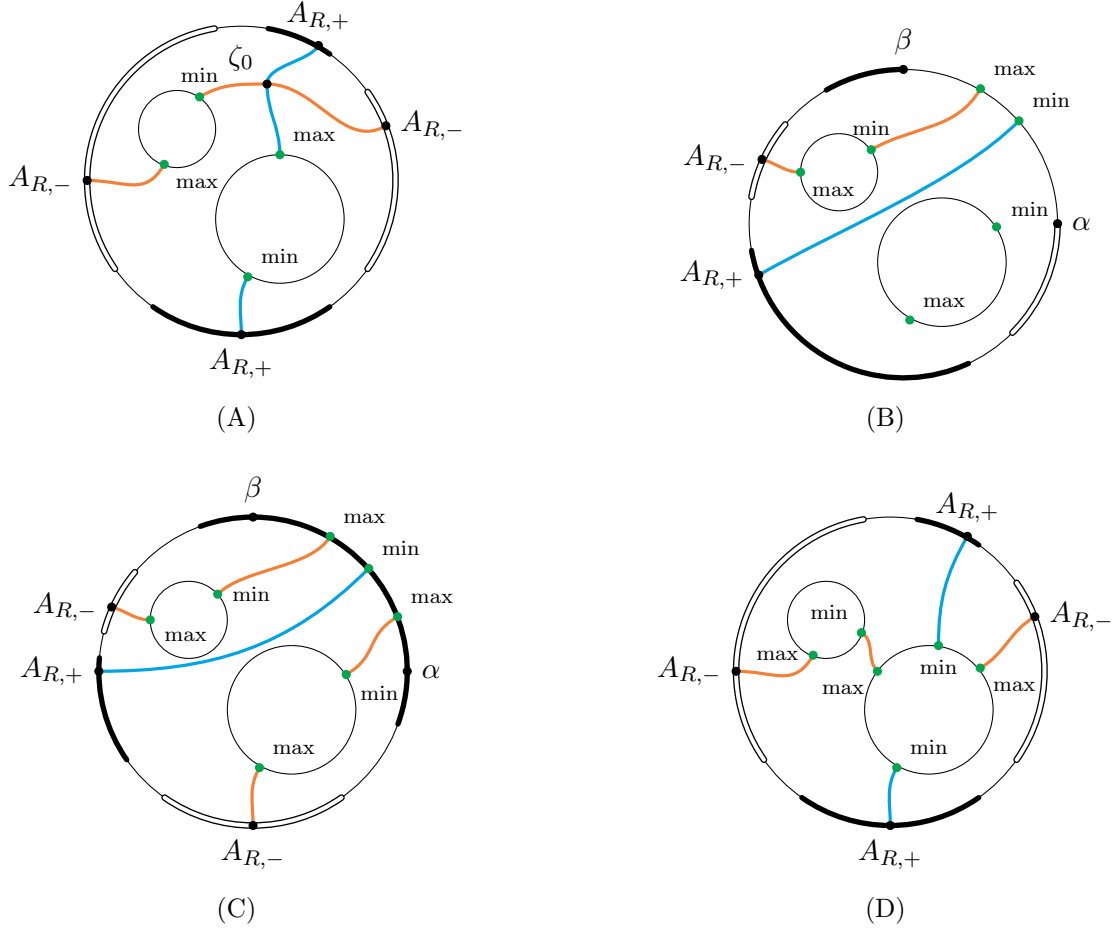

We now extend the curves defined in Lemma~\ref{lem:steepest_descent_curves} to $\curve$. Define $C_{\zeta_0,\mathrm{desc}}\subset \Sigma$ (resp. $C_{\zeta_0,\mathrm{asc}}\subset \Sigma$) to be the $\sigma$-invariant extension of $C_{\zeta_0,\mathrm{desc}}^+\subset \Sigma_+$ (resp. $C_{\zeta_0,\mathrm{asc}}^+\subset \Sigma_+$) such that 
\[\sigma_*(C_{\zeta_0,\mathrm{desc}})=-C_{\zeta_0,\mathrm{desc}} \qquad (\text{resp. }\sigma_*(C_{\zeta_0,\mathrm{asc}})=-C_{\zeta_0,\mathrm{asc}}).
\] 
Note that we have not specified the orientation of the curves yet, however, we assume they are given some orientation. Any segment of the curve along~$\Sigma(\RR)$ is traversed in both directions, so its contribution to the integral cancels. 
By Lemma~\ref{lem:steepest_descent_curves}, the curve~$C_{\zeta_0,\mathrm{desc}}$ (resp. $C_{\zeta_0,\mathrm{asc}}$) is a curve of steepest descent (resp. ascent) and intersects~$A_0$ in at most two points. 

We then define \(C_{(x,y),\mathrm{desc}}\) and \(C_{(x,y),\mathrm{asc}}\) by
\[
C_{(x,y),\mathrm{desc}}
:=
\begin{cases}
C_{\zeta_0,\mathrm{desc}}, & \text{if } \zeta_0=\Omega(x,y)\in \Sigma_+^\circ,\\
\displaystyle\bigsqcup_{\zeta_0\text{ local minimum in $A_{(x,y)}$}}
C_{\zeta_0,\mathrm{desc}},
& \text{otherwise,}
\end{cases}
\]
and
\[
C_{(x,y),\mathrm{asc}}
:=
\begin{cases}
C_{\zeta_0,\mathrm{asc}}, & \text{if } \zeta_0=\Omega(x,y)\in \Sigma_+^\circ,\\
\displaystyle\bigsqcup_{\zeta_0\text{ local maximum in $A_{(x,y)}$}}
C_{\zeta_0,\mathrm{asc}},
& \text{otherwise}.
\end{cases}
\]

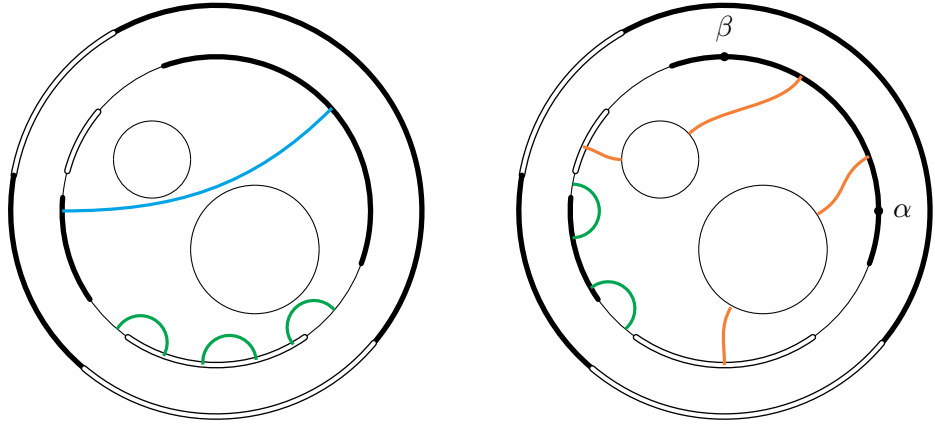
\begin{figure}
\centering

\tikzset{
  hollowedgem/.style={
    draw=black,
    line cap=round,
    line join=round,
    line width=0.5pt,
    double=white,
    double distance=1.5pt
  },
  solidedgem/.style={
    draw=black,
    line cap=round,
    line join=round,
    line width=2pt
  }
}

\begin{tikzpicture}[scale=0.17]
	
		\node  (1) at (12, 0) {};
		\node  (2) at (0, 12) {};
		\node  (3) at (0, -12) {};
		\node  (4) at (-12, 0) {};
		\node  (6) at (-2, 4) {};
		\node  (7) at (-5, 7) {};
		\node  (8) at (-5, 1) {};
		\node  (9) at (-8, 4) {};
		\node  (10) at (8, -3) {};
		\node  (11) at (3, 2) {};
		\node  (12) at (3, -8) {};
		\node  (13) at (-2, -3) {};
		\coordinate (16) at (-8, 4);
		\coordinate (17) at (-2.75, 6);
		\coordinate (21) at (7.25, -0.25);
		\coordinate (22) at (0.5, -7.5);
		\coordinate (23) at (9, 8);
		\coordinate (24) at (0, -12);
		\coordinate (25) at (11.25, 4.25);
		\coordinate (26) at (-12, 0);
		\node  (27) at (9, -6) {};
		\coordinate (28) at (6, 10.5);
		\coordinate (29) at (-11, 5);
	
\draw (0,0) circle[radius=12];
\draw[solidedgem] (-20:12) arc(-20:110:12);
\draw[hollowedgem] (140:12) arc(140:165:12);
\draw[solidedgem] (175:12) arc(175:215:12);
\draw[hollowedgem] (235:12) arc(235:305:12);

\draw[solidedgem] (-40:16) arc(-40:120:16);
\draw[hollowedgem] (120:16) arc(120:170:16);
\draw[solidedgem] (170:16) arc(170:230:16);
\draw[hollowedgem] (230:16) arc(230:320:16);

		\draw [in=270, out=0] (8.center) to (6.center);
		\draw [in=0, out=90] (6.center) to (7.center);
		\draw [in=90, out=-180] (7.center) to (9.center);
		\draw [in=-180, out=-90] (9.center) to (8.center);
		\draw [in=270, out=0] (12.center) to (10.center);
		\draw [in=0, out=90] (10.center) to (11.center);
		\draw [in=90, out=-180] (11.center) to (13.center);
		\draw [in=-180, out=-90] (13.center) to (12.center);
		\draw [in=0, out=-135, Cerulean, very thick]  (23) to (26);

\coordinate (P) at (230:12);
\coordinate (Q) at (250:12);

\draw[Green, very thick]
  let
    \p1 = (P),
    \p2 = (Q),
    \p3 = ($(P)!0.5!(Q)$),
    \n1 = {veclen(\x1-\x2,\y1-\y2)/2},
    \n2 = {atan2(\y1-\y3,\x1-\x3)},
    \n3 = {\n2-180}
  in
    (P) arc[start angle=\n2, end angle=\n3, radius=\n1];

\coordinate (P) at (300:12);
\coordinate (Q) at (320:12);

\draw[Green, very thick]
  let
    \p1 = (P),
    \p2 = (Q),
    \p3 = ($(P)!0.5!(Q)$),
    \n1 = {veclen(\x1-\x2,\y1-\y2)/2},
    \n2 = {atan2(\y1-\y3,\x1-\x3)},
    \n3 = {\n2-180}
  in
    (P) arc[start angle=\n2, end angle=\n3, radius=\n1];

\coordinate (P) at (265:12);
\coordinate (Q) at (285:12);

\draw[Green, very thick]
  let
    \p1 = (P),
    \p2 = (Q),
    \p3 = ($(P)!0.5!(Q)$),
    \n1 = {veclen(\x1-\x2,\y1-\y2)/2},
    \n2 = {atan2(\y1-\y3,\x1-\x3)},
    \n3 = {\n2-180}
  in
    (P) arc[start angle=\n2, end angle=\n3, radius=\n1];

\end{tikzpicture} \hspace{5mm} \begin{tikzpicture}[ scale=0.17]
	
		\node  (1) at (12, 0) {};
		\node  (2) at (0, 12) {};
		\node  (3) at (0, -12) {};
		\node  (4) at (-12, 0) {};
		\node  (6) at (-2, 4) {};
		\node  (7) at (-5, 7) {};
		\node  (8) at (-5, 1) {};
		\node  (9) at (-8, 4) {};
		\node  (10) at (8, -3) {};
		\node  (11) at (3, 2) {};
		\node  (12) at (3, -8) {};
		\node  (13) at (-2, -3) {};
		\coordinate (16) at (-8, 4);
		\coordinate (17) at (-2.75, 6);
		\coordinate (21) at (7.25, -0.25);
		\coordinate (22) at (0.5, -7.5);
		\coordinate (23) at (9, 8);
		\coordinate (24) at (0, -12);
		\coordinate (25) at (11.25, 4.25);
		\coordinate (26) at (-12, 0);
		\node  (27) at (9, -6) {};
		\coordinate (28) at (6, 10.5);
		\coordinate (29) at (-11, 5);
	
\draw (0,0) circle[radius=12];
\draw[solidedgem] (-20:12) arc(-20:110:12);
\draw[hollowedgem] (140:12) arc(140:165:12);
\draw[solidedgem] (175:12) arc(175:215:12);
\draw[hollowedgem] (235:12) arc(235:305:12);

		\draw [in=270, out=0] (8.center) to (6.center);
		\draw [in=0, out=90] (6.center) to (7.center);
		\draw [in=90, out=-180] (7.center) to (9.center);
		\draw [in=-180, out=-90] (9.center) to (8.center);
		\draw [in=270, out=0] (12.center) to (10.center);
		\draw [in=0, out=90] (10.center) to (11.center);
		\draw [in=90, out=-180] (11.center) to (13.center);
		\draw [in=-180, out=-90] (13.center) to (12.center);
		\draw [in=90, out=-120, Orange, very thick]  (22) to (24);
		\draw [in=45, out=-120, looseness=0.75, Orange, very thick] (28) to (17);
		\draw [in=345, out=180, Orange, very thick]  (16) to (29);
		\draw [in=30, out=-150, looseness=1.25, Orange, very thick]  (25) to (21);
        \node[nvert,label=above:{$\beta$}] at (90:12) {};
        \node[nvert,label=right:{$\alpha$}] at (0:12) {};

\draw[solidedgem] (-40:16) arc(-40:120:16);
\draw[hollowedgem] (120:16) arc(120:170:16);
\draw[solidedgem] (170:16) arc(170:230:16);
\draw[hollowedgem] (230:16) arc(230:320:16);

\coordinate (P) at (170:12);
\coordinate (Q) at (190:12);

\draw[Green, very thick]
  let
    \p1 = (P),
    \p2 = (Q),
    \p3 = ($(P)!0.5!(Q)$),
    \n1 = {veclen(\x1-\x2,\y1-\y2)/2},
    \n2 = {atan2(\y1-\y3,\x1-\x3)},
    \n3 = {\n2-180}
  in
    (P) arc[start angle=\n2, end angle=\n3, radius=\n1];

\coordinate (P) at (210:12);
\coordinate (Q) at (230:12);

\draw[Green, very thick]
  let
    \p1 = (P),
    \p2 = (Q),
    \p3 = ($(P)!0.5!(Q)$),
    \n1 = {veclen(\x1-\x2,\y1-\y2)/2},
    \n2 = {atan2(\y1-\y3,\x1-\x3)},
    \n3 = {\n2-180}
  in
    (P) arc[start angle=\n2, end angle=\n3, radius=\n1];    

\end{tikzpicture}
\caption{
{
On the left-hand side, the union of the two black (resp. white) intervals in the outer circle corresponds to the angles in $I_{\bl_n,+}$ (resp. $I_{\bl_n,-}$), while union of the two black (resp. white) intervals in $A_0$ corresponds to $A_{R,+}$ (resp. $A_{R,-}$) as in Figure~\ref{fig:ascent/descent_contours}. The steepest descent curve (blue) in Figure~\ref{fig:ascent/descent_contours}(C) can be deformed to the curve $C(I_{\bl_n})$ (green) without crossing any angles in $I_{\bl_n,-}$. Similarly, the right-hand side shows that the steepest ascent curve (orange) can be deformed to $C(I_{\wh_n})$ (green) without crossing any angles in $I_{\wh_n,+}$ (note that $\alpha$ and $\beta$ are consecutive angles).
}
}\label{fig:deformation_sd_sa}
\end{figure}

The following lemma shows that the contours appearing in the inverse Kasteleyn formula~\eqref{eq:formula_inverse_K} may be deformed to the steepest descent and steepest ascent curves respectively without crossing the poles of the integrand as specified in the following claim {(see Figure~\ref{fig:deformation_sd_sa}).}

\begin{lemma}\label{lem:descent_curves}
Let~$(x,y)\in \domain^\circ$ lie in a phase region, and let~$\zeta_0=\zero(x,y)$. If~$(x,y)\in \mathcal F_{(\alpha,\beta)}$, where~$(\alpha,\beta)\subset A_0$ is the interval between consecutive angles~$\alpha,\beta$, assume that~$(x,y)\notin \line_{e(\alpha)}\cup \line_{e(\beta)}$. Then we may choose orientations of~$C_{(x,y),\mathrm{desc}}$ and~$C_{(x,y),\mathrm{asc}}$ such that, for all sufficiently large~$n$ and all
\(
\bl_n\in B(\compact_n),
\wh_n\in W(\compact_n),
\)
the curve~$C_{(x,y),\mathrm{desc}}$ is {homologous} to~$C(I_{\bl_n})$ in
\[
\Sigma\setminus \bigcup_{e\in I_{\bl_n,-}} \bm\alpha_e,
\]
and the curve~$C_{(x,y),\mathrm{asc}}$ is {homologous} to~$C(I_{\wh_n})$ in
\[
\Sigma\setminus \bigcup_{e\in I_{\wh_n,+}} \bm\alpha_e.
\]
\end{lemma}

\begin{proof}
We distinguish cases according to the location of~$(x,y)$.
\begin{enumerate}
\item Assume first that~$(x,y)\in \mathcal G_j$ for some~$j=1,\dots,g$, so that~$A_{(x,y)}=A_j$. The curves~$C_{(x,y),\mathrm{desc}}$ and~$C_{(x,y),\mathrm{asc}}$ divide~$\Sigma$ into four parts. The four zeros of~$\lambda_c(x,y)$ on $A_j$ alternate between local minima and local maxima. By Lemma~\ref{lem:steepest_descent_curves},~$C_{(x,y),\mathrm{desc}}$ and~$C_{(x,y),\mathrm{asc}}$ do not intersect, and so the same lemma implies that~$C_{(x,y),\mathrm{desc}}$ intersects~$A_0$ once in each component of~$A_{\cR,+}$, and~$C_{(x,y),\mathrm{asc}}$ intersects~$A_0$ once in each component of~$A_{\cR,-}$. 
\item If~$(x,y)\in \liq$, so that~$\Omega(x,y)\in \Sigma_+^{\circ}$, a similar argument shows that~$C_{(x,y),\mathrm{desc}}$ intersects~$A_0$ once in each component of~$A_{\cR,+}$, and~$C_{(x,y),\mathrm{asc}}$ intersects~$A_0$ once in each component of~$A_{\cR,-}$. 
\item Finally, if~$(x,y)\in \mathcal F_{(\alpha,\beta)}$ or~$(x,y)\in \mathcal Q_{(\alpha,\beta)}$ for some $(\alpha,\beta)\subset A_0$, there are two possibilities:
\begin{enumerate}
\item If~$A_{(x,y)} \subset A_{\cR,+}$ (resp. $A_{(x,y)} \subset A_{\cR,-}$), then there is one local minimum (resp. local maximum) and two local maxima (resp. local minima) in $A_{(x,y)}$, and the types of extrema alternate along~$A_{(x,y)}$. By the same argument as above, the curve~$C_{(x,y),\mathrm{desc}}$ (resp. $C_{(x,y),\mathrm{asc}}$) goes between~$A_{(x,y)}$ and the opposite component of~$A_{\cR,+}$ (resp. $A_{\cR,-}$), and~$C_{(x,y),\mathrm{asc}}$ (resp. $C_{(x,y),\mathrm{desc}}$) consists of two components, going between the two components of~$A_{\cR,-}$ (resp. $A_{\cR,+}$) and~$A_{(x,y)}$. 
\item If~$A_{(x,y)}$ is contained in neither~$A_{\cR,+}$ nor~$A_{\cR,-}$, then there is one local minimum and one local maximum in~$A_{(x,y)}$. By the same argument as above,~$C_{(x,y),\mathrm{desc}}$ (resp. $C_{(x,y),\mathrm{asc}}$) joins~$A_{(x,y)}$ and the opposite component of~$A_{\cR,+}$ (resp. $A_{\cR,-}$). 
\end{enumerate}
\end{enumerate}

In each case, the description of the curves given above together with Lemma~\ref{lem:asymp_intervals} implies that, if we choose the orientation correctly,~$C_{(x,y),\mathrm{desc}}$ (resp. $C_{(x,y),\mathrm{asc}}$) can be deformed to~$C(I_{\bl_n})$ (resp. $C(I_{\wh_n})$) in~$\Sigma \backslash \cup_{e\in I_{{\bl_n},-}}\bm \alpha_e$ (resp. $\Sigma \backslash \cup_{e\in I_{{\wh_n},+}}\bm \alpha_e$); see Figure~\ref{fig:ascent/descent_contours}.

\end{proof}

\begin{remark}\label{rem:technical_lines}
{Note that if~\((x,y)\) lies in a frozen region \(\mathcal F_{(\alpha,\beta)}\) at a concave vertex as in Figure~\ref{fig:local-configurations}(B), the steepest descent and ascent curves are fundamentally different depending on which side of the dashed line 
the point \((x,y)\) lies on; see Figure~\ref{fig:ascent/descent_contours}(B)--(C). When \((x,y)\in \mathcal F_{(\alpha,\beta)}\cap\left(\line_{e(\alpha)}\cup \line_{e(\beta)}\right)\), the analysis therefore becomes sensitive to the precise locations of \(\bl_n\) and \(\wh_n\). This is the reason that we exclude this case in our analysis.}
\end{remark}

\begin{proof}[Proof of Theorem~\ref{thm:inverse_convergence}]
We use the convention in Remark~\ref{rem:inverse_kasteleyn_signs}, so that
\begin{equation}\label{eq:inverse_n}
\kast_{\bl_n,\wh_n}^{-1}=\iint_{C(I_{\wh_n})\prec C(I_{\bl_n})}\omega_{\bl_n,\wh_n} - \langle \partial I_{\bl_n}\cap I_{\wh_n},\mathsf{A}_{\bl_n,\wh_n}\rangle.
\end{equation} 
Let~$I_{\bl_n}=[\ell_{\bl_n},r_{\bl_n}]$. Then, for sufficiently large $n$,
\begin{equation}
- \langle \partial I_{\bl_n}\cap I_{\wh_n},\mathsf{A}_{\bl_n,\wh_n}\rangle
=\mathbf 1_{\{\ell_{\bl_n}\in I_{\wh_n}\}}\mathsf A^{\ell_{\bl_n}}_{\bl_n,\wh_n}
-\mathbf 1_{\{r_{\bl_n}\in I_{\wh_n}\}}\mathsf A^{r_{\bl_n}}_{\bl_n,\wh_n}
=\mathsf{A}_{\bl_n,\wh_n}^{\ell_{\bl_n}}.
\end{equation}
{Indeed, by the choice of~$I_{\wh_n}$ and~$I_{\bl_n}$ in Remark~\ref{rem:inverse_kasteleyn_signs}, and by Lemma~\ref{lem:asymp_intervals},~$r_{\bl_n}\notin I_{\wh_n}$. If~$\ell_{\bl_n}\notin I_{\wh_n}$, $\sign_{\dR_{\bl_n}}(e)=+$ and~$\sign_{\dR_{\wh_n}}(e)=-$, where~$e$ is the edge immediately to the left of~$\ell_{\bl_n}$, so~$\ell_{\bl_n}\in \sector_{\bl_n,\wh_n}$ by Lemma~\ref{lem:sector}, and, hence,~$\mathsf{A}_{\bl_n,\wh_n}^{\ell_{\bl_n}}=0$.}


By Lemma~\ref{lem:descent_curves}, we may deform~$C(I_{\wh_n})$ to~$C_{(x,y),\mathrm{asc}}$ and~$C(I_{\bl_n})$ to~$C_{(x,y),\mathrm{desc}}$ in the double contour integral in~\eqref{eq:inverse_n} without crossing any poles of $\omega_{\bl_n,\wh_n}$ other than the diagonal $\zeta = \eta$. Thus, the contribution of this deformation is precisely the diagonal residue along the curve~$C_{\ell_{\bl_n}}^{\zero}$ defined in Definition~\ref{def:contours_whole_plane}. Therefore,
\begin{equation}
\kast_{\bl_n,\wh_n}^{-1}=\int_{C_{(x,y),\mathrm{asc}}}\int_{C_{(x,y),\mathrm{desc}}}\omega_{\bl_n,\wh_n} + \frac{1}{2 \pi \i} \int_{C_{\ell_{\bl_n}}^{\zero}}g_{\bl_n,\wh_n}+\mathsf{A}_{\bl_n,\wh_n}^{\ell_{\bl_n}}.
\end{equation} 
Since
\begin{equation}
\frac{1}{2 \pi \i}  \int_{C_{\ell_{\bl_n}}^{\zero}}g_{\bl_n,\wh_n}+\mathsf{A}_{\bl_n,\wh_n}^{\ell_{\bl_n}}=\mathsf{A}_{\bl_n,\wh_n}^{\zero} = \mathsf{A}_{\bl,\wh}^{\zero},
\end{equation}
it remains to prove that the double contour integral tends to zero as~$n\to\infty$.

Recall the action function \(\Lambda_c( \cdot ;x,y)\) from Definition~\ref{def:action_function}. Expanding all the factors, we write the kernel~\eqref{eq:kernel} as
\begin{equation}\label{eq:form_double_integral} \omega_{\bl_n,\wh_n}=\frac{1}{(2\pi\i)^2}\frac{\theta(-t+\bm D_{\bm\beta^n}+\eta-\zeta)\theta(t+\eta+\dabel(\wh_n))\theta(-t+\zeta-\dabel(\bl_n))}{E(\zeta,\eta)\theta(-t+\bm D_{\bm \beta^n})} \frac{\prod_{\alpha\in \bm \alpha}E(\alpha,\eta)^{\left(\dabel(\wh_n)+\bm D_{\bm \beta^n}\right)_\alpha}}{\prod_{\alpha\in \bm \alpha}E(\alpha,\zeta)^{\left(\dabel(\bl_n)+\bm D_{\bm \beta^n}\right)_\alpha}}. \end{equation}
By~\eqref{eq:prime_form_action_asymptotic}, the prime-form factor in~\eqref{eq:form_double_integral} is
\[
\exp \Bigl(
n\bigl(\Lambda_c(\eta;x,y)-\Lambda_c(\zeta;x,y)\bigr)+O(1)
\Bigr).
\]
Since \(d\Lambda_c(\cdot;x,y)=\lambda_c(x,y)\) and \(\lambda_c(x,y)\) is imaginary-normalized, the real part \(\Re\Lambda_c(\cdot;x,y)\) is well-defined on \(\curve\), and
\begin{equation}
\Re \Lambda_c(\zeta;x,y)-\Re \Lambda_c(\eta;x,y)=\Re \int_{\zeta_0}^\zeta\lambda_c(x,y)-\Re \int_{\zeta_0}^\eta\lambda_c(x,y).
\end{equation}

It follows that the curves \(C_{(x,y),\mathrm{asc}}\) and \(C_{(x,y),\mathrm{desc}}\) are steepest ascent and descent contours for the action function \(\Lambda_c(\cdot;x,y)\). Standard arguments (that go back to~\cite{Oko03}), using that the first factor in~\eqref{eq:form_double_integral} is bounded in~$n$ since the ovals are compact, show that 
\begin{equation}
\int_{C_{(x,y),\mathrm{asc}}}\int_{C_{(x,y),\mathrm{desc}}}\omega_{\bl_n,\wh_n}\to 0
\end{equation}
as~$n\to \infty$, which proves the theorem. 
\end{proof}

\subsection{The limit shape}\label{sec:limit_shape}
Let~$h_M$ be the height function corresponding to a dimer cover $M$ defined on~$\graphbetan$ with the extremal dimer cover~$M_0:=M_{u_0}$ corresponding to~$u_0\in V(N)$ as reference dimer cover. We choose the reference face~$\f_0\in F(\graphpl)\backslash F(\graphbetan)$ to be adjacent to~$\e_{u_0}\in E(\graphbetan)$, where~$\e_{u_0}$ is the intersection of two boundary zig-zag paths as in Section~\ref{sec:def_AZ_graphs}. The limit shape is the appropriate scaled large~$n$ limit of the height function: Let~$(j_n,k_n)=n(x,y)+\Ordo(1)$ for some~$(x,y) \in \domain^\circ$ and set~$\f_n=\f_0+(j_n,k_n)\in F(\graphbetan)$. The limit shape~$\ls$ is defined as
\begin{equation}
\ls(x,y):=\lim_{n\to\infty}\EE\left[\frac{1}{n}h_M(\f_n)\right].
\end{equation}
For dimer models with periodic edge weights, it is well known that the limit shape exists, is the solution to a variational problem, and that~$\frac{1}{n}h_M\to \ls$ in probability as~$n\to \infty$~\cite{CKP00, Gor21, KOS06, Kuc17}. This is also expected to be true for Fock's weights with periodic angles~\cite{BBS24}. 

We now give an integral formula for the limit shape. 
\begin{theorem}\label{thm:limit_shape}
{Assume that the angle function is periodic.}
For~$(x,y)\in \domain^\circ~$, 
\begin{equation}\label{eq:height_fn}
\ls(x,y)=-\frac{1}{2\pi\i}\int_{C_{u_0}^{\zero}}\lambda_c(x,y),
\end{equation}
where~$C_{u_0}^{\zero}$ is the curve going from~$\sigma(\zero)$ to~$\zero$ crossing~$A_0$ only once in the component of~$A_0\backslash \zz$ corresponding to~$u_0$ as in Definition~\ref{def:contours_whole_plane}. The point~$\zero\in \Sigma_+$ is defined as in Definition~\ref{def:zero} if~$(x,y)$ lies in a phase region and defined by continuity otherwise.
\end{theorem}

The proof starts by rewriting the expectation of the height function in a form that is suitable for steepest descent analysis, similar to the one carried out in Section~\ref{sec:steepest}. 
The key identity that allows for this simplification in our setting is given in the following proposition. This identity was proved in~\cite{BBRdT26} in the study of perfect t-embeddings of Aztec diamonds with Fock's weights. For the convenience of the reader, we include the proof from~\cite{BBRdT26} in Appendix~\ref{appendix:telescopic}.

\begin{proposition}[\cite{BBRdT26}]\label{prop:telescopic}
Let~$\f_+,\f_-\in F(\graphbetan)$ be adjacent faces with~$\wh\bl\in E(\graphbetan)$ as their common adjacent edge, and let~$\alpha$ and~$\beta$ be the strands containing~$\wh\bl$, as in Figure~\ref{fig:fock_kast}. Then
\begin{equation}
E(\zeta,\eta)\frac{g_{\wh}(\eta)}{g_{\bl}(\zeta)}\kast_{\wh,\bl}=G_{\f_+}(\zeta,\eta)-G_{\f_-}(\zeta,\eta),
\end{equation} 
where
\begin{equation}
G_{\f}(\zeta,\eta):=\frac{\theta(-t-\dabel(\f)+\zeta-\eta)}{\theta(t+\dabel(\f))}\frac{g_{\f}(\eta)}{g_{\f}(\zeta)},
\end{equation}
for all faces~$\f\in F(\graphbetan)$ and all~$\zeta,\eta\in \Sigma$.
\end{proposition}
\begin{proof}[Proof of Theorem~\ref{thm:limit_shape}]
We will consider points~$(x,y)\in \domain^\circ$ satisfying the following assumption.

\begin{assumption}\label{ass:limit}
Let~$(x,y)\in \mathcal D_c^\circ$ be a point in a phase region. If $(x,y)\in \mathcal F_{(\alpha,\beta)}$, assume also that $(x,y)\notin \line_{e(\alpha)}\cup \line_{e(\beta)}.$
\end{assumption}

In particular, we will prove the limit
\begin{equation}
\lim_{n\to\infty}\EE\left[\frac{1}{n}h_M(\f_n)\right]=-\frac{1}{2\pi\i}\int_{C_{u_0}^{\zero}}\lambda_c(x,y),
\end{equation}
for~$(x,y)$ satisfying Assumption~\ref{ass:limit}. 
Since~$\ls$ is a continuous function--indeed, Lipschitz continuous--the statement follows.

By definition of the height function; see Section~\ref{sec:height_function}, and by Theorems~\ref{thm:kenyon_inv_K} and~\ref{thm:gibbs_measure}, we have
\begin{equation}\label{eq:height_sum}
\EE\left[h_M(\f_n)\right]=\sum_{\e\in E(\graphbetan)}(\e \wedge \gamma_n^*)\EE[\one_{\e\in M}-\one_{\e\in M_0}]
=\sum_{\e=\bl\wh\in E(\graphbetan)}(\e \wedge \gamma_n^*)\left(\kast^{-1}_{\bl,\wh}\kast_{\wh,\bl}-\mathsf{A}_{\bl,\wh}^{u_0}\kast_{\wh,\bl}\right),
\end{equation}
where~$\gamma_n^*$ is a dual path in~$\graphbetan^*$ from~$\f_0$ to~$\f_n$, and~$\mathsf{A}_{\bl,\wh}^{u_0}$ is the whole-plane inverse Kasteleyn matrix defining the Gibbs measure associated with~$u_0$.

We will first compute the limit of the expected height difference between two faces in the same chamber. 
We use the convention for~$\kast_{\bl,\wh}^{-1}$ from Remark~\ref{rem:inverse_kasteleyn_signs}. Let~$\e=\bl\wh$ and let~$I_{\bl}=[\ell_{\bl},r_\bl]$. By Proposition~\ref{prop:telescopic} and definition of~$\mathsf A^{\zeta}_{\bl,\wh}$,
\begin{multline}\label{eq:height_one_step}
\EE[\one_{\e\in M}-\one_{\e\in M_0}]
=\iint_{C(I_{\wh})\prec C(I_{\bl})}\omega_{\bl,\wh}\kast_{\wh,\bl}+\left(\mathsf A^{\ell_{\bl}}_{\bl,\wh}-\mathsf A^{u_0}_{\bl,\wh}\right)\kast_{\wh,\bl}
=\frac{1}{2\pi\i}\int_{C([\ell_{\bl},u_0])}g_{\bl,\wh}\kast_{\wh,\bl} \\
+\frac{1}{(2\pi\i)^2}\iint_{C(I_{\wh})\prec C(I_{\bl})}\frac{\theta(-t+\bm D_{\bm \beta^n}+\eta-\zeta)}{\theta(-t+\bm D_{\bm \beta^n})E(\zeta,\eta)^2}\left(G_{\f_+}(\zeta,\eta)-G_{\f_-}(\zeta,\eta)\right)\frac{\Psi_{\bm \beta}(\eta)}{\Psi_{\bm \beta}(\zeta)},
\end{multline}
where~$\f_-,\f_+\in F(\graphbetan)$ are as in Figure~\ref{fig:fock_kast}.

More generally, let $\f_-,\f_+\in F(\graphbetan)$ be any two faces in the same chamber~$\dR$, and let~$\gamma^*$ be a dual path in~$\dR$ going from~$\f_-$ to~$\f_+$ keeping white vertices to the left. Then 
\begin{multline}\label{eq:height_change_chamber}
\EE\left[h_M(\f_+)-h_M(\f_-)\right]=\sum_{\e\in E(\dR)}(\e \wedge \gamma^*)\EE[\one_{\e\in M}-\one_{\e\in M_0}]
=\sum_{\e=\bl\wh\in E(\dR)}\frac{(\e \wedge \gamma^*)}{2\pi\i}\int_{C([\ell_{\bl_{\dR}},u_0])}g_{\bl,\wh}\kast_{\wh,\bl} \\
+\frac{1}{(2\pi\i)^2}\iint_{C(I_{\wh_{\dR}})\prec C(I_{\bl_{\dR}})}\frac{\theta(-t+\bm D_{\bm \beta^n}+\eta-\zeta)}{\theta(-t+\bm D_{\bm \beta^n})E(\zeta,\eta)^2}\left(G_{\f_+}(\zeta,\eta)-G_{\f_-}(\zeta,\eta)\right)\frac{\Psi_{\bm \beta}(\eta)}{\Psi_{\bm \beta}(\zeta)},
\end{multline}
where~$\bl_{\dR}\in B(\dR)$ and~$\wh_{\dR}\in W(\dR)$ are fixed vertices, and we use that the $G_{\f}(\zeta,\eta)$-terms telescope.

We consider first the single integral. If~$\e=\bl \wh \in E(\dR)$, then Lemma~\ref{lem:fays} implies that
\begin{equation}
\frac{1}{2\pi\i}\int_{C([\ell_{\bl_{\dR}},u_0])}g_{\bl,\wh}\kast_{\wh,\bl}=\frac{1}{2\pi\i}\int_{C([\ell_{\bl_{\dR}},u_0])}\omega_{\beta-\alpha},
\end{equation}
where~$\alpha$ and~$\beta$ are the two strands containing~$\e$ as in Figure~\ref{fig:fock_kast}. Summing over~$\gamma^*$ yields
\begin{equation}\label{eq:slope_chamber}
\sum_{\e=\bl\wh\in E(\dR)}\frac{(\e \wedge \gamma^*)}{2\pi\i}\int_{C([\ell_{\bl_{\dR}},u_0])}g_{\bl,\wh}\kast_{\wh,\bl}
=\frac{1}{2\pi\i}\int_{C([\ell_{\bl_{\dR}},u_0])}\omega_{\dabel(\f_+)-\dabel(\f_-)}.
\end{equation}
This is the height difference between~$\f_+$ and~$\f_-$ for the periodic dimer cover $M_{\ell_{\bl_{\dR}}}$ of $\graphpl$ with reference dimer cover $M_0$. We note that
\begin{equation}\label{eq:difference_forms}
\omega_{\dabel(\f_+)-\dabel(\f_-)}=\omega_{\dabel(\f_+)+\bm D_{\bm\beta^n}-\dabel(\f_-)-\bm D_{\bm\beta^n}}
=\omega_{\dabel(\f_+)+\bm D_{\bm\beta^n}}-\omega_{\dabel(\f_-)+\bm D_{\bm\beta^n}}.
\end{equation}

Now we turn to the double contour integral in~\eqref{eq:height_change_chamber}. For~$\f\in F(\graphbetan)$, let
\begin{equation}\label{eq:form_double_integral_height}
\omega_{\f}
:=\frac{1}{(2\pi\i)^2}\frac{1}{E(\zeta,\eta)^2}\frac{\theta(-t+\bm D_{\bm \beta^n}+\eta-\zeta)}{\theta(-t+\bm D_{\bm \beta^n})}\frac{\theta(-t-\dabel(\f)+\zeta-\eta)}{\theta(t+\dabel(\f))}
\frac{\prod_{\alpha\in \bm \alpha}E(\alpha,\eta)^{\left(\dabel(\f)+\bm D_{\bm \beta^n}\right)_\alpha}}{\prod_{\alpha\in \bm \alpha}E(\alpha,\zeta)^{\left(\dabel(\f)+\bm D_{\bm \beta^n}\right)_\alpha}},
\end{equation}
so that the double integral may be written as
\[
\iint_{C(I_{\wh_{\dR}})\prec C(I_{\bl_{\dR}})} (\omega_{\f_+} - \omega_{\f_-}).
\]

For~$n=1,2,\dots$, let~$\dR_n\subset \graphbetan$ be a chamber and assume that
~$\sign_{\dR_n}(e)$ is independent of~$n$ if~$n$ is large enough. Let~$\f_n=\f_0+(j_n,k_n)\in F(\dR_n)$ with~$(j_n,k_n)=n(x,y)+\Ordo(1)$ as~$n\to\infty$, where~$(x,y)$ satisfy {Assumption~\ref{ass:limit}}. Let~$\zero=\zero(x,y)$ be defined by~\eqref{eq:zeta_*}.

Comparing~\eqref{eq:form_double_integral_height} with~\eqref{eq:form_double_integral} makes it clear that we can reuse the steepest descent analysis performed in Section~\ref{sec:steepest}. We deform~$C(I_{\wh})$ to~$C_{(x,y),\mathrm{asc}}$ and~$C(I_{\bl})$ to~$C_{(x,y),\mathrm{desc}}$, where the curves are the steepest ascent and descent curves as in Section~\ref{sec:steepest}, and~$\wh=\wh_{\dR_n}$,~$\bl=\bl_{\dR_n}$. Similarly to the proof of Theorem~\ref{thm:inverse_convergence}, Lemma~\ref{lem:descent_curves} implies that the only contribution from such deformation comes from the pole at~$\zeta=\eta$ along the curve~$C^{\zero}_{\ell_{\bl}}$. The pole, however, is now a double pole instead of a simple pole, which leads to a very different contribution along this curve:  
\begin{multline}\label{eq:form_limit}
\iint_{C(I_{\wh})\prec C(I_{\bl})}\omega_{\f_{n}}
=\frac{1}{2\pi\i}\int_{C_{\ell_{\bl}}^{\zero}}\omega_{\dabel(\f_{n})+\bm D_{\bm \beta^n}}
+\int_{C_{(x,y),\mathrm{asc}}}\int_{C_{(x,y),\mathrm{desc}}}\omega_{\f_{n}} \\
+\frac{1}{2\pi\i}\sum_{k=1}^g\left(\frac{\partial \log \theta}{\partial z_k}(-t+\bm D_{\bm \beta^n})-\frac{\partial \log \theta}{\partial z_k}(-t-\dabel(\f_{n}))\right)\int_{C_{\ell_{\bl}}^{\zero}}\omega_k.
\end{multline}
Indeed, the residue at the double pole is computed by differentiating the regular part of the integrand in~\eqref{eq:form_double_integral_height} and then setting~$\zeta=\eta$. The derivative of the factor involving the prime forms gives the first term in~\eqref{eq:form_limit} by~\eqref{eq:omega_D_def}, whereas the derivative of the theta-function factor gives the third term. By the periodicity of~$\theta(z)$ for~$z\in\RR^g$, the third term is bounded in~$n$. The second term on the right-hand side of~\eqref{eq:form_limit} is likewise bounded in~$n$, by standard steepest descent arguments. Therefore, as~$n\to\infty$, the main contribution to~\eqref{eq:form_limit} comes from the first term on the right-hand side.

Let~$\f_{n,+}=\f_0+(i_{n,+},j_{n,+})\in F(\dR_n)$ and~$\f_{n,-}=\f_0+(i_{n,-},j_{n,-})\in F(\dR_n)$ with~$(i_{n,+},j_{n,+})=n(x_+,y_+)+\Ordo(1)$ and~$(i_{n,-},j_{n,-})=n(x_-,y_-)+\Ordo(1)$ as $n\to\infty$, for some~$(x_+,y_+),(x_-,y_-)\in \domain^\circ$ satisfying {Assumption~\ref{ass:limit}}. {The assumption on the limiting points is a technical condition ensuring that we can deform the integration contours to the corresponding steepest descent and ascent curves; see Remark~\ref{rem:technical_lines}.} 
Let~$\zeta_+=\zero(x_+,y_+)$ be defined by~\eqref{eq:zeta_*}, and similarly~$\zeta_-=\zero(x_-,y_-)$. Then, by~\eqref{eq:height_change_chamber},~\eqref{eq:slope_chamber},~\eqref{eq:difference_forms}, and~\eqref{eq:form_limit},
\begin{multline}
\EE\left[h_M(\f_{n,+})-h_M(\f_{n,-})\right]
=\frac{1}{2\pi\i}\int_{C([\ell_{\bl},u_0])}\omega_{\dabel(\f_{n,+})-\dabel(\f_{n,-})}+\frac{1}{2\pi\i}\int_{C_{\ell_{\bl}}^{\zeta_+}}\omega_{\dabel(\f_{n,+})+\bm D_{\bm \beta^n}} \\
-\frac{1}{2\pi\i}\int_{C_{\ell_{\bl}}^{\zeta_-}}\omega_{\dabel(\f_{n,-})+\bm D_{\bm \beta^n}}
+\Ordo(1)
=\frac{1}{2\pi\i}\int_{C_{u_0}^{\zeta_+}}\omega_{\dabel(\f_{n,+})+\bm D_{\bm \beta^n}} -\frac{1}{2\pi\i}\int_{C_{u_0}^{\zeta_-}}\omega_{\dabel(\f_{n,-})+\bm D_{\bm \beta^n}}+\Ordo(1),
\end{multline}
as~$n\to \infty$. By Lemma~\ref{lem:divisor_coefficients_asymptotic} we conclude that, as~$n\to \infty$,  
\begin{equation}\label{eq:height_limit}
\EE\left[h_M(\f_{n,+})-h_M(\f_{n,-})\right]
=-n\frac{1}{2\pi\i}\int_{C_{u_0}^{\zeta_+}}\lambda_c(x_+,y_+) + n\frac{1}{2\pi\i}\int_{C_{u_0}^{\zeta_-}}\lambda_c(x_-,y_-)+\Ordo(1).
\end{equation}
Thus, the height difference is consistent with~\eqref{eq:height_fn} when~$(x_+,y_+)$ and~$(x_-,y_-)$ lie in the same chamber. To treat the general case, we connect the two points by a curve and decompose the height difference into a sum of contributions from the chambers crossed by the curve. More precisely, we split the curve into long segments lying inside single chambers and short segments crossing the separating lines between chambers. The long segments are handled by~\eqref{eq:height_limit}, and we will show that the short segments contribute negligibly.

Let~$(x,y)$ {satisfy Assumption~\ref{ass:limit}} 
and let~$\f_n=\f_0+(j_n,k_n)\in F(\graphbetan)$ be such that~$(j_n,k_n)=n(x,y)+\Ordo(1)$ as~$n\to\infty$. Pick a curve~$\Gamma$ in~$\domain$ from~$(0,0)$ to~$(x,y)$, intersecting the lines~$\line_e$,~$e\in E(N)$, at a finite number of locations~$(x_k,y_k)$,~$k=1,\dots,m-1$. We assume that these intersections are transverse, that each point~$(x_k,y_k)$ lies on exactly one line~$\line_e$ with~$e\in E(N)$, and that the points~$(x_k,y_k)$ satisfy Assumption~\ref{ass:limit}. The last assumption is possible because the rough region is locally convex, so~$\arctic$ intersects each line~$\line_e$ only finitely many times.

If~$(x_k,y_k)\in \line_e$, let~$\cR_{k}$ and~$\cR_{k+1}$ be the 2-cells adjacent to the part of~$\line_e$ containing~$(x_k,y_k)$, so that~$\Gamma$ enters~$(x_k,y_k)$ from~$\cR_{k}$ and exits into~$\cR_{k+1}$. For large enough~$n$, Lemma~\ref{lem:asymp_intervals} implies that, for~$k=1,\dots,{m}$, there is a chamber~$\dR_{n,k}$ in~$\graphbetan$ such that
\begin{equation}
\sign_{\dR_{n,k}}(e)=\sign_{\cR_{k}}(e),
\end{equation}
for all~$e\in E(N)$. Let 
\begin{equation}
\f_{n,k,+}=\f_0+(i_{n,k,+},j_{n,k,+})\in F(\dR_{n,k})\quad \text{and} \quad \f_{n,k,-}=\f_0+(i_{n,k,-},j_{n,k,-})\in F(\dR_{n,k+1}),
\end{equation}
where 
\begin{equation}\label{eq:limit_faces_at_lines}
(i_{n,k,+},j_{n,k,+})=n(x_k,y_k)+\Ordo(1), \quad  (i_{n,k,-},j_{n,k,-})=n(x_k,y_k)+\Ordo(1) \quad \text{as} \quad n\to \infty.
\end{equation}
We also set~$\f_{n,0,+}=\f_0$,~$\f_{n,m,-}=\f_n$,
\begin{equation}
\f_{n,0,-}=\f_0+(i_{n,0,-},j_{n,0,-})\in F(\dR_{n,1})\quad \text{with} \quad (i_{n,0,-},j_{n,0,-})=\Ordo(1) \quad \text{as} \quad n\to \infty,
\end{equation}
and
\begin{equation}
\f_{n,m,+}=\f_0+(i_{n,m,+},j_{n,m,+})\in F(\dR_{n,m}) \quad \text{with} \quad (i_{n,m,+},j_{n,m,+})=n(x,y)+\Ordo(1) \quad \text{as} \quad n\to \infty.
\end{equation}
For~$k=1,\dots,m$, we define the dual path~$\gamma_{n,k,\mathrm{l}}^*$ contained in~$\dR_{n,k}$ going from~$\f_{n,k-1,-}$ to~$\f_{n,k,+}$, and for~$k=0,\dots,m$, we define the dual path~$\gamma_{n,k,\mathrm{s}}^*$ going from~$\f_{n,k,+}$ to~$\f_{n,k,-}$. {Here~$\mathrm l$ and~$\mathrm s$ stand for long and short respectively.} {The concatenation}
\[
\begin{tikzcd}[
  column sep=1.2em,
  row sep=1.4em,
  cells={nodes={inner sep=1pt}}
]
\f_0=\f_{n,0,+}
  \arrow[rr, "\gamma_{n,0,\mathrm{s}}^*"]
  \arrow[dr]
&& \f_{n,0,-}
  \arrow[rr, "\gamma_{n,1,\mathrm{l}}^*"]
  \arrow[dl]
&& \f_{n,1,+}
  \arrow[rr, "\gamma_{n,1,\mathrm{s}}^*"]
  \arrow[dr]
&& \f_{n,1,-}
  \arrow[rr, "\gamma_{n,2,\mathrm{l}}^*"]
  \arrow[dl]
&& \cdots
  \arrow[rr, "\gamma_{n,m,\mathrm{l}}^*"]
&& \f_{n,m,+}
  \arrow[rr, "\gamma_{n,m,\mathrm{s}}^*"]
  \arrow[dr]
&& \f_{n,m,-}=\f_n
  \arrow[dl]
\\
& (0,0)
&&&& (x_1,y_1)
&&& \cdots 
&&& (x,y)
\end{tikzcd}
\]
{of these paths defines a dual path $\gamma_n^*$ from~$\f_0$ to~$\f_n$, where the downward arrows indicate the points to which the faces converge in the scaling limit.}

We first bound the part of~\eqref{eq:height_sum} coming from the curves~$\gamma_{n,k,\mathrm{s}}^*$. By~\eqref{eq:limit_faces_at_lines}, we may pick~$\gamma_{n,k,\mathrm{s}}^*$ so that its length is uniformly bounded for all~$n$. We obtain
\begin{equation}
\sum_{\e\in E(\graphbetan)} (\e \wedge \gamma_{n,k,\mathrm{s}}^*)\EE[\one_{\e\in M}-\one_{\e\in M_0}]=\Ordo(1)
\end{equation}
as~$n\to \infty$, for~$k=0,\dots,m$.

We continue by summing over the curves~$\gamma_{n,k,\mathrm{l}}^*$. For~$k=1,\dots,m-1$, let~$\zeta_k=\zero(x_k,y_k)$ be given by \eqref{eq:zeta_*}. Let~${\zeta_0}\in A_0\backslash \zz$ be any point in the component corresponding to~$u_0$. Then, by~\eqref{eq:height_limit},
\begin{align*}
\frac{1}{n}\sum_{k=1}^m\EE\left[h_M(\f_{n,k,+})-h_M(\f_{n,k-1,-})\right] &
=\sum_{k=1}^m\left(-\frac{1}{2\pi\i}\int_{C_{u_0}^{\zeta_{k}}}\lambda_c(x_k,y_k) + \frac{1}{2\pi\i}\int_{C_{u_0}^{\zeta_{k-1}}}\lambda_c(x_{k-1},y_{k-1})+\Ordo(\frac 1 n)\right)\\
&=\frac{1}{2\pi\i}\int_{C_{u_0}^{\zeta_{0}}}\lambda_c(0,0)-\frac{1}{2\pi\i}\int_{C_{u_0}^{\zero}}\lambda_c(x,y)+\Ordo\left(\frac{1}{n}\right),
\end{align*}
as~$n\to \infty$. Since the curve~$C_{u_0}^{\zeta_{0}}$ does not enclose any angles, the above equality proves Theorem~\ref{thm:limit_shape}.
\end{proof}

Using the formulation~\eqref{eq:action_zw} for~$\lambda_c(x,y)$, we see that the slope of~$\ls$ is
\begin{equation}
\nabla \ls(x,y)=\left(-\frac{1}{2\pi\i}\int_{C_{u_0}^{\zero}} d\log z , -\frac{1}{2\pi\i}\int_{C_{u_0}^{\zero}}d \log w\right).
\end{equation}
The right-hand side agrees with the slope~\eqref{eq:slope_integral} of the Gibbs measure~$\PP^{\zero}$. We are led to the following corollary of Theorems~\ref{thm:inverse_convergence} and~\ref{thm:limit_shape}.

\begin{corollary}\label{cor:local_fluctuations}
{Assume that the angle function is periodic, and}
{suppose $(x,y) \in \domain^\circ$ {lies in a phase region and} satisfies Assumption~\ref{ass:limit}.} The local correlations in a neighborhood of~$(x,y)$ of the dimer model on~$\graphbetan$ defined by the Kasteleyn matrix \eqref{eq:kast_fock} converge to those of the Gibbs measure defined by the corresponding whole-plane Kasteleyn matrix with slope~$\nabla \ls(x,y)$.
\end{corollary}


\begin{remark}\label{rem:fluct}
In the case of the Aztec diamond, for~$(x,y)\in \mathcal G_j$,~$j=1,\dots,g$, the part of the integral~\eqref{eq:height_fn} coming from the third term in expression~\eqref{eq:action_zw} for~$\lambda_c(x,y)$, {namely}
\begin{equation}
-\frac{1}{2\pi\i}\int_{B_j} \omega_{\bm D_c}, 
\end{equation}
with~$\bm D_c$ as in \eqref{eq:div_c}, turned out to be important in the description of the global height fluctuations~\cite[Equation (6)]{BN25}. It is therefore natural to expect that the same quantity will play a similar role in the global height fluctuations of AZ graphs.
\end{remark}

\appendix

\section{Tropical limits and simulations}\label{appendix:tropical}

In this appendix, we explain how certain astroidal domains arise as local pieces of the tropical limit of the Aztec diamond with periodic weights. This allows us to use domino shuffling on the Aztec diamond~\cite{Pro03} to simulate limit shapes in more general astroidal domains.

Let $\graphpl$ be a minimal periodic bipartite graph and let $p: \graphpl \to \graphtor$ denote the quotient map to the corresponding graph $\graphtor$ on the torus. There is a construction, called the \emph{spectral transform}~\cite{KO06}, which allows one to view any periodic choice of edge weights on $\graphpl$ (equivalently, edge weights on $\graphtor$) as gauge-equivalent to Fock's dimer model for an M-curve \(\curve\), \(t \in \jac(\curve)\), and angle function \(\nu\) \cite{Foc15, BCT22}. We apply this to the family of edge weights on $\graphtor$ of the form
\[
\wt(\e;\tau) = w(\e) \exp(\tau \mathcal E(\e)),
\]
where 
\[
w: E(\graphtor) \to \RR_{>0}, \qquad  \mathcal E:E(\graphtor) \to \RR,  \qquad \tau \in \RR_{>0}.
\]
The corresponding M-curve \(\curve_\tau\) is an explicit algebraic curve given by the vanishing locus of the \emph{characteristic polynomial}
\[
\mathsf P(z,w;\tau) := \sum_{(j,k) \in \ZZ^2} \pm \left(\sum_{\text{dimer covers $M$ on $\graphtor$}: (h^y_M,-h^x_M) = (j,k)} \left( \prod_{\e \in M} \wt(\e;\tau) \right) z^j w^k \right),
\]
which is a Laurent polynomial in the variables \(z,w\) with Newton polygon equal to \(N\)~\cite{KOS06}. Here \(\pm\) denotes a sign that depends only on \((j,k) \in \ZZ^2\) and $h_M^x$ and $h_M^y$ are the height changes of $M$ along dual paths with homology classes $(1,0)$ and $(0,1)$ respectively.

The limit as \(\tau \to \infty\) is called the tropical limit; see~\cite{IMStropical, MStropical} for a general introduction to tropical geometry. The \emph{tropical characteristic polynomial} is defined as
\begin{equation}\label{eq:trop_char_poly}
    \mathsf P^t(x,y) := \max_{(j,k) \in \ZZ^2} \left(\max_{\text{dimer covers $M$ on $\graphtor$}: (h^y_M,-h^x_M) = (j,k)} \left(\sum_{\e \in M}\mathcal E(\e) \right) +jx+ky \right).
\end{equation}

The fundamental theorem of tropical geometry~\cite{MKLtropical} states that, as \(\tau \to \infty\), the rescaled \emph{amoeba}
\[
\frac{1}{\tau} \{(\log|z|,\log|w|) \in \RR^2 : \mathsf P(z,w;\tau) = 0 \}
\]
converges in the Hausdorff metric to the \emph{tropical spectral curve}
\[
\curve^t := \{(x,y) \in \RR^2 : \text{the maximum over $(j,k) \in \ZZ^2$ in~\eqref{eq:trop_char_poly} is achieved at least twice}\}.
\]

The graph of the tropical characteristic polynomial induces a regular subdivision of the Newton polygon \(N\) which is dual to \(\curve^t\) (see~Figure~\ref{fig:trop_curve_subdivision}). The tropical curve is called \emph{smooth} if all \(2\)-cells of the subdivision are unimodular triangles. We illustrate, by means of an example, how the arctic curve can be computed in the tropical limit by a generalization of the results of~\cite{BBtropical}. Their setting was for smooth tropical curves, whereas here we consider a non-smooth tropical curve.

\begin{figure}[t]
\centering

\begin{subfigure}[c]{0.48\textwidth}
\centering
\begin{tikzpicture}[scale=1.2, baseline=(current bounding box.center),
  elab/.style={midway, sloped, font=\scriptsize, inner sep=1.2pt, text opacity=1},
]
     \fill[gray!30]
    (0,0) -- (6,0) -- (6,4) -- (0,4)-- cycle;
       \coordinate[bvert] (b11) at (1,4) {};
       \coordinate[bvert] (b12) at (3,4) {};
       \coordinate[bvert] (b13) at (5,4) {};

       \coordinate[bvert] (b21) at (1,2) {};
       \coordinate[bvert] (b22) at (3,2) {};
       \coordinate[bvert] (b23) at (5,2) {};

       \coordinate[bvert] (b31) at (1,0) {};
       \coordinate[bvert] (b32) at (3,0) {};
       \coordinate[bvert] (b33) at (5,0) {};

       \coordinate[wvert] (w11) at (0,3) {};
       \coordinate[wvert] (w12) at (2,3) {};
       \coordinate[wvert] (w13) at (4,3) {};
       \coordinate[wvert] (w14) at (6,3) {};

       \coordinate[wvert] (w21) at (0,1) {};
       \coordinate[wvert] (w22) at (2,1) {};
       \coordinate[wvert] (w23) at (4,1) {};
       \coordinate[wvert] (w24) at (6,1) {};

\draw [] (w11) edge node[elab, above]{$\exp({-\tau})$} (b11) edge node[elab, above]{$\gamma-\beta$} (b21);

\draw [] (w12) edge node[elab, above]{$\epsilon-\gamma$} (b11) edge node[elab, above]{$\beta-\epsilon$} (b21) edge node[elab, above]{$\gamma-\beta$} (b12) edge node[elab, above]{$\exp({-4 \tau})$} (b22);

\draw [] (w13) edge node[elab, above]{$\delta-\gamma$} (b12) edge node[elab, above]{$\beta-\delta$} (b22) edge node[elab, above]{$\gamma-\beta$} (b13) edge node[elab, above]{$\exp({- \tau})$} (b23);

\draw [] (w14) edge node[elab, above]{$\beta-\gamma$} (b13) edge node[elab, above]{$\exp({-4 \tau})$} (b23);

\draw [] (w21) edge node[elab, above]{$\delta-\gamma$} (b21) edge node[elab, above]{$\gamma-\alpha$} (b31);

\draw [] (w22) edge node[elab, above]{$\epsilon-\delta$} (b21) edge node[elab, above]{$\alpha-\epsilon$} (b31) edge node[elab, above]{$\delta-\beta$} (b22) edge node[elab, above]{$\beta-\alpha$} (b32);

\draw [] (w23) edge node[elab, above]{$\exp({-\tau})$} (b22) edge node[elab, above]{$\alpha-\delta$} (b32) edge node[elab, above]{$\delta-\alpha$} (b23) edge node[elab, above]{$\exp({- 4\tau})$} (b33);

\draw [] (w24) edge node[elab, above]{$\alpha-\delta$} (b23) edge node[elab, above]{$\exp({- \tau})$} (b33);
\end{tikzpicture}
\end{subfigure}
\hfill
\begin{subfigure}[c]{0.48\textwidth}
\centering
\begin{tikzpicture}[scale=1.2, baseline=(current bounding box.center),
  elab/.style={midway, sloped, font=\scriptsize, inner sep=1.2pt, text opacity=1},
]
     \fill[gray!30]
    (0,0) -- (6,0) -- (6,4) -- (0,4)-- cycle;
       \coordinate[bvert] (b11) at (1,4) {};
       \coordinate[bvert] (b12) at (3,4) {};
       \coordinate[bvert] (b13) at (5,4) {};

       \coordinate[bvert] (b21) at (1,2) {};
       \coordinate[bvert] (b22) at (3,2) {};
       \coordinate[bvert] (b23) at (5,2) {};

       \coordinate[bvert] (b31) at (1,0) {};
       \coordinate[bvert] (b32) at (3,0) {};
       \coordinate[bvert] (b33) at (5,0) {};

       \coordinate[wvert] (w11) at (0,3) {};
       \coordinate[wvert] (w12) at (2,3) {};
       \coordinate[wvert] (w13) at (4,3) {};
       \coordinate[wvert] (w14) at (6,3) {};

       \coordinate[wvert] (w21) at (0,1) {};
       \coordinate[wvert] (w22) at (2,1) {};
       \coordinate[wvert] (w23) at (4,1) {};
       \coordinate[wvert] (w24) at (6,1) {};

\draw [] (w11) edge node[elab, above]{$\gamma-\beta$} (b21);

\draw [] (w12) edge node[elab, above]{$\epsilon-\gamma$} (b11) edge node[elab, above]{$\beta-\epsilon$} (b21) edge node[elab, above]{$\gamma-\beta$} (b12);

\draw [] (w13) edge node[elab, above]{$\delta-\gamma$} (b12) edge node[elab, above]{$\beta-\delta$} (b22) edge node[elab, above]{$\gamma-\beta$} (b13);

\draw [] (w14) edge node[elab, above]{$\beta-\gamma$} (b13);

\draw [] (w21) edge node[elab, above]{$\delta-\gamma$} (b21) edge node[elab, above]{$\gamma-\alpha$} (b31);

\draw [] (w22) edge node[elab, above]{$\epsilon-\delta$} (b21) edge node[elab, above]{$\alpha-\epsilon$} (b31) edge node[elab, above]{$\delta-\beta$} (b22) edge node[elab, above]{$\beta-\alpha$} (b32);

\draw [] (w23) edge node[elab, above]{$\alpha-\delta$} (b32) edge node[elab, above]{$\delta-\alpha$} (b23);

\draw [] (w24) edge node[elab, above]{$\alpha-\delta$} (b23);
\end{tikzpicture}
\end{subfigure}

\caption{A \((2\times 3)\)-fundamental domain of the square lattice with a family of edge weights (left), and the corresponding limiting graph obtained by deleting the edges whose weights vanish in the tropical limit (right). {This particular choice of weights was obtained via the tropical inverse spectral transform developed in the forthcoming work~\cite{GG2}.}
}
\label{fig:2x3_square lattice}
\end{figure}

Consider the \((2\times 3)\)-fundamental domain of the square lattice with the family of edge weights shown in Figure~\ref{fig:2x3_square lattice}(left). The Newton polygon \(N\) and the function \(c:E(N)\to \RR\) for the scaling limit of the Aztec diamond, are given by
\begin{center}
\begin{tikzpicture}[scale=0.7, baseline=(current bounding box.center)]
    \coordinate (p1) at (0,0);
    \coordinate (p2) at (2,0);
    \coordinate (p3) at (2,3);
    \coordinate (p4) at (0,3);

    \node[bvert] at (p1) {};
    \node[bvert] at (p2) {};
    \node[bvert] at (p3) {};
    \node[bvert] at (p4) {};
    \node[bvert] at (0,1) {};
    \node[bvert] at (1,1) {};
    \node[bvert] at (1,0) {};
    \node[bvert] at (1,3) {};
    \node[bvert] at (0,2) {};
    \node[bvert] at (1,2) {};
    \node[bvert] at (2,2) {};
    \node[bvert] at (2,1) {};
      
    \draw[] (p1) -- node[below]{$0$} (p2);
    \draw[] (p3) -- node[above]{$\frac 1 2$} (p4);
    \draw[] (p2) -- node[right]{$0$} (p3);
    \draw[] (p4) -- node[left]{$-\frac 1 3$} (p1);
\end{tikzpicture}
\end{center}

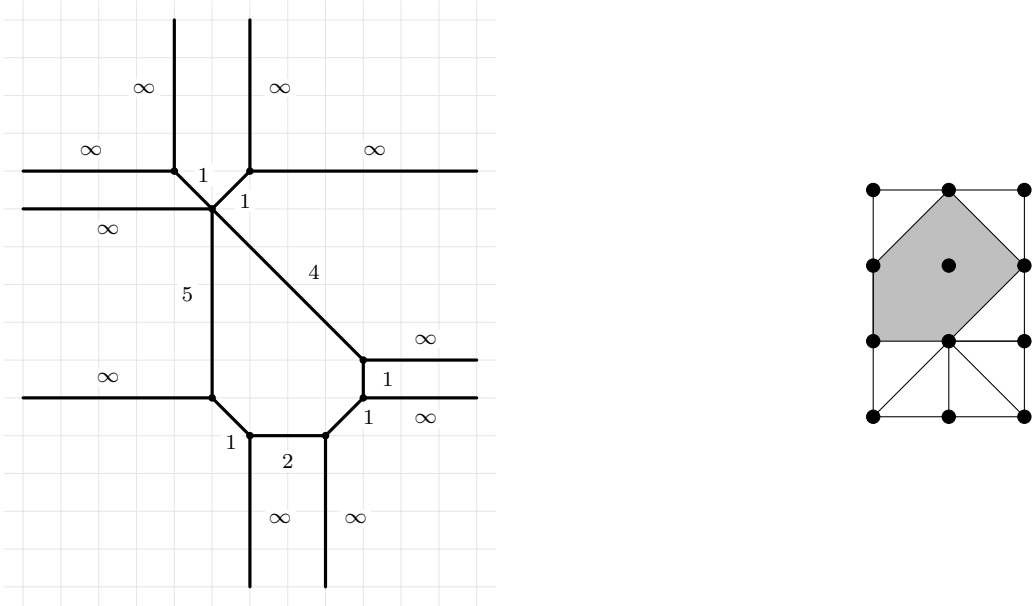
\begin{figure}
\centering

\begin{subfigure}[c]{0.58\textwidth}
\centering
\begin{tikzpicture}[
    scale=0.5,
    baseline=(current bounding box.center),
    line cap=round,
    line join=round,
    edgelab/.style={fill=white, inner sep=1.5pt, font=\scriptsize},
    vtx/.style={circle, fill=black, inner sep=1 pt}
]

    \draw[step=1, very thin, gray!20] (-5.5,-10.5) grid (7.5,5.5);

    \coordinate (O)  at (0,0);
    \coordinate (A)  at (-1,1);
    \coordinate (B)  at (1,1);
    \coordinate (C)  at (0,-5);
    \coordinate (D)  at (4,-4);
    \coordinate (L)  at (-5,0);
    \coordinate (AL) at (-5,1);
    \coordinate (AU) at (-1,5);
    \coordinate (BR) at (7,1);
    \coordinate (BU) at (1,5);
    \coordinate (CL) at (-5,-5);
    \coordinate (E)  at (1,-6);
    \coordinate (ED) at (1,-10);
    \coordinate (F)  at (3,-6);
    \coordinate (FD) at (3,-10);
    \coordinate (G)  at (4,-5);
    \coordinate (GR) at (7,-5);
    \coordinate (DR) at (7,-4);

    \draw[black, very thick]
        (O) -- (A) node[pos=.55, above right=0.5pt, edgelab] {$1$}
        -- (AL) node[pos=.55, above=4pt, edgelab] {$\infty$};

    \draw[black, very thick]
        (A) -- (AU) node[pos=.55, left=5pt, edgelab] {$\infty$};

    \draw[black, very thick]
        (O) -- (B) node[pos=.55, below right=0.5pt, edgelab] {$1$}
        -- (BR) node[pos=.55, above=4pt, edgelab] {$\infty$};

    \draw[black, very thick]
        (B) -- (BU) node[pos=.55, right=5pt, edgelab] {$\infty$};

    \draw[black, very thick]
        (O) -- (L) node[pos=.55, below=4pt, edgelab] {$\infty$};

    \draw[black, very thick]
        (O) -- (C) node[pos=.45, left=5pt, edgelab] {$5$}
        -- (CL) node[pos=.55, above=4pt, edgelab] {$\infty$};

    \draw[black, very thick]
        (C) -- (E) node[pos=.5, below=5pt, edgelab] {$1$}
        -- (ED) node[pos=.55, right=5pt, edgelab] {$\infty$};

    \draw[black, very thick]
        (E) -- (F) node[pos=.5, below=5pt, edgelab] {$2$}
        -- (FD) node[pos=.55, right=5pt, edgelab] {$\infty$};

    \draw[black, very thick]
        (F) -- (G) node[pos=.5, right=5pt, edgelab] {$1$}
        -- (GR) node[pos=.55, below=4pt, edgelab] {$\infty$};

    \draw[black, very thick]
        (G) -- (D) node[pos=.5, right=5pt, edgelab] {$1$}
        -- (DR) node[pos=.55, above=4pt, edgelab] {$\infty$};

    \draw[black, very thick]
        (O) -- (D) node[pos=.55, above right=4pt, edgelab] {$4$};

    \foreach \P in {O,A,B,C,D,E,F,G}
        \node[vtx] at (\P) {};

\end{tikzpicture}
\end{subfigure}
\hfill
\begin{subfigure}[c]{0.30\textwidth}
\centering
\begin{tikzpicture}[scale=1,  baseline=(current bounding box.center)]
    \coordinate (p1) at (0,0);
    \coordinate (p2) at (2,0);
    \coordinate (p3) at (2,3);
    \coordinate (p4) at (0,3);

    \coordinate (a) at (0,1);
    \coordinate (b) at (0,2);
    \coordinate (c) at (1,0);
    \coordinate (d) at (1,1);
    \coordinate (e) at (1,2);
    \coordinate (f) at (1,3);
    \coordinate (g) at (2,1);
    \coordinate (h) at (2,2);

    \fill[gray!50] (a) -- (d) -- (h) -- (f) -- (b) -- cycle;
    \draw[] (a) -- (d) -- (h) -- (f) -- (b) -- cycle;

    \draw[] (p1) -- (p2) -- (p3) -- (p4) -- cycle;

    \draw[] (d) edge (g) edge (p1) edge (c) edge (p2) edge (g);

    \foreach \P in {p1,p2,p3,p4,a,b,c,d,e,f,g,h}
        \node[bvert] at (\P) {};
\end{tikzpicture}
\end{subfigure}

\caption{The tropical spectral curve with the integral lengths of its edges (left), and the dual regular subdivision of the Newton polygon with a shaded pentagonal $2$-cell (right).}
\label{fig:trop_curve_subdivision}
\end{figure}

The tropical curve for this family of edge weights is shown in Figure~\ref{fig:trop_curve_subdivision}(left).

\begin{figure}
\centering

\begin{subfigure}[c]{0.58\textwidth}
\centering
\begin{tikzpicture}[
    scale=0.5,
    baseline=(current bounding box.center),
    line cap=round,
    line join=round,
    edgelab/.style={fill=white, inner sep=1.5pt, font=\scriptsize},
    vtx/.style={circle, fill=black, inner sep=1pt},
    intedge/.style={black, very thick, -{Stealth[length=4pt,width=5pt]}},
    bdryedge/.style={black, very thick, -{Stealth[length=4pt,width=5pt]}},
    every label/.style={font=\scriptsize}
]

    \draw[step=1, very thin, gray!20] (-5.5,-10.5) grid (7.5,5.5);

    \coordinate[label=-135:{$\scriptsize O$}] (O)  at (0,0);
    \coordinate[label=135:{$\scriptsize A$}] (A)  at (-1,1);
    \coordinate[label=45:{$\scriptsize B$}] (B)  at (1,1);
    \coordinate[label=135:{$\scriptsize C$}] (C)  at (0,-5);
    \coordinate[label=90:{$\scriptsize D$}] (D)  at (4,-4);
    \coordinate (L)  at (-5,0);
    \coordinate (AL) at (-5,1);
    \coordinate (AU) at (-1,5);
    \coordinate (BR) at (7,1);
    \coordinate (BU) at (1,5);
    \coordinate (CL) at (-5,-5);
    \coordinate[label=90:{$\scriptsize{E}$}] (E)  at (1,-6);
    \coordinate (ED) at (1,-10);
    \coordinate[label=90:{$\scriptsize{F}$}] (F)  at (3,-6);
    \coordinate (FD) at (3,-10);
    \coordinate[label=45:{$\scriptsize{G}$}] (G)  at (4,-5);
    \coordinate (GR) at (7,-5);
    \coordinate (DR) at (7,-4);

    \draw[bdryedge]
        (A) -- (AL) node[pos=.55, above=4pt, edgelab] {$-\frac13$};
    \draw[bdryedge]
        (A) -- (AU) node[pos=.55, left=5pt, edgelab] {$\frac12$};
    \draw[bdryedge]
        (B) -- (BR) node[pos=.55, above=4pt, edgelab] {$0$};
    \draw[bdryedge]
        (B) -- (BU) node[pos=.55, right=5pt, edgelab] {$\frac12$};
    \draw[bdryedge]
        (O) -- (L) node[pos=.55, below=4pt, edgelab] {$-\frac13$};
    \draw[bdryedge]
        (C) -- (CL) node[pos=.55, above=4pt, edgelab] {$-\frac13$};
    \draw[bdryedge]
        (E) -- (ED) node[pos=.55, right=5pt, edgelab] {$0$};
    \draw[bdryedge]
        (F) -- (FD) node[pos=.55, right=5pt, edgelab] {$0$};
    \draw[bdryedge]
        (G) -- (GR) node[pos=.55, below=4pt, edgelab] {$0$};
    \draw[bdryedge]
        (D) -- (DR) node[pos=.55, above=4pt, edgelab] {$0$};

    \draw[intedge]
        (O) -- (A) node[pos=.52, above right=1pt, edgelab] {$\frac16$};
    \draw[intedge]
        (O) -- (B) node[pos=.52, below right=1pt, edgelab] {$\frac12$};
    \draw[intedge]
        (C) -- (O) node[pos=.45, left=6pt, edgelab] {$\frac{3}{14}$};
    \draw[intedge]
        (D) -- (O) node[pos=.53, above right=4pt, edgelab] {$\frac{5}{42}$};
    \draw[intedge]
        (C) -- (E) node[pos=.5, below left=5pt, edgelab] {$\frac{5}{42}$};
    \draw[intedge]
        (E) -- (F) node[pos=.5, below=5pt, edgelab] {$\frac{5}{42}$};
    \draw[intedge]
        (F) -- (G) node[pos=.5, below right=5pt, edgelab] {$\frac{5}{42}$};
    \draw[intedge]
        (G) -- (D) node[pos=.5, left=5pt, edgelab] {$\frac{5}{42}$};

    \foreach \P in {O,A,B,C,D,E,F,G}
        \node[vtx] at (\P) {};
\end{tikzpicture}
\end{subfigure}
\hfill
\begin{subfigure}[c]{0.34\textwidth}
\centering
\begin{tikzpicture}[
    scale=0.28,
    baseline=(current bounding box.center),
    line cap=round,
    line join=round,
    cell/.style={fill=gray!50, draw=black},
    lab/.style={font=\scriptsize,  inner sep=1pt}
]

\coordinate (At) at (-14,-21);
\coordinate (Bt) at (0,-21);
\coordinate (Ct) at (-14,-9);
\coordinate (Dt) at (0,-5);
\coordinate (Et) at (-5,0);

\draw (-14,-21) -- (0,-21) -- (0,0) -- (-14,0) -- cycle;

\coordinate (P1) at (-14,-21);
\coordinate (P2) at (-7,-14);
\coordinate (P3) at (-8,-13);
\coordinate (P4) at (-4,-9);
\coordinate (P5) at (-14,-9);

\draw(P2) -- (Bt);
\draw (P3) -- (Dt);

\draw[cell] (P1) -- (P2) -- (P3) -- (P4) -- (P5) -- cycle;

\draw (Ct) -- (Et);
\draw (Et) -- (Dt);

\node[lab, below left=2pt]  at (At) {$A$};
\node[lab, below right=2pt] at (Bt) {$B$};
\node[lab, left=2pt]        at (Ct) {$C$};
\node[lab, right=2pt]       at (Dt) {$D=G$};
\node[lab, above=2pt]       at (Et) {$E=F$};
\node[lab] at (-9.8,-12.2) {$O$};

\end{tikzpicture}
\end{subfigure}

\caption{The harmonic \(1\)-form on \(\curve^t\) (left), and the image of \((\Omega^{-1})^t\) in \(\mathcal D_c\) (right).}
\label{fig:curve_and_diamond}
\end{figure}
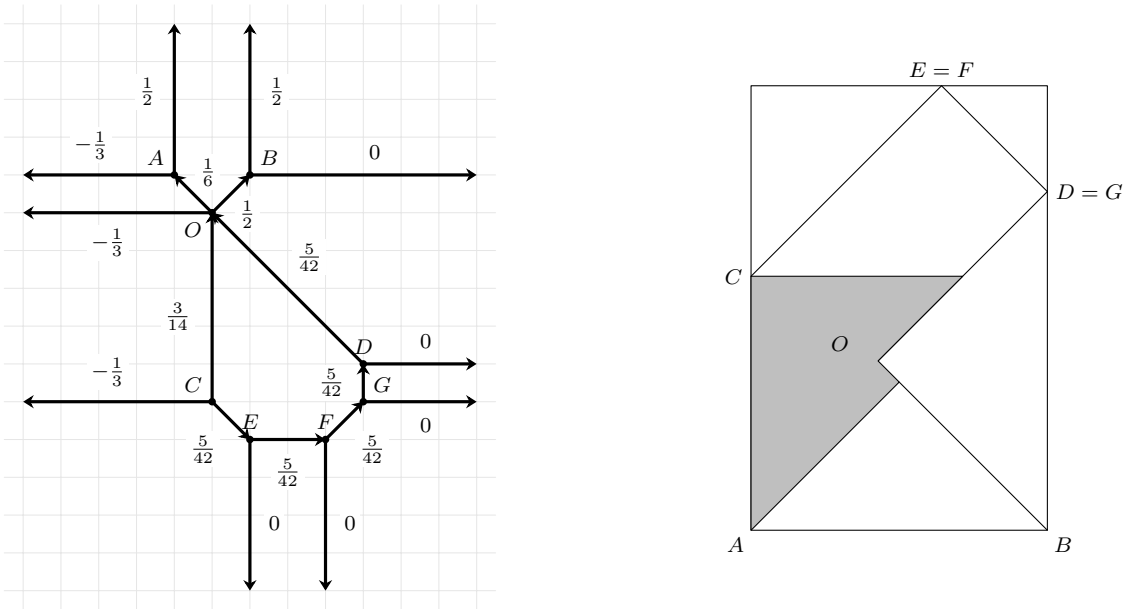

A \emph{discrete \(1\)-form} \(\omega\) on \(\curve^t\) is an antisymmetric function on the oriented edges of \(\curve^t\). It is called \emph{harmonic} if it is
\begin{itemize}
    \item \emph{closed}, \emph{i.e.} for every face \(f\) of \(\curve^t\),
    \[
    \sum_{e\subset \partial f} |e|_{\ZZ} \omega(e)=0,
    \]
    where the edges in \(\partial f\) are oriented counterclockwise, and
    \item \emph{co-closed}, \emph{i.e.} for every vertex \(u\in V(\curve^t)\),
    \[
    \sum_{e=uv} \omega(e)=0,
    \]
    where the sum is over all edges incident to \(u\), each oriented from \(u\) to \(v\).
\end{itemize}

Given any outflows on the unbounded edges of \(\curve^t\) that sum to \(0\), there is a unique harmonic \(1\)-form with those boundary values. Taking the boundary values given by the function $c$, we obtain the harmonic \(1\)-form shown in Figure~\ref{fig:curve_and_diamond}(left).

The map \(\Omega^{-1}\) tropicalizes, in the limit $\tau \to \infty$, to the map
\[
(\Omega^{-1})^t:\curve^t \to \mathcal D_c
\]
sending \(v \in V(\curve^t)\) to the point \((x,y) \in \RR^2\) such that 
\begin{equation}\label{eq:trop_1_form_xy}
    \frac{1}{|e|_\ZZ} \langle (x,y), e \rangle + \omega(e)=0
\end{equation}
for every edge \(e\) of \(\curve^t\) incident to \(v\), and sending non-boundary edges of \(\curve^t\) to the corresponding line segments between the images of their endpoints. Its image is shown in Figure~\ref{fig:curve_and_diamond}(right). Note that the degree-\(5\) vertex \(O\) maps to a pentagonal astroidal domain $\mathcal D_O$ bounded by the lines defined by~\eqref{eq:trop_1_form_xy}. The condition that $\omega$ is co-closed at $O$ is equivalent to the admissibility condition for $\mathcal D_O$.

Letting $\tau \to \infty$ and deleting edges with zero weight, we obtain the graph shown on the right-hand side of Figure~\ref{fig:2x3_square lattice}, which is a minimal graph for the pentagonal Newton polygon shown shaded in Figure~\ref{fig:trop_curve_subdivision}(right). Its M-curve $\curve_\infty$ is given by the vanishing of the \emph{initial polynomial}
\[
 \lim_{\tau \to \infty} \mathsf P(z,w;\tau).
\]
Our choice of weights is such that $\curve_\infty$ has genus $0$, and the weights $\wt(\e;\infty)$ are isoradial weights as in Section~\ref{sec:pentagon_example}. As in that section, the prime form is given by 
\[
E(\zeta,\eta)=\frac{\eta-\zeta}{\sqrt{d\zeta}\sqrt{d\eta}}.
\]
Thus, the functions in~\eqref{eq:imff} can be written explicitly and globally as
\[
f_1(\zeta) = -\sum_{\alpha \in \zz} b_{e (\alpha)} \log(\zeta-\alpha),\qquad  f_2(\zeta) = \sum_{\alpha \in \zz} a_{e (\alpha)} \log(\zeta-\alpha), \qquad f_3(\zeta) = \sum_{\alpha \in 
\zz} c{(e(\alpha))} \log(\zeta-\alpha).
\]
Choosing the angles 
\[
\alpha = 1,\qquad
\beta = e^{\pi \ii/25},\qquad
\gamma = e^{4\pi \ii/5},\qquad
\delta = e^{6\pi \ii/5},\qquad
\epsilon = e^{8\pi \ii/5},
\]
and using the parametrization~\eqref{eq:fsol}, we obtain the arctic curve in \(\mathcal D_O\) shown in Figure~\ref{fig:tropical_astroidal_arctic_curve}(right). For comparison, a simulation of the Aztec diamond is shown in Figure~\ref{fig:tropical_astroidal_arctic_curve}(left).

\section{Proof of Proposition~\ref{prop:telescopic}}\label{appendix:telescopic}

{The following proof is reproduced from~\cite{BBRdT26}.}
By the product form~\eqref{eq:g_prod} of \(g\), we have
\[
g_{\x}=g_{\f_+}\,g_{\f_+,\x},
\qquad
\x\in\{\bl,\wh,\f_-,\f_+\}.
\]
Canceling the \(g_{\f_+}\)-factors, it suffices to prove that
\begin{equation}\label{eq:fay_cons}
\begin{aligned}
E(\zeta,\eta){\frac{g_{\f_+,\wh}(\eta)}{g_{\f_+,\bl}(\zeta)}}\kast_{\wh,\bl}
&=
\frac{\theta(-t-\dabel(\f_+)+\zeta-\eta)}
     {\theta(t+\dabel(\f_+))}
\\
&\quad-
\frac{\theta(-t-\dabel(\f_-)+\zeta-\eta)}
     {\theta(t+\dabel(\f_-))}
\frac{E(\beta,\zeta)}{E(\alpha,\zeta)}
\frac{E(\alpha,\eta)}{E(\beta,\eta)}.
\end{aligned}
\end{equation}
By definition,
\begin{equation*}
\begin{aligned}
E(\zeta,\eta){\frac{g_{\f_+,\wh}(\eta)}{g_{\f_+,\bl}(\zeta)}}\kast_{\wh,\bl}
&=
\frac{
\theta(-t-\dabel(\bl)+\zeta)\,
\theta(t+\dabel(\wh)+\eta)\,
E(\alpha,\beta)\,
E(\zeta,\eta)}
{
\theta(t+\dabel(\f_+))\,
\theta(t+\dabel(\f_-))\,
E(\alpha,\zeta)\,
E(\beta,\eta)}.
\end{aligned}
\end{equation*}

We now apply Fay's trisecant identity~\cite{Fay73} (see also~\cite[Theorem~27]{BCT22}): for \(z\in\CC^g\) and \(a,b,c,d\in\curve\),
\begin{equation*}
\begin{aligned}
&\theta(z+c-a)\theta(z+d-b)E(c,b)E(a,d)
+\theta(z+c-b)\theta(z+d-a)E(c,a)E(d,b)
\\
&\qquad\qquad
=\theta(z+c+d-a-b)\theta(z)E(c,d)E(a,b).
\end{aligned}
\end{equation*}
Dividing by \(\theta(z)\theta(z+d-a)E(c,a)E(d,b)\), we obtain
\begin{equation*}
\begin{aligned}
&\frac{\theta(z+c-a)\theta(z+d-b)E(c,b)E(a,d)}
      {\theta(z)\theta(z+d-a)E(c,a)E(d,b)}
\\
&\qquad=
\frac{\theta(z+c+d-a-b)E(c,d)E(a,b)}
     {\theta(z+d-a)E(c,a)E(d,b)}
-\frac{\theta(z+c-b)}{\theta(z)}.
\end{aligned}
\end{equation*}
Now set
\[
z=-t-\dabel(\f_+),\qquad
a=\alpha,\qquad
b=\eta,\qquad
c=\zeta,\qquad
d=\beta.
\]
Using
\[
z-a=-t-\dabel(\bl),\qquad
z+d=-t-\dabel(\wh),\qquad
z+d-a=-t-\dabel(\f_-),
\]
we get
\begin{equation*}
\begin{aligned}
&\frac{
\theta(-t-\dabel(\bl)+\zeta)\,
\theta(-t-\dabel(\wh)-\eta)\,
E(\zeta,\eta)\,
E(\alpha,\beta)}
{
\theta(-t-\dabel(\f_+))\,
\theta(-t-\dabel(\f_-))\,
E(\zeta,\alpha)\,
E(\beta,\eta)}
\\
&\qquad=
\frac{
\theta(-t-\dabel(\f_-)+\zeta-\eta)\,
E(\zeta,\beta)\,
E(\alpha,\eta)}
{
\theta(-t-\dabel(\f_-))\,
E(\zeta,\alpha)\,
E(\beta,\eta)}
-
\frac{\theta(-t-\dabel(\f_+)+\zeta-\eta)}
     {\theta(-t-\dabel(\f_+))}.
\end{aligned}
\end{equation*}

Finally, using that the Riemann theta function is even and that the prime form is antisymmetric, the previous identity simplifies to~\eqref{eq:fay_cons}. This proves the proposition.
\qed{}

\bibliographystyle{plain}
\bibliography{bibliotek}

\end{document}